\documentclass[11pt]{amsart} 
\usepackage{amsmath} \usepackage{amsxtra}
\usepackage{amscd} \usepackage{amsthm}  
\usepackage{amsfonts}\usepackage{amssymb} 
\usepackage{eucal}
\makeindex

\newtheorem{Cor}[subsubsection]{Corollary}
\newtheorem{Lm}[subsubsection]{Lemma}
\newtheorem{Pp}[subsubsection]{Proposition}
\newtheorem{Con}[subsubsection]{Conjecture}
\newtheorem{Thm}[subsubsection]{Theorem}
\newtheorem{Def}[subsubsection]{Definition}
\newtheorem{Rem}[subsubsection]{Remark}

\theoremstyle{definition}

\theoremstyle{remark}
\newcommand{\nc}{\newcommand}
\nc{\renc}{\renewcommand}
\nc{\ssec}{\subsection}
\nc{\sssec}{\subsubsection}
\nc{\on}{\operatorname}

\nc\ol{\overline}
\nc\wt{\widetilde}
\nc\tboxtimes{\wt{\boxtimes}}
\newcommand{\cA}{{\mathcal A}}
\newcommand{\cB}{{\mathcal B}}
\newcommand{\cC}{{\mathcal C}}

\newcommand{\cH}{{\mathcal H}}
\newcommand{\cE}{{\mathcal E}}

\newcommand{\cJ}{{\mathcal J}}
\newcommand{\cO}{{\mathcal O}}
\newcommand{\cL}{{\mathcal L}}
\newcommand{\cM}{{\mathcal M}}
\newcommand{\cN}{{\mathcal N}}
\newcommand{\cF}{{\mathcal F}}
\newcommand{\cK}{{\mathcal K}}
\newcommand{\cP}{{\mathcal P}}

\newcommand{\cR}{{\mathcal R}}
\newcommand{\cS}{{\mathcal S}}
\newcommand{\cT}{{\mathcal T}}
\newcommand{\cU}{{\mathcal U}}
\newcommand{\cV}{{\mathcal V}}
\newcommand{\cW}{{\mathcal W}}
\newcommand{\cX}{{\mathcal X}}
\newcommand{\cY}{{\mathcal Y}}
\newcommand{\cZ}{{\mathcal Z}}

\newcommand{\FF}{{\mathbb F}}
\newcommand{\GG}{{\mathbb G}}

\newcommand{\ZZ}{{\mathbb Z}}
\newcommand{\QQ}{{\mathbb Q}}

\newcommand{\PP}{{\mathbb P}}
\newcommand{\RR}{{\mathbb R}}

\newcommand{\EE}{{\mathbb E}}

\newcommand{\VV}{{\mathbb V}}

\renewcommand{\gg}{\mathfrak{g}}  %gotic
\newcommand{\gm}{\mathfrak{m}}    

\newcommand{\gp}{\mathfrak{p}}
\newcommand{\gq}{\mathfrak{q}}
\newcommand{\gu}{\mathfrak{u}}

\newcommand{\gt}{\mathfrak{t}}

\newcommand{\gU}{\mathfrak{U}}

\nc{\gM}{\mathfrak{M}}

\nc{\uZ}{\underline{\cZ}}

\newcommand{\Rep}{{\on{Rep}}}
\newcommand{\Sch}{{\on{Sch}}}

\newcommand{\Qlb}{\mathbb{\bar Q}_\ell}
\newcommand{\Gm}{\mathbb{G}_m}

\newcommand{\A}{\mathbb{A}}

\newcommand{\toup}[1]{\stackrel{#1}{\to}}

\newcommand{\hook}[1]{\stackrel{#1}{\hookrightarrow}}

\newcommand{\getsup}[1]{\stackrel{#1}{\gets}}

\newcommand{\Sp}{\on{\mathbb{S}p}}
\newcommand{\Spin}{\on{\mathbb{S}pin}}
\newcommand{\GSpin}{\on{G\mathbb{S}pin}}
\newcommand{\GSp}{\on{G\mathbb{S}p}}
\newcommand{\IC}{\on{IC}}

\newcommand{\Hom}{\on{Hom}}

\newcommand{\Ext}{\on{Ext}}

\newcommand{\Sym}{\on{Sym}}
\newcommand{\SO}{\on{S\mathbb{O}}}

\newcommand{\Ker}{\on{Ker}}

\newcommand{\Aut}{\on{Aut}}

\newcommand{\RG}{\on{R\Gamma}}

\newcommand{\Bun}{\on{Bun}}

\newcommand{\Bunb}{\on{\overline{Bun}} }
\newcommand{\Bunt}{\on{\widetilde\Bun}}

\newcommand{\Spec}{\on{Spec}}

\newcommand{\supp}{\on{supp}}

\newcommand{\Gr}{\on{Gr}}

\newcommand{\GL}{\on{GL}}
\newcommand{\PSL}{\on{PSL}}

\newcommand{\Fr}{{\on{Fr}}}

\newcommand{\pr}{\on{pr}}
\newcommand{\id}{\on{id}}
\newcommand{\QED}{$\square$} 
  
\newcommand{\Fp}{\mathbb{F}_p}  % what for??????    

\newcommand{\iso}{{\widetilde\to}}

\newcommand{\comp}{\circ}

\renewcommand{\H}{{\on{H}}}   %cohomologies

\newcommand{\R}{\on{R}\!}   %highest direct images
\newcommand{\DD}{\mathbb{D}}  %for duality
\newcommand{\D}{\on{D}}       %for derived categories     
\newcommand{\ov}[1]{\overline{#1}}
\newcommand{\select}[1]{{\it{#1}}}

\newcommand{\und}[1]{\underline{#1}}

\newcommand{\<}{\langle}
\renewcommand{\>}{\rangle}

\newcommand{\ev}{\mathit{ev}}

\newcommand{\Loc}{\on{Loc}}
\newcommand{\Lie}{\on{Lie}}

\newcommand{\Res}{\on{Res}}

\newcommand{\act}{\on{act}}
\newcommand{\dimrel}{\on{dim.rel}}
\newcommand{\codim}{\on{codim}}

\newcommand{\SL}{\on{SL}}

\newcommand{\ra}{\rightarrow}
\newcommand{\la}{\leftarrow}

\nc{\Perv}{\on{Perv}}

\nc{\Gra}{\on{Gra}}
\nc{\PPerv}{\on{{\PP}erv}}

\nc{\oX}{\overset{\scriptscriptstyle\circ}{X}}
\nc{\ocZ}{\overset{\scriptscriptstyle\circ}{\cZ}}
\nc{\ocL}{\overset{\scriptscriptstyle\circ}{\cL}}
\nc{\gRes}{\on{gRes}}
\nc{\Sign}{\on{Sign}}
\nc{\goodat}{\rm{good\, at}}
\nc{\Whit}{\on{Whit}}
\nc{\add}{\on{add}}
\nc{\FS}{\on{FS}}
\nc{\oo}[1]{\overset{\scriptscriptstyle\circ}{#1}}
\nc{\can}{\on{can}}
\nc{\summ}{\on{sum}}
\nc{\SiSu}{\on{SS}}
\nc{\Irr}{\on{Irr}}
\nc{\og}[1]{\overset{\scriptscriptstyle\bullet}{#1}}

\begin{document}

\title{Twisted Whittaker models for metaplectic groups}

\author{S. Lysenko}
%\address{}
%\email{}

\maketitle

\tableofcontents

\section*{Introduction} \label{intr}

\sssec{Motivation} In this paper inspired by \cite{G} we study the twisted Whittaker categories for metaplectic groups (in the sense of \cite{L1}). This is a part of the quantum geometric Langlands program \cite{S}, \cite{G2}, (\cite{Fr}, Section~6.3), \cite{GL}. 

 Let $G$ be a connected reductive group over an algebraically closed field $k$ of characteristic $p>0$. Let $\cO=k[[t]]\subset F=k((t))$, write $\Gr_G=G(F)/G(\cO)$ for the affine grassmannian of $G$. Let us briefly describe the aspect of the quantum geometric Langlands program, which motivates our study. Assume given a central extension $1\to \Gm\to \EE \to G(F)\to 1$ in the category of group ind-schemes over $k$ together with a splitting over $G(\cO)$. Let $N\ge 1$ be invertible in $k$, $\zeta: \mu_N(k)\to \Qlb^*$ an injective character. Here $\ell$ is invertible in $k$. %and $\cL_{\zeta}$ the corresponding local system on $B(\mu_N)$, on which $\mu_N(k)$ acts by $\zeta$. 
 Let $\wt \Gr_G$ be the stack quotient of $\EE/G(\cO)$ by the $\Gm$-action, where $z\in\Gm$ acts as $z^N$. Let $\PPerv_{G,\zeta}$ be the category of $G(\cO)$-equivariant perverse sheaves on $\wt\Gr_G$, on which $\mu_N(k)$ acts by $\zeta$. 
 
 To get an extension of the Satake equivalence to this case (as well as, conjecturally, of the whole nonramified geometric Langlands program), one needs to impose additional assumptions and structures on the gerbe $\wt\Gr_G\to \Gr_G$. We believe that it should come from factorization gerbe over $\Gr_G$ (\cite{GL}). Let us explain the idea of what it is. 
 
  Let $X$ be a smooth projective curve. Write $\Bun_G$ for the stack of $G$-torsors on $X$. Let $\cR(X)$ be the Ran space of $X$, $\Gr_{\cR(X)}$ the Beilinson-Drinfeld affine grassmannian of $G$ over $\cR(X)$ (cf. \cite{G3}). Recall the factorization structure of $\Gr_{\cR(X)}$
\begin{multline}
\label{Gr_BD_factorozation}
\Gr_{\cR(X)}\times_{\cR(X)} (\cR(X)\times\cR(X))_{disj}\,\iso\,\\
(\Gr_{\cR(X)}\times \Gr_{\cR(X)})\times_{\cR(X)\times\cR(X)} (\cR(X)\times\cR(X))_{disj},
\end{multline}
here $(\cR(X)\times\cR(X))_{disj}$ is the space classifying pairs of disjoint finite subsets in $X$. A factorization $\mu_N$-gerbe over $\Gr_G$ is a $\mu_N$-gerbe $\wt\Bun_G$ over $\Bun_G$ together with a factorization structure of its restriction under $\Gr_{\cR(X)}\to\Bun_G$ extending (\ref{Gr_BD_factorozation}) and satisfying some compatibility properties (\cite{GL}, 2.2.5). 

 There is an easier notion of a factorization line bundle on $\Gr_{\cR(X)}$ (\cite{GL}, 2.2.8 and 3.3). It can be seen as a line bundle $\cL$ on $\Bun_G$ together with a factorization structure of its restriction to $\Gr_{\cR(X)}$. For such a line bundle the gerbe $\Bunt_G$ of its $N$-th roots has the above factorization structure. 
 
 In \cite{L1} the following twisted version of Satake equivalence was proved. 
Our input data was a factorization line bundle on $\Gr_{\cR(X)}$ of some special form (sufficient to construct all the factorization $\mu_N$-gerbes on $\Gr_G$ up to an isomorphism)\footnote{This allowed to completely avoid the higher category theory heavily used in \cite{GL}.}. These data are described in Section~\ref{section_input_data}.  
 
  Given such data for $G$, we equipped $\PPerv_{G,\zeta}$ with a structure of a tensor category and established an equivalence of tensor categories $\PPerv_{G,\zeta}^{\natural}\,\iso\, \Rep(\check{G}_{\zeta})$, where $\check{G}_{\zeta}$ is a connected reductive group, an analog of the Langlands dual group in the metaplectic setting. Here $\PPerv_{G,\zeta}^{\natural}$ is a symmetric monoidal category obtained from $\PPerv_{G,\zeta}$ by some modification of the commutativity constraint. 
 
 Let $\D_{\zeta}(\Bunt_G)$ be the derived category of $\Qlb$-sheaves on $\Bunt_G$, on which $\mu_N(k)$ acts by $\zeta$. In Section~\ref{Section_action_Hecke_on_BuntG} we define an action of $\Rep(\check{G}_{\zeta})$ on $\D_{\zeta}(\Bunt_G)$ by Hecke functors. An extension of the geometric Langlands program to this case is the problem of the spectral decomposition of $\D_{\zeta}(\Bunt_G)$ under this action. 

\sssec{Gaitsgory's conjecture} Recall the Whittaker category from \cite{FGV}. Let $U\subset G$ be a maximal unipotent subgroup, $\cL_{\psi}$ the Artin-Schreier sheaf on $\A^1$ corresponding to an injective character $\psi: \Fp\to\Qlb^*$. Pick a non-degenerate character $\chi: U(F)\to\GG_a$ with zero conductor. Heuristically, the Whittaker category is the category of $(U(F), \chi^*\cL_{\psi})$-equivariant perverse sheaves on $\Gr_G$. 
 
  Recall that the orbits of $U(F)$ on $\Gr_G$ are infinite-dimensional, and there are two equivalent ways to make sense of the above definition. The first (technically difficult) local definition is found in \cite{Ber}. The second one is via the Drinfeld compactification denoted $\gM$ and using a smooth projective curve $X$ as an input datum. Let $x\in X$ and $U_{out}$ be the group-subscheme of $U(F)$ of maps $(X-x)\to U$. The character $\chi$ is trivial on $U_{out}$, so the objects of the Whittaker category `live' on $U_{out}\backslash \Gr_G$. Let $\gM_x$ be the ind-stack classifying a $G$-torsor on $X$ with a generalized reduction to $U$ over $X-x$. Then $U_{out}\backslash \Gr_G$ is included naturally into $\gM_x$. The surrogate Whittaker category $\Whit_x$ is defined as the category of perverse sheaves on $\gM_x$ with a certain equivariance condition, which restores the $(U(F), \chi^*\cL_{\psi})$-equivariance property on $\Gr_G$ (cf. \cite{FGV}, \cite{G}).  The main result of \cite{FGV} established an equivalence of categories $\Whit_x\,\iso\, \Rep(\check{G})$. Here $\check{G}$ is the Langlands dual to $G$.
Gaitsgory's conjecture (\cite{G}, Conjecture~0.4) is a quantum deformation of the above equivalence. 

 The definition of the twisted Whittaker category for $G$ from \cite{G} extends to our (a bit more general) setting of $G$ equipped with a factorizable line bundle on $\Gr_{\cR(X)}$ of our special form. Heuristically, the twisted Whittaker category is the category of $(U(F), \chi^*\cL_{\psi})$-equivariant perverse sheaves on $\wt\Gr_G$. This makes sense, because a central extension of $U(F)$ by $\Gm$ splits canonically. In this paper we adopt the second definition of the twisted Whittaker category denoted $\Whit^{\kappa}_x$, here $\kappa$ refers to our metaplectic data. Let $\wt\gM_x$ be the restriction of the gerbe $\wt\Bun_G$ to $\gM_x$. Then $\Whit^{\kappa}_x$ is defined as the category of perverse sheaves on $\wt\gM_x$ with some equivariance condition (the same as in the untwisted case).  
    
Gaitsgory's conjecture attaches to our metaplectic data a big quantum group $U_q(\check{G})$ (Lusztig's version with $q$-divided powers) such that one should have an equivalence 
\begin{equation}
\label{equiv_Lurie_main}
\Whit^{\kappa}_x\,\iso\, \Rep(U_q(\check{G}))
\end{equation}
with the category of its finite-dimensional representations. This paper is a step towards the proof of this conjecture in our setting. 

 Both categories are actually factorization categories in the sense of \cite{R1}, and the above equivalence should be compatible with these structures. This is the reason for which in this paper (as in \cite{G}) we  also consider the versions of the twisted Whittaker category for $n$ points of $X$ denoted $\Whit^{\kappa}_n$. 
 
 The group $\check{G}_{\zeta}$ is expected to be the quantum Frobenius quotient of $U_q(\check{G})$. So, the category $\Rep(\check{G}_{\zeta})$ will act on both sides of (\ref{equiv_Lurie_main}), and the equivalence has to be compatible with these actions. Besides, the action on the basic object of $\Whit^{\kappa}_x$ realizes $\Rep(\check{G}_{\zeta})$ as its full subcategory.
 
\sssec{Main results} One of the main ideas of \cite{G} was the construction of the functor $G_n: \Whit^c_n\to \FS^c_n$ from the twisted Whittaker category of $G$ to the category of factorizable sheaves assuming that the quantum parameter denoted $c$ in \select{loc.cit.} is irrational (i.e., $q=\exp(\pi i c)$ is not a root of unity). The main result of \cite{BFS} identified the category of factorizable sheaves $\FS^c_n$ with the category $\Rep(\og{u}_q(\check{G}))$ of representations of the corresponding graded small quantum group $\og{u}_q(\check{G})$. When $q$ is not a root of unity, the latter coincides with the big quantum group $U_q(\check{G})$, and this has led to a proof of the above conjecture in that case (\cite{G}). 

 In the metaplectic case, corresponding to $q$ being a root of unity, $\og{u}_q(\check{G})$ and  $U_q(\check{G})$ are substantially different, and the construction of $G_n$ breaks down. 
 
  A possible strategy of the proof of (\ref{equiv_Lurie_main}) in the metaplectic case is to construct a corrected version of the functor $G_n$, and then upgrade this functor to the desired equivalence. 
%This strategy is pursued in \cite{GL2}.  

 One of our main results is a construction of a corrected version of the functor $G_n$ in the metaplectic case. This is the purpose of Part I of this paper, in Part II we study properties of this functor.
 
\sssec{} The definitions of the twisted Whittaker category $\Whit^{\kappa}_n$ and the category $\wt\FS^{\kappa}_n$ of factorizable sheaves are given in Sections~\ref{Section_The twisted Whittaker category} and \ref{Section_The FS category}. Our Theorem~\ref{Thm_main_functor_ovFF} provides a functor
$$
\ov{\FF}: \Whit^{\kappa}_n\to \wt\FS^{\kappa}_n
$$
exact for the perverse t-structures and commuting with the Verdier duality. It is constructed under the assumption that our metaplectic parameter, the quadratic form $\varrho$, satisfies what we call the \select{subtop cohomology property}. This is a local property that we prove for all the simple simply-connected reductive groups and most of parameters $\varrho$ in Theorem~\ref{Th_114} (and Remark~\ref{Rem_E_6_E_7_E_8}), which is one of our main results. We formulate Conjecture~\ref{Con_main} describing those quadratic forms $\varrho$ for which we expect the subtop cohomology property to hold. These are precisely those $\varrho$ for which our construction of $\ov{\FF}$ makes sense.
\footnote{When this paper has been written, D. Gaitsgory has informed
the author that the definition of the functor $\ov{\FF}$ has been known to him at the time of working on \cite{G} around 2007, as well as some version of Conjecture~\ref{Con_main} of our paper. But since this conjecture was not proved, the definition of $\ov{\FF}$ was not made public at that moment.}
 
  Let $U^-\subset G$ be the opposite maximal unipotent subgroup. The functor $G_n$ in \cite{G} was defined, roughly, by taking cohomologies along $U^-(F)$-orbits on $\Gr_G$. More precisely, these cohomology complexes are put in families over the configuration spaces $X^{\mu}_n$
of divisors on $X$ as direct images under $\cZ^{\mu}_n\to X^{\mu}_n$
giving rise to perverse sheaves on (some gerbes over) these $X^{\mu}_n$. The perverse sheaves so obtained are moreover factorizable in a natural sense. Here $\cZ^{\mu}_n$ are Zastava spaces largely used in the geometric Langlands program (\cite{BFGM}).  
  
 To construct the functor $\ov{\FF}$, we introduce natural compactifications of Zastava spaces in Section~\ref{Section_Compactified Zastava}. Our proof also essentially uses the description of the twisted $\IC$-sheaves of Drinfeld compactifications $\Bunb_B$ from \cite{L2}. 
 
 As predicted by the above conjecture, an irreducible object of the twisted Whittaker category $\Whit_x^{\kappa}$ is of the form $\cF_{x,\lambda}$ for some $G$-dominant coweight $\lambda$. Assuming the subtop cohomology property we show that 
$$
\ov{\FF}(\cF_{x,\lambda})\,\iso\, \mathop{\oplus}\limits_{\mu\le\lambda} \cL_{x,\mu}\otimes V^{\lambda}_{\mu},
$$ 
where $\cL_{x,\mu}$ are the irreducible objects of $\wt\FS^{\kappa}_x$, and $V^{\lambda}_{\mu}$ are some multiplicity vector spaces (cf. Corollary~\ref{Corollary_392} and Proposition~\ref{Pp_ovFF_is_exact}). One of our main results is a description of the space $V^{\lambda}_{\mu}$ in Theorem~\ref{Th_4.12.1}. We show that $V^{\lambda}_{\mu}$ admits a canonical base, which is naturally a subset of $B(\lambda)$. Here $B(\lambda)$ is the crystal of the canonical base of the irreducible $\check{G}$-representation $\VV^{\lambda}$ of highest weight $\lambda$. 

 The dominant weights of $\check{G}_{\zeta}$ form naturally a subset of the set $\Lambda^+$ of $G$-dominant coweights. Our Theorem~\ref{Thm_special_case_of_V_lambda_mu} shows that if $\lambda$ is a dominant weight of $\check{G}_{\zeta}$ then $V^{\lambda}_{\mu}$ identifies with the $\mu$-weight space in the irreducible representation $V(\lambda)$ of $\check{G}_{\zeta}$ of highest weight $\lambda$.
 
\sssec{Other results inspired by quantum groups}  In Section~\ref{Section_Hecke functors} we define the action of the category $\Rep(\check{G}_{\zeta})$ of representations of $\check{G}_{\zeta}$ by
Hecke functors on the twisted derived category $\D_{\zeta}(\Bunt_G)$ of $\Bun_G$, and on the twisted Whittaker category $\D\Whit^{\kappa}_x$. The main result of this Section is Theorem~\ref{Th_Hecke_action}. It shows that the Hecke functors are exact for the perverse t-structure on the twisted Whittaker category. It also shows that acting on the basic object of $\Whit^{\kappa}_x$ by the Hecke functor corresponding to an irreducible representation of $\check{G}_{\zeta}$, one gets the corresponding irreducible object of $\Whit^{\kappa}_x$. This is an analog of (\cite{FGV}, Theorem~4) in the metaplectic setting.   

 In Section~\ref{Sec_Objects that remain irreducible} we introduce a notion of restricted dominant coweights $\lambda$ of $G$ and show that the corresponding irreducible objects $\cF_{x,\lambda}\in \Whit^{\kappa}_x$ remain irreducible after applying $\ov{\FF}$. This is an analog of the corresponding result for the restriction functor $\Rep(U_q(\check{G}))\to \Rep(\og{u}_q(\check{G}))$, see (\cite{ABBGM}, Proposition~1.1.8).  
 
 We also prove an analog in our setting of the Lusztig-Steinberg tensor product theorem for quantum groups (\cite{ABBGM}, Theorem~1.1.4). It describes the structure of the semi-simple part $\Whit^{\kappa, ss}_x$ of the twisted Whittaker category $\Whit^{\kappa}_x$ as a module over $\Rep(\check{G}_{\zeta})$ acting by Hecke functors. If $[\check{G}_{\zeta},  
\check{G}_{\zeta}]$ is simply-connected then we show that this is a complete description. 

 Let $\check{T}_{\zeta}\subset \check{G}_{\zeta}$ be the canonical maximal torus. In Section~\ref{section_H_functors_for_FS} we define the action of $\Rep(\check{T}_{\zeta})$ by Hecke functors on $\wt\FS^{\kappa}_x$. This corresponds to an action of representations of the maximal torus of the quantum Frobenius quotient on $\Rep(\og{u}_q(\check{G}))$ defined in (\cite{ABBGM}, Section~1.1.6).  
 
 One of our main results is Theorem~\ref{Thm_functoriality} showing that $\ov{\FF}: \Whit^{\kappa}_x\to \wt\FS^{\kappa}_x$ commutes with the actions of Hecke functors with respect to the inclusion $\check{T}_{\zeta}\hook{} \check{G}_{\zeta}$. This is an analog of the similar property of the restriction functor $\Rep(U_q(\check{G}))\to \Rep(\og{u}_q(\check{G}))$ (\cite{ABBGM}, Proposition~1.1.11). 
 
  Further, we compute the Kazhdan-Lusztig's type polynomials for $\Whit^{\kappa}_x$ in some special cases (cf. Section~\ref{Section_Examples_KL_polynomials}). Acting on the basic object of $\FS^{\kappa}_x$, one gets a full embedding $\Rep(\check{T}_{\zeta})\subset \FS^{\kappa}_x$. We also show that, assuming the subtop cohomology property, this full subcategory is closed under extensions (cf. Proposition~\ref{Pp_closed_under_extensions}). 
  
  In Section~\ref{C-Sh-formula_basics} we formulate a metaplectic analog of the Casselman-Shalika problem (and its analog for quantum groups). Then we calculate the top cohomology group of the corresponding Casselman-Shalika complex in Proposition~\ref{Pp_H^0_of_C-Sh-formula}. 
  
  In Appendix~B we prove Proposition~\ref{Pp_reformulating_subtop_coh_property}, which reformulates the subtop cohomology property as some categorical property of $\Whit^{\kappa}_x$ saying that $\Ext^1$ in $\Whit^{\kappa}_x$ between some irreducible objects vanish. The corresponding property is known to hold for $\Rep(U_q(\check{G}))$ (\cite{Bez}). In Section~\ref{Section_comparison_ABBGM} we compare our setting with that of \cite{ABBGM}, where the quantum Frobenius quotient of $U_q(\check{G})$ identifies with $G$.  

\sssec{Notation} 
\label{sssec_Notation}
We work over an algebraically closed ground field $k$ of characteristic $p>0$. Let $X$ be a smooth projective connected curve of genus $g$. \index{$k, p, X, g, \Omega$}\index{$\cO, F, \Omega^{\frac{1}{2}}$}%
Let $\Omega$ denote the canonical line bundle on $X$. We fix a square root $\Omega^{\frac{1}{2}}$ of $\Omega$. Set $\cO=k[[t]]\subset F=k((t))$. 
 Let $G$ be a connected reductive group over $k$ with $[G,G]$ simply-connected. Let $B\subset G$ be a Borel subgroup, $B^-\subset G$ its opposite and $T=B\cap B^-$ a maximal torus. \index{$G, B, B^-, T, U, U^-$}
Let $U$ (resp., $U^-$) denote the unipotent radical of $B$ (resp., of $B^-$). Let $\Lambda$ denote the coweights of $T$, $\check{\Lambda}$ the weights of $T$. The canonical pairing between the two is denoted by $\< , \>$. By $\Lambda^+$ (resp., $\check{\Lambda}^+$) we denote the semigroup of dominant coweights (resp., dominant weights) for $G$. Let $\rho$ be the half-sum of positive coroots of $G$. Let $\Lambda^{pos}$ denote the $\ZZ_+$-span of positive coroots in $\Lambda$. \index{$\Lambda, \check{\Lambda}, \<,\>$}\index{$\Lambda^+, \check{\Lambda}^+, \rho, \Lambda^{pos}$}
 
 Set $G_{ab}=G/[G,G]$, let $\Lambda_{ab}$ (resp., $\check{\Lambda}_{ab}$) denote the coweights (resp., weights) of $G_{ab}$. Let $J$ denote the set of connected components of the Dynkin diagram of $G$. For $j\in J$ write $\cJ_j$ for the set of vertices of the $j$-th connected component of the Dynkin diagram, $\cJ=\cup_{j\in J} \cJ_i$. For $j\in \cJ$ let $\alpha_j$ (resp., $\check{\alpha}_j$) denote the corresponding simple coroot (resp., simple root). \index{$G_{ab}, \Lambda_{ab}, \check{\Lambda}_{ab}$}\index{$J, \cJ_j, \cJ, \alpha_j, \check{\alpha}_j$} 
 One has $G_{ad}=\prod_{j\in J} G_j$, where $G_j$ is a simple adjoint group. Let $\gg_j=\Lie G_j$. For $j\in J$ let $\kappa_j: \Lambda\otimes\Lambda\to\ZZ$ be the Killing form for $G_j$, so
$$
\kappa_j=\sum_{\check{\alpha}\in \check{R}_j} \check{\alpha}\otimes\check{\alpha},
$$
\index{$G_j, \gg_j, \kappa_j, \check{R}_j$} where $\check{R}_j$ is the set of roots of $G_j$. For a standard Levi subgroup $M$ of $G$ we denote by $\Lambda^{pos}_M$ the $\ZZ_+$-span of simple coroots of $M$ in $\Lambda$. Our notation $\mu\le_M\lambda$ for $\lambda,\mu\in\Lambda$ means that $\lambda-\mu\in \Lambda^{pos}_M$. For $M=G$ we write $\le$ instead of $\le_G$. \index{$\Lambda^{pos}_M, \le_M$}\index{$\cE^s(T)$}

By a super line we mean a $\ZZ/2\ZZ$-graded line. As in \cite{L1}, we denote by $\cE^s(T)$ the groupoid of pairs: a symmetric bilinear form $\kappa: \Lambda\otimes\Lambda\to\ZZ$, and a central super extension $1\to k^*\to \tilde\Lambda^s \to\Lambda\to 1$ whose commutator is $(\gamma_1,\gamma_2)_c=(-1)^{\kappa(\gamma_1,\gamma_2)}$. 
 
 Let $\Sch/k$ denote the category of $k$-schemes of finite type with Zarisky topology. The $n$-th Quillen $K$-theory group of a scheme form a presheaf on $\Sch/k$. As in \cite{BD}, $K_n$ will denote the associated sheaf on $\Sch/k$ for the Zariski topology.
 
 Pick a prime $\ell$ invertible in $k$. We work with (perverse) $\Qlb$-sheaves on $k$-stacks for the \'etale topology. Pick an injective character $\psi: \Fp\to\Qlb^*$, let $\cL_{\psi}$ be the corresponding Artin-Schreier sheaf on $\A^1$. The trivial $G$-torsor over some base is denoted $\cF^0_G$. 
\index{$\Sch/k,  K_n, \ell, \Qlb$}\index{$\psi, \cL_{\psi}, \cF^0_G$}
 
 Denote by $\cP^{\theta}(X,\Lambda)$ the Picard groupoid of $\theta$-data (\cite{BD2}, Section~3.10.3). Its object is a triple $\theta=(\kappa,\lambda, c)$, where $\kappa:\Lambda\otimes\Lambda\to\ZZ$ is a symmetric bilinear form, $\lambda$ is a rule that assigns to each $\gamma\in\Lambda$ a super line bundle $\lambda^{\gamma}$ on $X$, and $c$ is a rule that assigns to each pair $\gamma_1,\gamma_2\in\Lambda$ an isomorphism $c^{\gamma_1,\gamma_2}: \lambda^{\gamma_1}\otimes \lambda^{\gamma_2}\,\iso\, \lambda^{\gamma_1+\gamma_2}\otimes\Omega^{\kappa(\gamma_1,\gamma_2)}$ on $X$. They are subject to the conditions from (\cite{BD2}, Section~3.10.3). 
  
 For a reductive group $H$ we denote by $\Bun_H$ the stack of $H$-torsors on $X$. \index{$\cP^{\theta}(X,\Lambda), \Bun_H$}

\sssec{Input data} 
\label{section_input_data}
We fix the following data as in (\cite{L1}, Section~2.3). Write $\Gr_G=G(F)/G(\cO)$ for the affine grassmannian of $G$. For $j\in J$ let $\cL_j$ denote the ($\ZZ/2\ZZ$-graded purely of parity zero) line bundle on $\Gr_G$ with fibre $\det(\gg_j(\cO): \gg_j(\cO)^g)$ at $gG(\cO)$ (the definition of this relative determinant is found in \cite{FL}). Let $E^a_j$ be the punctured total space of the pull-back of $\cL_j$ to $G(F)$. This is a central extension
$$
1\to \Gm\to E^a_j\to G(F)\to 1,
$$
it splits canonically over $G(\cO)$. \index{$\Gr_G, \cL_j, E^a_j$}

Pick an even symmetric bilinear form $\beta: \Lambda_{ab}\otimes\Lambda_{ab}\to\ZZ$.  Assume given a central extension
\begin{equation}
\label{ext_Lambda_ab_by_V_beta}
1\to \Gm\to V_{\beta}\to \Lambda_{ab}\to 1
\end{equation}
over $k$ whose commutator is $(\gamma_1,\gamma_2)_c=(-1)^{-\beta(\gamma_1,\gamma_2)}$. It is given for each $\gamma\in\Lambda_{ab}$ by a line $\epsilon^{\gamma}$ over $k$ together with isomorphisms
$$
c^{\gamma_1,\gamma_2}: \epsilon^{\gamma_1}\otimes\epsilon^{\gamma_2}\,\iso\, \epsilon^{\gamma_1+\gamma_2}
$$
for $\gamma_i\in \Lambda_{ab}$ subject to the conditions in the definition of $\cE^s(G_{ab})$ (\cite{L3}, Section~3.2.1). \index{$\beta, V_{\beta}$}

  Let $N\ge 1$ be invertible in $k$. Let $\zeta: \mu_N(k)\to\Qlb^*$ be an injective character, we write $\cL_{\zeta}$ for the canonical rank one local system on $B(\mu_N)$ such that $\mu_N(k)$ acts on it by $\zeta$. We have a map $s_N: \Gm\to B(\mu_N)$ corresponding to the $\mu_N$-torsor $\Gm\to\Gm, z\mapsto z^N$. The local system $s_N^*\cL_{\zeta}$ is sometimes also denoted by $\cL_{\zeta}$. \index{$N, \zeta, \cL_{\zeta}$}
  
  For each $j\in J$ pick $c_j\in \ZZ$. To these data we associate the even symmetric bilinear form $\bar\kappa: \Lambda\otimes\Lambda\to\ZZ$ given by
\begin{equation}
\label{def_bar_kappa}
\bar\kappa=-\beta-\sum_{j\in J} c_j\kappa_j
\end{equation}
and the quadratic form $\varrho: \Lambda\to\QQ$ given by $\varrho(\mu)=\frac{\bar\kappa(\mu,\mu)}{2N}$. \index{$c_j, \bar\kappa, \varrho$}

\sssec{Central extensions} To the above unput data we associate the following objects. According to  (\cite{BD}, Theorem~3.16), to the central exension (\ref{ext_Lambda_ab_by_V_beta}) one canonically attaches a central extension
\begin{equation}
\label{ext_G_ab_by_K_2}
1\to K_2\to \cV_{\beta}\to G_{ab}\to 1
\end{equation}
of sheaf of groups on $\Sch/k$ with the following property. Write $(\cdot, \cdot)_{st}: F^*\times F^*\to k^*$ for the tame symbol map (\cite{L1}, Section~2.3). Passing to $F$-points in (\ref{ext_G_ab_by_K_2}) and further taking the push-out by the tame symbol $(\cdot, \cdot)_{st}: K_2(F)\to \Gm$ one gets a central extension
\begin{equation}
\label{ext_E_beta_of_G_ab(F)}
1\to\Gm\to E_{\beta}\to G_{ab}(F)\to 1,
\end{equation}
\index{$\cV_{\beta}, (\cdot, \cdot)_{st}, E_{\beta}, \EE$}
which is actually a central extension in the category of ind-schemes over $k$, and its commutator $(\cdot, \cdot)_c: G_{ab}(F)\times G_{ab}(F)\to\Gm$
satisfies 
$$
(\lambda_1\otimes f_1, \lambda_2\otimes f_2)_c=(f_1, f_2)_{st}^{-\beta(\lambda_1,\lambda_2)}
$$
for $\lambda_i\in \Lambda_{ab}, f_i\in F^*$. The pull-back of (\ref{ext_E_beta_of_G_ab(F)}) under $G(F)\to G_{ab}(F)$ is also denoted by $E_{\beta}$ by abuse of notation. 

  Recall that $\cV_{\beta}(\cO)\to G_{ab}(\cO)$ is surjective, and the composition of the tame symbol with $K_2(\cO)\to K_2(F)$ is trivial. For this reason  (\ref{ext_E_beta_of_G_ab(F)}) is equipped with a canonical section over $G_{ab}(\cO)$. The $\Gm$-torsor $E_{\beta}/G_{ab}(\cO)\to \Gr_{G_{ab}}$ over $t^{\gamma}G_{ab}(\cO)$ is constant with fibre $\epsilon^{\gamma}-\{0\}$, the group $G_{ab}(\cO)$ acts on it by the character 
$$
G_{ab}(\cO)\to G_{ab}\toup{\beta(\gamma)} \Gm
$$  

 The sum of the extensions $E_{\beta}$ and $(E^a_j)^{c_j}$, $j\in J$ is the central extension denoted
\begin{equation}
\label{ext_EE_of_G(F)}
1\to \Gm\to \EE\to G(F)\to 1.
\end{equation}
It is equipped with the induced section over $G(\cO)$. Let
\begin{equation}
\label{ext_V_EE_of_Lambda}
1\to\Gm\to V_{\EE}\to \Lambda\to 1
\end{equation}
be the pull-back of (\ref{ext_EE_of_G(F)}) under $\Lambda\to G(F)$, 
$\lambda\mapsto t^{\lambda}$. The commutator in (\ref{ext_V_EE_of_Lambda}) is given by
$$
(\lambda_1,\lambda_2)_c=(-1)^{\bar\kappa(\lambda_1,\lambda_2)}
$$
\index{$V_{\EE}, \Gra_G$}\index{$\wt\Gr_G, \PPerv_{G,\zeta}$}

Set $\Gra_G=\EE/G(\cO)$. Let $\wt\Gr_G$ be the stack quotient of $\Gra_G$ under the $\Gm$-action such that $z\in\Gm$ acts as $z^N$. Let $\PPerv_{G,\zeta}$ be the category of $G(\cO)$-equivariant perverse sheaves on $\wt\Gr_G$ on which $\mu_N(k)$ acts by $\zeta$.  

\sssec{Line bundles}  As in (\cite{L1}, Section~2.6) we associate to the pair $(V_{\beta}, -\beta)\in \cE^s(G_{ab})$ a line bundle $\cL_{\beta}$ on $\Bun_{G_{ab}}$. This is done in two steps. First, $(V_{\beta}, -\beta)$ yields an object $\theta_0\in \cP^{\theta}(X,\Lambda_{ab})$ as in (\cite{L3}, Lemma~4.1). Second, $\theta_0$ yields a super line bundle on $\Bun_{G_{ab}}$ as in (\cite{L1}, Section~2.6), it is actually of parity zero.
The restriction of $\cL_{\beta}$ to the Ran version of $\Gr_{G_{ab}}$ is a factrorization line bundle (\cite{GL}, 2.2.8). 

For $\mu\in \Lambda_{ab}$ consider the map $i_{\mu}: X\to\Bun_{G_{ab}}$, $x\mapsto \cO(-\mu x)$. One has canonically
$$
i_{\mu}^*\cL_{\beta}\,\iso\, \Omega^{\frac{\beta(\mu,\mu)}{2}}\otimes \epsilon^{\mu}
$$ 
\index{$\cL_{\beta}, i_{\mu}$}

 For $j\in J$ let $\cL_{j,\Bun_G}$ be the line bundle on $\Bun_G$ whose fibre at $\cF\in\Bun_G$ is
$$
\det\RG(X, (\gg_j)_{\cF^0_G})\otimes \det\RG(X, (\gg_j)_{\cF})^{-1}
$$
Denote by $\cL^{\bar\kappa}$ the line bundle $\cL_{\beta}\otimes (\mathop{\otimes}\limits_{j\in J} \cL_{j,\Bun_G}^{c_j})$ on $\Bun_G$. Its restriction to the Ran version of $\Gr_G$ is naturally a factorization line bundle. 
\index{$\cL_{j,\Bun_G}$}\index{$\cL^{\bar\kappa}, \Gr_{G,x}$}

For $x\in X$ let $\Gr_{G,x}$ denote the affine grassmannian classifying a $G$-torsor $\cF$ on $X$ with a trivialization $\cF\,\iso\, \cF^0_G\mid_{X-x}$. The restriction of $\cL^{\bar\kappa}$ (with zero section removed) under the forgetful map $\Gr_{G,x}\to \Bun_G$ identifies with $\Gra_G$ (once we pick an isomorphism $D_x\,\iso\, \Spec \cO$ for the formal disk $D_x$ around $x$). 

\sssec{Metaplectic dual group} 
\label{section_Metaplectic dual group} 
In \cite{L1} we equipped $\PPerv_{G,\zeta}$ with a structure of a symmetric monoidal category, we introduced a symmetric monoidal category $\PPerv^{\natural}_{G,\zeta}$ obtained from $\PPerv_{G,\zeta}$ by some modification of the commutativity constraint. 

 Set $\Lambda^{\sharp}=\{\lambda\in \Lambda\mid \bar\kappa(\lambda)\in N\check{\Lambda}\}$. Let $\check{T}_{\zeta}=\Spec k[\Lambda^{\sharp}]$ be the torus whose weights lattice is $\Lambda^{\sharp}$. 
Let $\check{G}_{\zeta}$ be the reductive group over $\Qlb$ defined in (\cite{L1}, Theorem~2.1), it is equipped with canonical inclusions $\check{T}_{\zeta}\subset \check{B}_{\zeta}\subset \check{G}_{\zeta}$, where $\check{T}_{\zeta}$ is a maximal torus, and $\check{B}_{\zeta}$ is a Borel subgroup dual to $T\subset B\subset G$. 
\index{$\Lambda^{\sharp}$}\index{$\check{T}_{\zeta}, \check{G}_{\zeta}$}
\index{$\check{B}_{\zeta}, \PPerv^{\natural}_{G,\zeta}$}
  
 To get a fibre functor on $\PPerv^{\natural}_{G,\zeta}$ one needs to pick an additional input datum. We make this choice as in \cite{L1}. Namely, let $\bar V_{\EE}$ be the stack quotient of $V_{\EE}$ by the $\Gm$-action, where $z\in\Gm$ acts as $z^N$. It fits into an exact sequence of group stacks
\begin{equation}
\label{ext_barV_EE_of_Lambda}
1\to B(\mu_N)\to \bar V_{\EE} \to\Lambda\to 1
\end{equation}
We pick a morphism of group stacks $\gt_{\EE}:\Lambda^{\sharp}\to \bar V_{\EE}$, which is a section of (\ref{ext_barV_EE_of_Lambda}) over $\Lambda^{\sharp}$. This yields as in (\cite{L1}, Theorem~2.1) an equivalence of tensor categories $\PPerv^{\natural}_{G,\zeta}\,\iso\, \Rep(\check{G}_{\zeta})$. \index{$\bar V_{\EE}, \gt_{\EE}$}

 Let $\wt\Gr_T$ be obtained from $\wt\Gr_G$ by the base change $\Gr_T\to\Gr_G$. Write $\PPerv_{T,G, \zeta}$ for the category of $T(\cO)$-equivariant perverse sheaves on $\wt\Gr_T$ on which $\mu_N(k)$ acts by $\zeta$. As in (\cite{L1}, Section~3.2), the datum of $\gt_{\EE}$ yields an equivalence $\Loc_{\zeta}: \Rep(\check{T}_{\zeta})\,\iso\,\PPerv_{T,G, \zeta}$. 
\index{$\wt\Gr_T, \PPerv_{T,G, \zeta}, \Loc_{\zeta}$}

\sssec{} Let $\Omega^{\rho}$ denote the $T$-torsor on $X$ obtained from $\Omega^{\frac{1}{2}}$ via the extension of scalars for $2\rho: \Gm\to T$. 
We denote by $^{\omega}\cL^{\bar\kappa}$ the line bundle on $\Bun_G$ whose fibre at $\cF\in\Bun_G$ is $\cL^{\bar\kappa}_{\cF}\otimes (\cL^{\bar\kappa}_{\Omega^{\rho}})^{-1}$. From (\cite{L3}, Proposition~4.1) one gets the following. 
\index{$\Omega^{\rho}, {^{\omega}\cL^{\bar\kappa}}$}

\begin{Lm} 
\label{Lm_fibre_of_cL_beta}
Let $D=\sum_x \mu_x x$ be a $\Lambda$-valued divisor on $X$. The fibre of $\cL_{\beta}$ at $\Omega^{\rho}(-D)$ identifies canonically with
$$
(\cL_{\beta})_{\Omega^{\rho}}\otimes (\otimes_{x\in X} (\Omega^{\frac{1}{2}}_x)^{\beta(\mu_x, \mu_x+2\rho)}\otimes \epsilon^{\bar\mu_x}),
$$
where $\bar\mu_x\in\Lambda_{ab}$ is the image of $\mu_x$. 
\end{Lm}

\sssec{Langlands program for metaplectic groups} 
\label{section_Langlands_program_metapl_groups}
Let $\Bunt_G$ be the gerbe of $N$-th roots of $^{\omega}\cL^{\bar\kappa}$ over $\Bun_G$. Its restriction to the Ran version of $\Gr_G$ is a factorization gerbe (\cite{GL}, 2.3.2). Let $\D_{\zeta}(\Bunt_G)$ denote the derived category of $\Qlb$-sheaves on $\Bunt_G$, on which $\mu_N(k)$ acts by $\zeta$. 

 As in \cite{L2}, where the case of $G$ simple simply-connected was considered, we define an action of $\PPerv_{G,\zeta}^{\natural}$ on $\D_{\zeta}(\Bunt_G)$ by Hecke functors (see Section~\ref{Section_action_Hecke_on_BuntG}). The geometric Langlands program for metaplectic groups could be the problem of finding a spectral decomposition of $\D_{\zeta}(\Bunt_G)$ under this action. Our study of the twisted Whittaker model in this setting is motivated by this problem.  
\index{$\Bunt_G, \D_{\zeta}(\Bunt_G)$}

\part{Construction of the functor $\ov{\FF}$}
%trial for parts
 
\medskip 
 
\section{Local problem: subtop cohomology}
\label{section_Local problem: subtop cohomology}

\ssec{} 
\label{Section_1.1}
In this section we formulate and partially prove Conjecture~\ref{Con_main} that is used in Proposition~\ref{Pp_3.11.2}.

 For a free $\cO$-module $M$ write $M_{\bar c}=M\otimes_{\cO} k$. 
For $\mu\in\Lambda$ let $\Gr_B^{\mu}$ (resp., $\Gr_{B^-}^{\mu}$) denote the $U(F)$-orbit (resp., $U^-(F)$-orbit) in $\Gr_G$ through $t^{\mu}$. For $\mu$ in the coroot lattice, the $\Gm$-torsor $\Gra_G\times_{\Gr_G} \Gr_B^{\mu}\to \Gr_B^{\mu}$ is constant with fibre $\Omega_{\bar c}^{-\bar\kappa(\mu,\mu)}-0$, and $T(\cO)$ acts on it by the character $T(\cO)\to T\toup{-\bar\kappa(\mu)}\Gm$. The $\Gm$-torsor $\Gra_G\times_{\Gr_G} \Gr_{B^-}^{\mu}\to \Gr_{B^-}^{\mu}$ is constant with fibre $\Omega_{\bar c}^{-\bar\kappa(\mu,\mu)}-0$, and $T(\cO)$ acts on it by $T(\cO)\to T\toup{-\bar\kappa(\mu)}\Gm$.
\index{$M_{\bar c}, \Gr_B^{\mu}, \Gr_{B^-}^{\mu}$} 

 As in (\cite{FGV}, Section~7.1.4), for $\eta\in\Lambda$ we write $\chi_{\eta}: U(F)\to \A^1$ for an additive character of conductor $\bar\eta$, where $\bar\eta$ is the image of $\eta$ in the coweights lattice of $G_{ad}$. For $\eta+\nu\in\Lambda^+$ we also write $\chi^{\nu}_{\eta}: \Gr_B^{\nu}\to\A^1$ for any $(U(F), \chi_{\eta})$-equivariant function.
\index{$\chi_{\eta}, \chi^{\nu}_{\eta}, \wt\Gr_B^{\mu},\ev$}

 For $\mu\in\Lambda$ let 
$$
\wt\Gr_B^{\mu}=\Gr_B^{\mu}\times_{\Gr_G}\wt\Gr_G
$$ 
Pick $\chi_0: U(F)\to\A^1$ and define $\chi^0_0: \Gr_B^0\to\A^1$
by $\chi^0_0(uG(\cO))=\chi_0(u)$ for $u\in U(F)$. Set $\ev=\chi^0_0$. For the canonical trivialization $\wt\Gr_B^0\,\iso\, \Gr_B^0\times B(\mu_N)$, we consider $\cL_G:=\ev^*\cL_{\psi}\boxtimes \cL_{\zeta}$ as a local system on $\wt\Gr_B^0$.

 For $\mu$ in the coroot lattice any trivialization of $\Omega_{\bar c}^{-\bar\kappa(\mu,\mu)}$ yields a section $s_{\mu}: \Gr^{\mu}_{B^-}\to \wt\Gr_{B^-}^{\mu}$. Recall that $\Gr_B^0\cap \Gr_{B^-}^{-\lambda}$ is empty unless $\lambda\ge 0$, and for $\lambda\ge 0$ this is a scheme of finite type and pure dimension $\<\lambda, \check{\rho}\>$ by (\cite{BFGM}, Section~6.3). 
Recall the quadratic form $\varrho$ from Section~\ref{section_input_data}.  
\index{$\cL_G, s_{\mu}$}
 
\begin{Def} 
\label{Def_subtop_coh_property}
We say that \select{the subtop cohomology property is satisfied for $\varrho$} if for any $\lambda>0$, which is not a simple coroot, 
\begin{equation}
\label{complex_main}
\RG_c(\Gr_B^0\cap \Gr_{B^-}^{-\lambda}, s_{-\lambda}^*\cL_G)
\end{equation}
is placed in degrees $\le top-2$, where $top=\<\lambda, 2\check{\rho}\>$.
\end{Def}

\begin{Con} 
\label{Con_main}
Assume that $\varrho(\alpha_i)\notin \ZZ$ for any simple coroot $\alpha_i$. Then the subtop cohomology property is satisfied for 
$\varrho$.
\end{Con}

 This conjecture is motivated by our definition of the functor $\ov{\FF}$ in Section~\ref{Section_4.6.1}, this is precisely the local property needed in Proposition~\ref{Pp_3.11.2}. The assumption $\varrho(\alpha_i)\notin \ZZ$ is used in the construction of $\ov{\FF}$ to get the correct answer over $\oo{X}{}^{\mu}$ (see Proposition~\ref{Pp_3.3.4}). 
 
\begin{Rem} 
\label{Rem_passing_to_G_semisimple}
i) The input data of Section~\ref{section_input_data} are functorial in a suitable sense. In particular, we may restrict them from $G$ to $[G,G]$. Then $\bar\kappa$ gets replaced by its restriction to the coroot lattice. The subtop cohomology property holds for $[G,G]$ (with the induced input data) if and only if it holds for $G$. \\
ii) We may pick a torus $T_1$ and an inlcusion $Z([G,G])\hook{} T_1$, where $Z([G,G])$ is the center of $[G,G]$. Then $G_1:=([G,G]\times T_1)/Z([G,G])$ has a connected center, here $Z([G,G])$ is included diagonally in the product. One may also extend the input data of Section~\ref{section_input_data} to $G_1$ and assume, if necessary, that $G$ has a connected center. 
\end{Rem}

\begin{Def}  
\label{Def_property_C}
If the center $Z(G)$ of $G$ is not connected, replace $G$ by the group $G_1$ as in Remark~\ref{Rem_passing_to_G_semisimple}, so we may assume $Z(G)$ connected. Then pick fundamental weights $\omega_i\in\Lambda$ of $\check{G}$ corresponding to $\check{\alpha}_i$ for $i\in \cJ$. Say that $\varrho$ \select{satisfies the property} (C) if the following holds. If $i\in \cJ$, $\lambda>\alpha_i$ such that $\omega_i-\lambda$ appears as a weight of the fundamental representation $\VV^{\omega_i}$ of $\check{G}$ then $\bar\kappa(\lambda-\alpha_i)$ is not divisible by $N$ in $\check{\Lambda}$. 
\end{Def}
\index{$Z(G), \omega_i, \VV^{\omega_i}$}

Here is the main result of this section. 
 
\begin{Thm} 
\label{Pp_C_implies_subtop cohomology property} 
If $\varrho$ satisfies the property (C) then the subtop cohomology property is satisfied for $\varrho$. 
\end{Thm}

 The proof of the following is given case by case in Appendix~A. 
 
\begin{Thm} 
\label{Th_114} The quadratic form $\varrho$ satisfies the property (C), and hence the subtop cohomology property, in the following cases:
\begin{itemize}
\item  $G$ is of type $C_2$ or $A_n$ for $n\ge 1$, and $\varrho(\alpha_i)\notin \ZZ$ for any simple coroot $\alpha_i$. 
\item $G$ is of type $B_n, C_n, D_n$ for $n\ge 1$ or $G_2$, and $\varrho(\alpha_i)\notin \frac{1}{2}\ZZ$ for any simple coroot $\alpha_i$. 
\item $G$ is of type $F_4$, and $\varrho(\alpha_i)\notin \frac{1}{2}\ZZ$,  $\varrho(\alpha_i)\notin \frac{1}{3}\ZZ$
for any simple coroot $\alpha_i$.
%% add the case of E_6,E_7,E_8.
\end{itemize}
\end{Thm}

% maybe add here a remark explaining why the results are not optimal, a source for improving them.

\begin{Rem} 
\label{Rem_E_6_E_7_E_8}
Let $G$ be of type $E_n$ with $6\le n\le 8$. As in the proof of Theorem~\ref{Th_114}, one shows that there is a collection of positive integers $d_1,\ldots, d_r$ (depending on $n$) with the following property. If $\varrho(\alpha_i)\notin \frac{1}{d_1}\ZZ, \ldots, \frac{1}{d_r}\ZZ$ for any simple coroot $\alpha_i$ then the property (C) is satisfied for $\varrho$. This collection can be found in principle in a way similar to the one we use for other types, however, this requires a lot of explicit calculations. They could certainly be done with a suitable computer program (like \cite{FK}). 

 In Section~A.2 of Appendix A, we consider $G$ of type $E_8$ and
establish a necessary condition for the property (C). Namely, one needs at least that $\varrho(\alpha_i)\notin \frac{1}{10}\ZZ, \frac{1}{8}\ZZ, \frac{1}{6}\ZZ$ for the property (C) to hold for $\varrho$ in this case. 
\end{Rem}

\ssec{Proof of Theorem~\ref{Pp_C_implies_subtop cohomology property}}
\label{Section_Property (C)}

\sssec{} 
\label{Section_121}
Over $\Gr_B^0\cap \Gr_{B^-}^{-\lambda}$ we get two different trivializations of the $\Gm$-torsor $\Gra_G\to\Gr_G$, the first coming from $\Gr_B^0$, the second one from that over $\Gr_{B^-}^{-\lambda}$. The discrepancy  between the two trivializations is a map $\gamma_G: \Gr_B^0\cap \Gr_{B^-}^{-\lambda}\to \Gm$ that intertwines the natural $T(\cO)$-action on the source with the $T(\cO)$-action on $\Gm$ by the character $T(\cO)\to T\toup{\bar\kappa(\lambda)}\Gm$. To be precise, for the corresponding sections $s^0_B: \Gr_B^0\to \Gra_G$ and $s^{-\lambda}_{B^-}: \Gr_{B^-}^{-\lambda}\to\Gra_G$ one has $s^{-\lambda}_{B^-}=\gamma_G s^0_B$. Note that $s^*_{-\lambda}\cL_G\,\iso\, \ev^*\cL_{\psi}\otimes \gamma_G^*\cL_{\zeta}$. 

 Recall that the restriction of $\ev: \Gr_B^0\cap \Gr_{B^-}^{-\lambda}\to\A^1$ to each irreducible component of $\Gr_B^0\cap \Gr_{B^-}^{-\lambda}$ is dominant (\cite{G}, Section~5.6). So, (\ref{complex_main}) is placed in degrees $\le top-1$. 
 
\sssec{Recollections on crystals}  
\label{Section_Recollections on crystals} 
As in \cite{BFG}, write $B_{\gg}(\lambda)$ for the set of irreducible components of $\Gr_B^0\cap \Gr_{B^-}^{-\lambda}$. One has the structure of a crystal on $B_{\gg}=\cup_{\lambda\ge 0} \, B_{\gg}(\lambda)$ defined in (\cite{BFG}, Sections~13.3-13.4). We recall the part of this crystal structure used in our proof.
 
  For a standard parabolic $P\subset G$ with Levi quotient $M$ let $\gq_P: \Gr_P\to \Gr_M$ be the natural map. Write $B(M)$ and $B^-(M)$ for the corresponding Borel subgroups of $M$. For $\lambda\ge 0$ the scheme $\Gr_B^0\cap \Gr_{B^-}^{-\lambda}$ is stratified by locally closed subschemes
$\Gr_B^0\cap \gq_P^{-1}(\Gr_{B^-(M)}^{-\mu})\cap \Gr_{B^-}^{-\lambda}$ indexed by $0\le_M \mu\le \lambda$. For such $\mu$ and any $g\in \Gr_{B^-(M)}^{-\mu}$ one has an isomorphism
\begin{equation}
\label{iso_decomp_intersection_first}
\Gr_B^0\cap \gq_P^{-1}(\Gr_{B^-(M)}^{-\mu})\cap \Gr_{B^-}^{-\lambda}\,\iso\, (\Gr_{B(M)}^0\cap \Gr_{B^-(M)}^{-\mu})\times (\gq_P^{-1}(g)\cap \Gr_{B^-}^{-\lambda})
\end{equation}
Denote by $B^{\gm, \ast}_{\gg}(\lambda-\mu)$ the set of irreducible components of $\gq_P^{-1}(g)\cap \Gr_{B^-}^{-\lambda}$ of (maximal possible) dimension $\<\lambda-\mu, \check{\rho}\>$. This set is independent of $g\in \Gr_{B^-(M)}^{-\mu}$ in a natural sense (see \select{loc.cit.}). One gets the bijection 
$$
B_{\gg}(\lambda)\,\iso\, \cup_{\mu}  B^{\gm, \ast}_{\gg}(\lambda-\mu)\times B_{\gm}(\mu)
$$
sending an irreducible component $b$ of $\Gr_B^0\cap \Gr_{B^-}^{-\lambda}$ to the pair $(b_1, b_2)$ defined as follows. First, there is a unique $\mu\in\Lambda$ with $0\le_M \mu\le \lambda$ such that $b\cap \gq_P^{-1}(\Gr_{B^-(M)}^{-\mu})$ is dense in $b$. Then $b\cap \gq_P^{-1}(\Gr_{B^-(M)}^{-\mu})$ corresponds via (\ref{iso_decomp_intersection_first}) to $(b_1, b_2)$.
\index{$B_{\gg}(\lambda), B_{\gg}, \gq_P$}\index{$B(M), B^-(M), \gamma_G$}\index{$B^{\gm, \ast}_{\gg}(\lambda-\mu), B_{\gm}(\mu)$}%%%%%%

 For $i\in \cJ$ the operation $f_i: B_{\gg}\to B_{\gg}\cup 0$ is defined as follows. Let $P_i$ be the standard parabolic whose Levi $M_i$ has a unique simple coroot $\alpha_i$. Our convention is that $f_i: B_{\gm_i}\to B_{\gm_i}\cup 0$ sends the unique element of $B_{\gm_i}(\nu)$ to the unique element of $B_{\gm_i}(\nu-\alpha_i)$ for $\nu\ge_{M_i} \alpha_i$ (resp., to $0$ for $\nu=0$). For the corresponding decomposition
$$
B_{\gg}(\lambda)\,\iso\, \cup_{\mu}  B^{\gm_i, \ast}_{\gg}(\lambda-\mu)\times B_{\gm_i}(\mu)
$$
write $b\in B_{\gg}(\lambda)$ as $(b_1, b_2)$. Then $f_i(b_1, b_2)=(b_1, f_i(b_2))$ by definition. 

 For $i\in\cJ$, $b\in B_{\gg}(\nu)$ set $\phi_i(b)=\max\{m\ge 0\mid f_i^m b\ne 0\}$. 

 Let $B(-\infty)$ denote the standard crystal of the canonical base in $U(\check{\gu})$, here $\check{\gu}$ is the Lie algebra of the unipotent radical of the Borel $\check{B}\subset \check{G}$. It coincides with the crystal introduced in (\cite{K}, Remark~8.3). A canonical isomorphism $B_{\gg}\,\iso\, B(-\infty)$ is established in \cite{BFG}. For $\lambda\in\Lambda$ denote by $T_{\lambda}$ the crystal with the unique element of weight $\lambda$, the notation from (\cite{K}, Example 7.3) and (\cite{BauG}, Section~2.2). 
\index{$B(-\infty), \check{\gu}, T_{\lambda}, B(\lambda)$}%
For $\lambda\in\Lambda^+$ denote by $B(\lambda)$ the crystal 
of the canonical base of the irreducible $\check{G}$-representation $\VV^{\lambda}$ of highest weight $\lambda$. We identify it canonically with the crystal denoted by $B^G(\lambda)$ in (\cite{BG1}, Section~3.1).
So, an element of $B(\lambda)$ is an irreducible component of $\Gr_B^{\nu}\cap \Gr_G^{\lambda}$ for some $\nu\in\Lambda$ appearing as a weight of $\VV^{\lambda}$. Recall from (\cite{BauG}, Section~2.2) that for $\lambda\in\Lambda^+$ there is a canonical embedding $B(\lambda)\hook{} T_{w_0(\lambda)}\otimes B(-\infty)$ whose image is
\begin{equation}
\label{image_of_finite_crystal}
\{t_{w_0(\lambda)}\otimes b\mid b\in B(-\infty), \phi_i(b^*)\le -\<w_0(\check{\alpha}_i), \lambda\>\;\mbox{for all}\; i\in\cJ\}
\end{equation}
Here $B(-\infty)\to B(-\infty), b\mapsto b^*$ is the involution defined in (\cite{K}, Section~8.3), see also (\cite{BauG}, Section~2.2). This inclusion is described in the geometric terms in (\cite{BauG}, Proposition~4.3). The involution $*$ is also described in geometric terms as the one coming from an automorphism of $G$ in (\cite{BauG}, Section~4.1, p. 100). 

\sssec{} Let $\bar\mu=\{\mu_i\}_{i\in \cJ}$ with $\mu_i\in \Lambda$, $\lambda\ge \mu_i\ge_{M_i} 0$. We have the corresponding maps $\gq_{P_i}: \Gr_{P_i}\to\Gr_{M_i}$. Set
$$
Y^{\bar \mu}=(
\mathop{\cap}\limits_{i\in \cJ} \, \gq_{P_i}^{-1}(\Gr_{B^-(M_i)}^{-\mu_i}))\cap \Gr_B^0\cap \Gr_{B^-}^{-\lambda}\; .
$$
The scheme $\Gr_B^0\cap \Gr_{B^-}^{-\lambda}$ is stratified by locally closed subshemes $Y^{\bar\mu}$ for the collections $\bar\mu$ as above (some strata could be empty). Our strategy is to show that each stratum $Y^{\bar \mu}$ does not contribute to $top-1$ cohomology in (\ref{complex_main}).
\index{$\gq_{P_i}: \Gr_{P_i}\to\Gr_{M_i}$}

Set $Z^{\bar\mu}=\prod_{i\in \cJ} \Gr_{B(M_i)}^0\cap \Gr_{B^-(M_i)}^{-\mu_i}$. Let 
$$
\gq^{\bar\mu}: Y^{\bar \mu}\to Z^{\bar\mu}
$$ 
be the product of the maps $\gq_{P_i}$. Write $U(M_i)$ for the unipotent radical of $B(M_i)$. For each $i\in \cJ$ define $\ev_i: \Gr_{B(M_i)}^0\to\A^1$ by $\ev_i(uM_i(\cO))=\chi_0(u)$ for $u\in U(M_i)(F)$. We have used here some section $M_i\hook{} P_i$. For $\ev^{\bar\mu}:Z^{\bar\mu}\to \A^1$ given by $\ev^{\bar\mu}=\sum_{i\in \cJ} \ev_i$ the restriction $\ev\mid_{Y^{\bar \mu}}$ equals $\ev^{\bar\mu}\gq^{\bar\mu}$. 
 
 By Definition~\ref{Def_property_C}, we assume $Z(G)$ connected and pick fundamental coweights $\omega_i$ of $G$. Note that $\gamma_G^*\cL_{\zeta}$ is equivariant under the action of $\Ker(T(\cO)\to T)$. If there is $i\in\cJ$ such that $\mu_i\ge_{M_i} 2\alpha_i$ then under the action of $\Ker(\cO^*\toup{\omega_i}T(\cO)\to T)$ the sheaf $\ev_i^*\cL_{\psi}$ on $\Gr_{B(M_i)}^0\cap \Gr_{B^-(M_i)}^{-\mu_i}$ changes by a nontrivial additive character. Therefore, $\ev^*\cL_{\psi}\otimes \gamma_G^*\cL_{\zeta}$ on $Y^{\bar \mu}$ also changes by a nontrivial additive character under the action of this group. So, the integral over this stratum vanishes by (\cite{Ngo}, Lemma~3.3). 
 
 Assume from now on that each $\mu_i$ is either $\alpha_i$ or zero. The stratum $Y^{\bar\mu}$, where $\mu_i=0$ for all $i$, is of dimension $< \<\lambda, \check{\rho}\>$ by (\cite{G}, Section~5.6). 
 
 Consider a stratum $Y^{\bar\mu}$ such that $\mu_i\ne 0$ for precisely $m$ different elements $i\in\cJ$ with $m\ge 2$. Recall that $\Gr_{B(M_i)}^0\cap \Gr_{B^-(M_i)}^{-\alpha_i}\,\iso\, \Gm$. The group $T$ acts transitively on $Z^{\bar\mu}$. Since $\gq^{\bar\mu}$ is $T(\cO)$-equivariant, the dimensions of the fibres of $\gq^{\mu}$ are $\le \<\lambda, \check{\rho}\>-m$. Our claim in this case is reduced to the following. For any $T(\cO)$-equivariant constructible sheaf $F$ on $Z^{\bar\mu}$, the complex $\RG_c(Z^{\bar\mu}, F\otimes (\ev^{\bar\mu})^*\cL_{\psi})$ is placed in degrees $\le m$. This is easy to check.
 
  The only remaining case is the stratum $Y^{\bar\mu}$ such that there is $i\in \cJ$ with $\mu_i=\alpha_i$ and $\mu_j=0$ for $j\ne i$. In particular, $\lambda\ge\alpha_i$. We may assume that $Y^{\bar\mu}$ contains an irreducible component $b$ of dimension $\<\lambda, \check{\rho}\>$, otherwise this stratum does not contribute to $top-1$ cohomology in (\ref{complex_main}). The closure of $b$ in $\Gr_B^0\cap \Gr_{B^-}^{-\lambda}$ is an element $\bar b\in B_{\gg}(\lambda)$ such that $f_j \bar b=0$ for $j\ne i$ and $f_i^2\bar b=0$. The following is derived from (\cite{K}, Proposition~8.2, Section~8.3), see the formula (\ref{image_of_finite_crystal}). 
  
\begin{Pp} 
\label{Pp_Kashiwara}
Pick $i\in \cJ$. If $\nu>0$ and $\bar b\in B_{\gg}(\nu)$ such that $f_j \bar b=0$ for all $j\ne i$, and $f_i^2 \bar b=0$ then $\omega_i-\nu$ appears in the fundamental representation $\VV^{\omega_i}$ of $\check{G}$ with highest weight $\omega_i$. In other words, $w(\omega_i-\nu)\le \omega_i$ for all $w\in W$. 
\end{Pp}  
  We conclude that $\omega_i-\lambda$ appears in $\VV^{\omega_i}$ (for other $\lambda$ the proof is already finished). 
  
   For $P=P_i$ and $g=t^{-\alpha_i}$ the isomorphism (\ref{iso_decomp_intersection_first}) becomes
\begin{equation}
\label{iso_decomp_intresection_first_with_t_mu}  
\Gr_B^0\cap \gq_{P_i}^{-1}(\Gr_{B^-(M_i)}^{-\alpha_i})\cap \Gr_{B^-}^{-\lambda}\,\iso\, (\Gr_{B(M_i)}^0\cap \Gr_{B^-(M_i)}^{-\alpha_i})\times (\gq_{P_i}^{-1}(t^{-\alpha_i})\cap \Gr_{B^-}^{-\lambda})
\end{equation}  
We let $T(\cO)$ act on the right hand side of (\ref{iso_decomp_intresection_first_with_t_mu}) as the product of the natural actions of $T(\cO)$ on the two factors. Then (\ref{iso_decomp_intresection_first_with_t_mu}) is $T(\cO)$-equivariant (see Section~\ref{Section_equivariance_general}). The $\Gm$-torsor $\Gra_G\to\Gr_G$ is constant over $\gq_{P_i}^{-1}(t^{-\alpha_i})$ with fibre $\Omega_{\bar c}^{-\bar\kappa(\alpha_i,\alpha_i)}-0$, and $T(\cO)$ acts on it by the character 
$$
T(\cO)\to T\toup{\bar\kappa(\alpha_i)}\Gm
$$
Pick any trivialization of $\Omega_{\bar c}^{-\bar\kappa(\alpha_i,\alpha_i)}$, let $\bar s_i: \gq_{P_i}^{-1}(t^{-\alpha_i})\to \Gra_G$ be the corresponding section of the $\Gm$-torsor. We get the discrepancy function $\gamma_i: \gq_{P_i}^{-1}(t^{-\alpha_i})\cap \Gr_{B^-}^{-\lambda}\to\Gm$ such that $s^{-\lambda}_{B^-}=\gamma_i \bar s_i$ over $\gq_{P_i}^{-1}(t^{-\alpha_i})\cap \Gr_{B^-}^{-\lambda}$. The map $\gamma_i$ interwines the natural $T(\cO)$-action on $\gq_{P_i}^{-1}(t^{-\alpha_i})\cap \Gr_{B^-}^{-\lambda}$ with the action on $\Gm$ by $T(\cO)\to T\toup{\bar\kappa(\lambda-\alpha_i)}\Gm$. 

 Let $\Gra_{M_i}$ be the restriction of $\Gra_G$ under $\Gr_{M_i}\to \Gr_G$. As for $G$, one defines the discrepancy function $\gamma_{M_i}: \Gr_{B(M_i)}^0\cap \Gr_{B^-(M_i)}^{-\alpha_i}\to\Gm$.
The map
$$
(\Gr_{B(M_i)}^0\cap \Gr_{B^-(M_i)}^{-\alpha_i})\times (\gq_{P_i}^{-1}(t^{-\alpha_i})\cap \Gr_{B^-}^{-\lambda})\toup{\gamma_{M_i}\gamma_i}\Gm
$$
coincides with the restriction of $\gamma_G$. 
  
  There is a $T(\cO)$-invariant subscheme $\cY\subset \gq_{P_i}^{-1}(t^{-\alpha_i})\cap \Gr_{B^-}^{-\lambda}$ such that (\ref{iso_decomp_intresection_first_with_t_mu}) restricts to an isomorphism
$$
Y^{\bar\mu}\,\iso\, (\Gr_{B(M)}^0\cap \Gr_{B^-(M)}^{-\alpha_i})\times \cY
$$  
The contribution of $Y^{\bar\mu}$ becomes
$$
\RG_c(\Gr_{B(M)}^0\cap \Gr_{B^-(M)}^{-\alpha_i}, \ev_i^*\cL_{\psi}\otimes \gamma_{M_i}^*\cL_{\zeta})\otimes \RG_c(\cY, \gamma_i^*\cL_{\zeta})
$$
We have $\dim(\cY)\le \<\lambda, \check{\rho}\>-1$. To finish the proof it suffices to show that $\gamma_i^*\cL_{\zeta}$ is nonconstant on each irreducible component of $\cY$ of dimension $\<\lambda, \check{\rho}\>-1$. This is the case, because the character $\bar\kappa(\lambda-\alpha_i)$ is not divisible by $N$ in $\check{\Lambda}$, so that $\gamma_i^*\cL_{\zeta}$ changes under the $T(\cO)$-action by a nontrivial character. Theorem~\ref{Pp_C_implies_subtop cohomology property} is proved. 

\sssec{Equivariant decomposition} 
\label{Section_equivariance_general} If $G$ is a group scheme, and $f: Y\to Z$ is a $G$-equivariant map such that $G$ acts transitively on $Z$, assume that for any $y\in Y$, the inclusion $Stab_G(y, Y)\subset Stab_G(f(y), Z)$ is an equality. Then a choice of $z\in Z$ yields an isomorphism $\xi: Z\times f^{-1}(z)\,\iso\, Y$. Namely, let $S=Stab_G(z, Z)$. The map $(G/S)\times f^{-1}(z)\to Y$, $(gS, y)\mapsto gy$ is well defined and gives this isomorphism.

 Assume in addition we have a semi-direct product $1\to G\to \tilde G\to H\to 1$ with a section $H\hook{}\tilde G$ as a subgroup. Assume $f$ is in addition $\tilde G$-equivariant. Assume $z\in Z$ is fixed by $H$. Then $SH$ is a subgroup of $\tilde G$ equal to $Stab_{\tilde G}(z, Z)$. So, $H$ acts on $S$ by conjugation. If we identify $G/S\,\iso\, Z$, $gS\mapsto gz$ then the action of $h\in H$ on $gS\in G/S\,\iso\, Z$ sends $gS$ to $hgh^{-1}S$. 
Now $\xi:  Z\times f^{-1}(z)\,\iso\, Y$ becomes $H$-equivariant if we let $h\in H$ act on $Z\times f^{-1}(z)$ as the product of the actions, that is, $h\in H$ acts on $(z_1, y)\in Z\times f^{-1}(z)$ as $(hz_1, hy)$.

\section{The twisted Whittaker category}
\label{Section_The twisted Whittaker category}

\ssec{} 
\label{Sec_2.1.}
The definition of the twisted Whittaker category from (\cite{G}, Section~2) naturally extends to our setting, we give the detailed exposition. For $\lambda\in\Lambda^+$ denote by $\cV^{\lambda}$ the corresponding Weyl module for $G$ as in \cite{J}.  For $n\ge 0$ let $\gM_n$ be the stack classifying: 
\begin{itemize}
\item $(x_1,\ldots, x_n)\in X^n$, a $G$-torsor $\cF$ on $X$,
\item for each $\check{\lambda}\in \check{\Lambda}^+$ a non-zero map
\begin{equation}
\label{maps_kappa_check_lambda}
\kappa^{\check{\lambda}}: \Omega^{\<\check{\lambda}, \rho\>} \to \cV^{\check{\lambda}}_{\cF},
\end{equation}
which is allowed to have any poles at $x_1,\ldots, x_n$. The maps $\kappa^{\check{\lambda}}$ are required to satisfy the Pl\"ucker relations as in \cite{BG}. 
\end{itemize}
For $n=0$ the stack $\gM_n$ is rather denoted by $\gM_{\emptyset}$.
Let $\gp: \gM_n\to\Bun_G$ be the map sending the above point to $\cF$. 

 Let $\cP^{\bar\kappa}$ denote the line bundle $\gp^*(^{\omega}\cL^{\bar\kappa})$ on $\gM_n$. By $\wt\gM_n$ we denote the gerbe of $N$-th roots of $\cP^{\bar\kappa}$ over $\gM_n$. Let $\D_{\zeta}(\gM_n)$ denote the derived category of $\Qlb$-sheaves on $\wt\gM_n$, on which $\mu_N(k)$ acts by $\zeta$. This category does not change (up to an equivalence) if $\bar\kappa$ and $N$ are multiplied by the same integer, so essentially depends only on $\varrho$. 
\index{$\cV^{\lambda}, \gM_n, \kappa^{\check{\lambda}}$}%
\index{$\gM_{\emptyset}, \gp$}\index{$\cP^{\bar\kappa}, \D_{\zeta}(\gM_n)$}
 
\ssec{} 
\label{Secion_1.2}
Pick $y\in X$. Write $D_y$ (resp., $D_y^*$) for the formal disk (resp., punctured formal disk) around $y\in X$. Let $\Omega^{\rho}_B$ be the $B$-torsor on $X$ obtained from $\Omega^{\rho}$ via extension of scalars $T\to B$. Let $^{\omega}\cN$ be the group scheme over $X$ of automorphisms of $\Omega^{\rho}_B$ acting trivially on the induced $T$-torsor. Let $\cN_y^{reg}$ (resp., $\cN_y^{mer}$) be the group scheme (resp., group ind-scheme) of sections of $^{\omega}\cN$ over $D_y$ (resp., $D_y^*$). Recall that
$$
\cN^{mer}_y/[\cN^{mer}_y, \cN^{mer}_y]\,\iso\, \Omega\mid_{D_y^*}\times\ldots\times \Omega\mid_{D_y^*}, 
$$
the product taken over simple roots of $G$. Taking the sum of residues in this product, one gets the character $\chi_y: \cN^{mer}_y\to \A^1$. 

 As in (\cite{G}, Section~2.3) for a collection of distinct points $\bar y:=y_1,\ldots, y_m$ let $\cN^{reg}_{\bar y}$ (resp., $\cN^{mer}_{\bar y}$) denote the product of the corresponding groups $\cN^{reg}_{y_i}$ (resp., $\cN^{mer}_{y_i}$). The sum of the corresponding characters gives the character $\chi_{\bar y}: \cN^{mer}_{\bar y}\to\A^1$. 
 
 Let $(\gM_n)_{\goodat\, \bar y}\subset \gM_n$ be the open substack given by the property that all $x_i$ are different from the points of $\bar y$, and $\kappa^{\check{\lambda}}$ have no zeros at $\bar y$. A point of $(\gM_n)_{\goodat\, \bar y}$ defines a $B$-torsor $\cF_B$ over $D_{\bar y}=\prod_{j=1}^m D_{y_j}$ equipped with a trivialization $\epsilon_B: \cF_B\times_B T\,\iso\, \Omega^{\rho}$ over $D_{\bar y}$. 
\index{$\Omega^{\rho}_B, {^{\omega}\cN}, \cN_y^{reg}, \cN_y^{mer}, \chi_y$} \index{$\cN^{reg}_{\bar y}, \cN^{mer}_{\bar y}, \chi_{\bar y}$}
  
  Let $_{\bar y}\gM_n$ denote the $\cN^{reg}_{\bar y}$-torsor over $(\gM_n)_{\goodat\, \bar y}$ classifying a point of $(\gM_n)_{\goodat\, \bar y}$ as above together with a trivialization $\cF_B\,\iso\, \Omega^{\rho}_B\mid_{D_{\bar y}}$ compatible with $\epsilon_B$.   
  
  Now $_{\bar y}\gM_n$ can be seen as the stack classifying: $(x_1,\ldots, x_n)\in X^n$ different from $\bar y$, a $G$-torsor $\cF$ over $X-\bar y$ with a trivialization $\epsilon_{\cF}: \cF\,\iso\, \Omega^{\rho}_B\times_B G\mid_{D_{\bar y}^*}$, for $\check{\lambda}\in\check{\Lambda}^+$ non-zero maps (\ref{maps_kappa_check_lambda}) over $X-\bar y-\bar x$ satisfying the Pl\"ucker relations and compatible with the trivialization $\epsilon_{\cF}$. Here we denoted $D_{\bar y}^*\,\iso\, \prod_{j=1}^m D_{y_j}^*$. 
\index{$(\gM_n)_{\goodat\, \bar y},  \; {_{\bar y}\gM_n}$}  
  
   The group $\cN^{mer}_{\bar y}$ acts on $_{\bar y}\gM_n$ by changing the trivialization $\epsilon_{\cF}$ via its action on $\Omega^{\rho}_B\mid_{D_{\bar y}^*}$. The composition $_{\bar y}\gM_n\to \gM_n\toup{\gp}\Bun_G$ sends the above point to the gluing of $\cF\mid_{X-\bar y}$ with $\Omega^{\rho}_B\times_B G\mid_{D_{\bar y}}$ via $\epsilon_{\cF}:\cF\,\iso\, \Omega^{\rho}_B\times_B G\mid_{D_{\bar y}^*}$.    

 Denote by $_{\bar y}\cP^{\bar\kappa}$ the restriction of $\cP^{\bar\kappa}$ to $_{\bar y}\gM_n$. As in (\cite{G}, Lemma~2.4), the action of $\cN^{mer}_{\bar y}$ on $_{\bar y}\gM_n$ lifts naturally to an action on $_{\bar y}\cP^{\bar\kappa}$. 
 
 Let $\wt\gM_n$ (resp., $_{\bar y}\wt\gM_n, (\wt\gM_n)_{\goodat\, \bar y}$) be the gerbe of $N$-th roots of the corresponding line bundle $\cP^{\bar\kappa}$ (resp., its restriction). We denote by $\Perv_{\zeta}((\wt\gM_n)_{\goodat\, \bar y})$ the category of perverse sheaves on $(\wt\gM_n)_{\goodat\, \bar y}$, on which $\mu_N(k)$ acts by $\zeta$.
 Write $(\Whit^{\kappa}_n)_{\goodat\, \bar y}$ for the full subcategory of $\Perv_{\zeta}(\wt\gM_n)_{\goodat\, \bar y})$ consisting of perverse sheaves, whose restriction to $_{\bar y}\wt\gM_n$ is $(\cN^{mer}_{\bar y}, \chi_{\bar y}^*\cL_{\psi})$-equivariant (as in \cite{G}, Section~2.5). 
\index{$\wt\gM_n, \;{_{\bar y}\wt\gM_n}, (\wt\gM_n)_{\goodat\, \bar y}$}
\index{$\Perv_{\zeta}((\wt\gM_n)_{\goodat\, \bar y}), \; {_{\bar y}\cP^{\bar\kappa}}$}
 
  If $\bar y'$ and $\bar y''$ are two collections of points, set $\bar y=\bar y'\cup \bar y''$. Over $(\wt\gM_n)_{\goodat\, \bar y}$ one gets the corresponding torsors with respect to each of the groups 
$$
\cN_{\bar y'}^{reg}, \cN^{reg}_{\bar y''}, \cN^{reg}_{\bar y}
$$ 
As in (\cite{G}, Section~2.5), the three full subcategories of $\Perv_{\zeta}((\wt\gM_n)_{\goodat\, \bar y})$ given by the equivariance condition with respect to one of these groups are equal.
 
 Let $\Whit^{\kappa}_n\subset \Perv_{\zeta}(\wt\gM_n)$ be the full subcategory of $F\in \Perv_{\zeta}(\wt\gM_n)$ such that for any $\bar y$ as above, the restriction of $F$ to $(\wt\gM_n)_{\goodat\, \bar y}$ lies in $(\Whit^{\kappa}_n)_{\goodat\, \bar y}$. As in (\cite{G1}, Lemma~4.8), the full subcategory $\Whit^{\kappa}_n\subset \Perv_{\zeta}(\wt\gM_n)$ is stable under sub-quotients and extensions, and is therefore a Serre subcategory. So, we also define the full triangulated subcategory $\D\!\Whit^{\kappa}_n\subset \D_{\zeta}(\wt\gM_n)$ of complexes with
all perverse cohomology lying in $\Whit^{\kappa}_n$. 
\index{$(\Whit^{\kappa}_n)_{\goodat\, \bar y}$}\index{$\Whit^{\kappa}_n$}\index{$\D\Whit^{\kappa}_n$}

  The Verdier duality preserves $\Whit^{\kappa}_n$ (up to replacing $\psi$ by $\psi^{-1}$ and $\zeta$ by $\zeta^{-1}$), because the corresponding action maps are smooth (as in \cite{G1}, Section~4.7).
    
\ssec{} 
\label{Section_2.3}
For a $n$-tuple $\bar\lambda=(\lambda_1,\ldots,\lambda_n)$ of dominant coweights of $G$ let $\gM_{n,\le\bar\lambda}\subset \gM_n$ be the closed substack given by the property that for each $\check{\lambda}\in \check{\Lambda}^+$ the map
\begin{equation}
\label{map_kappa^check_lambda_with_poles}
\kappa^{\check{\lambda}}: \Omega^{\<\rho, \check{\lambda}\>}\to \cV^{\check{\lambda}}_{\cF}(\sum_i\<\lambda_i x_i, \check{\lambda}\>)
\end{equation}
is regular over $X$. For $\bar x=(x_1,\ldots, x_n)\in X^n$ fixed let $\gM_{\bar x}$ denote the fibre of $\gM_n$ over this point of $X^n$. Write $\Whit^{\kappa}_{\bar x}$ for the corresponding version of the Whittaker category of twisted perverse sheaves on $\gM_{\bar x}$. (By a twisted perverse sheaf on a base we mean a perverse sheaf on some gerbe over this base). 
\index{$\gM_{n,\le\bar\lambda}, \gM_{\bar x}$}\index{$\Whit^{\kappa}_{\bar x}$}\index{$\gM_{\bar x,\le\bar\lambda}, \gM_{\bar x, \bar\lambda}, j_{\bar x,\bar\lambda}$}

 Assume $(x_1,\ldots, x_n)$ pairwise different. Define the closed substack $\gM_{\bar x,\le\bar\lambda}\subset \gM_{\bar x}$ as above.
  The irreducible objects of $\Whit^{\kappa}_{\bar x}$ are as follows. Let $\gM_{\bar x, \bar\lambda}\subset \gM_{\bar x,\le\bar\lambda}$ be the open substack given by the property that for each $\check{\lambda}\in \check{\Lambda}^+$ the map (\ref{map_kappa^check_lambda_with_poles}) has no zeros over $X$. Let 
$$
j_{\bar x,\bar\lambda}: \gM_{\bar x, \bar\lambda}\hook{} \gM_{\bar x,\le\bar\lambda}
$$ 
be the corresponding open immersion. Recall that $j_{\bar x,\bar\lambda}$ is affine (\cite{FGV}, Proposition~3.3.1). 

 In the same way, one defines the version of the Whittaker category of twisted perverse sheaves on $\gM_{\bar x,\bar\lambda}$. As in (\cite{G}, Lemma~2.7), this category is non-canonically equivalent to that of vector spaces. Let $\bar\cF_{\bar x,\bar\lambda}$ denote the unique (up to a non-canonical scalar automorphism) irreducible object of this category. As in (\cite{FGV}, Section~4.2.1), one defines a canonical evaluation map $\ev_{\bar x,\bar\lambda}: \gM_{\bar x,\bar\lambda}\to \A^1$. The restriction of the line bundle $\cP^{\bar\kappa}$ to $\gM_{\bar x,\bar\lambda}$ is constant with fibre
\begin{equation}
\label{fibre_of_cP_barkappa_at_T-torsor}
^{\omega}\cL^{\bar\kappa}_{\Omega^{\rho}(-\sum_i \lambda_ix_i)}
\end{equation}
Any trivialization of (\ref{fibre_of_cP_barkappa_at_T-torsor}) yields a trivialization $\wt\gM_{\bar x,\bar\lambda}\,\iso\, \gM_{\bar x,\bar\lambda}\times B(\mu_N)$ of the gerbe $\wt\gM_{\bar x,\bar\lambda}\to\gM_{\bar x,\bar\lambda}$. There is an isomorphism 
$$
\bar\cF_{\bar x,\bar\lambda}\,\iso\, \ev_{\bar x,\bar\lambda}^*\cL_{\psi}\boxtimes\cL_{\zeta}[\dim \gM_{\bar x,\bar\lambda}]
$$ 
For $\bar\lambda=0$ the line (\ref{fibre_of_cP_barkappa_at_T-torsor}) is canonically trivialized. So, $\bar\cF_{\bar x, 0}$ is defined up to a canonical isomorphism.  
\index{$\bar\cF_{\bar x,\bar\lambda}, \ev_{\bar x,\bar\lambda}$}
\index{$\cF_{\bar x, \bar\lambda, "!}, \cF_{\bar x, \bar\lambda, *}, \cF_{\bar x,\bar\lambda}$}

 Let $\cF_{\bar x, \bar\lambda, !}$ (resp., $\cF_{\bar x, \bar\lambda, *}$, $\cF_{\bar x,\bar\lambda}$) denote the extension of $\bar\cF_{\bar x, \bar\lambda}$  by $j_{\bar x,\bar\lambda, !}$ (resp., $j_{\bar x,\bar\lambda, *}$, $j_{\bar x,\bar\lambda, !*}$). Since $j_{\bar x,\bar\lambda}$ is affine, these are perverse sheaves. As in (\cite{FGV}, Proposition~6.2.1), one checks that all of three are objects of $\Whit^{\kappa}_{\bar x}$, and the version of (\cite{G}, Lemma~2.8) holds: 
 
\begin{Lm} 
\label{Lm_irr_objects_of_Whit}
(a) Every irreducible object in $\Whit^{\kappa}_{\bar x}$ is of the form $\cF_{\bar x,\bar\lambda}$ for some $n$-tuple of dominant coweights $\bar\lambda$.\\
(b) The cones of the canonical maps
\begin{equation}
\label{maps_for_!_and_*-objects_in_Lm}
\cF_{\bar x, \bar\lambda, !}\to \cF_{\bar x, \bar\lambda}\to \cF_{\bar x, \bar\lambda, *}
\end{equation}
are extensions of objects $\cF_{\bar x, \bar\lambda'}$ for $\bar\lambda'<\bar\lambda$.
\end{Lm}

 Here the notation $\bar\lambda'<\bar\lambda$ means that $\lambda'_i\le \lambda_i$ for all $1\le i\le n$ and for at least one $i$ the inequality is strict. Recall that the maps (\ref{maps_for_!_and_*-objects_in_Lm}) are not isomorphisms in general. 
Let $\D\!\Whit^{\kappa}_{\bar x}\subset \D_{\zeta}(\wt\gM_{\bar x})$ denote the full subcategory of objects whose all perverse cohomologies lie in $\Whit^{\kappa}_{\bar x}$.  
 
\ssec{} 
\label{Section_1.4}
The basic object of the category $\Whit^{\kappa}_{\emptyset}$ is denoted $\cF_{\emptyset}$. Recall the open substack $\gM_{\emptyset,0}\subset \gM_{\emptyset}$ given by the property that the maps (\ref{maps_kappa_check_lambda}) have neither zeros nor poles over $X$. 
Since there are no dominant weights $<0$, from Lemma~\ref{Lm_irr_objects_of_Whit} we learn that the canonical maps
$$
j_{\emptyset, 0, !}(\cF_{\emptyset,0})\,\iso\,j_{\emptyset, 0, !*}(\cF_{\emptyset,0})\,\iso\, j_{\emptyset, 0, *}(\cF_{\emptyset,0})  
$$
are isomorphisms.
\index{$\D\Whit^{\kappa}_{\bar x}$}\index{$\cF_{\emptyset}, \gM_{\emptyset,0}$}

\ssec{} For $n\ge 0$ and $\mu\in\Lambda$ let $X^{\mu}_n$ be the ind-scheme classifying $(x_1,\ldots, x_n)\in X^n$, and a $\Lambda$-valued divisor $D$ on $X$ of degree $\mu$ which is anti-effective away from $x_1,\ldots, x_n$. This means that for any $\check{\lambda}\in\check{\Lambda}^+$, $\<\check{\lambda}, D\>$ is anti-effective away from $x_1,\ldots, x_n$. 
 
 For $n=0$ we rather use the notation $X^{\mu}_{\emptyset}$ or $X^{\mu}$ instead of $X^{\mu}_0$. If $\mu=-\sum_{i\in\cJ} m_i\alpha_i$ with $m_i\ge 0$ then $X^{\mu}=\prod_i X^{(m_i)}$. 
 
 For a $n$-tuple $\bar\lambda=(\lambda_1,\ldots, \lambda_n)$ of elements of $\Lambda$ denote by $X^{\mu}_{n, \le\bar\lambda}\subset X^{\mu}_n$ the closed subscheme classifying $(x_1,\ldots, x_n, D)\in X^{\mu}_n$ such that 
$$
D-\sum_{i=1}^n \lambda_i x_i
$$ 
is anti-effective over $X$. We have an isomorphism 
$
X^n\times X^{\mu-\lambda_1-\ldots-\lambda_n}\,\iso\, X^{\mu}_{n, \le\bar\lambda}
$ 
sending $(x_1,\ldots, x_n, D')$ to $D'+\sum_{i=1}^n\lambda_i x_i$. For another collection $\bar\lambda'=(\lambda'_1, \ldots,\lambda'_n)$ with $\lambda'_i\ge \lambda_i$ one has a natural closed embedding $X^{\mu}_{n, \le \bar\lambda}\hook{} X^{\mu}_{n, \le\bar\lambda'}$, and 
$$
X^{\mu}_n=\underset{\bar\lambda}{\underset{\longrightarrow}{\lim}}\; 
X^{\mu}_{n,\le\bar\lambda}
$$
\index{$X^{\mu}_n, X^{\mu}, X^{\mu}_{n, \le\bar\lambda}$}

\sssec{} 
\label{Section_151}
By abuse of notation, the restriction of $^{\omega}\cL^{\bar\kappa}$ under $\Bun_T\to \Bun_G$ is still denoted by $^{\omega}\cL^{\bar\kappa}$. 
Let $AJ: X^{\mu}_n\to\Bun_T$ be the Abel-Jacobi map sending $(x_1,\ldots, x_n, D)$ to $\Omega^{\rho}(-D)$. The line bundle $AJ^*(^{\omega}\cL^{\bar\kappa})$ is denoted by $\cP^{\bar\kappa}$ by abuse of notations. 

 Denote by $^{\omega}\cL_{j, \Bun_G}$ the line bundle on $\Bun_G$ whose fibre at $\cF\in\Bun_G$ is $(\cL_{j, \Bun_G})_{\cF}\otimes (\cL_{j, \Bun_G})^{-1}_{\Omega^{\rho}}$. For $D=\sum_x \mu_x x\in X^{\mu}_n$ one has
$$
(^{\omega}\cL_{j, \Bun_G})_{\Omega^{\rho}(-D)}\,\iso\, \otimes_{x\in X} (\Omega^{\frac{1}{2}}_x)^{\kappa_j(\mu_x, \mu_x+2\rho)}
$$
This isomorphism uses a trivialization of all the positive root spaces of $\gg$ that we fix once and for all (they yield also trivializations of all the negative root spaces). 
\begin{Lm} 
\label{Lm_fibre_of^omega_cL_bar_kappa}
For $D=\sum_x \mu_x x\in X^{\mu}_n$ one has
$$
(^{\omega}\cL^{\bar\kappa})_{\Omega^{\rho}(-D)}\,\iso\, \otimes_{x\in X}(\Omega^{\frac{1}{2}}_x)^{-\bar\kappa(\mu_x,\mu_x+2\rho)}\otimes\epsilon^{\bar\mu_x}\,\iso\, (\otimes_{x\in X}(\Omega^{\frac{1}{2}}_x)^{-\bar\kappa(\mu_x,\mu_x+2\rho)})\otimes(\otimes_{i=1}^n \epsilon^{\bar\mu_{x_i}})
$$
where $\bar\mu_x\in\Lambda_{ab}$ is the image of $\mu_x$. 
\end{Lm}
\begin{proof} Use Lemma~\ref{Lm_fibre_of_cL_beta} and the fact that $\epsilon^0$ is trivialized.
\end{proof}
\index{$AJ, {^{\omega}\cL_{j, \Bun_G}}$}\index{$\wt X^{\mu}_n$}
\index{$\Perv_{\zeta}(X^{\mu}_n)$}\index{$\D_{\zeta}(X^{\mu}_n)$}

 Let $\wt X^{\mu}_n$ denote the gerbe of $N$-th roots of $\cP^{\bar\kappa}$ over $X^{\mu}_n$. Write $\Perv_{\zeta}(X^{\mu}_n)$ for the category of perverse sheaves on $\wt X^{\mu}_n$, on which $\mu_N(k)$ acts by $\zeta$. 
Similarly, one has the derived category $\D_{\zeta}(X^{\mu}_n)$.

\ssec{} For $\mu\in\Lambda$ denote by $_{\mu}\gM_n\subset \gM_n$ the ind-substack classifying $(x_1,\ldots, x_n, D)\in X^{\mu}_n$, a $B$-torsor $\cF_B$ on $X$ with an isomorphism $\cF_B\times_B T\,\iso\, \Omega^{\rho}(-D)$. As $\mu$ varies in $\Lambda$ these ind-stacks form a stratification of $\gM_n$. Let $\pi_{\gM}: {_{\mu}\gM_n}\to X^{\mu}_n$ be the map sending the above point to $(x_1,\ldots, x_n, D)$. 

 For a collection $\bar\lambda=(\lambda_1,\ldots, \lambda_n)\in\Lambda^n$ let 
$_{\mu}\gM_{n,\le\bar\lambda}$ be obtained from $_{\mu}\gM_n$ by the base change $\gM_{n,\le\bar\lambda}\to \gM_n$. The map $\pi_{\gM}$ restricts to a morphism still denoted $\pi_{\gM}: {_{\mu}\gM_{n,\le\bar\lambda}}\to X^{\mu}_{n, \le\bar\lambda}$.  

 By the same token, one defines the version of the Whittaker category $\Whit^{\kappa}(_{\mu}\gM_n)\subset \Perv_{\zeta}(_{\mu}\wt\gM_n)$ and its derived version $\D\!\Whit^{\kappa}(_{\mu}\gM_n)\subset \D_{\zeta}(_{\mu}\wt\gM_n)$. 
\index{$_{\mu}\gM_n, \pi_{\gM}, \; {_{\mu}\gM_{n,\le\bar\lambda}}$} 
\index{$\Whit^{\kappa}(_{\mu}\gM_n), \D\Whit^{\kappa}(_{\mu}\gM_n)$}
 
 Let $^{+}X^{\mu}_n\hook{} X^{\mu}_n$ be the closed subscheme given by the condition $\<D, \check{\alpha}\>\ge 0$ for any simple root $\check{\alpha}$ of $G$. Let $^{+}_{\mu}\gM_n$ be the preimage of $^{+}X^{\mu}_n$ in $_{\mu}\gM_n$. As above, we have the natural evaluation map $\ev: {^{+}_{\mu}\gM_n}\to\A^1$. The derived category $\D_{\zeta}(^{+}X^{\mu}_n)$ is defined as in Section~\ref{Section_151}. Since the map $\pi_{\gM}: {_{\mu}\gM_n}\to X^{\mu}_n$ has contractible fibres, as in (\cite{G1}, Proposition~4.13), one gets the following.
 
\begin{Lm} Each object of $\D\!\Whit^{\kappa}(_{\mu}\gM_n)$ is the extension by zero from $^{+}_{\mu}\gM_n$. The functor $\D_{\zeta}(^{+}X^{\mu}_n)\to 
\D\!\Whit^{\kappa}(_{\mu}\gM_n)$ sending $K$ to $\pi_{\gM}^*K\otimes\ev^*\cL_{\psi}$ is an equivalence.
\end{Lm}
\index{$^{+}X^{\mu}_n,\; {^{+}_{\mu}\gM_n}$}\index{$\D_{\zeta}(^{+}X^{\mu}_n)$}

As in (\cite{G1}, Lemma~4.11), one gets the following.

\begin{Lm} i) Let $\mu\in\Lambda$. The $*$ and $!$ restrictions send $\D\!\Whit^{\kappa}_n$ to $\D\!\Whit^{\kappa}(_{\mu}\gM_n)$.\\
ii) The $*$ and $!$ direct images send $\D\!\Whit^{\kappa}(_{\mu}\gM_n)$ to $\D\!\Whit^{\kappa}_n$.\\
iii) An object $K\in \D_{\zeta}(\gM_n)$ lies in $\D\!\Whit^{\kappa}_n$ if and only if its $*$-restrictions (or, equivalently, $!$-restrictions) to all $_{\mu}\gM_n$  belong to $\D\!\Whit^{\kappa}(_{\mu}\gM_n)$.
\end{Lm}

\begin{Rem} i) Consider a point $(x_1,\ldots, x_n, D)\in {^{+}X^{\mu}_n}$. Assume $(y_1,\ldots, y_m)\in X^m$ pairwise different such that $\{y_1,\ldots, y_m\}=\{x_1,\ldots, x_n\}$. Then there is a collection of $G$-dominant coweights $(\mu_1,\ldots, \mu_m)$ such that $D=\sum_{i=1}^m \mu_i y_i$ with $\sum_{i=1}^m \mu_i=\mu$. In particular, $^{+}X^{\mu}_n$ is empty unless $\mu$ is $G$-dominant.\\
ii) Let $\bar x=(x_1,\ldots, x_n)\in X^n$ be a $k$-point with $x_i$ pairwise different. Define $^{+}X^{\mu}_{\bar x}$ as the fibre of $^{+}X^{\mu}_n$ over $\bar x\in X^n$. Let $\bar\lambda\in \Lambda^n$ with $\mu\le\sum_i\lambda_i$. Define the closed subscheme $^{+}X^{\mu}_{\bar x, \le\bar\lambda}$ by the condition $D\le \sum_i\lambda_i x_i$. Then $^{+}X^{\mu}_{\bar x, \le\bar\lambda}$ is a discrete finite set of points.
\end{Rem}
\index{$^{+}X^{\mu}_{\bar x}, \; {^{+}X^{\mu}_{\bar x, \le\bar\lambda}}$}

\subsection{} Let $x\in X$. In Appendix~B we show that the subtop cohomology property admits the following reformulation in terms of $\Whit^{\kappa}_x$.

\begin{Pp} 
\label{Pp_reformulating_subtop_coh_property}
The following properties are equivalent.\\
i) The subtop cohomology property is satisfied for $\varrho$.\\
ii) Let $\lambda>0$, which is not a simple coroot. For $\mu\in\Lambda^{\sharp}$ deep enough in the dominant chamber the complex $j_{x, \mu-\lambda}^*\cF_{x,\mu}$ over $\wt\gM_{x, \mu-\lambda}$ is placed in perverse degrees $\le -2$.\\
iii) Let $\lambda>0$, which is not a simple coroot. For $\mu\in\Lambda^{\sharp}$ deep enough in the dominant chamber one has $\Ext^1(\cF_{x, \mu-\lambda}, \cF_{x, \mu})=0$ in $\Whit^{\kappa}_x$.
\end{Pp}

Based on this proposition, we propose the following.

\begin{Con} Let $\mu<\mu'$ be dominant coweights such that $\mu'-\mu$ is not a simple coroot. Then $\Ext^1(\cF_{x,\mu}, \cF_{x, \mu'})=0$ in $\Whit^{\kappa}_x$.
\end{Con}

\section{The FS category}
\label{Section_The FS category}

\subsection{} The definition of the category of factorizable sheaves from (\cite{G}, Section~3) extends to our setting, we give a detailed exposition for the convenience of the reader.

 For a partition $n=n_1+n_2$, $\mu=\mu_1+\mu_2$ with $\mu_i\in\Lambda$, let 
$$
\add_{\mu_1,\mu_2}: X^{\mu_1}_{n_1}\times X^{\mu_2}_{n_2}\to X^{\mu}_n
$$ 
be the addition map. Given $n_1$-tuple $\bar\lambda_1$, $n_2$-tuple $\bar\lambda_2$ of coweights let
$$
(X^{\mu_1}_{n_1,\le\bar\lambda_1}\times X^{\mu_2}_{n_2,\le\bar\lambda_2})_{disj}
$$
be the open part of the product given by the property that the supports of the two divisors do not intersect. The restriction of $\add_{\mu_1,\mu_2}$ to the above scheme is an \'etale map to $X^{\mu}_{n, \le \bar\lambda_1\cup\bar\lambda_2}$. 
 
 Recall the line bundles $\cP^{\bar\kappa}$ from Sections~\ref{Sec_2.1.}, \ref{Section_151}.
 From Lemma~\ref{Lm_fibre_of^omega_cL_bar_kappa} we obtain the following factorization property
\begin{equation}
\label{eq_factorization_cP}
\add^*_{\mu_1,\mu_2}\cP^{\bar\kappa}\mid_{(X^{\mu_1}_{n_1,\le\bar\lambda_1}\times X^{\mu_2}_{n_2,\le\bar\lambda_2})_{disj}}\,\iso\,
\cP^{\bar\kappa}\boxtimes\cP^{\bar\kappa}\mid_{(X^{\mu_1}_{n_1,\le\bar\lambda_1}\times X^{\mu_2}_{n_2,\le\bar\lambda_2})_{disj}}
\end{equation}
compatible with refinements of partitions.
\index{$\add_{\mu_1,\mu_2}$}\index{$(X^{\mu_1}_{n_1,\le\bar\lambda_1}\times X^{\mu_2}_{n_2,\le\bar\lambda_2})_{disj}$} 
\index{$(X^{\mu_1}\times X^{\mu_2}_n)_{disj}$}\index{$\add_{\mu_1,\mu_2, disj}$}
 
 Let $(X^{\mu_1}\times X^{\mu_2}_n)_{disj}$ denote the ind-subscheme of $X^{\mu_1}\times X^{\mu_2}_n$ consisting of points 
$$
(D_1\in X^{\mu_1}, (\bar x, D_2)\in X^{\mu_2}_n)
$$
such that $D_1$ is disjoint from both $\bar x$ and $D_2$. Let
$
\add_{\mu_1,\mu_2, disj}:  (X^{\mu_1}\times X^{\mu_2}_n)_{disj}\to X^{\mu}_n
$
denote the restriction of $\add_{\mu_1,\mu_2}$. For a $n$-tuple $\bar\lambda$
the restriction is \'etale
$$
\add_{\mu_1,\mu_2, disj}:  (X^{\mu_1}\times X^{\mu_2}_{n, \le \bar\lambda})_{disj}\to X^{\mu}_{n, \le \bar\lambda}.
$$ 
Over $(X^{\mu_1}\times X^{\mu_2}_n)_{disj}$ we get an isomorphism
\begin{equation}
\label{iso_factorization_cP_disj}
\add_{\mu_1,\mu_2, disj}^*\cP^{\bar\kappa}\,\iso\, \cP^{\bar\kappa}\boxtimes\cP^{\bar\kappa}
\end{equation}

\ssec{} 
\label{section_2.4}

 For $\mu\in -\Lambda^{pos}$ let $\oX{}^{\mu}\subset X^{\mu}$ be the open subscheme classifying divisors of the form $D=\sum_k \mu_k y_k$ with $y_k$ pairwise different and each $\mu_k$ being a minus simple coroot. Denote by $j^{diag}: \oX{}^{\mu}\subset X^{\mu}$ the open immersion. 

 If $\alpha$ is a simple coroot then $\bar\kappa(-\alpha, -\alpha+2\rho)=0$. Therefore, $\cP^{\bar\kappa}\mid_{\oX{}^{\mu}}$ is canonically trivialized. We get a canonical equivalence 
$$
\Perv(\oX{}^{\mu})\,\iso\, \Perv_{\zeta}(\oX{}^{\mu})
$$ 
Let $\ocL{}^{\mu}_{\emptyset}\in \Perv_{\zeta}(\oX{}^{\mu})$ be the object corresponding via the above equivalence to the sign local system on $\oX{}^{\mu}$. If $\mu=-\sum m_i\alpha_i$ with $m_i\ge 0$ then the sign local system on $\oX{}^{\mu}$ is by definition the product of sign local systems on $\oX{}^{(m_i)}$ for all $i$. Set
$$
\cL^{\mu}_{\emptyset}=j^{diag}_{!*}(\ocL{}^{\mu}_{\emptyset}),
$$
the intermediate extension being taken in $\Perv_{\zeta}(X^{\mu})$. 
\index{$\oX{}^{\mu}, j^{diag}$}\index{$\Perv_{\zeta}(\oX{}^{\mu})$}\index{$\ocL{}^{\mu}_{\emptyset}, \cL^{\mu}_{\emptyset}$}

 Note that for $\mu=\mu_1+\mu_2$ with $\mu_i\in -\Lambda^{pos}$ we have a canonical isomorphism
\begin{equation}
\label{fact_isom_for_cL^mu_emptyset}  
\add^*_{\mu_1,\mu_2,disj}(\cL^{\mu}_{\emptyset})\,\iso\, \cL^{\mu_1}_{\emptyset}\boxtimes \cL^{\mu_2}_{\emptyset}
\end{equation}
  
\ssec{} 
As in (\cite{G}, Section~3.5), we first define $\wt \FS^{\kappa}_n$ as the category, whose objects are collections $\cL^{\mu}_n\in \Perv_{\zeta}(X^{\mu}_n)$ for each $\mu\in \Lambda$ equipped with the factorization isomorphisms: for any partition $\mu=\mu_1+\mu_2$ with $\mu_2\in\Lambda$, $\mu_1\in -\Lambda^{pos}$ for the map
$$
\add_{\mu_1,\mu_2, disj}: (X^{\mu_1}\times X^{\mu_2}_n)_{disj}\to X^{\mu}_n
$$
we must be given an isomorphism
\begin{equation}
\label{iso_for_def_tilde_FS}
\add_{\mu_1,\mu_2, disj}^*\cL^{\mu}_n\,\iso\, \cL^{\mu_1}_{\emptyset}\boxtimes \cL^{\mu_2}_n
\end{equation}
compatible with refinements of partitions with respect to (\ref{fact_isom_for_cL^mu_emptyset}). 
\index{$\wt\FS^{\kappa}_n$}\index{$\cL^{\mu}_n$}\index{$(X^{\mu_0}\times X^{\mu_1}\times X^{\mu_2}_n)_{disj}$}

 For $\mu_0, \mu_1\in -\Lambda^{pos}, \mu_2\in\Lambda$ let $(X^{\mu_0}\times X^{\mu_1}\times X^{\mu_2}_n)_{disj}$ be the open subscheme classifying $(D_0, D_1, x_1,\ldots, x_n, D_2)\in X^{\mu_0}\times X^{\mu_1}\times X^{\mu_2}_n$ such that $D_0, D_1$ are mutually disjoint and disjoint with $\bar x, D_2$. Compatibility with refinements of partitions means that for $\mu=\mu_1+\mu_2$ the diagram
$$
\begin{array}{ccc}
(X^{\mu_0}\times X^{\mu_1}\times X^{\mu_2}_n)_{disj} & \to & (X^{\mu_0+\mu_1}\times X^{\mu_2}_n)_{disj}\\
\downarrow && \downarrow\\
(X^{\mu_0}\times X^{\mu}_n)_{disj} & \to & X^{\mu_0+\mu}_n
\end{array}
$$
yields the commutative diagram of isomorphisms over $(X^{\mu_0}\times X^{\mu_1}\times X^{\mu_2}_n)_{disj}$
$$
\begin{array}{ccc}
\cL^{\mu_0+\mu}_n & \iso &  \cL^{\mu_0}_{\emptyset}\boxtimes \cL^{\mu}_n\\
\downarrow && \downarrow\\
\cL^{\mu_0+\mu_1}_{\emptyset}\boxtimes \cL^{\mu_2}_n & \toup{(\ref{fact_isom_for_cL^mu_emptyset})} & \cL^{\mu_0}_{\emptyset}\boxtimes \cL^{\mu_1}_{\emptyset}\boxtimes \cL^{\mu_2}_n,
\end{array}
$$
where to simplify the notations we omited the corresponding functors $\add^*$.  

 A morphism from a collection $\{^1\cL^{\mu}_n\}$ to another collection $\{^2\cL^{\mu}_n\}$ is a collection of maps $^1\cL^{\mu}_n\to {^2\cL^{\mu}_n}$ in $\Perv_{\zeta}(X^{\mu}_n)$ compatible with the isomorphisms (\ref{iso_for_def_tilde_FS}). 
  
  Let $j^{poles}: \dot{X}^n\hook{} X^n$ be the complement to all the diagonals. For $\mu\in\Lambda$ set $X^{\mu}_{\dot{n}}=X^{\mu}_n\times_{X^n} \dot{X}^n$. 
  By the same token, one defines the category $\wt\FS^{\kappa}_{\dot{n}}$ consisting of collections $\cL^{\mu}_n\in \Perv_{\zeta}(X^{\mu}_{\dot{n}})$ with factorization isomorphisms. Both $\wt \FS^{\kappa}_n$ and $\wt\FS^{\kappa}_{\dot{n}}$ are abelian categories.
  
  We have the restriction functor
$(j^{poles})^*: \wt \FS^{\kappa}_n\to \wt \FS^{\kappa}_{\dot{n}}$ and its left adjoint
$$
j^{poles}_!: \wt \FS^{\kappa}_{\dot{n}}\to \wt \FS^{\kappa}_n
$$
well-defined because $j^{poles}$ is an affine open embedding.
\index{$j^{poles}: \dot{X}^n\hook{} X^n, \; X^{\mu}_{\dot{n}}$}\index{$\wt\FS^{\kappa}_{\dot{n}}, \; \vartriangle_{\bar n}, \dot{\vartriangle}_{\bar n}$}

 If $\bar n=n_1+\ldots+n_k$ is a partition of $n$, let $\vartriangle_{\bar n}: X^k\to X^n$ and $\dot{\vartriangle}_{\bar n}:\dot{X}^k\to X^n$ be the corresponding diagonal and its open subscheme. We have the natural functors
$$
(\vartriangle_{\bar n})_!: \wt\FS^{\kappa}_k\to \wt\FS^{\kappa}_n\;\;\;\;\mbox{and}\;\;\;\; (\dot{\vartriangle}_{\bar n})_!: \wt\FS^{\kappa}_{\dot{k}}\to \wt\FS^{\kappa}_n
$$
The corresponding restriction functors are well-defined on the level of derived categories (the latter are understood as the derived categories of the corresponding abelian categories):
$$
(\vartriangle_{\bar n})^*: \D(\wt\FS^{\kappa}_n)\to \D(\wt\FS^{\kappa}_k)\;\;\;\;\mbox{and}\;\;\;\; (\dot{\vartriangle}_{\bar n})^*: \D(\wt\FS^{\kappa}_n)\to \D(\wt\FS^{\kappa}_{\dot{k}})
$$
They coincide with the same named functors on the level of derived categories of $\Qlb$-sheaves on the corresponding gerbes.

\ssec{} For a $k$-scheme $Y$ and $F\in \D(Y)$ we denoted by $\SiSu(F)$ the singular support of $F$ in the sense of Beilinson \cite{Be}. Define the full subcategory $\FS^{\kappa}_n\subset \wt\FS^{\kappa}_n$ as follows. A collection $\cL_n\in \wt\FS^{\kappa}_n$ lies in $\FS^{\kappa}_n$ if the following conditions are satisifed:

\begin{itemize}
\item[(i)] $\cL^{\mu}_n$ may be nonzero only for $\mu$ belonging to finitely many cosets in $\pi_1(G)$. For each $\tau\in \pi_1(G)$ there is a collection $\bar\nu=(\nu_1,\ldots,\nu_n)\in\Lambda^n$ with $\sum_i \nu_i=\tau\in \pi_1(G)$ such that for any $\mu\in\Lambda$ over $\tau$ the perverse sheaf $\cL^{\mu}_n$ is the extension by zero from $X^{\mu}_{n,\le\bar\nu}$. 

\item[(ii)] The second condition is first formulated over $\dot{X}^n$, that is, we first define the subcategory $\FS^{\kappa}_{\dot{n}}\subset \wt\FS^{\kappa}_{\dot{n}}$. Let $\cL_{\dot{n}}\in\wt\FS^{\kappa}_{\dot{n}}$, $\mu\in \Lambda$ and $\bar\nu\in\Lambda^n$ with $\sum_i\nu_i=\mu\in \pi_1(G)$ such that $\cL^{\mu}_{\dot{n}}$ is the extension by zero from $\wt X^{\mu}_{\dot {n}, \le \bar\nu}$. Then there are only finitely many collections $(\mu_1,\ldots,\mu_n)\in\Lambda^n$ with $\sum_i \mu_i=\mu$ such that $\SiSu(\cL^{\mu}_{\dot{n}})$ contains the conormal to the subscheme $\dot{X}^n\hook{} X^{\mu}_{\dot {n}, \le \bar\nu}$, $(x_1,\ldots, x_n)\mapsto \sum_i \mu_i x_i$. 

 Now the condition (ii) over $X^n$ is that for any partition $n=n_1+\ldots+n_k$ each of the cohomologies of $(\dot{\vartriangle}_{\bar n})^*(\cL_n)$, which is an object of $\wt\FS^{\kappa}_{\dot{k}}$, belongs to $\FS^{\kappa}_{\dot{n}}$. 
\end{itemize}
\index{$\SiSu(F), \FS^{\kappa}_n$}

\ssec{} For $\bar x=(x_1,\ldots, x_n)\in X^n$ fixed let $X^{\mu}_{\bar x}$ denote the fibre of $X^{\mu}_n$ over $\bar x\in X^n$. In a similar way, one introduces the abelian category $\wt\FS^{\kappa}_{\bar x}$. We define $\FS^{\kappa}_{\bar x}$ as the full subcategory of objects of finite length in $\wt\FS^{\kappa}_{\bar x}$. As in Section~\ref{section_2.4}, one defines the category $\Perv_{\zeta}(X^{\mu}_{\bar x})$. 
\index{$X^{\mu}_{\bar x}$}\index{$\wt\FS^{\kappa}_{\bar x}$}\index{$\FS^{\kappa}_{\bar x}, \Perv_{\zeta}(X^{\mu}_{\bar x})$}

 Pick $\bar x\in X^n$ with $x_i$ pairwise distinct. Let $\bar\lambda=(\lambda_1,\ldots,\lambda_n)$ be a $n$-tuple of elements of $\Lambda$. For $\mu\in \Lambda$ with $(\sum_i\lambda_i)-\mu\in\Lambda^{pos}$ consider the closed subscheme $X^{\mu}_{\bar x,\le\bar\lambda}=X^{\mu}_{\bar x}\cap X^{\mu}_{n,\le\bar\lambda}$. Let $X^{\mu}_{\bar x, =\bar\lambda}\subset X^{\mu}_{\bar x,\le\bar\lambda}$ be the open subscheme classifying divisors of the form
$$
(\sum_{i=1}^n \lambda_ix_i)-D',
$$
where $D'$ is $\Lambda^{pos}$-valued divisor on $X$ of degree $(\sum_i \lambda_i)-\mu$, and $x_i$ is not in the support of $D'$ for any $1\le i\le n$. One similarly defines the categories $\Perv_{\zeta}(X^{\mu}_{\bar x,\le\bar\lambda})$ and $\Perv_{\zeta}(X^{\mu}_{\bar x, =\bar\lambda})$. 
Let 
$$
\oX{}^{\mu}_{\bar x,\le\bar\lambda}\subset X^{\mu}_{\bar x, =\bar\lambda}
$$ 
be the open subscheme given by requiring that $D'$ is of the form $D'=\sum \mu_ky_k$, where $y_k$ are pairwise distinct, and each $\mu_k$ is a simple coroot of $G$. Here, of course, $y_i$ is different from all the $x_i$. 
Denote the corresponding open immersions by
$$
\oX{}^{\mu}_{\bar x,\le\bar\lambda}\;\toup{'j^{poles}}\;  X^{\mu}_{\bar x, =\bar\lambda}\;\toup{''j^{poles}}\; X^{\mu}_{\bar x,\le\bar\lambda}
$$
\index{$X^{\mu}_{\bar x,\le\bar\lambda}, X^{\mu}_{\bar x, =\bar\lambda}$}\index{$\Perv_{\zeta}(X^{\mu}_{\bar x,\le\bar\lambda}), \Perv_{\zeta}(X^{\mu}_{\bar x, =\bar\lambda})$}
\index{$\oX{}^{\mu}_{\bar x,\le\bar\lambda},  \; {'j^{poles}}, {''j^{poles}}$}

\begin{Lm} The restriction of $\cP^{\bar\kappa}$ to $\oX{}^{\mu}_{\bar x,\le\bar\lambda}$ is trivial with fibre
\begin{equation}
\label{line_for_Lm271}
\otimes_{i=1}^n (\Omega_x^{\frac{1}{2}})^{-\bar\kappa(\lambda_i, \lambda_i+2\rho)}\otimes\epsilon^{\bar\lambda_i},
\end{equation}
where $\bar\lambda_i\in\Lambda_{ab}$ is the image of $\lambda_i$.
\end{Lm}
\begin{proof}
If $\alpha$ is a simple coroot then $\bar\kappa(-\alpha, -\alpha+2\rho)=0$. Now apply Lemma~\ref{Lm_fibre_of^omega_cL_bar_kappa}. 
\end{proof}  
  
  If $(\sum_i \lambda_i)-\mu=\sum_{j\in \cJ} m_j\alpha_j$ then $\prod_{j\in\cJ} X^{(m_j)}\,\iso\, X^{\mu-\sum_i\lambda_i}$ via the map sending $\{D_j\}_{j\in \cJ}$ to $-\sum_{j\in \cJ} D_j\alpha_j$. 
  
  We have an open immersion $j^{\mu}_{\bar\lambda}: X^{\mu}_{\bar x, =\bar\lambda}\hook{} X^{\mu-\sum_i\lambda_i}$ sending $D$ to $D-\sum_{i=1}^n \lambda_i x_i$. The line bundle $\cP^{\bar\kappa}$ over $X^{\mu}_{\bar x, =\bar\lambda}$ identifies with the tensor product of $(j^{\mu}_{\bar\lambda})^*\cP^{\bar\kappa}$ with (\ref{line_for_Lm271}).
So, for any trivialization of the line (\ref{line_for_Lm271}), we get the restriction functor 
$$
\Perv_{\zeta}(X^{\mu-\sum_i\lambda_i})\to \Perv_{\zeta}(X^{\mu}_{\bar x, =\bar\lambda})
$$ 
We denote by $\ocL{}^{\mu}_{\bar x,\bar\lambda}$ the image of $\cL^{\mu-\sum_i\lambda_i}_{\emptyset}$ under the latter functor. So, $\ocL{}^{\mu}_{\bar x,\bar\lambda}$ is defined up to a non-unique scalar automorphism. 
Set
$$
\cL^{\mu}_{\bar x,\bar\lambda, !}={''j^{poles}_!} (\ocL{}^{\mu}_{\bar x,\bar\lambda}),\;\;\;\;\;\; \cL^{\mu}_{\bar x,\bar\lambda}={''j^{poles}_{!*}} (\ocL{}^{\mu}_{\bar x,\bar\lambda}),\;\;\;\;\;\; \cL^{\mu}_{\bar x,\bar\lambda, *}={''j^{poles}_*} (\ocL{}^{\mu}_{\bar x,\bar\lambda})
$$
\index{$j^{\mu}_{\bar\lambda}, \;\ocL{}^{\mu}_{\bar x,\bar\lambda}$}%
\index{$\cL^{\mu}_{\bar x,\bar\lambda, "!},  \;\cL^{\mu}_{\bar x,\bar\lambda}, \;\cL^{\mu}_{\bar x,\bar\lambda, *}$}
Define the collection $\cL_{\bar x, \bar\lambda, !}=\{\cL_{\bar x, \bar\lambda, !}^{\mu}\}_{\mu\in \Lambda}$ by the property
$$
\cL_{\bar x, \bar\lambda, !}^{\mu}=\left\{
\begin{array}{cl} 
\cL^{\mu}_{\bar x,\bar\lambda, !}, & \mu\in (\sum_i \lambda_i)-\Lambda^{pos}\\
0, & \mbox{otherwise}
\end{array}
\right.
$$
It is understood that we use the same trivialization of (\ref{line_for_Lm271}) for all $\mu$ in the above formula. One similarly defines the collections $\cL_{\bar x, \bar\lambda}, \cL_{\bar x, \bar\lambda, *}$. All the three are objects of $\wt\FS^{\kappa}_{\bar x}$. 
\index{$\cL_{\bar x, \bar\lambda, "!}, \;\cL_{\bar x, \bar\lambda}, \;\cL_{\bar x, \bar\lambda, *}$}

\begin{Lm} i) For any irreducible object $F$ of $\wt\FS^{\kappa}_{\bar x}$ there is a collection $\bar\lambda\in\Lambda^n$ such that $F$ is isomorphic to $\cL_{\bar x, \bar\lambda}$. \\
ii) The kernels and cokernels of the natural maps
$$
\cL_{\bar x, \bar\lambda, !}\to \cL_{\bar x, \bar\lambda}\to \cL_{\bar x, \bar\lambda, *}
$$
in $\wt\FS^{\kappa}_{\bar x}$ are extensions of objects of the form $\cL_{\bar x, \bar\lambda'}$ for $\bar\lambda'<\bar\lambda$.
\end{Lm}
\begin{proof}
i) Let $\bar\lambda\in \Lambda^n$ be such that the $*$-fibre of $F$ at $\sum_{i=1}^n\lambda_ix_i\in X^{\mu}_{\bar x}$ is nonzero for some $\mu\in\Lambda$. We may assume (changing $\bar\lambda$ if necessary) that for any $\nu\in \Lambda$ with $\nu=\mu$ in $\pi_1(G)$ the twisted perverse sheaf $F^{\nu}\in \Perv_{\zeta}(X^{\nu}_{\bar x})$ is the extension by zero from $X^{\nu}_{\bar x, \le\lambda}$. Then from the factorization property we see that we must have $F\,\iso\, \cL_{\bar x, \bar\lambda}$.
\end{proof}

\begin{Lm} Let $\bar x=(x_1,\ldots, x_n)$ with $x_i$ pairwise different, $\bar\lambda\in\Lambda^n$. Then the objects $\cL_{\bar x,\bar\lambda, !}$, $\cL_{\bar x,\bar\lambda, *}\in \wt \FS^{\kappa}_{\bar x}$ are of finite length.
\end{Lm}
\begin{proof}
Set $\tilde\kappa=-\sum_{j\in J} c_j\kappa_j$. Write $D\in X^{\mu}_{\bar x, \le\bar\lambda}$ as $D=(\sum_{y\in X} \mu_y y)+\sum_{i=1}^n \lambda_i x_i$ with $\mu_y\in -\Lambda^{pos}$ for all $y\in X$. Denote by $\cP^{\tilde\kappa}$ the line bundle on $X^{\mu}_{\bar x, \le\bar\lambda}$ whose fibre at the above point $D$ is
$$
\otimes_{y\in X} (\Omega_y^{\frac{1}{2}})^{-\tilde\kappa(\mu_y, \mu_y+2\rho)}
$$
The line bundle $\cP^{\bar\kappa}\otimes (\cP^{\tilde\kappa})^{-1}$ on the scheme $X^{\mu}_{\bar x, \le\bar\lambda}$ is trivial. So, it suffices to prove our claim under the assumption $\beta=0$. The latter is done in (\cite{G}, Lemma~3.8(b)).
\end{proof}

\section{Zastava spaces}

\ssec{} 
\label{section_4.1_Zastava}
Our purpose is to construct an exact functor $\Whit^{\kappa}_n\to \wt \FS^{\kappa}_n$. We first adopt the approach from (\cite{G}, Section~4) to our setting, it produces an approximation of the desired functor. We will further correct it to get the desired one. 
 
 For $\mu\in \Lambda$ let $\Bun_{B^-}^{\mu}$ denote the connected component of $\Bun_{B^-}$ classifying $B^-$-torsors on $X$ such that the induced $T$-torsor is of degree $(2g-2)\rho-\mu$. Recall that a point of $\Bun_{B^-}^{\mu}$ can be seen as a collection: a $G$-torsor $\cF$ on $X$, a $T$-torsor $\cF_T$ on $X$ of degree $(2g-2)\rho-\mu$, a collection of surjective maps of coherent sheaves
$$
\kappa^{\check{\lambda},-}: \cV^{\check{\lambda}}_{\cF}\to \cL^{\check{\lambda}}_{\cF_T}, \;\; \check{\lambda}\in \check{\Lambda}^+
$$
satisfying the Pl\"ucker relations. Define $\gp^-, \gq^-$ as the projections in the diagram
$$
\Bun_G\getsup{\gp^-} \Bun_{B^-}^{\mu} \toup{\gq^-} \Bun_T
$$
The line bundle $(\gp^-)^*(^{\omega}\cL^{\bar\kappa})$ is denoted by $\cP^{\bar\kappa}$ by abuse of notations. One has naturally $\cP^{\bar\kappa}\,\iso\, (\gq^-)^*(^{\omega}\cL^{\bar\kappa})$. 
\index{$\Bun_{B^-}^{\mu}, \;\gp^-, \;\gq^-, \;\cP^{\bar\kappa}$}

 Denote by $\cZ^{\mu}_n\subset \gM_n\times_{\Bun_G} \Bun^{\mu}_{B^-}$ the open substack given by the property that for each $G$-dominant weight $\check{\lambda}$ the composition
\begin{equation}
\label{maps_def_of_cZ^mu_n}
\Omega^{\<\check{\lambda}, \rho\>} \;\toup{\kappa^{\check{\lambda}}} \;\cV^{\check{\lambda}}_{\cF} \;\toup{\kappa^{\check{\lambda}, -}} \;\cL^{\check{\lambda}}_{\cF_T},
\end{equation}
which is a map over $X-\cup_i x_i$, is not zero. Let $'\gp$, $'\gp_B$ denote the projections in the diagram
$$
\gM_n \;\getsup{'\gp} \;\cZ^{\mu}_n \;\toup{'\gp_B} \;\Bun^{\mu}_{B^-}
$$  
Let $\pi^{\mu}: \cZ^{\mu}_n\to X^{\mu}_n$ be the map sending the above point to $(x_1,\ldots, x_n, D)$ such that the maps (\ref{maps_def_of_cZ^mu_n}) induce an isomorphism $\Omega^{\rho}(-D)\,\iso\, \cF_T$.
\index{$\cZ^{\mu}_n, {'\gp}, {'\gp_B}$}\index{$\pi^{\mu}, \cZ^{\mu}_{n,\le\bar\lambda}$}

 For any $n$-tuple $\bar\lambda\in \Lambda^n$ define the closed substack $\cZ^{\mu}_{n,\le\bar\lambda}$ by the base change $\gM_{n,\le \bar\lambda}\hook{} \gM_n$. The map $\pi^{\mu}$ restricts to a map
$$
\pi^{\mu}: \cZ^{\mu}_{n,\le\bar\lambda}\to X^{\mu}_{n,\le\bar\lambda}
$$ 
However, the preimage of $X^{\mu}_{n,\le\bar\lambda}$ under $\pi^{\mu}: \cZ^{\mu}_n\to X^{\mu}_n$ is not $\cZ^{\mu}_{n,\le\bar\lambda}$. 

\begin{Rem} 
\label{Rem311}
For $\mu\in\Lambda$ let $\Gr_{^{\omega}\cN^-, X^{\mu}_n}$ be  the ind-scheme classifying $(x_1,\ldots, x_n, D)\in X^{\mu}_n$, a $B^-$-torsor $\cF$ on $X$ with compatible isomorphisms $\cF\times_{B^-} T\,\iso\, \Omega^{\rho}(-D)$ over $X$ and $\cF\,\iso\, \Omega^{\rho}\times_T B^-\mid_{X-D-\cup_i x_i}$. We have a closed immersion $\cZ^{\mu}_n\hook{} \Gr_{^{\omega}\cN^-, X^{\mu}_n}$ given by the property that the corresponding maps
$$
\Omega^{\<\rho, \check{\lambda}\>}\to \cV^{\check{\lambda}}_{\cF}
$$
for $\check{\lambda}\in \check{\Lambda}^+$ are regular over $X-\cup_i x_i$. 
Since the projection $\Gr_{^{\omega}\cN^-, X^{\mu}_n}\to X^{\mu}_n$ is ind-affine, the map $\pi^{\mu}: \cZ^{\mu}_n\to X^{\mu}_n$ is also ind-affine. 
\end{Rem}

\ssec{} The ind-scheme $\cZ^{\mu}_0$ is rather denoted $\cZ^{\mu}$. Recall that for $\mu_1\in -\Lambda^{pos}, \mu_2\in \Lambda$ and $\mu=\mu_1+\mu_2$ we have the factorization property (\cite{G}, Proposition~4.7)
\begin{equation}
\label{iso_factorization_cZ} 
(X^{\mu_1}\times X^{\mu_2}_n)_{disj}\times_{X^{\mu}_n} \cZ^{\mu}_n\,\iso\, 
(X^{\mu_1}\times X^{\mu_2}_n)_{disj}\times_{(X^{\mu_1}\times X^{\mu_2}_n)}
(\cZ^{\mu_1}\times \cZ^{\mu_2}_n)
\end{equation}
  
  Recall that the diagram commutes
\begin{equation}
\label{diag_for_the_functor_FF}
\begin{array}{cccc}
\gM_n \getsup{'\gp} &
\cZ^{\mu}_n & \toup{'\gp_B} & \Bun^{\mu}_{B^-}\\
&\downarrow\lefteqn{\scriptstyle \pi^{\mu}} &&  \downarrow\lefteqn{\scriptstyle \gq^-}\\
&X^{\mu}_n & \toup{AJ} & \Bun_T
\end{array}
\end{equation} 
and $('\gp)^*\cP^{\bar\kappa}\,\iso\, (\pi^{\mu})^*\cP^{\bar\kappa}$ canonically, this line bundle is also denoted $\cP^{\bar\kappa}$. Let $\wt\cZ^{\mu}_n$ denote the gerbe of $N$-th roots of $\cP^{\bar\kappa}$ over $\cZ^{\mu}_n$, let $\D_{\zeta}(\cZ^{\mu}_n)$ denote the corresponding derived category of twisted $\Qlb$-sheaves.
\index{$\cZ^{\mu}, \cP^{\bar\kappa}$}\index{$\wt\cZ^{\mu}_n, \D_{\zeta}(\cZ^{\mu}_n)$}

This allows to define the following functors. First, we have the functor $F^{\mu}: \D_{\zeta}(\gM_n)\to \D_{\zeta}(\cZ^{\mu}_n)$ given by
$$
F^{\mu}(K)=('\gp)^*K[\dimrel('\gp)]
$$
As in (\cite{G}, Section~4.8), this functor commutes with the Verdier duality for $\mu$ satisfying $\<\mu, \check{\alpha}\><0$ for any simple root $\check{\alpha}$. Using the factorization property, we will be able to assume that $\mu$ satisfies the latter inequality, so this functor \select{essentially always} commutes with the Verdier duality. We get the functor $\FF:\D_{\zeta}(\gM_n)\to \D_{\zeta}(X^{\mu}_n)$ given by
$$
\FF(K)=\pi^{\mu}_!('\gp)^*(K)[\dimrel('\gp)]
$$
\index{$F^{\mu}, \FF$}

\ssec{} The analog of (\cite{G}, Proposition~4.13) holds in our setting:

\begin{Pp} 
\label{Pp_331_factorization_of_F^mu}
Let $\mu_1\in -\Lambda^{pos}, \mu_2\in\Lambda$, $\mu=\mu_1+\mu_2$ and $\cF\in \Whit^{\kappa}_n$. Under the isomorphism (\ref{iso_factorization_cZ}), the complex
$$
\add_{\mu_1,\mu_2, disj}^*F^{\mu}(\cF)\in \D_{\zeta}((X^{\mu_1}\times X^{\mu_2}_n)_{disj}\times_{X^{\mu}_n} \cZ^{\mu}_n)
$$
identifies with
$$
F^{\mu_1}(\cF_{\emptyset})\boxtimes F^{\mu_2}(\cF)\in \D_{\zeta}((X^{\mu_1}\times X^{\mu_2}_n)_{disj}\times_{(X^{\mu_1}\times X^{\mu_2}_n)}
(\cZ^{\mu_1}\times \cZ^{\mu_2}_n))
$$
\end{Pp}
\begin{proof}
We write down the complete proof for the convenience of the reader and to correct some misprints in (\cite{G}, proof of Proposition~4.13). Set $\ocZ{}^{\mu_1}=\cZ^{\mu_1}\times_{\gM_{\emptyset}} \gM_{\emptyset,0}$. Let $(\gM_n)_{\goodat\, \mu_1}\subset X^{\mu_1}\times \gM_n$ be the open substack given by the property that $D\in X^{\mu_1}$ does not contain pole points $(x_1,\ldots, x_n)$, and all $\kappa^{\check{\lambda}}$ are morphisms of vector bundles in a neighbourhood of $\supp(D)$. 

 Let $\cN^{reg}_{\mu_1}$ (resp., $\cN^{mer}_{\mu_1}$) be the group scheme (resp., group ind-scheme) over $X^{\mu_1}$, whose fibre at $D$ is the group scheme (resp., group ind-scheme) of sections of $^{\omega}\cN$ over the formal neighbourhood of $D$ (resp., the punctured formal neighbourhood of $D$). As in Section~\ref{Secion_1.2}, we have the character $\chi_{\mu_1}: \cN^{mer}_{\mu_1}\to\A^1$.  
\index{$\ocZ{}^{\mu_1}$}\index{$(\gM_n)_{\goodat\, \mu_1}$}\index{$\cN^{reg}_{\mu_1}, \cN^{mer}_{\mu_1}, \act_{\mu_1}$}
 
 For a point of $(\gM_n)_{\goodat\, \mu_1}$ we get a $B$-torsor $\cF_B$ over the formal neighbourhood $\bar D$ of $D$ with a trivialization $\epsilon_B: \cF_B\times_B T\,\iso\, \Omega^{\rho}$ over $\bar D$. Let $_{\mu_1}\gM_n$ denote the $\cN^{reg}_{\mu_1}$-torsor over $(\gM_n)_{\goodat\, \mu_1}$
classifying a point of $(\gM_n)_{\goodat\, \mu_1}$ together with a trivialization 
$\cF_B\,\iso\, \Omega^{\rho}_B\mid_{\bar D}$ compatible with $\epsilon_B$.  
The group ind-scheme $\cN^{mer}_{\mu_1}$ acts on $_{\mu_1}\gM_n$ over $X^{\mu_1}$, this action lifts naturally to an action on $\cP^{\bar\kappa}$. Let
$$
\act_{\mu_1}: \cN^{mer}_{\mu_1}\times^{\cN^{reg}_{\mu_1}} (_{\mu_1}\gM_n)\to  (\gM_n)_{\goodat\, \mu_1}
$$
be the action map. For each $\cF\in\Whit^{\kappa}_n$ one has an isomorphism of twisted perverse sheaves
$$
\act_{\mu_1}^*(\cF)\,\iso\, \chi_{\mu_1}^*\cL_{\psi}\boxtimes \cF
$$
As the fibre $\cN^{mer}_{\mu_1}/\cN^{reg}_{\mu_1}$ at $D\in X^{\mu_1}$ can be written as an inductive system of affine spaces, the above system of isomorphisms makes sense, see (\cite{G1}, Section~4).

The preimage of $(\gM_n)_{\goodat\, \mu_1}$ under the map
$$
(X^{\mu_1}\times X^{\mu_2}_n)_{disj}\times_{X^{\mu}_n} \cZ^{\mu}_n\toup{'\gp} X^{\mu_1}\times \gM_n
$$
goes over under the isomorphism (\ref{iso_factorization_cZ}) to 
\begin{equation}
\label{open_substack_for_Pp331}
(X^{\mu_1}\times X^{\mu_2}_n)_{disj}\times_{(X^{\mu_1}\times X^{\mu_2}_n)}(\ocZ{}^{\mu_1}\times \cZ^{\mu_2}_n)
\end{equation}

Note that $\cN^{mer}_{\mu_1}/\cN^{reg}_{\mu_1}$ can be seen as the ind-scheme classifying $D\in X^{\mu_1}$, a $B$-torsor $\cF$ on $X$ with compatible isomorphisms $\cF\times_B T\,\iso\, \Omega^{\rho}$ over $X$ and $\cF\,\iso\, \Omega^{\rho}_B\mid_{X-D}$. The character $\chi_{\mu_1}$ decomposes as 
$$
\cN^{mer}_{\mu_1}/\cN^{reg}_{\mu_1}\to \gM_{\emptyset, 0}\toup{\ev_{\emptyset, 0}}\A^1
$$ 
We have a locally closed embedding over $X^{\mu_1}$
$$
\ocZ{}^{\mu_1}\hook{} \cN^{mer}_{\mu_1}/\cN^{reg}_{\mu_1}
$$
given by the property that for each $\check{\lambda}\in \check{\Lambda}^+$ the map 
$
\kappa^{\check{\lambda}, -}: \cV^{\check{\lambda}}_{\cF}\to \cL^{\check{\lambda}}_{\Omega^{\rho}(-D)},
$
initially defined over $X-D$, is regular over $X$ and surjective. 

 For $\cF\in\Whit^{\kappa}_n$ its pull-back to
$$
(X^{\mu_1}\times \gM_n)\times_{(X^{\mu_1}\times X^n)}(X^{\mu_1}\times X^n)_{disj}
$$
is the extension by $*$ and also by $!$ from $(\gM_n)_{\goodat\, \mu_1}$, because there are no dominant coweight strictly smaller than $0$ (see Section~\ref{Section_1.4}). So, it suffices to prove the desired isomorphism over the open substack (\ref{open_substack_for_Pp331}). 

 The composition 
$$ 
(X^{\mu_1}\times X^{\mu_2}_n)_{disj}\times_{(X^{\mu_1}\times X^{\mu_2}_n)}(\ocZ{}^{\mu_1}\times \cZ^{\mu_2}_n)\to (X^{\mu_1}\times X^{\mu_2}_n)_{disj}\times_{X^{\mu}_n} \cZ^{\mu}_n\to X^{\mu_1}\times \gM_n
$$
factors as
\begin{multline*}
(X^{\mu_1}\times X^{\mu_2}_n)_{disj}\times_{(X^{\mu_1}\times X^{\mu_2}_n)}(\ocZ{}^{\mu_1}\times \cZ^{\mu_2}_n)\to\\
(X^{\mu_1}\times X^{\mu_2}_n)_{disj}\times_{(X^{\mu_1}\times X^{\mu_2}_n)}
(\cN^{mer}_{\mu_1}/\cN^{reg}_{\mu_1} \times \cZ^{\mu_2}_n)\\
\iso\, (X^{\mu_1}\times X^{\mu_2}_n)_{disj}\times_{(X^{\mu_1}\times X^{\mu_2}_n)} (\cN^{mer}_{\mu_1}\times^{\cN^{reg}_{\mu_1}}(_{\mu_1}\gM_n\times_{\gM_n} \cZ^{\mu_2}_n))\\
\to \cN^{mer}_{\mu_1}\times^{\cN^{reg}_{\mu_1}}{_{\mu_1}\gM_n}\;\toup{\act_{\mu_1}}\; (\gM_n)_{\goodat\, \mu_1}\hook{} X^{\mu_1}\times \gM_n,
\end{multline*}
where the second arrow used the trivialization of the $\cN^{reg}_{\mu_1}$-torsor 
$$
(_{\mu_1}\gM_n\times_{\gM_n} \cZ^{\mu_2}_n)\times_{(X^{\mu_1}\times X^{\mu_2}_n)} (X^{\mu_1}\times X^{\mu_2}_n)_{disj}
$$
(see Remark~\ref{Rem311}). 
\end{proof}

\medskip

\begin{Cor} 
\label{Cor_332}
For $\cF\in \Whit^{\kappa}_n$, $\mu_1\in -\Lambda^{pos}, \mu_2\in\Lambda$ and $\mu=\mu_1+\mu_2$ one has 
$$
\add^*_{\mu_1,\mu_2, disj}\FF(\cF)\,\iso\, \FF(\cF_{\emptyset})\boxtimes \FF(\cF)
$$
in $\D_{\zeta}((X^{\mu_1}\times X^{\mu_2}_n)_{disj})$. These isomorphisms are compatible with refinements of partitions.
\end{Cor}

We will use the following.

\begin{Rem} 
\label{Rem_Levi}
Let $M\subset G$ be a standard Levi, $\Lambda^{pos}_M$ the $\ZZ_+$-span of $M$-positive coroots in $\Lambda$. For $\mu\in-\Lambda^{pos}$ let $Z^{\mu}_G$ denote the Zastava space classifying $D\in X^{\mu}$, $U^-$-torsor $\cF$ on $X$, a trivialization $\cF\,\iso\, \cF^0_{U^-}\mid_{X-D}$ that gives rise to a generalized $B$-structure on $\cF_G:=\cF\times_{U^-} G$ over $X$ with the corresponding $T$-torsor $\cF^0_T(D)$. That is, for each $\check{\lambda}\in \check{\Lambda}^+$ the natural map
$$
\kappa^{\check{\lambda}}: \cO(\<D, \check{\lambda}\>)\to \cV^{\check{\lambda}}_{\cF}
$$
is regular over $X$. Assume in addtion $\mu\in -\Lambda^{pos}_M$. Then we have the similarly defined ind-scheme $Z^{\mu}_M$ for $M$. The natural map $Z^{\mu}_M\to Z^{\mu}_G$ is an isomorphism over $X^{\mu}$.
\end{Rem}
\index{$Z^{\mu}_G, Z^{\mu}_M$}

\begin{Pp} 
\label{Pp_3.3.4}
Assume $\varrho(\alpha_i)\notin \ZZ$ for any simple coroot $\alpha_i$. Then for $\mu\in -\Lambda^{pos}$ we have a (non-canonical) isomorphism $\cL^{\mu}_{\emptyset}\,\iso\, \FF(\cF_{\emptyset})$ in $\D_{\zeta}(\oX{}^{\mu})$.
\end{Pp}
\begin{proof} Consider first the case $\mu=-\alpha$, where $\alpha$ is a simple coroot of $G$. Then $X^{\mu}=X$. Applying Remark~\ref{Rem_Levi} for the corresponding subminimal Levi, we get $\cZ^{-\alpha}\,\iso\, X\times\A^1$, and $\ocZ{}^{-\alpha}\,\iso\, X\times\Gm$ is the complement to the zero section. The line bundle $\cP^{\bar\kappa}$ over $X^{\mu}$ is trivialized canonically. However, over $\ocZ{}^{-\alpha}$ we get another trivialization of $\cP^{\bar\kappa}$ inherited from the trivialization of $\cP^{\bar\kappa}\mid_{\gM_{\emptyset, 0}}$. The discrepancy between the two trivializations is the map
$$
\ocZ{}^{-\alpha}\,\iso\, X\times\Gm\toup{\pr}\Gm\toup{z\mapsto z^d}\Gm,
$$
where $d=\frac{-\bar\kappa(\alpha,\alpha)}{2}$. Since our answer here is different from that of (\cite{G}, Section~5.1), we give more details. Let $M$ be the standard subminimal Levi corresponding to the coroot $\alpha$, $M_0$ be the derived group of $M$, so $M_0\,\iso\, \SL_2$. Pick $x\in X$. Let $\PP$ denote the projective line classifying lattices $\cM$ included into
\begin{equation}
\label{eq_for_M_for_P1}
\Omega^{-\frac{1}{2}}(-x)\oplus\Omega^{\frac{1}{2}}\subset \cM\subset \Omega^{-\frac{1}{2}}\oplus\Omega^{\frac{1}{2}}(x)
\end{equation} 
such that
$\cM/(\Omega^{-\frac{1}{2}}(-x)\oplus\Omega^{\frac{1}{2}})$ is 1-dimensional. This defines a map $\PP\to \Bun_{M_0}$ sending $\cM$ to $\cM$ viewed as a $M_0$-torsor on $X$. Let $\cL$ denote the line bundle on $\PP$ with fibre
$$
\frac{\det\RG(X, \Omega^{\frac{1}{2}})\otimes \det\RG(X, \Omega^{-\frac{1}{2}})}{\det\RG(X, \cM)}
$$
at $\cM$. The restriction of $^{\omega}\cL^{\bar\kappa}$ under the composition $\PP\to \Bun_{M_0}\to\Bun_G$ identifies with $\cL^{\frac{-\bar\kappa(\alpha,\alpha)}{2}}$. The fibre $\cZ^{-\alpha}$ over $D=-\alpha x$ is the open subscheme of $\PP$ given by the property that $\Omega^{-\frac{1}{2}}(-x)\subset\cM$ is a subbundle. The formula for $d$ follows from the fact that $\cL\,\iso\, \cO(1)$ on $\PP$.

 So, if $\varrho(\alpha)\notin\ZZ$ then $\FF(\cF_{\emptyset})\,\iso\, \Qlb[1]$ non-canonically in $\D_{\zeta}(X^{-\alpha})$. 
 
 Let now $\mu=-\sum m_i\alpha_i\in -\Lambda^{pos}$ with $m_i\ge 0$. Applying Corollary~\ref{Cor_332} and the above computation, one gets the desired isomorphism after the pull-back to $\prod_i X^{m_i}-\vartriangle$, where $\vartriangle$ is the diagonal divisor. From the K\"unneth formula one sees that the product of the corresponding symmetric groups $\prod_i S_{m_i}$ acts by the sign character because the Gauss sum $\RG_c(\Gm, \cL_{\psi}\otimes\cL_{\zeta}^d)$ is concentrated in the degree 1 for $d\notin N\ZZ$. 
\end{proof}

The isomorphism of Proposition~\ref{Pp_3.3.4} does not hold in $\D_{\zeta}(X^{\mu})$. This is already seen in the following special case.
\begin{Lm} 
\label{Lm_335_case_SL_2}
Assume $G=\SL_2$ and $\varrho(\alpha)\notin\ZZ$ for the simple coroot $\alpha$. Then for $\mu\in -\Lambda^{pos}$, 
$\FF(\cF_{\emptyset})\in \D_{\zeta}(X^{\mu})$ is the extension by zero from $\oX{}^{\mu}$.
\end{Lm}
\begin{proof} Take $\mu=-m\alpha$, $m\ge 0$. So, $X^{(m)}\,\iso\, X^{\mu}$ via the map $D\mapsto -D\alpha$. The scheme $\cZ^{\mu}$ is a vector bundle over $X^{\mu}$ with fibre 
$$
\Ext^1(\Omega^{\frac{1}{2}}(D)/\Omega^{\frac{1}{2}}, \Omega^{-\frac{1}{2}}(-D))=\Omega^{-1}(-D)/\Omega^{-1}(-2D)
$$ 
at $-D\alpha$. A point of $\cZ^{\mu}$ is given by $D\in X^{(m)}$ and a diagram
$$
\begin{array}{cccc}
0\to \Omega^{-\frac{1}{2}}(-D)\to  M& \to &\Omega^{\frac{1}{2}}(D) & \to 0\\
&\nwarrow &\uparrow\\
&& \Omega^{\frac{1}{2}}
\end{array}
$$ 
The line bundle $\cP^{\bar\kappa}$ over $X^{(m)}$ identifies canonically with $\cO(-4c_j\!\vartriangle)$, where $\vartriangle\subset X^{(m)}$ is the divisor of the diagonals. Here $c_j$ is a part of our input data from (Section~\ref{section_input_data}, formula~(\ref{def_bar_kappa})). 

 For a line bundle $L$ on $X$ and an $D\in X^{(m)}$ let $(L(D)/L)_{max}\subset L(D)/L$ be the open subscheme consisting of those $v\in L(D)/L$ such that for any $0\le D'<D$, $v\notin L(D')/L$. Note that $(L(D)/L)_{max}$ identifies canonically with $(L^{-1}(-D)/L^{-1}(-2D)_{max}$.
\index{$(L(D)/L)_{max}$} 
 
 The fibre of $\ocZ{}^{\mu}$ over $D\in X^{(m)}$ is $(\Omega^{-1}(-D)/\Omega^{-1}(-2D))_{max}\,\iso\, (\Omega(D)/\Omega)_{max}$. Let $D=\sum_k m_k x_k\in X^{(m)}$. Then $(\Omega(D)/\Omega)_{max}\,\iso\, \prod_k (\Omega(m_kx_k)/\Omega)_{max}$. The fibre of $\cP^{\bar\kappa}$ at $-D\alpha\in X^{\mu}$ is
$$
(\otimes_k \;\Omega_{x_k}^{m_k^2-m_k})^{4c_j}
$$
Write a point of $\prod_k (\Omega(m_kx_k)/\Omega)_{max}$ as $v=(v_k)$, $v_k\in (\Omega(m_kx_k)/\Omega)_{max}$. Let $\bar v_k$ be the image of $v_k$ in the geometric fibre $(\Omega(m_kx_k))_{x_k}=\Omega_{x_k}^{1-m_k}$. The canonical section of $\pi^{\mu *}\cP^{\bar\kappa}$ over $\ocZ{}^{\mu}$ sends $v$ to $(\otimes_k \,\bar v_k^{-m_k})^{4c_j}$. So, the $*$-fibre of $\FF(\cF_{\emptyset})$ at $-D\alpha\in X^{\mu}$ identifies (up to a shift) with the tensor product over $k$ of the complexes
\begin{equation}
\label{complex_for_Lm335}
\RG_c((\Omega(m_kx_k)/\Omega)_{max}, \ev^*\cL_{\psi}\otimes \eta_k^*\cL_{\zeta^{4c_jm_k}}),
\end{equation}
where $\eta_k$ is the map
$$
\eta_k: (\Omega(m_kx_k)/\Omega)_{max}\to (\Omega(m_kx_k))_{x_k}\toup{\tau_k} \Gm 
$$
for some isomorphisms $\tau_k$. Calculate (\ref{complex_for_Lm335}) via the composition $(\Omega(m_kx_k)/\Omega)_{max}\to (\Omega(m_kx_k))_{x_k}\to \Spec k$. If $m_k>1$ for some $k$ then the sheaf $\ev^*\cL_{\psi}$ on $(\Omega(m_kx_k)/\Omega)_{max}$ changes under the action of the vector space $\Omega((m_k-1)x_k)/\Omega$ by the Artin-Schreier character, so (\ref{complex_for_Lm335}) vanishes for this $k$. Our claim follows.
\end{proof}

\begin{Rem} Assume that $\varrho(\alpha_i)\notin\ZZ$ for any simple coroot $\alpha_i$. For $G=\SL_2$ the fibres of $\cL^{\mu}_{\emptyset}$ are calculated in \cite{BFS}, it is not the extension by zero from $\oX{}^{\mu}$. As in (\cite{G}, Proposition~4.10), one may show that for any $K\in \Whit^{\kappa}_n$ the object $\FF(K)$ is placed in perverse cohomological degree zero (this is essentially done in Proposition~\ref{Pp_ovFF_is_exact}). However, Lemma~\ref{Lm_335_case_SL_2} shows that the functor $\FF$ does not produce an object of $\wt\FS^{\kappa}_n$, and should be corrected. 
\end{Rem}

\ssec{Compactified Zastava} 
\label{Section_Compactified Zastava}
For $\mu\in\Lambda$ let $\Bunb^{\mu}_{B^-}$ be the Drinfeld compactification of $\Bun_{B^-}^{\mu}$. Namely, this is the stack classifying a $G$-torsor $\cF$ on $X$, a $T$-torsor $\cF_T$ on $X$ of degree $(2g-2)\rho-\mu$, and a collection of nonzero maps of coherent sheaves for $\check{\lambda}\in \check{\Lambda}^+$
$$
\kappa^{\check{\lambda}, -}: \cV^{\check{\lambda}}_{\cF}\to \cL^{\check{\lambda}}_{\cF_T}
$$ satisfying the Pl\"ucker relations. This means that for any $\check{\lambda}, \check{\mu}\in \check{\Lambda}^+$ the composition 
$$
\cV^{\check{\lambda}+\check{\mu}}_{\cF}\to (\cV^{\check{\lambda}}\otimes\cV^{\check{\mu}})_{\cF}\;\stackrel{\kappa^{\check{\lambda}, -}\otimes\kappa^{\check{\mu}, -}}{\longrightarrow} \;\cL^{\check{\lambda}+\check{\mu}}_{\cF_T}
$$
coincides with $\kappa^{\check{\lambda}+\check{\mu}, -}$, and $\kappa^{0,-}: \cO\to \cO$ is the identity map. Let $\bar\gq^-: \Bunb^{\mu}_{B^-}\to\Bun_T$ be the map sending the above point to $\cF_T$. 

 For $\theta\in\Lambda^{pos}$ denote by $_{\le\theta}\Bunb_{B^-}^{\mu}\subset \Bunb^{\mu}_{B^-}$ the open substack given by the property that for any $\check{\lambda}\in\check{\Lambda}^+$ the cokernel of $\kappa^{\check{\lambda}, -}$ is a torsion sheaf of length $\le\<\theta, \check{\lambda}\>$. 
\index{$\Bunb^{\mu}_{B^-}, \; \bar\gq^-$}\index{${_{\le\theta}\Bunb_{B^-}^{\mu}}$}
 
 For $n\ge 0$ denote by $\ov{\cZ}^{\mu}_n$ the open substack of $\gM_n\times_{\Bun_G} \Bunb^{\mu}_{B^-}$ given by the property that for each $\check{\lambda}\in \check{\Lambda}^+$ the composition
\begin{equation}
\label{composition_for_bar_cZ^mu_n}
\Omega^{\<\check{\lambda}, \rho\>}\toup{\kappa^{\check{\lambda}}} \cV^{\check{\lambda}}_{\cF} \toup{\kappa^{\check{\lambda}, -}} \cL^{\check{\lambda}}_{\cF_T},
\end{equation}
which is regular over $X-\cup_i x_i$, is not zero. Define the projections by the diagram
$$
\gM_n\; \getsup{'\bar\gp}\; \ov{\cZ}^{\mu}_n \;\toup{'\bar\gp_B} \;\Bunb^{\mu}_{B^-}
$$
Let $\bar\pi^{\mu}: \ov{\cZ}^{\mu}_n\to X^{\mu}_n$ be the map sending the above point to $(x_1,\ldots, x_n, D)$ such that the maps (\ref{composition_for_bar_cZ^mu_n}) induce an isomorphism $\Omega^{\rho}(-D)\,\iso\, \cF_T$. Note that $\cZ^{\mu}_n\subset \ov{\cZ}^{\mu}_n$ is open. 
\index{$\ov{\cZ}^{\mu}_n, \;{'\bar\gp}, \;{'\bar\gp_B}$}\index{$\bar\pi^{\mu}, \; \ov{\cZ}^{\mu}_{n,\le \bar\lambda}$}

 For a $n$-tuple $\bar\lambda\in \Lambda^n$ define the closed substack  $\ov{\cZ}^{\mu}_{n,\le \bar\lambda}$ by the base change $\gM_{n,\le\bar\lambda}\to \gM_n$. The map $\bar\pi^{\mu}$ restricts to a map
\begin{equation}
\label{map_bar_pi^mu_le_bar_lambda}
\bar\pi^{\mu}:\ov{\cZ}^{\mu}_{n,\le \bar\lambda}\to X^{\mu}_{n,\le\bar\lambda}
\end{equation} 
The stack $\ov{\cZ}^{\mu}_0$ is rather denoted $\ov{\cZ}^{\mu}$.
As in (\cite{G}, Proposition~4.5), one gets the following.

\begin{Lm} 
\label{Lm_4.4.1.}
Let $(\bar x, \cF, \cF_T, (\kappa^{\check{\lambda}}), (\kappa^{\check{\lambda}, -}))$ be a point of $\ov{\cZ}^{\mu}_n$, whose image under 
$\bar\pi^{\mu}$ is $(\bar x, D)$. Then the restriction of $\cF$ to $X-D-\cup_i x_i$ is equipped with an isomorphism $\cF\,\iso\, \Omega^{\rho}\times_T G$ with the tautological maps $\kappa^{\check{\lambda}}, \kappa^{\check{\lambda}, -}$. In particular, $\ov{\cZ}^{\mu}_n$ is an ind-scheme over $k$. 
\end{Lm}

 Let $\Gr_{^{\omega}G, X^{\mu}_n}$ denote the ind-scheme classifying $(x_1,\ldots, x_n, D)\in X^{\mu}_n$, a $G$-torsor $\cF$ on $X$, a trivialization $\cF\,\iso\, \Omega^{\rho}\times_T G$ over $X-D-\cup_i x_i$. The projection $\Gr_{^{\omega}G, X^{\mu}_n}\to X^{\mu}_n$ is ind-proper.
\index{$\ov{\cZ}^{\mu}, \Gr_{^{\omega}G, X^{\mu}_n}$}
 
 We have a closed immersion $\ov{\cZ}^{\mu}_n\hook{} \Gr_{^{\omega}G, X^{\mu}_n}$ given by the property that for each $\check{\lambda}\in \check{\Lambda}^+$ the natural map $\kappa^{\check{\lambda}, -}: \cV^{\check{\lambda}}_{\cF}\to \Omega^{\<\rho, \check{\lambda}\>}(-\<D,\check{\lambda}\>)$ is regular over $X$, and 
$$
\kappa^{\check{\lambda}}:  \Omega^{\<\rho, \check{\lambda}\>}\to 
\cV^{\check{\lambda}}_{\cF}
$$
is regular over $X-\cup_i x_i$. So, $\bar\pi^{\mu}: \ov{\cZ}^{\mu}_n\to X^{\mu}_n$ is ind-proper. 
 
\begin{Lm} For $\mu_1\in -\Lambda^{pos}, \mu_2\in \Lambda$ and $\mu=\mu_1+\mu_2$ we have the following factorization property
\begin{equation}
\label{iso_factrorization_for_bar_cZ}
(X^{\mu_1}\times X^{\mu_2}_n)_{disj}\times_{X^{\mu}_n} \ov{\cZ}^{\mu}_n\,\iso\,  (X^{\mu_1}\times X^{\mu_2}_n)_{disj}\times_{(X^{\mu_1}\times X^{\mu_2}_n)} (\ov{\cZ}^{\mu_1}\times \ov{\cZ}^{\mu_2}_n)
\end{equation}
compatible with (\ref{iso_factorization_cZ}). 
\end{Lm}
\begin{proof} The argument follows (\cite{BFGM}, Proposition~2.4). Consider a point of the LHS given by $D_1\in X^{\mu_1}, (\bar x, D_2)\in X^{\mu_2}_n, \cF\in\Bun_G$. By Lemma~\ref{Lm_4.4.1.}, $\cF$ is equipped with an isomorphism 
$$
\beta: \cF\,\iso\, \Omega^{\rho}\times_T G\mid_{X-D_1-D_2-\cup_i x_i}
$$ 

 Let $\cF^1$ denote the gluing of the $G$-torsors $\Omega^{\rho}\times_T G\mid_{X-D_1}$ with $\cF\mid_{X-D_2-\cup_i x_i}$ via $\beta$ over the intersection $X-D_1-D_2-\cup_i x_i$ of these open subsets of $X$. It is equipped with the induced isomorphism $\beta^1: \cF^1\,\iso\, \Omega^{\rho}\times_T G\mid_{X-D_1}$. 
 
 Let $\cF^2$ denote the gluing of the $G$-torsors $\Omega^{\rho}\times_T G\mid_{X-D_2-\cup_i x_i}$ with $\cF\mid_{X-D_1}$ via $\beta$ over the intersection $X-D_1-D_2-\cup_i x_i$ of these open subsets of $X$. It is equipped with the induced isomorphism $\beta^2: \cF^2\,\iso\, \Omega^{\rho}\times_T G\mid_{X-D_2-\cup_i x_i}$. 
 
 The map (\ref{iso_factrorization_for_bar_cZ}) sends the above point to 
$$
D_1\in X^{\mu_1}, (\bar x, D_2)\in X^{\mu_2}_n, \; (\cF^1, D_1, \beta^1)\in \ov{\cZ}^{\mu_1}, \; (\cF^2,\beta^2, \bar x, D_2)\in \ov{\cZ}^{\mu_2}_n
$$
\end{proof}

 The diagram (\ref{diag_for_the_functor_FF}) extends to the diagram
\begin{equation}
\label{diag_for_the_functor_FF_extended}
\begin{array}{cccc}
\gM_n \getsup{'\bar\gp} &
\ov{\cZ}^{\mu}_n & \toup{'\bar\gp_B} & \Bunb^{\mu}_{B^-}\\
&\downarrow\lefteqn{\scriptstyle \bar\pi^{\mu}} &&  \downarrow\lefteqn{\scriptstyle \bar\gq^-}\\
&X^{\mu}_n & \toup{AJ} & \Bun_T
\end{array}
\end{equation} 

Now we face the difficulty that the line bundles $'\bar\gp^*\cP^{\bar\kappa}$ and $(\bar\pi^{\mu})^*\cP^{\bar\kappa}$ are not isomorphic over $\ov{\cZ}^{\mu}_n$, but only over its open part $\cZ^{\mu}_n$. 

\ssec{Description of fibres} Let $\cO_x$ denote the completed local ring of $X$ at $x$, $F_x$ its fraction field. For $\mu\in \Lambda$ we have the point $t^{\mu}\in \Gr_{G,x}=G(F_x)/G(\cO_x)$. Recall that $\Gr_B^{\mu}$ is the $U(F_x)$-orbit in $\Gr_{G,x}$ through $t^{\mu}$. We also have the closed ind-subscheme $\ov{\Gr}^{\mu}_B\subset \Gr_{G,x}$ defined in (\cite{FGV}, Section~7.1.1). It classifies a $G$-torsor $\cF$ on $X$ with a trivialization $\cF\,\iso\, \cF^0_G\mid_{X-x}$ such that for each $\check{\lambda}\in\check{\Lambda}^+$ the map 
$$
\kappa^{\check{\lambda}}: \cO(-\<\mu, \check{\lambda}\>)\to \cV^{\check{\lambda}}_{\cF}
$$
is regular over $X$. This is a scheme-theoretical version of the closure of $\Gr_B^{\mu}$.
\index{$\cO_x, F_x, \ov{\Gr}^{\mu}_B, \ov{\Gr}^{\mu}_{B^-}$}

Recall that $\Gr^{\mu}_{B^-}$ is the $U^-(F_x)$-orbit through $t^{\mu}$ in $\Gr_{G,x}$. Similarly, one defines $\ov{\Gr}^{\mu}_{B^-}\subset \Gr_{G,x}$. To be precise, $\ov{\Gr}^{\mu}_{B^-}$ classifies a $G$-torsor $\cF$ on $X$ with a trivialization $\cF\,\iso\, \cF^0_G\mid_{X-x}$ such that for any $\check{\lambda}\in\check{\Lambda}^+$ the map 
$$
\kappa^{\check{\lambda}, -}: \cV^{\check{\lambda}}_{\cF}\to \cO(-\<\mu, \check{\lambda}\>)
$$ 
is regular over $X$. Note that if $\Gr^{\nu}_{B^-}\subset \ov{\Gr}^{\mu}_{B^-}$ for some $\nu\in\Lambda$ then $\nu\ge\mu$. If $\Gr^{\nu}_B\subset \ov{\Gr}^{\mu}_B$ then $\nu\le \mu$. 

 Let $\mu\in -\Lambda^{pos}$. The fibre $\ov{\cZ}^{\mu}_{loc, x}$ of $\ov{\cZ}^{\mu}$ over $\mu x\in X^{\mu}$ identifies naturally with 
\begin{equation}
\label{fibre_of_ov_cZ^mu}
(\ov{\Gr}^0_B\cap \ov{\Gr}_{B^-}^{\mu})\times^{T(\cO_x)} \Omega^{\rho}\mid_{D_x},
\end{equation}
where $\Omega^{\rho}\mid_{D_x}$ denotes the corresponding $T(\cO_x)$-torsor. 
\index{$\ov{\cZ}^{\mu}_{loc, x}$}
  
\begin{Lm} 
\label{Lm_ovcZ^mu_is_scheme}
If $\mu\in -\Lambda^{pos}$ then (\ref{fibre_of_ov_cZ^mu}) is a projective scheme of finite type and of dimension $\le -\<\mu, \check{\rho}\>$
(and not just an ind-scheme).
\end{Lm}
\begin{proof}
Let $\nu\in\Lambda$ be such that $\Gr_{B^-}^{\nu}\subset \ov{\Gr}_{B^-}^{\mu}$, so $\nu\ge\mu$. We know from (\cite{BFGM}, Section~6.3) that $\ov{\Gr}^0_B\cap \Gr_{B^-}^{\nu}$ can be nonempty only for $\nu\le 0$, and in this case it is a scheme of finite type and of dimension $\le -\<\nu, \check{\rho}\>$. Since the set of $\nu\in\Lambda$ satisfying $\mu\le\nu\le 0$ is finite, we are done.
\end{proof}

 Lemma~\ref{Lm_ovcZ^mu_is_scheme} implies that $\bar\pi^{\mu}: \ov{\cZ}^{\mu}\to X^{\mu}$ is proper, its fibres are projective schemes of finite type of dimension $\le -\<\mu, \check{\rho}\>$. 
  
 Let $\mu\in\Lambda$. The fibre of $\ov{\cZ}^{\mu}_1$ over $\mu x_1$ identifies naturally with $\ov{\Gr}_{B^-}^{\mu}\times^{T(\cO_x)} \Omega^{\rho}\mid_{D_x}$. For $n\ge 1$ the fibre of $\bar\pi^{\mu}: \ov{\cZ}^{\mu}_n\to X^{\mu}_n$ over $(\bar x, D)$ is only an ind-scheme (not a scheme). Let also $\lambda\in\Lambda$. Then the fibre of $\ov{\cZ}^{\mu}_{1,\le\lambda}$ over $\mu x_1$ identifies naturally with 
$$
(\ov{\Gr}_B^{\lambda}\cap \ov{\Gr}_{B^-}^{\mu})\times^{T(\cO_x)} \Omega^{\rho}\mid_{D_x}
$$ 
This could be non-empty only for $\mu\le\lambda$, and in that case this is a projective scheme of dimension $\le \<\lambda-\mu, \check{\rho}\>$. 

 Now if $\bar\lambda\in\Lambda^n$ from the factorization property we see that 
the map (\ref{map_bar_pi^mu_le_bar_lambda}) is proper, its fibres are projective schemes of finite type.

\ssec{} 
\label{section_4.6_for_introducing_IC_zeta}
In Section~\ref{section_Langlands_program_metapl_groups} we defined $\Bunt_G$ as the gerbe of $N$-th roots of $^{\omega}\cL^{\bar\kappa}$ over $\Bun_G$, similarly for $\Bunt_T$. 

 Let $\Bunb_{B^-, \tilde G}=\Bunb_{B^-}\times_{\Bun_G}\Bunt_G$ and $\Bunb_{\tilde B^-}=\Bunb_{B^-,\tilde G}\times_{\Bun_T}\Bunt_T$. Set also $\Bun_{B^-, \tilde G}=\Bun_{B^-}\times_{\Bun_G}\Bunt_G$. Let $\Bun_{\tilde B^-}$ be the preimage of $\Bun_{B^-}$ in $\Bunb_{\tilde B^-}$.
\index{$\Bunb_{B^-, \tilde G}, \; \Bunb_{\tilde B^-}$}\index{$\Bun_{B^-, \tilde G}, \; \Bun_{\tilde B^-}$}
 
 A point of $\Bunb_{\tilde B^-}$ is given by $(\cF, \cF_T, \kappa^{\check{\lambda}, -})$ and lines $\cU, \cU_G$ equipped with isomorphisms
$$
\cU^N\,\iso\, (^{\omega}\cL^{\bar\kappa})_{\cF_T}, \;\;\; \cU_G^N\,\iso\, (^{\omega}\cL^{\bar\kappa})_{\cF}
$$
Let $\D_{\zeta^{-1}, \zeta}(\Bunb_{B^-})$ denote the derived category of $\Qlb$-sheaves on $\Bunb_{\tilde B^-}$ on which $\mu_N(k)\subset \Aut(\cU)$ acts by $\zeta$, and $\mu_N(k)\subset \Aut(\cU_G)$ acts by $\zeta^{-1}$. We define the irreducible perverse sheaf $\IC_{\zeta}\in \Perv_{\zeta^{-1}, \zeta}(\Bunb_{B^-})$ as follows (see \cite{L2}, Definition~3.1). One has the isomorphism
\begin{equation}
\label{iso_Bun_tilde_B^-}
B(\mu_N)\times \Bun_{B^-, \tilde G}\,\iso\, \Bun_{\tilde B^-}
\end{equation}
sending $(\cF_{B^-}, \cU_G, \cU_0\in B(\mu_N))$ with $\cU_0^N\,\iso\, k$ to $(\cF_{B^-}, \cU_G, \cU)$ with $\cU=\cU_G\otimes \cU_0$. View
$\cL_{\zeta}\boxtimes \IC(\Bun_{B^-, \tilde G})$ as a perverse sheaf on $\Bun_{\tilde B^-}$ via (\ref{iso_Bun_tilde_B^-}). Let $\IC_{\zeta}$ be its intermediate extension to $\Bunb_{\tilde B^-}$. 
\index{$\D_{\zeta^{-1}, \zeta}(\Bunb_{B^-})$}\index{$\IC_{\zeta}$}
 
\sssec{} 
\label{Section_4.6.1}
Let $\wt{\ov{\cZ}}{}^{\mu}_n$ denote the gerbe of $N$-th roots of $(\bar\pi^{\mu})^*\cP^{\bar\kappa}$, $\D_{\zeta}(\ov{\cZ}^{\mu}_n)$ denote the derived category of $\Qlb$-sheaves on $\wt{\ov{\cZ}}{}^{\mu}_n$, on which $\mu_N(k)$ acts by $\zeta$. For $\mu\in\Lambda$ define the functor $\bar F^{\mu}: \D_{\zeta}(\gM_n)\to \D_{\zeta}(\ov{\cZ}^{\mu}_n)$ by
$$
\bar F^{\mu}(K)={'\bar\gp^*}K\otimes ('\bar\gp_B)^*\IC_{\zeta}[-\dim\Bun_G]
$$
We write $\bar F^{\mu}_{\zeta}:=\bar F^{\mu}$ if we need to express the dependence on $\zeta$. 
Define the functor $\ov{\FF}: \D_{\zeta}(\gM_n)\to \D_{\zeta}(X^{\mu}_n)$ by
$$
\ov{\FF}(K)=(\bar\pi^{\mu})_!\bar F^{\mu}(K)
$$
We will see below that the functor $\bar F^{\mu}: \Whit^{\kappa}_n\to  \D_{\zeta}(\ov{\cZ}^{\mu}_n)$ commutes with the Verdier duality (up to replacing $\zeta$ by $\zeta^{-1}$).  
\index{$\wt{\ov{\cZ}}{}^{\mu}_n, \,\D_{\zeta}(\ov{\cZ}^{\mu}_n)$}
\index{$\bar F^{\mu}, \, \bar F^{\mu}_{\zeta}, \ov{\FF}$}

\ssec{} For $\mu\in -\Lambda^{pos}$ set $\oo{\ov{\cZ}}{}^{\mu}=\ov{\cZ}{}^{\mu}\times_{\gM_{\emptyset}} \gM_{\emptyset, 0}$. 
\index{$\oo{\ov{\cZ}}{}^{\mu}$}

\begin{Pp} 
\label{Pp_factorization_of_barF^mu}
Let $\mu_1\in -\Lambda^{pos}, \mu_2\in\Lambda, \mu=\mu_1+\mu_2$ and $\cF\in \Whit^{\kappa}_n$. Under the isomorphism (\ref{iso_factrorization_for_bar_cZ}) the complex
$$
\add_{\mu_1,\mu_2, disj}^*\bar F^{\mu}(\cF)\in \D_{\zeta}((X^{\mu_1}\times X^{\mu_2})_{disj}\times_{X^{\mu}_n} \ov{\cZ}^{\mu}_n)
$$
identifies with
$$
\bar F^{\mu_1}(\cF_{\emptyset})\boxtimes \bar F^{\mu_2}(\cF)\in \D_{\zeta}((X^{\mu_1}\times X^{\mu_2})_{disj}\times_{(X^{\mu_1}\times X^{\mu_2}_n)}(\ov{\cZ}^{\mu_1}\times \ov{\cZ}^{\mu_2}_n))
$$
\end{Pp}
\begin{proof} The preimage of $(\gM_n)_{\goodat\, \mu_1}$ under the map
$$
(X^{\mu_1}\times X^{\mu_2})_{disj}\times_{X^{\mu}_n} \ov{\cZ}^{\mu}_n \toup{'\bar\gp} X^{\mu_1}\times \gM_n
$$
goes over under the isomorphism (\ref{iso_factrorization_for_bar_cZ}) to
\begin{equation}
\label{open_part_of_factorized_barcZ}
(X^{\mu_1}\times X^{\mu_2})_{disj}\times_{(X^{\mu_1}\times X^{\mu_2}_n)} (\oo{\ov{\cZ}}{}^{\mu_1}\times \ov{\cZ}^{\mu_2}_n)
\end{equation}

Recall that $\cN^{mer}_{\mu_1}/\cN^{reg}_{\mu_1}$ is the ind-scheme classifying $D\in X^{\mu_1}$, a $B$-torsor $\cF$ on $X$ with compatible isomorphisms $\cF\times_B T\,\iso\,\Omega^{\rho}$ over $X$ and $\cF\,\iso\, \Omega^{\rho}_B\mid_{X-D}$. We have the closed embedding over $X^{\mu_1}$
$$
\oo{\ov{\cZ}}{}^{\mu_1}\hook{} \cN^{mer}_{\mu_1}/\cN^{reg}_{\mu_1}
$$
given by the property that for each $\check{\lambda}\in\check{\Lambda}^+$ the map $\kappa^{\check{\lambda}, -}: \cV^{\check{\lambda}}_{\cF}\to \cL^{\check{\lambda}}_{\Omega^{\rho}(-D)}$, initially defined over $X-D$, is regular over $X$. 

 The two complexes we want to identify are extensions by zero from the open substack (\ref{open_part_of_factorized_barcZ}), so, it sufiices to establish the desired isomorphism over (\ref{open_part_of_factorized_barcZ}). By (\cite{L2}, Theorem~4.1), the complex $\add_{\mu_1,\mu_2, disj}^*('\bar\gp_B^*\IC_{\zeta})$ goes over under (\ref{iso_factrorization_for_bar_cZ}) to the complex $'\bar\gp_B^*\IC_{\zeta}\boxtimes ('\bar\gp_B)^*\IC_{\zeta}$ up to a shift. 
 
 The composition 
$$ 
(X^{\mu_1}\times X^{\mu_2}_n)_{disj}\times_{(X^{\mu_1}\times X^{\mu_2}_n)}(\oo{\ov{\cZ}}{}^{\mu_1}\times \ov{\cZ}^{\mu_2}_n)\to (X^{\mu_1}\times X^{\mu_2}_n)_{disj}\times_{X^{\mu}_n} \ov{\cZ}^{\mu}_n\to X^{\mu_1}\times \gM_n
$$
factors as
\begin{multline*}
(X^{\mu_1}\times X^{\mu_2}_n)_{disj}\times_{(X^{\mu_1}\times X^{\mu_2}_n)}(\oo{\ov{\cZ}}{}^{\mu_1}\times \ov{\cZ}^{\mu_2}_n)\to\\
(X^{\mu_1}\times X^{\mu_2}_n)_{disj}\times_{(X^{\mu_1}\times X^{\mu_2}_n)}
(\cN^{mer}_{\mu_1}/\cN^{reg}_{\mu_1} \times \ov{\cZ}^{\mu_2}_n)\\
\iso\, (X^{\mu_1}\times X^{\mu_2}_n)_{disj}\times_{(X^{\mu_1}\times X^{\mu_2}_n)} (\cN^{mer}_{\mu_1}\times^{\cN^{reg}_{\mu_1}}(_{\mu_1}\gM_n\times_{\gM_n} \ov{\cZ}^{\mu_2}_n))\\
\to \cN^{mer}_{\mu_1}\times^{\cN^{reg}_{\mu_1}}{_{\mu_1}\gM_n}\;\toup{\act_{\mu_1}}\; (\gM_n)_{\goodat\, \mu_1}\hook{} X^{\mu_1}\times \gM_n,
\end{multline*}
where the second arrow used the trivialization of the $\cN^{reg}_{\mu_1}$-torsor
$$
(_{\mu_1}\gM_n\times_{\gM_n} \ov{\cZ}^{\mu_2}_n)\times_{(X^{\mu_1}\times X^{\mu_2}_n)} (X^{\mu_1}\times X^{\mu_2}_n)_{disj}
$$
as in Proposition~\ref{Pp_331_factorization_of_F^mu}. One finishes the proof as in Proposition~\ref{Pp_331_factorization_of_F^mu}.
\end{proof}

\ssec{Generalizing the ULA property} 
\label{section_Generalizing_ULA}
Let $S_1$ be a smooth equidimensional stack. Let $p_1: Y_1\to S_1$ and $q_1: S\to S_1$ be morphisms of stacks locally of finite type. Let $Y=Y_1\times_{S_1} S$. Let 
$p: Y\to S$ and $q: Y\to Y_1$ denote the projections. Denote by $g: Y\to Y_1\times S$ the map $(q,p)$. For $L\in \D(Y_1)$ consider the functor $\cF_L: \D(S)\to\D(Y)$ given by 
$$
\cF_L(K)=p^*K\otimes q^*L\<-\frac{\dim S_1}{2}\>,
$$ 
where $\<d\>=[2d](d)$.
\begin{Lm}
i) For $K\in \D(Y_1\times S)$ there is a canonical morphism functorial in $K$
\begin{equation}
\label{map_looks_like_trace}
g^*K\<-\frac{\dim S_1}{2}\>\to g^!K\<\frac{\dim S_1}{2}\>,
\end{equation}
ii) There is a canonical morphism functorial in $K\in \D(S), L\in \D(Y_1)$
\begin{equation}
\label{map_duality_for_cF_L}
\cF_{\DD L}(\DD K)\to \DD(\cF_L(K))
\end{equation}
\end{Lm}
\begin{proof}
i) We have a diagram, where the squares are cartesian
$$
\begin{array}{ccc}
S_1 & \toup{\vartriangle} & S_1\times S_1\\
\uparrow\lefteqn{\scriptstyle q_1} && \uparrow\lefteqn{\scriptstyle \id\times q_1}\\
S & \to & S_1\times S\\
\uparrow\lefteqn{\scriptstyle p} && \uparrow\lefteqn{\scriptstyle p_1\times\id}\\
Y & \toup{g} & Y_1\times S
\end{array}
$$
One has $\vartriangle^!\!\Qlb\,\iso\, \Qlb\<-\dim S_1\>$, because $S_1$ is smooth. By (\cite{SGA4}, XVII 2.1.3), one has the base change morphism $p^*q_1^*\vartriangle^!\to g^!(p_1\times q_1)^*$. Applying it to the previous isomorphism, one gets a canonical map $\can: \Qlb\<-\dim S_1\>\to g^!\Qlb$. 

 According to (\cite{BG}, Section~5.1.1), there is a canonical morphism $g^*K\otimes g^! K'\to g^! (K\otimes K')$ functorial in $K, K'\in\D(Y_1\times S)$. Taking $K'=\Qlb$ we define (\ref{map_looks_like_trace}) as the composition 
$$
g^*K\<-\dim S_1\>\toup{\id\otimes\can} g^*K\otimes g^!\Qlb\to g^!K
$$
ii) Apply (\ref{map_looks_like_trace}) to $\DD L\boxtimes \DD K$. 
\end{proof}

\begin{Def} 
\label{Def_ULA_for_diagram}
Let $\oo{Y}\subset Y$ be an open substack. Say that $L\in \D(Y_1)$ is locally acyclic with respect to the diagram $S\getsup{p} \oo{Y} \toup{q} Y_1$ if for any $K\in \D(S)$ the map (\ref{map_duality_for_cF_L}) is an isomorphism over $\oo{Y}$. Say that $L\in \D(Y_1)$ is universally locally acyclic with respect to the diagram $S\getsup{p} \oo{Y} \toup{q} Y_1$ if the same property holds after any smooth base change $S'_1\to S_1$. 
\end{Def}

\sssec{} 
\label{Section_383}
Here are some properties of the above ULA condition:
\begin{itemize} 
\item[1)] If $S_1=\Spec k$ then any $L\in \D(Y_1)$ is ULA with respect to the diagram $S\getsup{p} \oo{Y} \toup{q} Y_1$.
\item[2)] If $r_1: V_1\to Y_1$ is smooth of fixed relative dimension, and $L\in\D(Y_1)$ is ULA with respect to $S\getsup{p} \oo{Y}\toup{q} Y_1$ then $r_1^*L$ is ULA with respect to the diagram $S\gets \oo{V} \to V_1$. Here we defined $r: V\to Y$ as the base change of $r_1: V_1\to Y_1$ by $q: Y\to Y_1$, and $\oo{V}$ is the preimage of $\oo{Y}$ in $V$. Conversely, if $r_1: V_1\to Y_1$ is smooth and surjective, and $r_1^*L$ is ULA  with respect to the diagram $S\gets \oo{V} \to V_1$, then $L\in\D(Y_1)$ is ULA with respect to $S\getsup{p} \oo{Y}\toup{q} Y_1$.
\item[3)] Assume given a diagram as above $S\getsup{p} Y \toup{q} Y_1$ such that both $S_1$ and $S$ are smooth and equidimensional. Assume $L\in \D(Y_1)$, and the natural map $q^*L\<\dim S-\dim S_1\>\to q^!L$ is an isomorphism. Then $\DD(q^*L)$ is locally acyclic with respect to $p: \oo{Y}\to S$ if and only if $L$ is locally acyclic with respect to the diagram $S\getsup{p} \oo{Y} \toup{q} Y_1$.
\end{itemize}
\begin{proof}
3) Let $\bar p: Y\to Y\times S$ be the graph of $p: Y\to S$. By (\cite{BG}, Section~5.1.1), we have a canonical morphism, say $\alpha: \bar p^*(\cdot)\<-\dim S\>\to \bar p^!$. Since $S$ and $S_1$ are smooth, $q_1^!\Qlb\,\iso\, \Qlb\<\dim S-\dim S_1\>$. As in Section~\ref{section_Generalizing_ULA}, since the map $q\times\id: Y\times S\to Y_1\times S$ is obtained from $q_1$ by base change, the above isomorphism yields a canonical map $\can: \Qlb\<\dim S-\dim S_1\>\to (q\times\id)^!\Qlb$. For $K\in \D(Y_1\times S)$ we get a canonical map 
$$
\beta: (q\times\id)^*K\<\dim S-\dim S_1\>\to (q\times\id)^!K
$$
defined as the composition $(q\times\id)^*K\<\dim S-\dim S_1\>\toup{\id\otimes\can} (q\times\id)^*K\otimes (q\times\id)^!\Qlb \to (q\times\id)^!K$. The composition $Y\toup{\bar p} Y\times S\toup{q\times\id} Y_1\times S$ equals $g$. For $K\in \D(Y_1\times S)$ the map (\ref{map_looks_like_trace}) equals the composition 
$$
\bar p^*(q\times\id)^*K\<-\dim S_1\>\toup{\beta}  \bar p^*(q\times \id)^!K\<-\dim S\>\toup{\alpha} \bar p^!(q\times \id)^!K
$$
 
 Let now $K\in \D(S)$. By our assumptions, the map $\beta: (q\times\id)^*(\DD L\boxtimes \DD K)\<\dim S-\dim S_1\>\,\iso\, (q\times\id)^! (\DD L\boxtimes \DD K)$ is an isomorphism. The map $\DD(q^*L)$ is locally acyclic with respect to $p: \oo{Y}\to S$ if and only if the map $\alpha: \bar p^*(\DD(q^*L)\boxtimes \DD K)\<-\dim S\>
 \to \bar p^!(\DD(q^*L)\boxtimes \DD K)$ is an isomorphism over $\oo{Y}$ for any $K\in \D(S)$. Our claim follows.
\end{proof}

\sssec{} We say that for a morphism $p_1: Y_1\to S_1$ an object $L\in \D(Y_1)$ is ULA with respect to $p_1$ if it satisfies (\cite{SGA4demi}, Definition~2.12). One may check that this definition is equivalent to (\cite{BG}, Definition~5.1). In the latter one requires that local acyclicity holds after any smooth base change, whence in the former one requires it to hold after any base change $q_1: S\to S_1$.
 
  Assume given a cartesian square as in Section~\ref{section_Generalizing_ULA}
\begin{equation}
\label{diag_square_S_i_Y_i}
\begin{array}{ccc}
Y &\toup{q} & Y_1\\
\downarrow\lefteqn{\scriptstyle p} && \downarrow\lefteqn{\scriptstyle p_1}\\
S & \toup{q_1} & S_1
\end{array}
\end{equation}
with $S_1$ smooth equidimensional.

\begin{Pp} 
\label{Pp_ULA_over_S_1_implies_ULA}
Assume $q_1$ representable. Let $L\in \D(Y_1)$ be ULA with respect to $p_1$. Then $L$ is ULA with respect to the diagram $S\getsup{p} Y\toup{q} Y_1$.
\end{Pp}
 
 To establish Proposition~\ref{Pp_ULA_over_S_1_implies_ULA} we need the following.
\begin{Lm} 
\label{Lm_q^*_and_q^!_are_the_same}
Assume given a diagram (\ref{diag_square_S_i_Y_i}), where $S, S_1$ are smooth of dimensions $d, d_1$ respectively, and $q_1$ is representable. If $L\in D(Y_1)$ is ULA with respect to $p_1$ then the natural map $\eta: q^*L\<\frac{d-d_1}{2}\>\to q^!L\<\frac{d_1-d}{2}\>$ is an isomorphism.
\end{Lm} 
\begin{proof} One has canonical maps $p^*q_1^!\Qlb\to q^!\Qlb$ and $q^*L\otimes q^!\Qlb\to q^!L$, the second one is defined in (\cite{BG}, Section~5.1.1). One has $q_1^!\Qlb\,\iso\,\Qlb\<d-d_1\>$ canonically. Recall that $\eta$ is defined as the composition $q^*L\<d-d_1\>\to q^*L\otimes q^!\Qlb\to q^!L$.

If $q_1$ is smooth then our claim is well known. If $q_1$ is a closed immersion then this follows from (\cite{BG}, Lemma~B.3). In general, write $q_1$ as the composition $S\toup{\id\times q_1} S\times S_1\toup{\pr_2} S_1$. Localizing on $S_1$ in smooth topology, we may assume $S_1$ is a smooth affine scheme. Then $\id\times q_1$ is a closed immersion.
\end{proof}

\begin{proof}[Proof of Proposition~\ref{Pp_ULA_over_S_1_implies_ULA}] 
Let $K\in \D(S)$. Localizing on $S_1$ in smooth topology we may assume $S_1$ is a smooth affine scheme of dimension $d_1$. Let $i_1: S_0\to S$ be a locally closed smooth subscheme with $\dim S_0=d_0$, $E$ a local system on $S_0$. Decomposing $K$ in the derived category, it is enough to treat the case of $K=(i_1)_*E$. We must show that for this $K$ the map (\ref{map_duality_for_cF_L}) is an isomorphism over $Y$. Let $i: Y_0\hook{}Y$ be obtained from $i_1$ by the base change $p: Y\to S$. Let $p_0: Y_0\to S_0$ be the projection. By Lemma~\ref{Lm_q^*_and_q^!_are_the_same}, 
$$
i^*q^*L\<d_0-d_1\>\,\iso\, i^!q^!L
$$ 
Since $i^*q^*L$ is ULA over $S_0$, by 3) of Section~\ref{Section_383}, $L$ is locally acyclic with respect to the diagram $S_0\getsup{p_0} Y_0\,\toup{q\comp i} \,Y_1$. That is, one has an isomorphism over $Y_0$
\begin{equation}
\label{iso_over_Y_0_preliminary}
\DD(p_0^*E\otimes i^*q^*L)\,\iso\, p_0^*(\DD E)\otimes i^*q^*(\DD L)\<-d_1\>
\end{equation}
We must show that the natural map
\begin{equation}
\label{desired_iso_over_Y}
q^*(\DD L)\otimes p^*(i_1)_*E^*)\<d_0-d_1\>\to \DD(q^*L\otimes p^*(i_1)_*E)
\end{equation}
is an isomorphism over $Y$. By (\cite{F}, Theorem~7.6.9), $q^*L\otimes p^*(i_1)_*E\,\iso\, i_*(i^*q^*L\otimes p_0^*E)$. So, both sides of (\ref{desired_iso_over_Y}) are extensions by zero under $i$, and over $Y_0$ the desired isomorphism reduces to (\ref{iso_over_Y_0_preliminary}).
\end{proof}
 
\ssec{} The above notion of ULA was introduced, because we hoped that  for $\mu\in\Lambda$, $\bar\lambda\in\Lambda^n$ the perverse sheaf $\IC_{\zeta}\in \Perv_{\zeta^{-1},\zeta}(\Bunb_{B^-}^{\mu})$ is ULA with respect to the diagram 
$$
\gM_{n,\le\lambda}\getsup{'\bar\gp}\ov{\cZ}^{\mu}_{n,\le\lambda}\toup{'\bar\gp_B} \Bunb_{B^-}^{\mu}
$$ 
Unfortunately, this claim is not literally true. However, it is used in the proof of following result. For $\mu\in\Lambda, K\in \D_{\zeta}(\gM_n)$ the map (\ref{map_duality_for_cF_L}) defines a canonical morphism 
\begin{equation}
\label{map_barF^mu_duality}
\bar F^{\mu}_{\zeta^{-1}}(\DD K)\to \DD (\bar F^{\mu}(K))
\end{equation}
\begin{Pp} 
\label{Pp_391}
For any $K\in \Whit_n^{\kappa}$ the map (\ref{map_barF^mu_duality}) is an isomorphism.
\end{Pp} 

 For $\theta\in\Lambda^{pos}$ denote by $_{\le\theta}\Bunb_{\tilde B^-}^{\mu}$ the preimage of $_{\le\theta}\Bunb_{B^-}^{\mu}$ in $\Bunb_{\tilde B^-}^{\mu}$. Our proof of Proposition~\ref{Pp_391} uses
the following result of Campbell. 
\index{$_{\le\theta}\Bunb_{\tilde B^-}^{\mu}$}
% The proof is based on the following unpublished results of J. Campbell and D. Gaitsgory\footnote{The author is obliged to both of them for the possibility to cite their unpublished results.}. 

%\begin{Rem} The map $\bar\gp^K: \Bunb^K_{B^-}\to \Bunb_{B^-}$ is an isomorphism over the union of the open substacks $_{\le\alpha_i}\Bunb_{B^-}$ for all the simple coroots $\alpha_i$.
%\end{Rem}

% For the convenience of the reader recall the following property immediate from the definition of the singular support (\cite{Be}). 
%\begin{Lm} 
%\label{Lm_def_SS_transversality}
%Let $Y,Z$ be smooth schemes of finite type, $f: Y\to Z$ a morphism, $K\in \D(Y)$. If $f: Y\to Z$ is $SS(K)$-transversal in the sense of (\cite{Be}, Section~1.1) then $K$ is ULA with respect to $f$. \QED
%\end{Lm}
 
\begin{Pp}[\cite{C}, 4.2.1]
\label{Pp_Dennis_helped}
Let $\theta\in\Lambda^{pos}$, $\mu\in\Lambda$. Assume that for any $0\le \theta'\le\theta$ and any positive root $\check{\alpha}$ one has $\<\check{\alpha}, (2g-2)\rho-\mu+\theta'\> >2g-2$. Then the perverse sheaf $\IC_{\zeta}$ is ULA with respect to the projection $_{\le\theta}\Bunb_{\tilde B^-}^{\mu}\to\Bunt_G$.
\end{Pp}

In the untwisted case the latter result becomes (\cite{C}, 4.1.1.1).

\begin{proof}[Proof of Proposition~\ref{Pp_391}]
Pick a collection of  dominant coweights $\bar\lambda=(\lambda_1,\ldots,\lambda_n)$ and $\mu\in \Lambda$ with $\mu\le\sum_i\lambda_i$. We assume $K$ is the extension by zero from $\gM_{n,\le\bar\lambda}$.
We must show that (\ref{map_barF^mu_duality}) is an isomorphism over $\ov{\cZ}^{\mu}_{n,\le\bar\lambda}$. 

 For $\theta\ge 0$ we denote by $_{\le\theta}\ov{\cZ}^{\mu}_n$ the preimage of $_{\le\theta}\Bunb^{\mu}_{B^-}$ under $'\bar\gp_B: \ov{\cZ}^{\mu}_n\to \Bunb^{\mu}_{B^-}$. Set $\theta=(\sum\lambda_i)-\mu\in\Lambda^{pos}$. Note that $\ov{\cZ}^{\mu}_{n,\le\bar\lambda}$ is contained in $_{\le\theta}\ov{\cZ}^{\mu}_n$. 
\index{$_{\le\theta}\ov{\cZ}^{\mu}_n$} 
 
 For $\eta\le 0$ denote by $\cW^{\eta,\mu}$ the scheme  
$$
(\oo{X}{}^{\eta}\times X^{\mu}_n)_{disj}\times_{(X^{\eta}\times X^{\mu}_n)} 
(\cZ^{\eta}\times {_{\le\theta}\ov{\cZ}^{\mu}_n})
$$
By the factorization property, the natural map $\cW^{\eta,\mu}\to {_{\le\theta}\ov{\cZ}^{\eta+\mu}_n}$ is \'etale.

  By Proposition~\ref{Pp_factorization_of_barF^mu}, it suffices to show that the canonical map
$$
\bar F^{\mu+\eta}_{\zeta^{-1}}(\DD K)\to \DD \bar F^{\mu+\eta}(K)
$$
is an isomorphism over $_{\le\theta}\ov{\cZ}^{\eta+\mu}_n$. If $\eta\le 0$ is small enough then $\IC_{\zeta}$ is ULA with respect to
$_{\le\theta}\Bunb^{\eta+\mu}_{\tilde B^-}\to\Bunt_G$ by Proposition~\ref{Pp_Dennis_helped}. Our claim now follows from Proposition~\ref{Pp_ULA_over_S_1_implies_ULA}.
\end{proof}

 Let $\ov{\cZ}^{\mu}_{\bar x,\bar\lambda}\subset \ov{\cZ}^{\mu}_n$ (resp., $\ov{\cZ}^{\mu}_{\bar x,\le\bar\lambda}\subset \ov{\cZ}^{\mu}_n$) be the substack obtained from $\ov{\cZ}^{\mu}_n$ by the base change $\gM_{\bar x,\bar\lambda}\to \gM_n$ (resp., $\gM_{\bar x,\le\bar\lambda}\to \gM_n$). Let $\cZ^{\mu}_{\bar x,\bar\lambda}$ be the preimage of $\Bun_{B^-}^{\mu}$ in $\ov{\cZ}^{\mu}_{\bar x,\bar\lambda}$. 
\index{$\ov{\cZ}^{\mu}_{\bar x,\bar\lambda}, \ov{\cZ}^{\mu}_{\bar x,\le\bar\lambda}, \cZ^{\mu}_{\bar x,\bar\lambda}$}

\begin{Cor} 
\label{Corollary_392}
i)  If $\mu\in -\Lambda^{pos}$ then $\bar F^{\mu}(\cF_{\emptyset})$ is an irreducible perverse sheaf, the extension by zero from $\oo{\ov{\cZ}}{}^{\mu}$.\\
ii) Let $\bar x=(x_1,\ldots, x_n)\in X^n$ be pairwise different, $\bar\lambda=(\lambda_1,\ldots,\lambda_n)$ with $\lambda_i\in\Lambda^+$, $\mu\in\Lambda$ with $\mu\le\sum_i\lambda_i$. Then $\bar F^{\mu}(\cF_{\bar x, \bar\lambda, !})$ is perverse, and $\DD\bar F^{\mu}(\cF_{\bar x, \bar\lambda, !})\,\iso\, \bar F^{\mu}_{\zeta^{-1}}(\DD\cF_{\bar x, \bar\lambda, !})$.\\
iii) The complex $\bar F^{\mu}(\cF_{\bar x, \bar\lambda})$ is an irreducible perverse sheaf, the intermediate extension from $\ov{\cZ}^{\mu}_{\bar x,\bar\lambda}$. So, $\ov{\FF}(\cF_{\bar x, \bar\lambda})$ is a direct sum of (shifted) irreducible perverse sheaves. 
\end{Cor}
\begin{proof}
i) and ii). The fact that $\bar F^{\mu}(\cF_{\bar x, \bar\lambda, !})$ is an irreducible perverse sheaf over $\ov{\cZ}^{\mu}_{\bar x,\bar\lambda}$ is essentially explained in \cite{BFGM} (see also \cite{L2}). Our claim follows now from Proposition~\ref{Pp_391} and the fact that $\cF_{\emptyset}$ is self-dual (up to replacing $\psi$ by $\psi^{-1}$).\\
iii) For each collection of dominant coweights $\bar \lambda'<\bar\lambda$ the $*$-restriction of $\cF_{\bar x,\bar\lambda}$ to $\wt\gM_{\bar x, \bar\lambda'}$ is placed in perverse degrees $<0$. Therefore, the $*$-restriction of $\bar F^{\mu}(\cF_{\bar x, \bar\lambda})$ to $\ov{\cZ}^{\mu}_{\bar x, \bar\lambda'}$ is placed in perverse degrees $<0$ by ii). Our claim follows.
\end{proof}
\begin{Rem} Let us list the dimensions of stacks mentioned in Corollary~\ref{Corollary_392}. As in (\cite{BFGM}, Section~5.2) one checks that $\ov{\cZ}^{\mu}_{\bar x,\bar\lambda}$ is irreducible of dimension $\<-\mu+\sum_i\lambda_i, 2\check{\rho}\>$. The stack $\gM_{\bar x,\bar\lambda}$ is smooth irreducible of dimension 
$$
(g-1)\dim U-\<(2g-2)\rho-\sum_i\lambda_i, 2\check{\rho}\>,
$$ 
and 
$\dim\Bunb_{B^-}^{\mu}=(g-1)\dim B+\<2\check{\rho}, (2g-2)\rho-\mu\>$.
The $*$-restriction of $\bar F^{\mu}(\cF_{\bar x, \bar\lambda, !})$ to $\cZ^{\mu}_{\bar x,\bar\lambda}$ is a local system  placed in the usual degree 
$\<\mu-\sum_i\lambda_i, 2\check{\rho}\>$.
\end{Rem}

\ssec{} 
\label{section_description_of_IC_zeta}
The $*$-restrictions of $\IC_{\zeta}$ to a natural stratification have been calculated in (\cite{L2}, Theorem~4.1) under the additional assumption that $G$ is simple, simply-connected, but the answer and the argument hold also in our case of $[G,G]$ simply-connected. This way one gets the following description.

 Let $\check{\gu}^-_{\zeta}$ denote the Lie algebra of the unipotent radical of the Borel subgroup $\check{B}_{\zeta}^-\subset \check{G}_{\zeta}$ corresponding to $B^-$. For $\nu\in \Lambda^{\sharp}$ and $V\in\Rep(\check{T}_{\zeta})$ write $V_{\nu}$ for the direct summand of $V$, on which $\check{T}_{\zeta}$ acts by $\nu$. 
\index{$\check{\gu}^-_{\zeta}, V_{\nu}, \gU(\theta), \mid\gU(\theta)\mid$}
 
  Let $\theta\in -\Lambda^{pos}$. We write $\gU(\theta)$ for an element of the free abelian semigroup generated by $-\Lambda^{pos}-0$. In other words, $\gU(\theta)$ is a way to write 
\begin{equation}
\label{def_gU(theta)}  
\theta=\sum_m n_m\theta_m,
\end{equation}
where $\theta_m\in {-\Lambda^{pos}-0}$ are pairwise different, and $n_m\ge 0$. Set $\mid\gU(\theta)\mid=\sum_m n_m$. We denote by $X^{\gU(\theta)}$ the corresponding partially symmetrized power of the curve $X^{\gU(\theta)}=\prod_m X^{(n_m)}$. Let $\oo{X}{}^{\gU(\theta)}\subset X^{\gU(\theta)}$ be the complement to all the diagonals in $X^{\gU(\theta)}$. We view $\oo{X}{}^{\gU(\theta)}$ as a locally closed subscheme of $X^{\theta}$ via the map $\oo{X}{}^{\gU(\theta)}\to X^{\theta}$, $(D_m)\mapsto \sum_m D_m\theta_m$.
\index{$X^{\gU(\theta)}, \oo{X}{}^{\gU(\theta)}, {_{\gU(\theta)}\Bunb_{B^-}}$}
  
   Set $_{\gU(\theta)}\Bunb_{B^-}=\Bun_{B^-}\times \oo{X}{}^{\gU(\theta)}$. 
We get locally closed immersions $_{\gU(\theta)}\Bunb_{B^-}\hook{}\Bun_{B^-}\times X^{\theta}\hook{} \Bunb_{B^-}$, the second one sending $(\cF, \cF_T, \kappa^-, D)$ to $(\cF, \cF_T(-D), \kappa^-)$. Let $_{\gU(\theta)}\Bunb_{\tilde B^-}$ be obtained from $_{\gU(\theta)}\Bunb_{B^-}$ by the base change $\Bunb_{\tilde B^-}\to \Bunb_{B^-}$. 
   
   Let $\cH^{+,\gU(\theta)}_{T}$ be the stack classifying $\cF_T\in\Bun_T$, $D\in \oo{X}{}^{\gU(\theta)}$ viewed as a point of $X^{\theta}$. Let $\cH^{+,\gU(\theta)}_{\tilde T}$ be the stack classifying a point of $\cH^{+,\gU(\theta)}_{T}$ as above, and lines $\cU,\cU_G$ equipped with
$$
\cU^N\,\iso\, (^{\omega}\cL^{\bar\kappa})_{\cF_T(-D)},\;\;\;\; \cU_G^N\,\iso\,   (^{\omega}\cL^{\bar\kappa})_{\cF_T} \; .
$$
\index{$_{\gU(\theta)}\Bunb_{\tilde B^-}$}\index{$\cH^{+,\gU(\theta)}_{T}$}
As in (\cite{L2}, Section~4.4.1), we have an isomorphism 
\begin{equation}
\label{def_stack__gU(theta)Bunb_tildeB^-}
_{\gU(\theta)}\Bunb_{\tilde B^-}\,\iso\, \Bun_{B^-}\times_{\Bun_T} \,\cH^{+,\gU(\theta)}_{\tilde T},
\end{equation}
where to define the fibred product we used the map $\cH^{+,\gU(\theta)}_{\tilde T}\to\Bun_T$ sending the above point to $\cF_T$.  
   
   Consider the line bundle on $\oo{X}{}^{\gU(\theta)}$, whose fibre at $D$ is $\cL^{\bar\kappa}_{\cF^0_T(-D)}$, here we view $\oo{X}{}^{\gU(\theta)}\subset X^{\theta}$ as a subscheme. Let $\wt\Gr_T^{+,\gU(\theta)}$ be the gerbe of $N$-th roots of this line bundle. Call $V\in\Rep(\check{T}_{\zeta})$ negative if each $\check{T}_{\zeta}$-weight appearing in $V$ lies in $-\Lambda^{pos}$. Actually, such a weight is in $-\Lambda^{\sharp, pos}$, where $\Lambda^{\sharp, pos}=\Lambda^{\sharp}\cap\Lambda^{pos}$.
\index{$\wt\Gr_T^{+,\gU(\theta)},  \Lambda^{\sharp, pos}$}  
   
  For $V\in\Rep(\check{T}_{\zeta})$ negative we get a perverse sheaf $\Loc^{\gU(\theta)}_{\zeta}(V)$ on $\wt\Gr_T^{+,\gU(\theta)}$ on which $\mu_N(k)$ acts by $\zeta$, and such that for $D=\sum_k \theta_k x_k\in \oo{X}{}^{\gU(\theta)}$ its restriction to 
$$
\prod_k \wt\Gr_{T, x_k}^{\theta_k}  
$$
is $(\boxtimes_k \Loc_{\zeta}(V_{\theta_k}))[\mid \gU(\theta)\mid]$. Here $\Gr_{T, x}^{\theta}$ is the connected component of $\Gr_{T,x}$ containing $t_x^{\theta}T(\cO)$, in other words, corresponding to $\cF^0_T(-\theta x)$ with the evident trivialization off $x$. The functor $\Loc_{\zeta}$ was defined in Section~\ref{section_Metaplectic dual group}. Note that $\Loc^{\gU(\theta)}_{\zeta}(V)$ vanishes unless in the decomposition (\ref{def_gU(theta)}) each term lies in $-\Lambda^{\sharp, pos}$. 
\index{$\Loc^{\gU(\theta)}_{\zeta}(V)$}\index{$\Gr_{T, x}^{\theta}$}

 For $V\in\Rep(\check{T}_{\zeta})$ negative define a perverse sheaf $\Loc^{\gU(\theta)}_{\Bun_T,\zeta}(V)$ on $\cH^{+,\gU(\theta)}_{\tilde T}$ as follows. Let $\Bun_{T, \gU(\theta)}$ denote the stack classifying $\cF_T\in\Bun_T, D\in \oo{X}{}^{\gU(\theta)}$, and a trivialization of $\cF_T$ over the formal neighbourhood of $D$. Let $\Bunt_{T, \gU(\theta)}=\Bun_{T, \gU(\theta)}\times_{\Bun_T}\Bunt_T$. Let $T_{\gU(\theta)}$ be the scheme classifying $D\in \oo{X}{}^{\gU(\theta)}$ and a section of $T$ over the formal neighbourhood of $D$, this is a group scheme over $\oo{X}{}^{\gU(\theta)}$. For $(\cF_T, D)\in \Bun_{T, \gU(\theta)}$ we have a natural isomorphism $(^{\omega}\cL^{\bar\kappa})_{\cF_T}\otimes (\cL^{\bar\kappa})_{\cF^0_T(-D)}\,\iso\, (^{\omega}\cL^{\bar\kappa})_{\cF_T(-D)}$. So, as in (\cite{L2}, Section~4.4.2), we get a $T_{\gU(\theta)}$-torsor
$$
\Bunt_{T, \cU(\theta)}\times_{\oo{X}{}^{\gU(\theta)}} \wt\Gr_T^{+,\gU(\theta)}\to \cH^{+,\gU(\theta)}_{\tilde T}
$$
\index{$\Loc^{\gU(\theta)}_{\Bun_T,\zeta}(V)$}\index{$\Bun_{T, \gU(\theta)}$}%
\index{$\Bunt_{T, \gU(\theta)}, T_{\gU(\theta)}$}
For $\cT\in\D(\Bunt_T)$ and a $T_{\gU(\theta)}$-equivariant perverse sheaf $S$ on $\wt\Gr_T^{+,\gU(\theta)}$ we may form their twisted product $\cT\tboxtimes S$ on $\cH^{+,\gU(\theta)}_{\tilde T}$ using the above torsor. The perverse sheaf  $\Loc^{\gU(\theta)}_{\zeta}(V)$ on $\wt\Gr_T^{+,\gU(\theta)}$ is naturally $T_{\gU(\theta)}$-equivariant. For $V\in\Rep(\check{T}_{\zeta})$ negative define
$$
\Loc^{\gU(\theta)}_{\Bun_T,\zeta}(V)=\IC(\Bunt_T)\tboxtimes \Loc^{\gU(\theta)}_{\zeta}(V)
$$
For the map $\gq^-: \Bun_{B^-}\to\Bun_T$ on (\ref{def_stack__gU(theta)Bunb_tildeB^-}) we get the perverse sheaf denoted
$$
\Loc^{\gU(\theta)}_{\Bun_B,\zeta}(V)=(\gq^-)^*\Loc^{\gU(\theta)}_{\Bun_T,\zeta}(V)[\dimrel(\gq^-)]
$$ 
\index{$\Loc^{\gU(\theta)}_{\Bun_B,\zeta}(V)$}
\begin{Thm}[\cite{L2}, Theorem~4.1] 
\label{Thm_3.10.1}
The $*$-restriction of $\IC_{\zeta}$ to $_{\gU(\theta)}\Bunb_{\tilde B^-}$ vanishes unless in the decomposition (\ref{def_gU(theta)}) each term lies in $-\Lambda^{\sharp, pos}$. In the latter case it is isomorphic to
$$
\Loc^{\gU(\theta)}_{\Bun_B, \zeta}(\mathop{\oplus}\limits_{i\ge 0} \Sym^i (\check{\gu}^-_{\zeta})[2i])\otimes \Qlb[-\mid\gU(\theta)\mid],
$$
where $\mathop{\oplus}\limits_{i\ge 0} \Sym^i (\check{\gu}^-_{\zeta})[2i]$ is viewed as a cohomologically graded $\check{T}_{\zeta}$-module.
\end{Thm}

\ssec{}

Our purpose now is to improve Proposition~\ref{Pp_3.3.4} as follows.

\begin{Pp}
\label{Pp_4.11.1}
 i) Assume $\varrho(\alpha)\notin \ZZ$ for any simple coroot $\alpha$. Then for $\mu\in -\Lambda^{pos}$ we have a (non-canonical) isomorphism $\cL^{\mu}_{\emptyset}\,\iso\, \ov{\FF}(\cF_{\emptyset})$ in $\D_{\zeta}(\oo{X}{}^{\mu})$.\\
ii) The complex $\ov{\FF}(\cF_{\emptyset})$ is perverse. If in addition the subtop cohomology property is satisfied for $\varrho$ then we have a (non-canonical) isomorphism $\cL^{\mu}_{\emptyset}\,\iso\, \ov{\FF}(\cF_{\emptyset})$ in $\D_{\zeta}(X^{\mu})$.
\end{Pp}
\begin{proof} i) If $-\mu$ is a simple coroot of $G$ then, by Theorem~\ref{Thm_3.10.1}, $\bar F^{\mu}(\cF_{\emptyset})$ is the extension by zero under $\cZ^{\mu}\hook{}\ov{\cZ}^{\mu}$. Therefore, over $\oo{X}{}^{\mu}$ the desired isomorphism follows from the factorization property combined with Proposition~\ref{Pp_3.3.4}. 

\smallskip\noindent
ii) Denote by $\ov{\FF}(\cF_{\emptyset})_{\mu x}$ the $*$-fibre of $\ov{\FF}(\cF_{\emptyset})$ at $\mu x\in X^{\mu}$. If $D=\sum_k \mu_k x_k\in X^{\mu}$ with $x_k$ pairwise different, the $*$-fibre of $\ov{\FF}(\cF_{\emptyset})$ at $D$, by factorization property, identifies with
$$ 
\boxtimes_k \; \ov{\FF}(\cF_{\emptyset})_{\mu_k x_k}
$$
Our claim is reduced to the following Proposition~\ref{Pp_3.11.2}. 
\end{proof}

\begin{Pp} 
\label{Pp_3.11.2} 
Let $x\in X$ and $\mu<0$. \\
i) The complex $\ov{\FF}(\cF_{\emptyset})_{\mu x}$ is placed in degree $\le -1$. \\
ii) Assume in addition that the subtop cohomology property is satisfied for $\varrho$. Then $\ov{\FF}(\cF_{\emptyset})_{\mu x}$ is placed in degree $< -1$ unless $-\mu$ is a simple coroot.
\end{Pp}
\begin{proof}
% Recall that $\ov{\cZ}^{\mu}_{loc, x}$ denotes the fibre of $\bar\pi^{\mu}: \ov{\cZ}^{\mu}\to X^{\mu}$ over $\mu x$. 
We are integrating over the fibre, say $Y$, of $\oo{\ov{\cZ}}{}^{\mu}$ over $\mu x$. From (\ref{fibre_of_ov_cZ^mu}), $Y$ identifies with $(\Gr_B^0\cap \ov{\Gr}_{B^-}^{\mu})\times^{T(\cO_x)}\Omega^{\rho}\mid_{D_x}$. The restriction of $\bar F^{\mu}(\cF_{\emptyset})$ to the stratum 
$$
(\Gr_B^0\cap \Gr_{B^-}^{\mu})\times^{T(\cO_x)}\Omega^{\rho}\mid_{D_x}
$$ 
is a local system placed in usual degree $\<\mu, 2\check{\rho}\>$. 

 Denote by $\ev_x: \Gr_B^0\times^{T(\cO_x)}\Omega^{\rho}\mid_{D_x}\to \A^1$ the restriction of the canonical map $\ev: \gM_{\emptyset, 0}\to\A^1$. As is explained in (\cite{G}, Section~5.6), the local system $\ev_x^*\cL_{\psi}$
is nonconstant on each irreducible component of $(\Gr_B^0\cap \Gr_{B^-}^{\mu})\times^{T(\cO_x)}\Omega^{\rho}\mid_{D_x}$ of dimension $-\<\mu, \check{\rho}\>$. So, the restriction of $\bar F^{\mu}(\cF_{\emptyset})$ to each such irreducible component is also nonconstant. Thus, the contribution of the stratum $\Gr_B^0\cap \Gr_{B^-}^{\mu}$ is placed in the usual degree $\le -1$. 
\index{$\ev_x$}

 For $\mu=\nu+\theta$ with $\nu,\theta< 0$ consider the stratum $Y_{\nu}:=(\Gr_B^0\cap \Gr_{B^-}^{\nu})\times^{T(\cO_x)}\Omega^{\rho}\mid_{D_x}$ of $Y$. Let $\gU(\theta)$ be the trivial decomposition $\theta=\theta$, so $\oo{X}{}^{\gU(\theta)}=X$. 
Pick some trivialization of the line $\cL^{\bar\kappa}_{\cF^0_T(-\theta x)}$.
This allows for $V\in\Rep(\check{T}_{\zeta})$ to see $\Loc_{\zeta}(V_{\theta})$ as a complex over $\Spec k$. Then the $*$-restriction of $\bar F^{\mu}(\cF_{\emptyset})$ to $Y_{\nu}$ identifies with
$$
\Loc_{\zeta}((\mathop{\oplus}\limits_{i\ge 0} \Sym^i (\check{\gu}^-_{\zeta})[2i])_{\theta})\otimes \ev_x^*\cL_{\psi}\otimes\cE[-\<2\check{\rho}, \nu\>],
$$
where $\cE$ is a rank one tame local system. If $\nu\ne 0$ then $\ev_x^*\cL_{\psi}\otimes\cE$ is nontrivial on each irreducible component of $Y_{\nu}$ of dimension $-\<\check{\rho}, \nu\>$. Since $\Loc_{\zeta}((\mathop{\oplus}\limits_{i\ge 0} \Sym^i (\check{\gu}^-_{\zeta})[2i])_{\theta})$ is placed in degrees $<0$, for $\nu\ne 0$ the contribution of $Y_{\nu}$ is placed in degrees $\le -2$. 

 For $\nu=0$ we get $Y_{\nu}=\Spec k$. The $*$-restriction of $\bar F^{\mu}(\cF_{\emptyset})$ to this point identifies with 
$$
\Loc_{\zeta}((\mathop{\oplus}\limits_{i\ge 0} \Sym^i (\check{\gu}^-_{\zeta})[2i])_{\mu}),
$$
the latter is placed in degrees $\le -2$. So, $\ov{\FF}(\cF_{\emptyset})_{\mu x}$ is placed in degree $\le -1$, and only the open stratum $Y_{\mu}$ may contribute to the cohomology group $\H^{-1}(\ov{\FF}(\cF_{\emptyset})_{\mu x})$. 

\smallskip\noindent
ii) By definition of the subtop cohomology property, the open stratum $Y_{\mu}$ does not contribute to $\H^{-1}(\ov{\FF}(\cF_{\emptyset})_{\mu x})$. 
\end{proof}

\begin{Rem} Conjecture~\ref{Con_main} would imply the following. Assume $\varrho(\alpha)\notin \ZZ$ for any simple coroot $\alpha$. Then $\cL^{\mu}_{\emptyset}\;\iso\; \ov{\FF}(\cF_{\emptyset})$ in $\D_{\zeta}(X^{\mu})$.
\end{Rem}

\begin{Pp} 
\label{Pp_ovFF_is_exact}
The functor $\ov{\FF}: \D\Whit^{\kappa}_n\to \D_{\zeta}(X^{\mu}_n)$ is exact for the perverse t-structures.
\end{Pp}
\begin{proof} Pick $K\in \Whit^{\kappa}_n$. Let $\eta: \{1,\ldots, n\}\to A$ be a surjection. Pick 
$\mu_a\in\Lambda$ for $a\in A$ with $\sum_a \mu_a=\mu$. Let $V\subset X^{\mu}_n$ be the subscheme classifying disjoint points $\{y_a\in X\}_{a\in A}$ such that $x_i=y_{\eta(i)}$ for each $i$, and $D=\sum_{a\in A} \mu_a y_a$. In view of the factorization property and Propositions~\ref{Pp_391}, \ref{Pp_4.11.1}, it suffices to show that the $*$-restriction of $\ov{\FF}(K)$ to $V$ is placed in perverse degrees $\le 0$. Let $\ov{\cZ}^{\mu}_V$ be the preimage of $V$ under $\bar\pi^{\mu}: \ov{\cZ}^{\mu}_n\to X^{\mu}_n$. The fibre of $\ov{\cZ}^{\mu}_V$ over $\{y_a\}$ is 
$$
\prod_a \ov{\Gr}_{B^-, y_a}^{\mu_a}\times^{T(\cO_{y_a})}\Omega^{\rho}\mid_{D_{y_a}}
$$
 
  Pick a collection $\bar\lambda=\{\lambda_a\}_{a\in A}$ with $\lambda_a\in\Lambda^+$, $\mu_a\le\lambda_a$. Let $\gM_{\eta, \bar\lambda}\subset \gM_n$ be the substack classifying a point of $V$ as above (this defines $x_i$), and such that for each $\check{\lambda}\in\check{\Lambda}^+$ the map
$$
\kappa^{\check{\lambda}}: \Omega^{\<\rho, \check{\lambda}\>}\to \cV^{\check{\lambda}}_{\cF}(\sum_a \<\lambda_a y_a, \check{\lambda}\>)
$$
is regular over $X$ and has no zeros over $X$. Let $\ov{\cZ}^{\mu}_{V,\bar\lambda}$ be obtained from $\ov{\cZ}^{\mu}_V$ by the base change $\gM_{\eta, \bar\lambda}\to \gM_n$. Let $\pi_{\eta}: \gM_{\eta, \bar\lambda}\to V$ be the projection, $\ev_{\bar\lambda}: \gM_{\eta, \bar\lambda}\to\A^1$ the corresponding evaluation map (as in Section~\ref{Section_2.3}). 
Let $K^{\bar\lambda}$ be a complex on $V$ placed in perverse degrees $\le 0$ such that the $*$-restriction $K\mid_{\gM_{\eta, \bar\lambda}}$ identifies with
$$
\pi_{\eta}^*K^{\bar\lambda}\otimes \ev_{\bar\lambda}^*\cL_{\psi}[\dim],
$$
where $\dim=(g-1)\dim U-\<(2g-2)\rho-\sum_a\lambda_a, 2\check{\rho}\>$. This is the relative dimension of $\pi_{\eta}$. 
  
  Only finite number of the strata $\ov{\cZ}^{\mu}_{V,\bar\lambda}$ of $\ov{\cZ}^{\mu}_V$ contribute to $\ov{\FF}(K)\mid_V$. Let $K_{\bar\lambda}$ denote the $!$-direct image under $\bar\pi^{\mu}: \ov{\cZ}^{\mu}_{V,\bar\lambda}\to V$ of the $*$-restriction $\bar F^{\mu}(K)\mid_{\ov{\cZ}^{\mu}_{V,\bar\lambda}}$. It suffices to show that $K_{\bar\lambda}$ is placed in perverse degrees $\le 0$. From Theorem~\ref{Thm_3.10.1} we conclude that $K_{\bar\lambda}\,\iso\, K^{\bar\lambda}\otimes M$, where $M$ is a complex on $V$ with locally constant cohomology sheaves. It remains to show that $M$ is placed in degrees $\le 0$.  
  
  The problem being local, we may and do assume that $A$ is the one element set. Write $\mu=\mu_a$, $\lambda_a=\lambda$, $y_a=y$. Then the fibre $Y$  of $\ov{\cZ}^{\mu}_{V, \lambda}$ over $y$ is 
$$
(\Gr_{B,y}^{\lambda}\cap
\ov{\Gr}_{B^-, y}^{\mu})\times^{T(\cO_{y})}\Omega^{\rho}\mid_{D_{y}}
$$
For $\mu\le\nu\le \lambda$ let $Y_{\nu}=(\Gr_{B,y}^{\lambda}\cap
\Gr_{B^-, y}^{\nu})\times^{T(\cO_{y})}\Omega^{\rho}\mid_{D_{y}}$, they form a stratification of $Y$. For $\mu=\nu+\theta$ with $\nu\le\lambda, \theta\le 0$ let $\gU(\theta)$ be the trivial decomposition $\theta=\theta$, so $\oo{X}{}^{\gU(\theta)}=X$. 
Pick some trivialization of the line $\cL^{\bar\kappa}_{\cF^0_T(-\theta x)}$.
This allows for $V\in\Rep(\check{T}_{\zeta})$ to see $\Loc_{\zeta}(V_{\theta})$ as a complex over $\Spec k$ (as in Proposition~\ref{Pp_3.11.2}). 
The $*$-restriction $\bar F^{\mu}(K)\mid_{Y_{\nu}}$ identifies with 
$$
\Loc_{\zeta}((\mathop{\oplus}\limits_{i\ge 0} \Sym^i (\check{\gu}^-_{\zeta})[2i])_{\theta})\otimes\ev_{\bar\lambda}^*\cL_{\psi}\otimes\cE\otimes K^{\bar\lambda}_y[\<\lambda-\nu, 2\check{\rho}\>],
$$
where $\cE$ is some rank one local system. Since $\dim Y_{\nu}\le \<\lambda-\nu, \check{\rho}\>$, we see that the contribution of $Y_{\nu}$ to the complex $M_y$ is placed in degrees $\le 0$. We are done. 
\end{proof}

 Combining Propositions~\ref{Pp_factorization_of_barF^mu}, \ref{Pp_ovFF_is_exact}, one gets the following.

\begin{Thm} 
\label{Thm_main_functor_ovFF}
Assume that $\varrho$ satisfies the subtop cohomology property. Then $\ov{\FF}$ gives rise to the functor $\ov{\FF}: \Whit^{\kappa}_n\to \wt\FS^{\kappa}_n$, which is exact for the perverse t-structures and commutes with the Verdier duality (up to replacing $\psi$ by $\psi^{-1}$ and $\zeta$ by $\zeta^{-1}$). 
\end{Thm}

\ssec{Multiplicity spaces} 
\label{Section_Multiplicity spaces}

\sssec{} For a topological space $\cX$ write $\Irr(\cX)$ for the set of irreducible components of $\cX$. Recall for $\nu\ge 0$ the notation $B_{\gg}(\nu)$ and the functions $\phi_i$ on this crystal from Section~\ref{Section_121}. 
\index{$\Irr(\cX)$}

Let $\mu\in\Lambda, \lambda\in\Lambda^+$ with $\mu\le\lambda$. Let $b\subset \Gr_B^{\lambda}\cap \Gr_{B^-}^{\mu}$ be an irreducible component. Denote by $\bar b\subset \Gr_B^0\cap\Gr_{B^-}^{\mu-\lambda}$ the component $t^{-\lambda}b$, so $\bar b\in B_{\gg}(\lambda-\mu)$. By Anderson's theorem (\cite{A}, Proposition~3) we have a bijection
\begin{equation}
\label{bij_Andersen}
\{a\in \Irr(\Gr_{B^-}^{\mu}\cap \Gr_B^{\lambda})\mid a\subset \ov{\Gr}_G^{\lambda}\}\,\iso\, \Irr(\Gr_G^{\lambda}\cap \Gr_{B^-}^{\mu})
\end{equation}
sending $a$ to the closure of $a\cap \Gr_G^{\lambda}$. 
\begin{Lm} 
\label{Lm_Andersen_again}
Under the above assumptions the following are equivalent.\\
i) For all $i\in\cJ$, $\phi_i(\bar b)\le \<\lambda, \check{\alpha}_i\>$, \\ 
ii) $b\subset \ov{\Gr}_G^{\lambda}$.
\end{Lm}
\begin{proof} Recall the canonical inclusion
$
B(-w_0(\lambda))\hook{} T_{-\lambda}\otimes B(-\infty)
$
from (\cite{BauG}, p. 87), see also Section~\ref{Section_Recollections on crystals}. Its image is the set of $t_{-\lambda}\otimes a$ such that $a\in B(-\infty)$, and for each $i\in \cJ$, $\phi_i(a^*)\le \<\check{\alpha}_i, \lambda\>$. So, i) is equivalent to $t_{-\lambda}\otimes \bar b^*\in B(-w_0(\lambda))$.  By (\cite{A}, Proposition~3), we have a canonical bijection of irreducible components (up to passing to the closure)
$$
\Irr(t^{\mu}\Gr_G^{-w_0(\lambda)}\cap \Gr_B^0)\,\iso\,\{a\in \Irr(\Gr_B^0\cap \Gr_{B^-}^{\mu-\lambda})\mid a\subset t^{\mu}\ov{\Gr}_G^{-w_0(\lambda)}\}
$$
So, i) is equivalent to the property that $t^{-\mu}\bar b^*\in \Irr(\Gr_G^{-w_0(\lambda)}\cap \Gr_B^{-\mu})$. Our claim follows now from the properties of the bijection $\ast: B(-\infty)\to B(-\infty)$ and (\ref{bij_Andersen}). 
\end{proof}
 
\sssec{Additional input data} 
\label{section_Additional input data}
Recall that the pull-back of the central extension (\ref{ext_V_EE_of_Lambda}) to $\Lambda^{\sharp}$ is abelian. Pick a splitting $\gt_{\EE}^0:\Lambda^{\sharp}\to V_{\EE}$ of the exact sequence (\ref{ext_V_EE_of_Lambda}) over $\Lambda^{\sharp}$. We assume $\gt_{\EE}^0$ is compatible with the section $\gt_{\EE}$ from Section~\ref{section_Metaplectic dual group}.
\index{$\gt_{\EE}^0, \delta_{\lambda}$}
  
  For each $\bar\lambda\in \Lambda/\Lambda^{\sharp}$ we make the following choice. Pick compatible trivializations $\delta_{\lambda}: (V_{\EE})_{\lambda}\,\iso\, \Gm$ of the fibre of $\Gra_G\to \Gr_{G}$ at $t^{\lambda}G(\cO)$ for all $\lambda\in\Lambda$ over $\bar\lambda$. Here compatible means equivariant under the action of $\Lambda^{\sharp}$ via $\gt_{\EE}^0$. 
     
\sssec{} 
\label{section_4.12.4}
For $\lambda,\mu\in\Lambda$ the above trivializations $\delta_{\lambda}$ yield sections $s^{\lambda}_B: \Gr_B^{\lambda}\to\Gra_G$, $s^{\mu}_{B^-}: \Gr_{B^-}^{\mu}\to\Gra_G$ of the $\Gm$-torsor $\Gra_G\to\Gr_G$. The discrepancy between them is a map that we denote by 
$$
\gamma^{\mu}_{\lambda}: \Gr_B^{\lambda}\cap \Gr_{B^-}^{\mu}\to \Gm
$$ 
and define by $s^{\mu}_{B^-}=\gamma^{\mu}_{\lambda} s^{\lambda}_B$. 
Note that if $\lambda-\mu\in\Lambda^{\sharp}$ then $\gamma^{\mu}_{\lambda}$ does not depend of the choice of $\delta$ (so depends only on $\gt^0_{\EE}$).  
\index{$s^{\lambda}_B, s^{\mu}_{B^-}, \gamma^{\mu}_{\lambda}, \cL_{x,\mu}$}

\begin{Thm} 
\label{Th_4.12.1}
Assume that $\varrho$ satisfies the subtop cohomology property. Pick $\lambda\in\Lambda^+$ and $x\in X$. There is a decomposition 
\begin{equation}
\label{decomp_FF_of_irred_object}
\ov{\FF}(\cF_{x,\lambda})\,\iso\, \mathop{\oplus}\limits_{\mu\le\lambda, \; \lambda-\mu\in\Lambda^{\sharp}} \; \cL_{x,\mu}\otimes V_{\mu}^{\lambda}
\end{equation}
in $\wt\FS^{\kappa}_x$, where $V_{\mu}^{\lambda}$ is the $\Qlb$-vector space with a canonical base indexed by those $b\in \Irr(\Gr^{\lambda}_{B,x}\cap \Gr_{B^-, x}^{\mu})$ 
that satisfy the following two properties:
\begin{itemize}
\item $b\subset \ov{\Gr}_{G,x}^{\lambda}$,
\item the local system $(\gamma^{\mu}_{\lambda})^*\cL_{\zeta}$ is trivial on $b$.
\end{itemize}
In particular, we have $V^{\lambda}_{\lambda}=\Qlb$. 
\end{Thm}

\sssec{Proof of Theorem~\ref{Th_4.12.1}} 
Recall that $\cF_{x,\lambda}$ is the extension by zero from $\wt\gM_{x,\le\lambda}$. Since $\bar\pi^{\mu}$ factors through $\bar\pi^{\mu}: \ov{\cZ}^{\mu}_{x,\le\lambda}\to X^{\mu}_{x,\le\lambda}$, $\ov{\FF}(\cF_{x,\lambda})$ is the extension by zero from $X^{\mu}_{x,\le\lambda}$. The latter scheme is empty unless $\mu\le\lambda$. So, the $\mu$-component of $\ov{\FF}(\cF_{x,\lambda})$ vanishes unless $\mu\le\lambda$. 

By Corollary~\ref{Corollary_392}, since $\bar\pi^{\mu}$ is proper for each $\mu$, there is a decomposition 
\begin{equation}
\label{decomp_FF_of_irred_object_preliminary}
\ov{\FF}(\cF_{x,\lambda})\,\iso\, \mathop{\oplus}\limits_{\mu\le\lambda} \; \cL_{x,\mu}\otimes V_{\mu}^{\lambda} \, .
\end{equation}
It remains to determine the spaces $V_{\mu}^{\lambda}$. Pick $\mu\le\lambda$. Set for brevity $\gamma=\gamma^{\mu}_{\lambda}$. Recall the notation $\chi^{\lambda}_0: \Gr_{B,x}^{\lambda}\to\A^1$ from Section~\ref{Section_1.1}. 

\begin{Lm} The space $V^{\lambda}_{\mu}$ in (\ref{decomp_FF_of_irred_object_preliminary}) has a canonical base consiting of those irreducible components of $\Gr_{B,x}^{\lambda}\cap \Gr_{B^-, x}^{\mu}$ over which the local system $(\chi^{\lambda}_0)^*\cL_{\psi}\otimes \gamma^*\cL_{\zeta}$ is constant. 
\end{Lm}
\begin{proof}
Since $\ov{\FF}(\cF_{x,\lambda})\in\wt\FS^{\kappa}_x$,
it suffices to determine the fibre $K:=\ov{\FF}(\cF_{x,\lambda})_{\mu x}$. 
By Proposition~\ref{Pp_ovFF_is_exact}, $K$ is placed in degrees $\le 0$. 
Pick a trivialization of $\cP^{\bar\kappa}$ at $\mu x\in X^{\mu}_{x,\le\lambda}$. This allows to see $K$ as a complex over $\Spec k$, it also determines $\cL_{x,\mu}$ up to a unique isomorphism, so yields an isomorphism 
$$
V^{\lambda}_{\mu}\,\iso\, 
\H^0(\ov{\FF}(\cF_{x,\lambda})_{\mu x})
$$ 

 The fibre of $\bar\pi^{\mu}: \ov{\cZ}^{\mu}_{x,\le\lambda}\to X^{\mu}_{x,\le\lambda}$ over $\mu x$ is 
$$
Y:=(\ov{\Gr}_{B,x}^{\lambda}\cap \ov{\Gr}_{B^-, x}^{\mu})\times^{T(\cO_x)}\Omega^{\rho}\mid_{D_x}
$$ 
For $\eta\in\Lambda^+$, $\eta\le \lambda$ let 
$$
Y_{\eta}=(\Gr_{B,x}^{\eta}\cap \ov{\Gr}_{B^-, x}^{\mu})\times^{T(\cO_x)}\Omega^{\rho}\mid_{D_x}
$$  
Denote by $K^{\eta}$ the constant complex over $\Spec k$ such that  $j_{x, \eta}^*\cF_{x,\lambda}\,\iso\, K^{\eta}\otimes \cF_{x, \eta}$.
Here $K^{\eta}$ is placed in degrees $<0$ for $\eta<\lambda$, and $K^{\lambda}=\Qlb$. 
 
 Let $K_{\eta}$ be the contribution of the $*$-restriction $\cF_{x,\lambda}\mid_{\wt\gM_{x,\eta}}$ to $K$. In other words, 
$$
K_{\eta}=\RG_c(Y_{\eta}, \bar F^{\mu}(\cF_{x, \lambda})\mid_{Y_{\eta}}),
$$
where we used the $*$-restriction to $Y_{\eta}$, and the above trivialization of $\cP^{\bar\kappa}$ at $\mu x\in X^{\mu}_{x,\le\lambda}$ to get rid of the corresponding gerbe. By Proposition~\ref{Pp_ovFF_is_exact}, if $\eta<\lambda$ then $K_{\eta}$ is placed in degrees $<0$. So, it suffices to analyze $K_{\lambda}$. 

 For $\mu\le\nu\le \lambda$ let 
$$
Y_{\lambda, \nu}=(\Gr_{B,x}^{\lambda}\cap \Gr_{B^-, x}^{\nu})\times^{T(\cO_x)}\Omega^{\rho}\mid_{D_x} \, .
$$  
The schemes $Y_{\lambda, \nu}$ with $\mu\le\nu\le \lambda$ form a stratification of $Y_{\lambda}$. 

 For $\mu=\nu+\theta$ with $\nu\le\lambda,\theta\le 0$ let $\gU(\theta)$ be the trivial decomposition $\theta=\theta$. Pick a trivialization of the line $\cL^{\bar\kappa}_{\cF^0_T(-\theta x)}$. As in the proof of Proposition~\ref{Pp_ovFF_is_exact} this allows for $V\in \Rep(\check{T}_{\zeta})$ to see $\Loc_{\zeta}(V_{\theta})$ as a complex over $\Spec k$. The $*$-restriction $\bar F^{\mu}(\cF_{x, \lambda})\mid_{Y_{\lambda,\nu}}$ identifies with 
$$
\Loc_{\zeta}((\mathop{\oplus}\limits_{i\ge 0} \Sym^i (\check{\gu}^-_{\zeta})[2i])_{\theta})\otimes\ev_{x, \lambda}^*\cL_{\psi}\otimes \cE[\<\lambda-\nu, 2\check{\rho}\>],
$$
where $\cE$ is some rank one local system. Recall that $Y_{\lambda, \nu}$ is of pure dimension $\<\lambda-\nu, \check{\rho}\>$. So, the contribution $K_{\lambda, \nu}$ of $Y_{\lambda,\nu}$ to $K_{\lambda}$ is 
$$
\Loc_{\zeta}((\mathop{\oplus}\limits_{i\ge 0} \Sym^i (\check{\gu}^-_{\zeta})[2i])_{\theta})\otimes \RG_c(Y_{\lambda,\nu}, \ev_{x, \lambda}^*\cL_{\psi}\otimes \cE)[\<\lambda-\nu, 2\check{\rho}\>]
$$
It is placed in degrees $\le 0$, and the inequality is strict unless $\theta=0$. There remains to analyze the complex 
$$
K_{\lambda, \mu}=\RG_c(Y_{\lambda,\mu}, \ev_{x, \lambda}^*\cL_{\psi}\otimes \cE)[\<\lambda-\mu, 2\check{\rho}\>]
$$ 
We see that only the open part $\cZ^{\mu}_{x, \lambda}\subset \ov{\cZ}^{\mu}_{x,\le\lambda}$ contributes to the $0$-th cohomology of $K$. This allows to describe the local system $\cE$ over $Y_{\lambda,\mu}$. 
From the definitions we get $\gamma^*\cL_{\zeta}\,\iso\, \cE$. So, $K_{\lambda,\mu}$ identifies with
$$
\RG_c(\Gr_{B,x}^{\lambda}\cap \Gr_{B^-, x}^{\mu}, (\chi^{\lambda}_0)^*\cL_{\psi}\otimes \gamma^*\cL_{\zeta})[\<\lambda-\mu, 2\check{\rho}\>]
$$
for some character $\chi_0: U(F_x)\to\A^1$ of conductor zero. Our claim follows.
\end{proof} 
 
\begin{Lm} 
\label{Lm_when_chi_lambda_0_is_dominant}
Let $\mu\le\lambda$, $\lambda\in\Lambda^+$. Let $b\subset \Gr_B^{\lambda}\cap \Gr_{B^-}^{\mu}$ be an irreducible component. Denote by $\bar b\subset \Gr_B^0\cap \Gr_{B^-}^{\mu-\lambda}$ the component $t^{-\lambda}b$, so $\bar b\in B_{\gg}(\lambda-\mu)$. The restriction $\chi^{\lambda}_0: b\to \A^1$ of $\chi^{\lambda}_0$ is dominant if and only if there is $i\in\cJ$ such that 
%$\<\lambda+\mu_i, \check{\alpha}_i\><0$.
$\phi_i(\bar b)> \<\lambda, \check{\alpha}_i\>$. 
\end{Lm}
\begin{proof}
For $i\in\cJ$ recall the maps $\gq_{P_i}: \Gr_{P_i}\to \Gr_{M_i}$.
For $i\in\cJ$ let $\mu_i\le\lambda$ be the unique element such that $\gq_{P_i}^{-1}(\Gr_{B^-(M_i)}^{\mu_i})\cap b$ is dense in $b$. 
 Note that $b\subset \Gr_B^{\lambda}\cap \Gr_{B^-}^{\mu}$ is a $T(\cO)$-invariant subscheme. Let 
$$
b_0=b\cap (\mathop{\cap}\limits_{i\in\cJ} \gq_{P_i}^{-1}(\Gr_{B^-(M_i)}^{\mu_i}),
$$
it is a dense $T(\cO)$-invariant subscheme of $b$. Set $\bar\mu=\{\mu_i\}_{i\in \cJ}$ and 
$$
Z^{\bar\mu}=\prod_{i\in \cJ} \Gr_{B(M_i)}^{\lambda}\cap \Gr_{B^-(M_i)}^{\mu_i}\, .
$$ 

 Let $\gq^{\bar\mu}: b_0\to Z^{\bar\mu}$ be the product of the maps $\gq_{P_i}$. This map is $T(\cO)$-equivariant. Since $T(\cO)$ acts transitively on $Z^{\bar\mu}$, the map $\gq^{\bar\mu}$ is surjective. 
For $i\in\cJ$ let $\ev_i$ be the composition 
$$
\Gr_{B(M_i)}^{\lambda}\cap \Gr_{B^-(M_i)}^{\mu_i}\hook{} \Gr_{B(M_i)}^{\lambda}\to \Gr_B^{\lambda}\toup{\chi^{\lambda}_0}\A^1
$$ 
Denote by $\ev^{\bar\mu}: Z^{\bar\mu}\to\A^1$ the map $\ev^{\bar\mu}=\sum_{i\in \cJ} \ev_i$. The restriction $\chi^{\lambda}_0\mid_{b_0}$ equals $\ev^{\bar\mu}\gq^{\bar\mu}$. 

% Let $U(M_i)\subset B(M_i)$ be the unipotent radical.
% Under the left action of $U(M_i)(F)$, the map $ev_i$ is $(U(M_i)(F), \chi_0)$-equivariant. 

 Clearly, $\ev^{\bar\mu}: Z^{\bar\mu}\to\A^1$ is dominant if and only if there is $i\in\cJ$ such that $\ev_i: 
\Gr_{B(M_i)}^{\lambda}\cap \Gr_{B^-(M_i)}^{\mu_i}\to \A^1$ is dominant. The latter condition is equivalent to 
$$
\phi_i(\bar b)=\<\lambda-\mu_i, \frac{\check{\alpha}_i}{2}\> >\<\lambda, \check{\alpha}_i\>
$$ 
Indeed, the multiplication by $t^{\lambda}$ identifies $\Gr_{B(M_i)}^0\cap \Gr_{B^-(M_i)}^{\mu_i-\lambda}\,\iso\, \Gr_{B(M_i)}^{\lambda}\cap \Gr_{B^-(M_i)}^{\mu_i}$. Under the latter isomorphism $\ev_i$ identifies with some map $\chi^0_{\lambda}: \Gr_{B(M_i)}^0\cap \Gr_{B^-(M_i)}^{\mu_i-\lambda}\to \A^1$ for the group $M_i$. Our claim follows.
\end{proof} 

 The local system $(\chi^{\lambda}_0)^*\cL_{\psi}\otimes \gamma^*\cL_{\zeta}$ is constant on $b$ if and only if $\chi^{\lambda}_0: b\to \A^1$ is not dominant and the local system $\gamma^*\cL_{\zeta}$ is constant on $b$. The map $\gamma$ intertwines the natural $T(\cO)$-action on $\Gr_B^{\lambda}\cap \Gr_{B^-}^{\mu}$ with the $T(\cO)$-action on $\Gm$ by the character $T(\cO)\to T\toup{\bar\kappa(\lambda-\mu)}\Gm$.
So, the condition $\lambda-\mu\in\Lambda^{\sharp}$ is necessary (but not sufficient) for $\gamma^*\cL_{\zeta}$ to be trivial.
Theorem~\ref{Th_4.12.1} follows now from Lemmas~\ref{Lm_when_chi_lambda_0_is_dominant} and \ref{Lm_Andersen_again}. \QED

\sssec{Special case} Our purpose now is to understand the spaces $V^{\lambda}_{\mu}$ under the additional assumption $\lambda\in\Lambda^{\sharp, +}$.

\begin{Lm} 
\label{Lm_other_three_intersections_local_systems_coincide}
Let $\mu\le\lambda$ with $\mu\in\Lambda, \lambda\in \Lambda^{\sharp, +}$. Then over $\Gr_G^{\lambda}\cap \Gr_B^{\lambda}\cap \Gr_{B^-}^{\mu}$ there is an isomorphism $(s^{\mu}_{B^-})^*\cA^{\lambda}_{\cE}\,\iso\, (\gamma^{\mu}_{\lambda})^*\cL_{\zeta}$ up to a shift.
\end{Lm}
\begin{proof}
Recall that for any $\lambda\in\Lambda^+$ we have a section $s_{\lambda}: \Gr_{G,x}^{\lambda}\to \wt\Gr_{G,x}^{\lambda}$ defined in (\cite{L1}, Section~2.4.2) and associated to a square root $\Omega^{\frac{1}{2}}(\cO_x)$ of $\Omega(\cO_x)$ picked in Section~\ref{sssec_Notation}. In turn, $s^{\lambda}_B: \Gr_B^{\lambda}\to \Gra_G$ yields a section denoted $s^{\lambda}_B: \Gr_B^{\lambda}\to \wt\Gr_B^{\lambda}$ by abuse of notation. Since $\Gr_B^{\lambda}\cap \Gr_G^{\lambda}$ is an affine space, the local system $(s^{\lambda}_B)^*\cA^{\lambda}_{\cE}$ is trivial on $\Gr_B^{\lambda}\cap \Gr_G^{\lambda}$.
Our claim follows.
\end{proof}

 For $\lambda\in\Lambda^{\sharp,+}$ write $V(\lambda)$ for the irreducible represenation of $\check{G}_{\zeta}$ of highest weight $\lambda$. For $\mu\in\Lambda^{\sharp}$ let $V(\lambda)_{\mu}\subset V(\lambda)$ denote the subspace of $\check{T}_{\zeta}$-weight $\mu$. 
\index{$V(\lambda), V(\lambda)_{\mu}$} 

\begin{Thm} 
\label{Thm_special_case_of_V_lambda_mu}
Let $\mu\in\Lambda^{\sharp}, \lambda\in\Lambda^{\sharp,+}$ with $\mu\le\lambda$. Then the vector space $V^{\lambda}_{\mu}$ in the formula (\ref{decomp_FF_of_irred_object}) of Theorem~\ref{Th_4.12.1} identifies canonically with $V(\lambda)_{\mu}$. 
\end{Thm}
\begin{proof} By (\cite{L1}, Lemma~3.2) applied to $B^-$ instead of $B$, the space $V(\lambda)_{\mu}$ admits a canonical base indexed by those $b\in \Irr(\Gr_G^{\lambda}\cap \Gr_{B^-}^{\mu})$ over which the shifted local system $(s^{\mu}_{B^-})^*\cA^{\lambda}_{\cE}$ is trivial. The space $V^{\lambda}_{\mu}$ has a canonical base of $b\in \Irr(\Gr_G^{\lambda}\cap \Gr_{B^-}^{\mu})$ such that $(\gamma^{\mu}_{\lambda})^*\cL_{\zeta}$ is trivial at the generic point of $b$. Our claim follows now from Lemma~\ref{Lm_other_three_intersections_local_systems_coincide}.
\end{proof}

\part{Properties of the functor $\ov{\FF}$}

\medskip

\section{Hecke functors}
\label{Section_Hecke functors}

\ssec{Action on $\D_{\zeta}(\Bunt_G)$} 
\label{Section_action_Hecke_on_BuntG}
In the case of $G$ simple simply-connected the Hecke functors on $\D_{\zeta}(\Bunt_G)$ are defined in (\cite{L2}, Section~3.2). Let us first define their analogs in our setting. 

 Write $\cH_G$ for the Hecke stack classifying $\cF, \cF'\in\Bun_G, x\in X$ and an isomorphism $\cF\,\iso\, \cF'\mid_{X-x}$. We have a diagram
$$
\Bun_G\times X\getsup{h^{\la}_G\times\pi} \cH_G\toup{h^{\ra}_G}\Bun_G,
$$
where $h^{\la}_G$ (resp., $h^{\ra}_G$) sends the above point to $\cF$ (resp., to $\cF'$). Here $\pi(\cF,\cF',x)=x$.  

 Let $\Gr_{G,X}$ be the ind-scheme classifying $x\in X$ and a $G$-torsor $\cF$ on $X$ with a trivialization $\cF\,\iso\, \cF^0_G\mid_{X-x}$. Let $G_X$ be the group scheme over $X$ classifying $x\in X$ and an automorphism of $\cF^0_G$ over $D_x$. The restriction of $\cL^{\bar\kappa}$ under $\Gr_{G,X}\to \Bun_G$ is also denoted $\cL^{\bar\kappa}$. Let $\wt\Gr_{G,X}$ denote the gerbe of $N$-th roots of $\cL^{\bar\kappa}$ over $\Gr_{G,X}$. 
\index{$\cH_G, h^{\la}_G, h^{\ra}_G, \pi$}\index{$\Gr_{G,X}, G_X, \wt\Gr_{G,X}$}
 
 Write $\Bun_{G,X}$ for the stack classifying $(\cF\in\Bun_G, x\in X,\nu)$, where $\nu:\cF\,\iso\, \cF^0_G\mid_{D_x}$ is a trivialization over $D_x$. Let $\Bunt_{G,X}=\Bun_{G,X}\times_{\Bun_G}\Bunt_G$. Denote by $\gamma^{\la}$ (resp., $\gamma^{\ra}$) the isomorphism 
$$
\Bun_{G,X}\times_{G_X}\Gr_{G,X}\,\iso\, \cH_G
$$ 
such that the projection to the first term corresponds to $h^{\la}_G$ (resp., $h^{\ra}_G$). The line bundle $^{\omega}\cL^{\bar\kappa}\boxtimes \cL^{\bar\kappa}$ on $\Bun_{G,X}\times \Gr_{G,X}$ is $G_X$-equivariant, we denote by $^{\omega}\cL^{\bar\kappa}\tboxtimes \cL^{\bar\kappa}$ its descent to 
$
\Bun_{G,X}\times_{G_X}\Gr_{G,X}
$. We have canonically 
\begin{equation}
\label{iso_over_cH_G_first}
(\gamma^{\ra})^*(h^{\la}_G)^*(^{\omega}\cL^{\bar\kappa})\,\iso\, {^{\omega}\cL^{\bar\kappa}}\tboxtimes \cL^{\bar\kappa}
\end{equation}
\index{$\Bun_{G,X}, \Bunt_{G,X}, \gamma^{\la}, \gamma^{\ra}$}%
\index{${^{\omega}\cL^{\bar\kappa}\tboxtimes \cL^{\bar\kappa}}, \cH_{\tilde G}$}

 Let $\cH_{\tilde G}$ be the stack obtained from $\Bunt_G\times\Bunt_G$ by the base change $h^{\la}_G\times h^{\ra}_G: \cH_G\to\Bun_G\times\Bun_G$. A point of $\cH_{\tilde G}$ is given by $(\cF,\cF',x)\in\cH_G$ and lines $\cU,\cU'$ equipped with 
\begin{equation}
\label{iso_cU_and_cU'_for_cH_tildeG} 
 \cU^N\,\iso\, (^{\omega}\cL^{\bar\kappa})_{\cF}, \; \cU'^N\,\iso\, (^{\omega}\cL^{\bar\kappa})_{\cF'}
\end{equation} 
We get the diagram of projections
$$
\Bunt_G\getsup{\tilde h^{\la}_G} \cH_{\tilde G} \toup{\tilde h^{\ra}_G} \Bunt_G
$$
\index{$\tilde h^{\la}_G, \tilde h^{\ra}_G, \tilde\gamma^{\ra}$}%
\index{$(\cT\tboxtimes\cS)^r, (\cT\tboxtimes\cS)^l, K\mapsto \ast K$}
As in (\cite{L2}, Section~3.2), the isomorphism (\ref{iso_over_cH_G_first}) yields a $G_X$-torsor 
$$
\tilde\gamma^{\ra}: \Bunt_{G,X}\times_X\wt\Gr_{G,X}\to \cH_{\tilde G}
$$ 
extending the $G_X$-torsor
$
\Bun_{G,X}\times_X \Gr_{G,X}\to \Bun_{G,X}\times_{G_X}\Gr_{G,X}\toup{\gamma^{\ra}}\cH_G
$. Namely, it sends 
$$
(x, \nu': \cF'\,\iso\, \cF^0_G\mid_{D_x}, \nu_1: \cF_1\,\iso\, \cF^0_G\mid_{X-x}, \cU'^N\,\iso\, (^{\omega}\cL^{\bar\kappa})_{\cF'}, \cU_1^N\,\iso\, \cL^{\bar\kappa}_{(\cF_1, \nu_1,x)})
$$ 
to  
$$
(\cF, \cF', \nu: \cF\,\iso\, \cF'\mid_{X-x}, \cU, \cU'),
$$
where $\cF$ is obtained as the gluing of $\cF'\mid_{X-x}$ with $\cF_1\mid_{D_x}$ via $\nu_1^{-1}\comp \nu': \cF'\,\iso\, \cF_1\mid_{D_x^*}$. We have canonically $(^{\omega}\cL^{\bar\kappa})_{\cF'}\otimes \cL^{\bar\kappa}_{(\cF_1, \nu_1, x)}\,\iso\,(^{\omega}\cL^{\bar\kappa})_{\cF}$, and $\cU=\cU'\otimes\cU_1$ is equipped with the induced isomorphism $\cU^N\,\iso\, (^{\omega}\cL^{\bar\kappa})_{\cF}$. 

 Given an object $\cS$ of the $G_X$-equivariant derived category on $\wt\Gr_{G,X}$ and $\cT\in \D(\Bunt_G)$ we can form their twisted external product $(\cT\tboxtimes\cS)^r$, which is the descent of $\cT\boxtimes \cS$ via $\tilde\gamma^{\ra}$. Similarly, one may define $\tilde\gamma^{\la}$ and the complex $(\cT\tboxtimes\cS)^l$ on $\cH_{\tilde G}$. If $\mu_N(k)$ acts on $\cS$ by $\zeta$, and $\cT\in \D_{\zeta}(\Bunt_G)$ then $(\tilde h^{\la}_G\times\pi)_! (\cT\tboxtimes\cS)^r\in \D_{\zeta}(\Bunt_G\times X)$. 
 
 In (\cite{L1}, Remark~2.2) we introduced a covariant functor $\PPerv_{G,\zeta}\to \PPerv_{G, \zeta^{-1}}, K\mapsto \ast K$. It is induced by the map $\EE\to\EE, z\mapsto z^{-1}$. 
 
 Our choice of $\Omega^{\frac{1}{2}}$ gives rise to the fully faithful functor $\tau^0: \PPerv_{G,\zeta}\to \PPerv_{G,\zeta, X}$ defined in (\cite{L1}, Section~2.6). The abelian category $\PPerv_{G,\zeta, X}$, defined in \select{loc.cit.}, is the category of $G_X$-equivariant perverse sheaves  (cohomologically shifted by 1 to the right) on $\wt\Gr_{G,X}$ on which $\mu_N(k)$ acts by $\zeta$. Now for $\cS\in \PPerv_{G,\zeta}$ we define following \cite{FGV}
$$
\begin{array}{l}
\H^{\la}_G: \PPerv_{G,\zeta^{-1}}\times \D_{\zeta}(\Bunt_G)\to \D_{\zeta}(\Bunt_G\times X)\\ \\
\H^{\ra}_G: \PPerv_{G,\zeta}\times \D_{\zeta}(\Bunt_G)\to \D_{\zeta}(\Bunt_G\times X)
\end{array}
$$
by
$$
\H^{\ra}_G(\cS, K)=(\tilde h^{\la}_G\times \pi)_!(K\tboxtimes \tau_0(\cS))^r\;\;\;\mbox{and}\;\;\; \H^{\la}_G(\cS, K)=(\tilde h^{\ra}_G\times\pi)_!(K\tboxtimes \tau_0(\ast\cS))^l
$$
\index{$\tau^0, \PPerv_{G,\zeta, X}$}\index{$\H^{\la}_G, \H^{\ra}_G$}

 Set $\Lambda^{\sharp, +}=\Lambda^{\sharp}\cap\Lambda^+$. For $\nu\in\Lambda^{\sharp, +}$ we have the associated irreducible object $\cA^{\nu}_{\cE}\in\PPerv_{G,\zeta}$ defined in (\cite{L1}, Section~2.4.2). % For $\nu\in\Lambda^{\sharp, +}$ let $\H^{\nu}_G: \D_{\zeta}(\Bunt_G)\to \D_{\zeta}(\Bunt_G\times X)$ be the functor $\H^{\nu}_G(K)=\H^{\la}_G(\cA^{\nu}_{\cE}, K)$. This notation agrees with \cite{BG}. 
Note that $\ast \cA^{\nu}_{\cE}\,\iso\,\cA^{-w_0(\nu)}_{\cE}$.  
\index{$\Lambda^{\sharp, +}, \cA^{\nu}_{\cE}$}\index{${_x\cH_{\tilde G}}, Z$}
 
\ssec{Action on $\D_{\zeta}(\gM_x)$.} 
\label{section_5.2_action_on_D_zeta}
Pick $x\in X$. Let $_x\cH_{\tilde G}$ denote the fibre of $\cH_{\tilde G}$ over $x\in X$. Set $Z={_x\cH_{\tilde G}}\times_{\Bunt_G} \wt\gM_x$, where we used the map $\tilde h^{\ra}_G: {_x\cH_{\tilde G}}\to\Bunt_G$ in the fibred product. 

\begin{Lm} There is a map $'h^{\la}: Z\to \wt\gM_x$ that renders the diagram
$$
\begin{array}{ccccc}
\wt\gM_x & \getsup{'h^{\la}} & Z & \toup{'h^{\ra}} & \wt\gM_x\\
\downarrow\lefteqn{\scriptstyle\tilde\gp}  && \downarrow\lefteqn{\scriptstyle\gp_Z} && \downarrow\lefteqn{\scriptstyle\tilde\gp} \\
\Bunt_G & \getsup{\tilde h^{\la}_G} & {_x\cH_{\tilde G}} & \toup{\tilde h^{\ra}_G} & \Bunt_G
\end{array}
$$
commutative. The left square in the above diagram is also cartesian.
\end{Lm}
\begin{proof} The stack $Z$ classifies $(\cF, \cF', \nu: \cF\,\iso\, \cF'\mid_{X-x}, \cU,\cU')$ with isomorphisms (\ref{iso_cU_and_cU'_for_cH_tildeG}), and inclusions for $\check{\lambda}\in\check{\Lambda}^+$
$$
\kappa^{\check{\lambda}}: \Omega^{\<\rho, \check{\lambda}\>}\to \cV^{\check{\lambda}}_{\cF'}(\infty x)
$$
subject to the Pl\"ucker relations. From $\kappa$ and $\nu$ we get a system of maps 
$$
\kappa'^{\check{\lambda}}: \Omega^{\<\rho, \check{\lambda}\>}\to \cV^{\check{\lambda}}_{\cF}(\infty x)
$$
satisfying the Pl\"ucker relations (\cite{FGV}, Proposition~5.3.4). Let the map $'h^{\la}$ send the above point to $(\cF, \kappa', \cU)$. 
\end{proof} 

 As in Section~\ref{Section_action_Hecke_on_BuntG}, given $\cS\in \PPerv_{G,\zeta}$ and $K\in \D_{\zeta}(\gM_x)$, we may form their twisted external product $(K\tboxtimes \cS)^r\in \D(Z)$ using the fibration $'h^{\ra}: Z\to \wt\gM_x$ with fibre $\wt\Gr_{G,x}$. Analogously, the map $'h^{\la}$ gives rise to $(K\tboxtimes \cS)^l\in \D(Z)$. We define
$$
\H^{\la}_G: \PPerv_{G,\zeta^{-1}}\times \D_{\zeta}(\gM_x)\to \D_{\zeta}(\gM_x)\;\;\;\mbox{and}\;\;\; \H^{\ra}_G: \PPerv_{G,\zeta}\times \D_{\zeta}(\gM_x)\to \D_{\zeta}(\gM_x)
$$
by 
$$
\H^{\ra}_G(\cS, K)=('h^{\la})_!(K\tboxtimes {\cS})^r\;\;\;\mbox{and}\;\;\; \H^{\la}_G(\cS, K)=('h^{\ra})_!(K\tboxtimes(\ast\cS))^l
$$
We have functorial isomorphisms
$$
\H^{\la}_G(\cS_1, \H^{\la}_G(\cS_2, K))\,\iso\, \H^{\la}_G(\cS_1\ast\cS_2, K)\;\;\;\mbox{and}\;\;\;
\H^{\ra}_G(\cS_1, \H^{\ra}_G(\cS_2, K))\,\iso\, \H^{\ra}_G(\cS_2\ast\cS_1, K)
$$

\begin{Lm} The functors $\H^{\la}_G, \H^{\ra}_G$ preserve the subcategory $\D\Whit^{\kappa}_x\subset \D_{\zeta}(\gM_x)$.
\end{Lm}
\begin{proof}
This is analogous to (\cite{G1}, Proposition~7.3). For a collection of points $\bar y$ the action of the Hecke groupoid on $\wt\gM_x$ yields an action on $(\wt\gM_x)_{\goodat\, \bar y}$, which in turn lifts to an action on the torsor $_{\bar y}\wt\gM_x$. 
\end{proof}

\ssec{} 
\label{section_5.3}
Write $\Whit^{\kappa, ss}_x\subset \Whit^{\kappa}_x$ for the full subcategory consisting of objects, which are finite direct sums of irreducible ones. 

\begin{Thm} 
\label{Th_Hecke_action}
i) The functor $\H^{\ra}_G: \PPerv_{G,\zeta}\times \D\Whit^{\kappa}_x\to \D\Whit^{\kappa}_x$ is exact for the perverse t-structures, so induces a functor 
$$
\H^{\ra}_G: \PPerv_{G,\zeta}\times \Whit^{\kappa}_x\to \Whit^{\kappa}_x$$
ii) For $\gamma\in \Lambda^{\sharp, +}$ we have $\H^{\ra}_G(\cA^{\gamma}_{\cE}, \cF_{\emptyset})\,\iso\, \cF_{x, \gamma}$. \\
iii) The functor $\H^{\ra}_G$ preserves the subcatgeory $\Whit^{\kappa, ss}_x$.
\end{Thm} 
  
  The point ii) of the above theorem is an analog of (\cite{FGV}, Theorem~4) in our setting. 
  
\ssec{Proof of Theorem~\ref{Th_Hecke_action}}  
\sssec{} 
\label{section_5.4.1_for_Th531}
Pick $\lambda\in\Lambda^+$, $\gamma\in\Lambda^{\sharp,+}$. First, we show that
\begin{equation}
\label{complex_first_for_Th_Hecke_action}
\H^{\ra}_G(\cA^{-w_0(\gamma)}_{\cE}, \cF_{x, \lambda})\,\iso\, ('h^{\la})_!(\cF_{x,\lambda}\tboxtimes \cA^{-w_0(\gamma)}_{\cE})^r
\end{equation}
is perverse. To simplify the notation, from now on we suppress the upper index $r$ in the latter formula.
  
  For $\nu\in\Lambda$ write $\gM_{x,\le\nu}\subset\gM_x$ for the substack given by the property that for any $\check{\lambda}$ the map 
\begin{equation}
\label{map_kappa_check_lambda_Sect5}
\Omega^{\<\rho, \check{\lambda}\>}\to \cV^{\check{\lambda}}_{\cF}(\<\nu, \check{\lambda}\>x)
\end{equation}
is regular over $X$. Let $\gM_{\tilde x, \le\nu}\subset  \gM_{x,\le\nu}$ be the open substack given by the property that (\ref{map_kappa_check_lambda_Sect5}) has no zeros in a neighbourhood of $x$. Let $\gM_{x,\nu}\subset  \gM_{\tilde x, \le\nu}$ be the open substack given by requiring that (\ref{map_kappa_check_lambda_Sect5}) has no zeros over $X$. Write $\wt\gM_{x,\nu}$, $\wt\gM_{\tilde x,\nu}$ and so on for the restriction of the gerbe $\wt\gM_x$ to the corresponding stack. 
\index{$\gM_{x,\le\nu}, \gM_{\tilde x, \le\nu}, \gM_{x,\nu}$}%
\index{$\wt\gM_{x,\nu}, \wt\gM_{\tilde x,\nu}$}
  
  Denote by $K^{\nu}_{\tilde x}$ (resp., $K^{\nu}$) the $*$-restriction of (\ref{complex_first_for_Th_Hecke_action}) to $\wt\gM_{\tilde x,\nu}$ (resp., to  $\wt\gM_{x,\nu}$). Since (\ref{complex_first_for_Th_Hecke_action}) is Verdier self-dual (up to replacing $\psi$ by $\psi^{-1}$ and $\zeta$ by $\zeta^{-1}$), it suffices to prove the following. 
  
\begin{Lm} 
\label{Lm_542}
If $\nu\in\Lambda$ then $K^{\nu}_{\tilde x}$ is placed in perverse degrees $\le 0$.  
\end{Lm}  

\sssec{} 
\label{section_543}
For $\nu, \nu'\in \Lambda$ define the locally closed substacks of $Z$
\begin{align*}
Z^{\nu,?}_{\tilde x}=('h^{\la})^{-1}(\wt\gM_{\tilde x, \le\nu}),\;\;\;\; Z^{\nu, ?}=('h^{\la})^{-1}(\wt\gM_{x,\nu})\\
Z^{?,\nu'}_{\tilde x}=('h^{\ra})^{-1}(\wt\gM_{\tilde x, \le\nu'}),\;\;\; Z^{?,\nu'}=('h^{\ra})^{-1}(\wt\gM_{x,\nu'})\\
Z^{\nu, \nu'}_{\tilde x}=Z^{\nu,?}_{\tilde x}\cap Z^{?,\nu'}_{\tilde x},\;\;\; 
Z^{\nu, \nu'}=Z^{\nu, ?}\cap Z^{?,\nu'}
\end{align*} 
  
  For $\mu\in\Lambda^+$ let $_x\cH^{\mu}$ be the locally closed substack $\gamma^{\la}(\Bun_{G,x}\times_{G(\cO_x)}\Gr_{G,x}^{\mu})\subset {_x\cH_G}$. Let $_x\cH^{\mu}_{\tilde G}$ be its preimage in $_x\cH_{\tilde G}$. 
Set
\begin{align*}
Z^{\nu,?,\mu}_{\tilde x}=Z^{\nu,?}_{\tilde x}\cap \gp_Z^{-1}(_x\cH^{\mu}_{\tilde G}),\;\;\; Z^{?,\nu',\mu}_{\tilde x}=Z^{?,\nu'}_{\tilde x}\cap \gp_Z^{-1}(_x\cH^{\mu}_{\tilde G})\\
Z^{\nu,\nu', \mu}_{\tilde x}=Z^{\nu, \nu'}_{\tilde x}\cap \gp_Z^{-1}(_x\cH^{\mu}_{\tilde G}),\;\;\; Z^{\nu,\nu', \mu}=Z^{\nu, \nu'}\cap  \gp_Z^{-1}(_x\cH^{\mu}_{\tilde G})
\end{align*}  
 Denote by $K^{\nu,\nu',\mu}_{\tilde x}$ the $!$-direct image under 
$'h^{\la}: Z^{\nu,\nu',\mu}_{\tilde x} \to \wt\gM_{\tilde x,\le \nu}$ of the $*$-restriction of $\cF_{x,\lambda}\tboxtimes \cA^{-w_0(\gamma)}_{\cE}$ to $Z^{\nu,\nu',\mu}_{\tilde x}$. Denote by $K^{\nu,\nu',\mu}$ the restriction of $K^{\nu,\nu',\mu}_{\tilde x}$ to the open substack $\wt\gM_{x,\nu}$. Lemma~\ref{Lm_542} is reduced to the following.
  
\begin{Lm}
\label{Lm_544}
(1) The complex $K^{\nu,\nu',\mu}_{\tilde x}$ is placed in perverse degrees $\le 0$, and the inequality is strict unless $\mu=\gamma$ and $\nu'=\lambda$.\\
(2) The $*$-restriction of $K^{\nu,\lambda,\gamma}_{\tilde x}$ to the closed substack $\wt\gM_{\tilde x,\le \nu}-\wt\gM_{x,\nu}$ vanishes.
\end{Lm}  
  
  Choose for each $\nu\in\Lambda$ a trivialization $\epsilon_{\nu}: \Omega^{\rho}(-\nu x)\,\iso\, \cF^0_T\mid_{D_x}$. They yield a $U(\cO_x)$-torsor $\cU^{\epsilon_{\nu}}_{\tilde x}$ (resp., $\cU^{\epsilon_{\nu}}$) over $\wt\gM_{\tilde x, \le\nu}$ (resp., over $\wt\gM_{x,\nu}$) classifying a point of the latter stack together with a trivialization of the corresponding $U$-torsor over $D_x$.  
The projection $'h^{\la}$ identifies $Z^{\nu, ?}_{\tilde x}$ (resp., $'h^{\ra}$ identifies $Z^{?,\nu'}_{\tilde x}$) with the fibration 
$$
\cU^{\epsilon_{\nu}}_{\tilde x}\times_{U(\cO_x)}\wt\Gr_{G,x}\to \wt\gM_{\tilde x, \le\nu}
$$ 
(resp., with the fibration $\cU^{\epsilon_{\nu'}}_{\tilde x}\times_{U(\cO_x)}\wt\Gr_{G,x}\to \wt\gM_{\tilde x,\le\nu'}$). As in (\cite{FGV}, Lemma~7.2.4), one has the following. 
\begin{Lm} 
\label{Lm_545}
(1) The stacks $Z^{\nu,\nu'}_{\tilde x}$ and $Z^{\nu, ?, \mu}_{\tilde x}$, when viewed as substack of $Z^{\nu, ?}_{\tilde x}$, are identified with
$$
\cU^{\epsilon_{\nu}}_{\tilde x}\times_{U(\cO_x)} \wt\Gr_{B,x}^{\nu'-\nu} \;\toup{'h^{\la}} \;\wt\gM_{\tilde x,\le\nu}\;\;\;\mbox{and}\;\;\;
\cU^{\epsilon_{\nu}}_{\tilde x}\times_{U(\cO_x)} \wt\Gr^{\mu}_{G,x}\;\toup{'h^{\la}}\; \wt\gM_{\tilde x,\le\nu}
$$
respectively.\\
(2) The stacks $Z^{\nu,\nu'}_{\tilde x}$ and $Z^{?,\nu',\mu}_{\tilde x}$, when viewed as substacks of $Z^{?,\nu'}_{\tilde x}$, are identified with
$$
\cU^{\epsilon_{\nu'}}_{\tilde x}\times_{U(\cO_x)} \wt\Gr_{B,x}^{\nu-\nu'}\;\toup{'h^{\ra}}\; \wt\gM_{\tilde x, \le\nu'}\;\;\;\mbox{and}\;\;\;
\cU^{\epsilon_{\nu'}}_{\tilde x}\times_{U(\cO_x)} \wt\Gr_{G,x}^{-w_0(\mu)}\;\toup{'h^{\ra}}\; \wt\gM_{\tilde x, \le\nu'}
$$
respectively. \QED
\end{Lm}
 
\begin{proof}[Proof of Lemma~\ref{Lm_544}] (1) By Lemma~\ref{Lm_545}, the $*$-restriction of $\cF_{x,\lambda}\tboxtimes \cA^{-w_0(\gamma)}_{\cE}$ to $Z^{?,\nu',\mu}_{\tilde x}$ is the twisted external product of complexes
$$
(\cF_{x,\lambda}\mid_{\wt\gM_{\tilde x,\le\nu'}})\tboxtimes (\cA^{-w_0(\gamma)}_{\cE}\mid_{\wt\Gr_{G,x}^{-w_0(\mu)}}) .
$$
It lives in perverse degrees $\le 0$, and the inequality is strict unless $\mu=\gamma$ and $\nu'=\lambda$. Recall also that the $*$-restriction of $\cA^{-w_0(\gamma)}_{\cE}$ to $\wt\Gr_{G,x}^{-w_0(\mu)}$ vanishes unless $\mu\in\Lambda^{\sharp,+}$. 

Since $\cA^{-w_0(\gamma)}_{\cE}\mid_{\wt\Gr_{G,x}^{-w_0(\mu)}}$ has locally constant cohomology sheaves, its $*$-restriction to $Z^{\nu,\nu',\mu}_{\tilde x}$ by Lemma~\ref{Lm_545} is placed in perverse degrees
$$
\le -\codim(\Gr_B^{\nu-\nu'}\cap \Gr_G^{-w_0(\mu)}, \Gr_G^{-w_0(\mu)})\le -\<\mu-\nu+\nu', \check{\rho}\>, 
$$
we have used here (\cite{FGV}, Proposition~7.1.3). From Lemma~\ref{Lm_545}(1) we now learn that the fibres of $'h^{\la}: Z^{\nu,\nu',\mu}_{\tilde x}\to \wt\gM_{\tilde x, \le \nu}$ are of dimension $\le \dim(\Gr_B^{\nu'-\nu}\cap \Gr_G^{\mu})\le \<\nu'-\nu+\mu,\check{\rho}\>$. If $f: Y\to W$ is a morphism of schemes of finite type, each fibre of $f$ is of dimension $\le d$, $K$ is a perverse sheaf on $Y$ then $f_!K$ is placed in perverse degrees $\le d$. We are done.

\smallskip\noindent
(2) the $*$-restriction of $\cF_{x,\lambda}$ to $\wt\gM_{\tilde x,\le\lambda}-\wt\gM_{x, \lambda}$ vanishes, because there are no dominant coweights $<0$. 
\end{proof}

 Theorem~\ref{Th_Hecke_action} i) is proved. Theorem~\ref{Th_Hecke_action} iii) follows from the decomposition theorem of \cite{BBD}. 
 
  To establish Theorem~\ref{Th_Hecke_action} ii), keep the above notation taking $\lambda=0$. We want to show that (\ref{complex_first_for_Th_Hecke_action}) identifies with $\cF_{x, -w_0(\gamma)}$. It remains to analyse the complex $K^{\nu,0,\gamma}$ on $\wt\gM_{x, \nu}$ placed in perverse degrees $\le 0$. We are reduced to the following.
 
\begin{Lm} i) The $0$-th perverse cohomology sheaf of $K^{\nu,0,\gamma}$ vanishes unless $\nu=-w_0(\gamma)$. \\
ii) The $0$-th perverse cohomology sheaf of $K^{-w_0(\gamma),0,\gamma}$ identifies with the restriction of $\cF_{x, -w_0(\gamma)}$ to
$\wt\gM_{x, -w_0(\gamma)}$.
\end{Lm}
\begin{proof} 
The situation with the additive characters is exactly the same as in (\cite{FGV}, Sections~7.2.6-7.2.8). Let $\ov{U(F_x)}^{\epsilon_{\nu}}$ be ind-group scheme over $\wt\gM_{x,\nu}$, the $\cU^{\epsilon_{\nu}}$-twist of $U(F_x)$ with respect to the adjoint action of $U(\cO_x)$ on $U(F_x)$. Then $Z^{\nu,\nu'}$ carries a natural $\ov{U(F_x)}^{\epsilon_{\nu}}$-action preserving $'h^{\la}: Z^{\nu,\nu'}\to \wt\gM_{x,\nu}$ and
defined via the identification of Lemma~\ref{Lm_545}(1). 

 The ind-group $\ov{U(F_x)}^{\epsilon_{\nu}}$ classifies a point $(\cF, \kappa, \cU)\in \wt\gM_{x,\nu}$ giving rise to the corresponding $B$-torsor $\cF_B$ on $D_x$ equipped with $\cF_B\times_B T\,\iso\, \Omega^{\rho}(-\nu x)$, and an automorphism $g: \cF_B\,\iso\,\cF_B$ over $D_x^*$ inducing the identity on $\cF_B\times_B T$. 
 
 The trivialization $\epsilon_{\nu}:\Omega^{\rho}(-\nu x)\,\iso\, \cF^0_T\mid_{D_x}$ gives for $i\in \cJ$ the character
$$
U/[U,U](F_x)\toup{\check{\alpha}_i}  F_x\toup{\epsilon_{\nu}^{-1}}\cL^{\check{\alpha}_i}_{\Omega^{\rho}(-\nu x)}\mid_{D_x^*}\,\iso\, \Omega(F_x)\toup{\Res}\A^1
$$
Their sum over $i\in \cJ$ is the character of conductor $\bar\nu$ denoted $\chi_{\nu}: U(F_x)\to \A^1$. Here $\bar\nu$ is the image of $\nu$ in the coweights lattice of $G_{ad}$. Twisting $U(F_x)$ by the $U(\cO_x)$-torsor $\cU^{\epsilon_{\nu}}$, one gets the character denoted $\bar\chi_{\nu}: \ov{U(F_x)}^{\epsilon_{\nu}}\to\A^1$. 
  
 For $\nu,\nu'\in\Lambda^+$ a $(U(F_x), \chi_{\nu})$-equivariant function $\chi^{\nu'-\nu}_{\nu}: \Gr_B^{\nu'-\nu}\to \A^1$ gives rise to a $(\ov{U(F_x)}^{\epsilon_{\nu}}, \bar\chi_{\nu})$-equivariant function $\bar\chi_{\nu}^{\nu'-\nu}: Z^{\nu,\nu'}\to \A^1$. For the convenience of the reader we recall the following.
\begin{Lm}[\cite{FGV}, Lemma~7.2.7] 
\label{Lm_547}
Assume $\nu'\in\Lambda^+$. Then\\
(1) the map $\ev_{x,\nu'}\comp {'h^{\ra}}: Z^{\nu,\nu'}\to \A^1$ is $(\ov{U(F_x)}^{\epsilon_{\nu}}, \bar\chi_{\nu})$-equivariant.\\
(2) If in addition $\nu\in\Lambda^+$ then $\ev_{x,\nu'}\comp {'h^{\ra}}$ coincides with the composition
$$
Z^{\nu,\nu'}\;\toup{\bar\chi^{\nu'-\nu}_{\nu}\times {'h^{\ra}}}\; \A^1\times \wt\gM_{x,\nu}\toup{\id\times \ev_{x,\nu}} \A^1\times\A^1\toup{\summ}\A^1
$$
for some $\chi^{\nu'-\nu}_{\nu}$. \QED
\end{Lm}

 The fibration $'h^{\la}: Z^{\nu,0, \gamma}\to \wt\gM_{x,\nu}$ identifes with $\cU^{\epsilon_{\nu}}\times_{U(\cO_x)}(\wt\Gr_{B,x}^{-\nu}\cap \wt\Gr_{G,x}^{\gamma})\to \wt\gM_{x,\nu}$. After a smooth localization $V\to \wt\gM_{x,\nu}$ the latter fibration becomes a direct product $V\times (\wt\Gr_{B,x}^{-\nu}\cap \wt\Gr_{G,x}^{\gamma})$. The $*$-restriction of $\cF_{\emptyset}\tboxtimes \cA_{\cE}^{-w_0(\gamma)}$ to $Z^{\nu,0, \gamma}$ decends  to $V\times (\Gr_{B,x}^{-\nu}\cap \Gr_{G,x}^{\gamma})$, and there becomes of the form
$$
\cE_V\boxtimes ((\chi^{-\nu}_{\nu})^*\cL_{\psi}\otimes \delta^*\cL_{\zeta})[\<\gamma-\nu, 2\check{\rho}\>],
$$
for a suitable discrepancy map $\delta: \Gr_{B,x}^{-\nu}\cap \Gr_{G,x}^{\gamma}\to\Gm$. Here $\cE_V$ is a perverse sheaf on $V$. 

The local system $(\chi^{-\nu}_{\nu})^*\cL_{\psi}\otimes \delta^*\cL_{\zeta}$ is nonconstant on any irreducible component by (\cite{FGV}, Proposition~7.1.7). This proves i). Since $\Gr_B^{w_0(\gamma)}\cap \Gr_G^{\gamma}$ is the point scheme, part ii) follows from Lemma~\ref{Lm_547} and \ref{Lm_545}.
\end{proof} 

Theorem~\ref{Th_Hecke_action} is proved.

\section{Objects that remain irreducible}
\label{Sec_Objects that remain irreducible}

\noindent
In this section we describe the irreducible objects $\cF_{x,\lambda}$ of $\Whit^{\kappa}_x$ such that $\ov{\FF}(\cF_{x,\lambda})\in \wt\FS^{\kappa}_x$ remain irreducible. As for quantum groups, we introduce the corresponding notion of restricted dominant coweights.

\ssec{Special elements in crystals}  Let $\lambda,\mu\in\Lambda$, $\mu\le\lambda$. Recall the map $\gamma^{\mu}_{\lambda}: \Gr_B^{\lambda}\cap \Gr_{B^-}^{\mu}\to\Gm$ defined in Section~\ref{section_4.12.4} via the equality $s^{\mu}_{B^-}=\gamma^{\mu}_{\lambda} s^{\lambda}_B$. 

\begin{Lm} 
\label{Lm_invariance_under_translations_1}
Let $\nu\in\Lambda$. Then there is $\epsilon\in k^*$ depending  on $\lambda,\mu,\nu$ such that the composition $\Gr_B^{\lambda}\cap \Gr_{B^-}^{\mu}\toup{t^{\nu}} \Gr_B^{\lambda+\nu}\cap \Gr_{B^-}^{\mu+\nu}\toup{\gamma^{\mu+\nu}_{\lambda+\nu}}\Gm$ equals $\epsilon \gamma^{\mu}_{\lambda}$.
\end{Lm}
\begin{proof}
Pick any $g\in\EE$ over $t^{\nu}\in G(F)$. The composition $\Gr_B^{\lambda}\toup{t^{\nu}} \Gr_B^{\lambda+\nu}\toup{s^{\lambda+\nu}_B} \Gra_G\toup{g^{-1}} \Gra_G$ equals $as^{\lambda}_B$ for some $a\in k^*$. Indeed, any map $\Gr_B^{\lambda}\to\Gm$ is constant. Similarly, the composition $\Gr_{B^-}^{\mu}\toup{t^{\nu}} \Gr_{B^-}^{\mu+\nu}\toup{s^{\mu+\nu}_{B^-}} \Gra_G\toup{g^{-1}}\Gra_G$ equals $b s^{\mu}_{B^-}$ for some $b\in k^*$. Our claim follows.
\end{proof}

 Our notations and conventions about the crystals are those of Section~\ref{Section_Recollections on crystals}.
 
\begin{Cor}
\label{Cor_allows_shifts_by_coweights}
 Let $\lambda,\mu,\nu\in \Lambda, b\in \Irr(\Gr_B^{\lambda}\cap \Gr_{B^-}^{\mu})$. Then $t^{\nu}b\in \Irr(\Gr_B^{\lambda+\nu}\cap \Gr_{B^-}^{\mu+\nu})$. The local system $(\gamma^{\mu}_{\lambda})^*\cL_{\zeta}$ is trivial on $b$ if and only if $(\gamma^{\mu+\nu}_{\lambda+\nu})^*\cL_{\zeta}$ is trivial on $t^{\nu}b$. 
So, the latter property is actually a property of $t^{-\lambda}b\in B_{\gg}(\lambda-\mu)$. 
\end{Cor}

\begin{Def} 
\label{Def_special_elements_in_B_gg}
For $\nu\in\Lambda^{pos}$ we call an element $\bar b\in B_{\gg}(\nu)$ special if the local system $(\gamma^{-\nu}_0)^*\cL_{\zeta}$ is constant on $\bar b$. Denote by $B_{\gg}^{sp}(\nu)$ the set of special elements of $B_{\gg}(\nu)$.
\end{Def}

  For $i\in\cJ$ denote by $\delta_i$ the denominator of $\frac{\bar\kappa(\alpha_i, \alpha_i)}{2N}$. Recall that $\frac{\check{\alpha}_i}{\delta_i}$ (resp., $\delta_i\alpha_i$) are the coroots (resp., roots) of $\check{G}_{\zeta}$.
\index{$\delta_i$} 

\begin{Rem} 
\label{Rem_special_elements_in_crystals}
i) If $\bar b\in B_{\gg}(\nu)$ is special then $\nu\in \oplus_{i\in \cJ} \; \ZZ_+(\delta_i\alpha_i)$.\\
ii) Let $\lambda\in\Lambda^+$, $x\in X$, $\mu\in\Lambda$. Assume that the subtop cohomology property holds for $\varrho$. If the multiplicity space $V^{\lambda}_{\mu}$ in the decomposition of $\ov{\FF}(\cF_{x,\lambda})$ from (\ref{decomp_FF_of_irred_object}) is nonzero then $\lambda-\mu\in \oplus_{i\in \cJ} \,\ZZ_+(\delta_i\alpha_i)$. 
\end{Rem}
\begin{proof} i) Pick $\lambda\in\Lambda^{\sharp, +}$ such that $\phi_i(\bar b)\le \<\lambda, \check{\alpha}_i\>$ for all $i\in\cJ$. Let $b=t^{\lambda}\bar b\in \Irr(\Gr_B^{\lambda}\cap \Gr_{B^-}^{\lambda-\nu})$. Then $b\subset \ov{\Gr}_G^{\lambda}$ by  Lemma~\ref{Lm_Andersen_again}. So, $b$ gives a base vector in the weight space $V(\lambda)_{\lambda-\nu}$ of the irreducible $\check{G}$-representation $V(\lambda)$ by (\cite{L1}, Lemma~3.2). Thus, $\nu$ is in the $\ZZ_+$-span of the simple roots of $\check{G}_{\zeta}$.\\
ii) Follows from i).
\end{proof}

 For a standard Levi $M\subset G$ recall that $\Gra_M=\Gr_M\times_{\Gr_G}\Gra_G$. The trivializations $\delta_{\lambda}$ picked in Section~\ref{section_Additional input data} yield for any $\lambda\in\Lambda$ sections $s^{\lambda}_{B(M)}:Ê\Gr_{B(M)}^{\lambda}\to \Gra_M$, $s^{\lambda}_{B^-(M)}:\Gr_{B^-(M)}^{\lambda}\to\Gra_M$ of the $\Gm$-torsor $\Gra_M\to \Gr_M$. The discrepancy between $s^{\mu}_{B^-(M)}$ and $s^{\lambda}_{B(M)}$ is a map that we denote by
$$
_M\gamma^{\mu}_{\lambda} : Ê\Gr_{B(M)}^{\lambda}\cap \Gr_{B^-(M)}^{\mu}\to\Gm
$$
and define by $s^{\mu}_{B^-(M)}=({_M\gamma^{\mu}_{\lambda}})s^{\lambda}_{B(M)}$. If $\lambda-\mu\in\Lambda^{\sharp}$ then $_M\gamma^{\mu}_{\lambda}$ does not depend on the choice of $\delta$ (so depends only on $\gt^0_{\EE}$).
\index{$s^{\lambda}_{B(M)}, s^{\lambda}_{B^-(M)}, \; {_M\gamma^{\mu}_{\lambda}}$}

Lemma~\ref{Lm_invariance_under_translations_1} and Corollary~\ref{Cor_allows_shifts_by_coweights} are immediately generalized for each standard Levi subgroup $M$ of $G$. Recall from Section~\ref{Section_Recollections on crystals} that for $\lambda\ge_M 0$ we denote 
$$
B_{\gm}(\lambda)=\Irr(\Gr_{B(M)}^0\cap \Gr_{B^-(M)}^{-\lambda})
$$ 
and $B_{\gm}=\cup_{\lambda\ge_M 0} \, B_{\gm}(\lambda)$, and for $\lambda\in\Lambda^{pos}$ we have the bijection 
\begin{equation}
\label{decomp_general_for_B_gg_and_P} 
 B_{\gg}(\lambda)\,\iso\, \cup_{\mu} \, B^{\gm, \ast}_{\gg}(\lambda-\mu)\times B_{\gm}(\mu)\, .
\end{equation} 
Here the union is over $\mu\in\Lambda$ satisfying $0\le_M \mu\le \lambda$. 

 For $\mu\in\Lambda$ the trivialization $\delta_{-\mu}$ from Section~\ref{section_Additional input data} yields a section $s_P^{-\mu}: \gq_P^{-1}(t^{-\mu})\to\Gra_G$ of the $\Gm$-torsor $\Gra_G\to\Gr_G$. Let 
$$
_P\gamma_{-\mu}^{-\lambda}: \gq_P^{-1}(t^{-\mu})\cap \Gr_{B^-}^{-\lambda}\to\Gm
$$ 
be the map defined by the equation $s^{-\lambda}_{B^-}=(_P\gamma_{-\mu}^{-\lambda}) s_P^{-\mu}$.
\index{$s_P^{-\mu}, \; {_P\gamma_{-\mu}^{-\lambda}}$}

\begin{Def} i) Let $M$ be a standard Levi of $G$. For $\nu\in \Lambda^{pos}_M$ we call an element $b\in B_{\gm}(\nu)$ special if the local system $(_M\gamma_0^{-\nu})^*\cL_{\zeta}$ is constant on $b$.\\
ii) We call $b\in B^{\gm, \ast}_{\gg}(\lambda-\mu)$ special if the local system $(_P\gamma_{-\mu}^{-\lambda})^*\cL_{\zeta}$ is constant on $b$.
\end{Def}

\begin{Lm} Let $\lambda\in\Lambda^{pos}$, $b\in B_{\gg }(\lambda)$. Let $0\le_M\mu\le\lambda$ be such that $b\cap \gq_P^{-1}(\Gr_{B^-(M)}^{-\mu})$ is dense in $b$. Let 
$$
(b_1,b_2)\in B^{\gm, \ast}_{\gg}(\lambda-\mu)\times B_{\gm}(\mu)
$$
correspond to $b$ via (\ref{decomp_general_for_B_gg_and_P}). Then $b$ is special if and only if both $b_1, b_2$ are special. In the latter case $\mu$ is in the $\ZZ_+$-span of the simple roots of $\check{M}_{\zeta}$. 
\end{Lm}
\begin{proof} As in Section~\ref{Pp_Kashiwara} we have a $T(\cO)$-equivariant isomorphism
$$
\Gr_B^0\cap \gq_P^{-1}(\Gr_{B^-(M)}^{-\mu})\cap \Gr_{B^-}^{-\lambda}\,\iso\, (\Gr_{B(M)}^0\cap \Gr_{B^-(M)}^{-\mu})\times (\gq_P^{-1}(t^{-\mu})\cap \Gr_{B^-}^{-\lambda}),
$$
where $T(\cO)$ acts on the right hand side as the product of the natural actions on the two factors. For brevity denote by $\gamma: \gq_P^{-1}(t^{-\mu})\cap \Gr_{B^-}^{-\lambda}\to\Gm$ the map $_P\gamma_{-\mu}^{-\lambda}$. Then the map
$$
(_M\gamma^{-\mu}_0)\gamma:
(\Gr_{B(M)}^0\cap \Gr_{B^-(M)}^{-\mu})\times (\gq_P^{-1}(t^{-\mu})\cap \Gr_{B^-}^{-\lambda})\to\Gm
$$
coincides with the restriction of $\gamma^{-\lambda}_0$. Our first claim follows. The second follows now from Remark~\ref{Rem_special_elements_in_crystals}. 
\end{proof} 

\begin{Cor} 
\label{Cor_divisibility_of_phi(b)}
Let $\nu\in\Lambda^{pos}$, $i\in \cJ$. If $b\in B_{\gg}(\nu)$ is special then $\phi_i(b)\in \ZZ_+\delta_i$. \QED
\end{Cor}

\ssec{} 
\label{section_6.2}
Pick $x\in X$. Our purpose now is to describe some irreducible objects $\cF_{x,\lambda}$ of $\Whit^{\kappa}_x$ such that $\ov{\FF}(\cF_{x,\lambda})\in\wt\FS^{\kappa}_x$ remains irreducible.

% As in \cite{L}, for $\lambda\in\Lambda^+$ we denote by $\VV^{\lambda}$ the irreducible $\check{G}$-representation with highest weight. $\lambda$. 
 
 For $a\in\Lambda/\Lambda^{\sharp}$ let $\Lambda^+_a$ denote the set of those $\lambda\in\Lambda^+$ whose image in $\Lambda/\Lambda^{\sharp}$ equals $a$. Set $\Lambda^{\sharp}_0=\{\mu\in\Lambda^{\sharp}\mid \<\mu, \check{\alpha}_i\>=0\;\mbox{for any}\, i\in\cJ\}$. Set
$$
\bar\cM=\{\lambda\in\Lambda^+\mid \mbox{for any}\; i\in\cJ, \<\lambda, \check{\alpha}_i\>< \delta_i\}
$$
As in (\cite{ABBGM}, Section~1.1.3), we call the elements of $\bar\cM$  \select{restricted} dominant coweights. Note that $\Lambda^{\sharp}_0$ acts on $\bar\cM$ by translations.
\index{$\cF_{x,\lambda}, \Lambda^+_a, \Lambda^{\sharp}_0, \bar\cM$}

 The following is an analog of (\cite{ABBGM}, Proposition~1.1.8). 

\begin{Thm} 
\label{Thm_objects_remain_irred}
Assume that the subtop cohomology property is satisfied for $\varrho$. For any $\lambda\in \bar\cM$ one has $\ov{\FF}(\cF_{x,\lambda})\,\iso\, \cL_{x,\lambda}$.
\end{Thm}
\begin{proof}
Let $\mu\in\Lambda$ with $w_0(\lambda)\le \mu<\lambda$. We must show that the multiplicity space $V^{\lambda}_{\mu}$ in (\ref{decomp_FF_of_irred_object})
vanishes. By Remark~\ref{Rem_special_elements_in_crystals}, we may assume 
$$
\lambda-\mu\in \oplus_{i\in \cJ} \,\ZZ_+(\delta_i\alpha_i)
$$ 
Let $b\in \Irr(\Gr_B^{\lambda}\cap \Gr_{B^-}^{\mu})$ with $b\subset \ov{\Gr}_G^{\lambda}$. For $\bar b=t^{-\lambda}b\in B_{\gg}(\lambda-\mu)$ from Lemma~\ref{Lm_Andersen_again} we get $\phi_i(\bar b)\le \<\lambda, \check{\alpha}_i\><\delta_i$ for all $i\in\cJ$. Assume $\bar b$ special. Then, by Corollary~\ref{Cor_divisibility_of_phi(b)}, $\phi_i(\bar b)\in\ZZ_+\delta_i$. So, $\phi_i(\bar b)=0$ for all $i$. The only element of $B_{\gg}$ with this property is the unique element of $B_{\gg}(0)$, a contradiction. So, $V^{\lambda}_{\mu}=0$.
\end{proof}

\section{Analog of the Lusztig-Steinberg tensor product theorem}

\ssec{} The purpose of this section is to prove Theorem~\ref{Thm_Lusztig-Steinberg_analog}, which is an analog in our setting of the Lusztig-Steinberg theorem for quantum groups. We use the notations of Section~\ref{Section_Hecke functors}. Pick $x\in X$. Recall that $\Lambda^{\sharp,+}=\Lambda^{\sharp}\cap\Lambda^+$. 

\begin{Thm} 
\label{Thm_Lusztig-Steinberg_analog}
Let $\lambda\in\bar\cM$ and $\gamma\in\Lambda^{\sharp,+}$. Then there is an isomorphism 
\begin{equation}
\label{iso_main_Steinberg-Lusztig_theorem}
\H^{\ra}_G(\cA^{\gamma}_{\cE}, \cF_{x, \lambda})\,\iso\, \cF_{x, \lambda+\gamma}
\end{equation}
\end{Thm}

\sssec{Proof of Theorem~\ref{Thm_Lusztig-Steinberg_analog}} Let $a$ denote the image of $\lambda$ in $\Lambda/\Lambda^{\sharp}$. Recall that the Weyl groups for $\check{G}$ and $\check{G}_{\zeta}$ are equal.
It is convenient for us to replace $\gamma$ by $-w_0(\gamma)$, so we must establish for $\gamma\in\Lambda^{\sharp,+}$ the isomorphism $\H^{\ra}_G(\cA^{-w_0(\gamma)}, \cF_{x, \lambda})\,\iso\, \cF_{x, \lambda-w_0(\gamma)}$. By definition,
\begin{equation}
\label{complex_for_L-St_theorem}
\H^{\ra}_G(\cA^{-w_0(\gamma)}, \cF_{x, \lambda})\,\iso\, ('h^{\la})_!(\cF_{x, \lambda}\tboxtimes \cA^{-w_0(\gamma)}_{\cE})^r
\end{equation}   
To simplify the notation, from now on we suppress the upper index $r$ in the above formula.

As in Section~\ref{section_5.4.1_for_Th531}, define the complex $K^{\nu}_{\tilde x}$ (resp., $K^{\nu}$) as the $*$-restriction of (\ref{complex_for_L-St_theorem}) to $\wt\gM_{\tilde x,\nu}$ (resp., to $\wt\gM_{x,\nu}$). Since (\ref{complex_for_L-St_theorem}) is Verdier self-dual (up to replacing $\psi$ by $\psi^{-1}$ and $\zeta$ by $\zeta^{-1}$), it suffices to prove the following.

\begin{Lm} 
\label{Lm_first_for_L-St_thm}
i) The complex $K^{\nu}_{\tilde x}$ is placed in perverse degrees $\le 0$.\\
ii) The $*$-restriction of $K^{\nu}_{\tilde x}$ to the closed substack
$\wt\gM_{\tilde x,\le\nu}-\wt\gM_{x,\nu}\subset \wt\gM_{\tilde x,\nu}$ vanishes.\\
iii) The $0$-th perverse cohomology sheaf of $K^{\nu}$ vanishes unless $\nu=\lambda-w_0(\gamma)$ and in the latter case it identifies with the restriction $\cF_{x, \lambda-w_0(\gamma)}\mid_{\wt\gM_{x, \lambda-w_0(\gamma)}}$.
\end{Lm}

 Lemma~\ref{Lm_first_for_L-St_thm} i) is a particular case of Lemma~\ref{Lm_542}. Recall the substacks of $Z$ 
\begin{align*}
Z^{\nu,?}_{\tilde x}, Z^{\nu, ?}, Z^{?,\nu'}_{\tilde x}, Z^{?,\nu'}, Z^{\nu,\nu'}_{\tilde x}, Z^{\nu, \nu'}\\
Z^{\nu, ?,\mu}_{\tilde x}, Z^{?, \nu', \mu}_{\tilde x}, Z^{\nu, \nu',\mu}_{\tilde x}, Z^{\nu,\nu',\mu}
\end{align*}
defined in Section~\ref{section_543} for $\nu,\nu'\in\Lambda, \mu\in \Lambda^+$. Denote by $K^{\nu,\nu',\mu}_{\tilde x}$ the $!$-direct image under 
$$
'h^{\la}: Z^{\nu,\nu',\mu}_{\tilde x}\to \wt\gM_{\tilde x,\le \nu}
$$ 
of the $*$-restriction of $\cF_{x, \lambda}\tboxtimes \cA^{-w_0(\gamma)}_{\cE}$ to $Z^{\nu,\nu',\mu}_{\tilde x}$. Let $K^{\nu,\nu',\mu}$ be the restriction of $K^{\nu,\nu',\mu}_{\tilde x}$ to the open substack $\wt\gM_{x,\nu}$. 

Using the standard spectral sequence, Lemma~\ref{Lm_first_for_L-St_thm} is reduced to the following.

\begin{Lm}
\label{Lm_second_for_L-St_thm}
(1) The complex $K^{\nu,\nu',\mu}_{\tilde x}$ is placed in perverse degrees $\le 0$, and the inequality is strict unless $\mu=\gamma$ and $\lambda=\nu'$. 

\smallskip\noindent
(2) The $*$-restriction of $K^{\nu,\nu',\mu}_{\tilde x}$ to $\wt\gM_{\tilde x,\le\nu}-\wt\gM_{x,\nu}$ vanishes.

\smallskip\noindent
(3) The $0$-th perverse cohomology of $K^{\nu,\lambda, \gamma}$ vanishes unless $\nu=\lambda-w_0(\gamma)$.

\smallskip\noindent
(4) The $0$-th perverse cohomology of $K^{\lambda-w_0(\gamma), \lambda, \gamma}$ identifies with $\cF_{x, \lambda-w_0(\gamma)}$.
\end{Lm}

The points (1) and (2) of Lemma~\ref{Lm_second_for_L-St_thm} follow from Lemma~\ref{Lm_544}. It remains to analyse the complex $K^{\nu,\lambda, \gamma}$ on $\wt\gM_{x,\nu}$ placed in perverse degrees $\le 0$. 

As in Section~\ref{Lm_544}, for each $\nu\in\Lambda$ we pick a trivialization $\epsilon_{\nu}: \Omega^{\rho}(-\nu x)\,\iso\, \cF^0_T\mid_{D_x}$. It yields a $U(\cO_x)$-torsor $\cU^{\epsilon_{\nu}}$ over $\wt\gM_{x,\nu}$ classifying a point of $\wt\gM_{x,\nu}$ together with a trivialization of the corresponding $U$-torsor over $D_x$. 

 Recall that for each $\nu\in\Lambda$ we fixed the section $s^{\nu}_B: \Gr_B^{\nu}\to\Gra_G$ of the $\Gm$-torsor $\Gra_G\to \Gr_G$ in Section~\ref{section_4.12.4}. By Lemma~\ref{Lm_545}, the fibration $'h^{\la}: Z^{\nu,\lambda, \gamma}\to \wt\gM_{x,\nu}$ identifies with
$$
\cU^{\epsilon_{\nu}}\times_{U(\cO_x)}(\wt\Gr_{B,x}^{\lambda-\nu}\cap \wt\Gr_{G,x}^{\gamma})\to \wt\gM_{x,\nu}
$$

 After a smooth localization $V\to \wt\gM_{x,\nu}$ the latter fibration becomes a direct product $V\times (\wt\Gr_{B,x}^{\lambda-\nu}\cap \wt\Gr_{G,x}^{\gamma})$. The $*$-restriction of $\cF_{x, \lambda}\tboxtimes \cA^{-w_0(\gamma)}_{\cE}$ to $Z^{\nu,\lambda, \gamma}$ descends  to $V\times (\Gr_{B,x}^{\lambda-\nu}\cap \Gr_{G,x}^{\gamma})$, and there becomes of the form
$$
\cE_V\boxtimes ((\chi^{\lambda-\nu}_{\nu})^*\cL_{\psi}\otimes (s_B^{\lambda-\nu})^*\cA^{\gamma}_{\cE})[\<\gamma+\lambda-\nu, 2\check{\rho}\>].
$$
Here $\cE_V$ is a locally constant perverse sheaf on $V$. So, $K^{\nu,\lambda, \gamma}$ vanishes unless $\nu\in\Lambda^+$, and in the latter case it identifies over $\wt\gM_{x,\nu}$ with
$$
\cF_{x,\nu}\otimes \RG_c(\Gr_{B,x}^{\lambda-\nu}\cap \Gr_{G,x}^{\gamma}, (\chi^{\lambda-\nu}_{\nu})^*\cL_{\psi}\otimes (s_B^{\lambda-\nu})^*\cA^{\gamma}_{\cE})[\<\gamma+\lambda-\nu, 2\check{\rho}\>]
$$

 Recall that $\Gr_B^{\lambda-\nu}\cap \Gr_G^{\gamma}$ is of pure dimension $\<\gamma+\lambda-\nu, \check{\rho}\>$. Consider the open subscheme 
\begin{equation}
\label{open_sunscheme_for_L-St_thm}
\Gr_B^{\lambda-\nu}\cap \Gr_{B^-}^{w_0(\gamma)}\cap \Gr_G^{\gamma}\subset 
\Gr_B^{\lambda-\nu}\cap \Gr_G^{\gamma}
\end{equation}
Anderson's theorem (\cite{A}, Proposition~3) implies that this open embedding induces a bijection on the set of irreducible components. 
The following is a version of Lemma~\ref{Lm_other_three_intersections_local_systems_coincide} obtained by exchaning the roles of $B$ and $B^-$.

\begin{Lm} 
\label{Lm_comparing_s_B_and_gamma}
Let $\mu\le\gamma$ with $\mu\in\Lambda, \gamma\in\Lambda^{\sharp,+}$. Then over $\Gr_G^{\gamma}\cap \Gr_{B^-}^{w_0(\gamma)}\cap \Gr_B^{\mu}$ there is an isomorphism
$(s^{\mu}_B)^*\cA^{\gamma}_{\cE}\,\iso\, (\gamma^{w_0(\gamma)}_{\mu})^*\cL_{\zeta^{-1}}$ up to a shift.
\end{Lm}
\begin{proof} By definition, $\gamma^{w_0(\gamma)}_{\mu} s^{\mu}_B=s^{w_0(\gamma)}_{B^-}$ over $\Gr_{B^-}^{w_0(\gamma)}\cap \Gr_B^{\mu}$. Our claim follows.
\end{proof}

So, we are analyzing the top cohomology of
\begin{equation}
\label{complex_almost_final_for_L-St_Thm}
\RG_c(\Gr_B^{\lambda-\nu}\cap \Gr_{B^-}^{w_0(\gamma)}\cap \Gr_G^{\gamma}, \, (\chi^{\lambda-\nu}_{\nu})^*\cL_{\psi}\otimes (\gamma^{w_0(\gamma)}_{\lambda-\nu})^*\cL_{\zeta^{-1}})
\end{equation} 

 Assume that $\nu\ne \lambda-w_0(\gamma)$ and (\ref{open_sunscheme_for_L-St_thm}) is not empty. Then the dimension of (\ref{open_sunscheme_for_L-St_thm}) is $>0$. In this case the local system 
$$
(\chi^{\lambda-\nu}_{\nu})^*\cL_{\psi}\otimes (\gamma^{w_0(\gamma)}_{\lambda-\nu})^*\cL_{\zeta^{-1}}
$$ 
is constant on an irreducible component $b$ of (\ref{open_sunscheme_for_L-St_thm}) if and only if both $(\chi^{\lambda-\nu}_{\nu})^*\cL_{\psi}$ and 
$(\gamma^{w_0(\gamma)}_{\lambda-\nu})^*\cL_{\zeta}$ are constant on $b$. So, the $0$-th perverse cohomology sheaf of $K^{\nu, \lambda, \gamma}$ vanishes unless 
$$
\bar b=t^{\nu-\lambda}b\in B_{\gg}(\lambda-\nu-w_0(\gamma))
$$ 
is special. By Remark~\ref{Rem_special_elements_in_crystals}, this implies
$\lambda-\nu-w_0(\gamma)\in \oplus_{i\in\cJ} \ZZ_+(\delta_i\alpha_i)$. In particular, $\nu\in \Lambda^+_a$. 

\begin{Lm} 
\label{Lm_2.1.6}
Let $\gamma\in\Lambda^+, \mu,\nu\in\Lambda$ satisfy $\mu+\nu\in\Lambda^+, \mu\ge w_0(\gamma)$. Let $b\subset \Gr_{B^-}^{w_0(\gamma)}\cap \Gr_B^{\mu}$ be an irreducible component. Let $\chi^{\mu}_{\nu}: \Gr^{\mu}_B\to \A^1$ be a $(U(F), \chi_{\nu})$-equivariant function, where $\chi_{\nu}: U(F)\to\A^1$ is an additive character of conductor $\bar\nu$. Here $\bar\nu$ is the image of $\nu$ in the coweights lattice of $G_{ad}$. Denote by 
$$
\bar b\subset \Gr_B^0\cap \Gr_{B^-}^{w_0(\gamma)-\mu}
$$ 
the component $t^{-\mu}b$, so $\bar b\in B_{\gg}(\mu-w_0(\gamma))$. Then $\chi^{\mu}_{\nu}: b\to \A^1$ is dominant if and only if there exists $i\in\cJ$ such that $\phi_i(\bar b)>\<\mu+\nu, \check{\alpha}_i\>$. 
\end{Lm}
\begin{proof} The proof is very close to that of Lemma~\ref{Lm_when_chi_lambda_0_is_dominant}. For $i\in\cJ$ recall the maps $\gq_{P_i}: \Gr_{P_i}\to \Gr_{M_i}$ from Section~\ref{Section_Recollections on crystals}. For $i\in\cJ$ let $\mu_i\le_{M_i} \mu$ be the unique element such that $\gq_{P_i}^{-1}(\Gr_{B^-(M_i)}^{\mu_i})\cap b$ is dense in $b$. Set
$$
b_0=b\cap (\mathop{\cap}\limits_{i\in\cJ} \gq_{P_i}^{-1}(\Gr_{B^-(M_i)}^{\mu_i}) \,.
$$
The subschemes $b_0\subset b\subset \Gr_{B^-}^{w_0(\gamma)}\cap \Gr_B^{\mu}$ are $T(\cO)$-invariant. Set $\bar\mu=\{\mu_i\}_{i\in \cJ}$ and
$$
Z^{\bar\mu}=\prod_{i\in \cJ} \Gr_{B(M_i)}^{\mu}\cap \Gr_{B^-(M_i)}^{\mu_i}\, .
$$
Let $\gq^{\bar\mu}: b_0\to Z^{\bar\mu}$ be the product of the maps $\gq_{P_i}$. Then $\gq^{\bar\mu}$ is $T(\cO)$-equivariant. Since $T(\cO)$ acts transitively on $Z^{\bar\mu}$, $\gq^{\bar\mu}$ is surjective. 
For $i\in\cJ$ denote by $\ev_i$ the composition
$$
\Gr_{B(M_i)}^{\mu}\cap \Gr_{B^-(M_i)}^{\mu_i}\hook{}\Gr_{B(M_i)}^{\mu}\to \Gr^{\mu}_B\toup{\chi^{\mu}_{\nu}} \A^1
$$
Denote by $\ev^{\bar\mu}: Z^{\bar\mu}\to\A^1$ the map $\ev^{\bar\mu}=\sum_{i\in\cJ} \ev_i$. We may assume that the restriction $\chi^{\mu}_{\nu}: b_0\to \A^1$ equals $\ev^{\bar\mu}\gq^{\bar\mu}$. 

 Now the morphism $\chi^{\mu}_{\nu}: b_0\to \A^1$ is dominant if and only if there is $i\in\cJ$ such that $\ev_i: \Gr_{B(M_i)}^{\mu}\cap \Gr_{B^-(M_i)}^{\mu_i}\to\A^1$ is dominant. The latter condition is equivalent to 
$$
\phi_i(\bar b)=\<\mu-\mu_i, \frac{\check{\alpha}_i}{2}\> >\<\mu+\nu, \check{\alpha}_i\>
$$
Indeed, the multiplication by $t^{\mu}$ gives an isomorphism 
$$
\Gr_{B(M_i)}^0\cap \Gr_{B^-(M_i)}^{\mu_i-\mu}\,\iso\, \Gr_{B(M_i)}^{\mu}\cap \Gr_{B^-(M_i)}^{\mu_i}
$$ 
Under this isomorphism $\ev_i$ identifies with some map $\chi^0_{\mu+\nu}: \Gr_{B(M_i)}^0\cap \Gr_{B^-(M_i)}^{\mu_i-\mu}\to \A^1$ for the group $M_i$. Our claim follows.
\end{proof}

\smallskip

 By Lemma~\ref{Lm_2.1.6}, $(\chi^{\lambda-\nu}_{\nu})^*\cL_{\psi}$ is constant on a given irreducible component $b$ of (\ref{open_sunscheme_for_L-St_thm}) if and only if $\phi_i(\bar b)\le \<\lambda, \check{\alpha}_i\>$ for all $i\in\cJ$. Since $\bar b$ is special, 
by Corollary~\ref{Cor_divisibility_of_phi(b)} we get $\phi_i(\bar b)\in \ZZ_+\delta_i$. Since $\<\lambda, \check{\alpha}_i\><\delta_i$, we conclude that $\phi_i(\bar b)=0$ for any $i\in\cJ$. This implies $\lambda=\nu+w_0(\gamma)$, a contradiction. Thus, Lemma~\ref{Lm_second_for_L-St_thm}(3) is proved.
  
 Let now $\nu=\lambda-w_0(\gamma)$. Then $\Gr_B^{\lambda-\nu}\cap \Gr_G^{\gamma}$ is a point, Lemma~\ref{Lm_second_for_L-St_thm}(4) follows. 
 
  Theorem~\ref{Thm_Lusztig-Steinberg_analog} is proved. \QED

\section{Simply-connectedness assumption}
\label{Section_Simply-connectedness assumption}

\ssec{} The purpose of this section is to identify the additional assumptions under which Theorem~\ref{Thm_Lusztig-Steinberg_analog} provides a complete decsription of the semi-simple part $\Whit^{\kappa, ss}_x$ of $\Whit^{\kappa}_x$ as a $\Rep(\check{G}_{\zeta})$-module.
\index{$\Whit^{\kappa, ss}_x$}

In general, the natural map $\bar\cM\to\Lambda/\Lambda^{\sharp}$ is not surjective. Here is an example.

\sssec{Example} Take $G=\SL_2$, so $\Lambda=\ZZ\alpha$, where $\alpha=\alpha_i$ is the simple coroot. Set $n=\delta_i$. Recall that $\delta_i$ was defined as the denominator of $\frac{\bar\kappa(\alpha_i,\alpha_i)}{2N}$.
Assume $n$ odd. Then $\check{G}_{\zeta}\,\iso\, \PSL_2$. The unique simple root of $\check{G}_{\zeta}$ is $n\alpha$, and $\Lambda^{\sharp}=n\Lambda$. We get $\bar\cM=\{a\alpha\mid 0\le a <n/2\}$.
So, for $n\ge 3$ the map $\bar\cM\to \Lambda/\Lambda^{\sharp}$ is not surjective. 

\sssec{} For $a\in\Lambda/\Lambda^{\sharp}$ recall that the set $\Lambda_a^+$ defined in Section~\ref{section_6.2} consists of those $\lambda\in\Lambda^+$ whose image in $\Lambda/\Lambda^{\sharp}$ equals $a$.

\begin{Def} 
\label{Def_free_module_of_rank_one}
For $a\in \Lambda/\Lambda^{\sharp}$ say that $\Lambda^+_a$ is a free module of rank one over $\Lambda^{\sharp,+}$ if and only if there is $\lambda_a\in \Lambda^+_a$ such that each $\lambda\in\Lambda^+_a$ can be written uniquely as $\lambda=\lambda_a+\mu$ with $\mu\in\Lambda^{\sharp,+}$.
\end{Def}

\begin{Lm} 
\label{Lm_elements_of_barcM_are_minimal}
i) For $\lambda\in\Lambda^{\sharp}$, $i\in \cJ$ one has $\<\lambda, \check{\alpha}_i\>\in \delta_i\ZZ$.\\
ii) Pick $\lambda_a\in\bar\cM$ over $a\in \Lambda/\Lambda^{\sharp}$.
Then each $\lambda\in\Lambda^+_a$ admits a unique decomposition $\lambda=\lambda_a+\mu$ with $\mu\in\Lambda^{\sharp,+}$. So, $\Lambda^+_a$ is a free module of rank one over $\Lambda^{\sharp,+}$.
\end{Lm}
\begin{proof}
i) Our claim follows from the fact that $\lambda$ is a weight of $\check{G}_{\zeta}$, and $\frac{\check{\alpha}_i}{\delta_i}$ is a coroot of $\check{G}_{\zeta}$.  \\
ii) Let $\lambda\in\Lambda^+_a$. Since $\lambda-\lambda_a\in\Lambda^{\sharp}$, we get $\<\lambda-\lambda_a, \check{\alpha}_i\>\in\delta_i\ZZ$ by i). Since $\<\lambda, \check{\alpha}_i\>\ge 0$, we get $\<\lambda-\lambda_a, \check{\alpha}_i\>\ge 0$. 
\end{proof}

\begin{Rem}
\label{Rem_order_for_G_semi-simple}
 i) If $G$ is semi-simple then we define an order on $\Lambda^+_a$ as follows. For $\lambda_1,\lambda_2\in\Lambda^+_a$ write $\lambda_1\prec\lambda_2$ if and only if $\lambda_2-\lambda_1\in \Lambda^+$ (the latter is also equivalent to $\lambda_2-\lambda_1\in \Lambda^{\sharp,+}$). Then $\Lambda^+_a$ is a free module of rank one over $\Lambda^{\sharp,+}$ if and only if there is a unique minimal element $\lambda_a$ in $\Lambda^+_a$ with respect to $\prec$. In general, $\Lambda^+_a$ is not a free module of rank one over $\Lambda^{\sharp,+}$.\\
ii) If $G$ is not semi-simple and  $\Lambda^+_a$ is a free module of rank one over $\Lambda^{\sharp,+}$ then $\lambda_a$ in Definition~\ref{Def_free_module_of_rank_one} is defined uniquely up to adding an element of $\Lambda^{\sharp}_0$. 
\end{Rem}

\ssec{Additional assumption}  For the rest of Section~\ref{Section_Simply-connectedness assumption} assume $[\check{G}_{\zeta}, \check{G}_{\zeta}]$ simply-connected. Under this assumption we can completely understand the structure of $\Whit^{\kappa, ss}_x$ as a module over $\Rep(\check{G}_{\zeta})$. Here, as in Section~\ref{section_5.3}, $\Whit^{\kappa, ss}_x\subset \Whit^{\kappa}_x$ denotes the full subcategory consisting of objects, which are finite direct sums of irreducible ones.

For $i\in \cJ$ pick a fundamental weight $\omega_i\in\Lambda^{\sharp, +}$ of $\check{G}_{\zeta}$ corresponding to the coroot $\frac{\check{\alpha}_i}{\delta_i}$. Note that $\<\omega_i, \check{\alpha}_i\>=\delta_i$. We get
$$
\Lambda^{\sharp}=\Lambda^{\sharp}_0\oplus (\oplus_{i\in \cJ} \,\ZZ\omega_i)\;\;\;\;\mbox{and}\;\;\;\; \Lambda^{\sharp,+}=\Lambda^{\sharp}_0\oplus (\oplus_{i\in \cJ} \,\ZZ_+\omega_i)
$$
\begin{Lm} 
\label{Lm_pr_is_surjective_for_bar_cM}
The map $\pr: \bar\cM\to\Lambda/\Lambda^{\sharp}$, $\lambda\mapsto \lambda+\Lambda^{\sharp}$ is surjective. If $\lambda\in \bar\cM$, and $a=\pr(\lambda)$ then the fibre of $\pr$ over $a$ is $\lambda+\Lambda^{\sharp}_0$.
\end{Lm}
\begin{proof} Note that $\Lambda=\Lambda^++\Lambda^{\sharp}$. Let $a\in\Lambda/\Lambda^{\sharp}$. Pick $\lambda\in\Lambda^+$ over $a$. If there is $i\in\cJ$ with $\<\lambda,\check{\alpha}_i\>\ge \delta_i$ then replace $\lambda$ by $\lambda-\omega_i$. We get $\lambda-\omega_i\in\Lambda^+$ and $\pr(\lambda-\omega_i)=a$. Continuing this procedure, one gets $\lambda\in \bar\cM$ with $\pr(\lambda)=a$. 

 If $\lambda, \lambda'\in \bar\cM$ with $\lambda-\lambda'\in\Lambda^{\sharp}$ then for any $i\in\cJ$, $\<\lambda-\lambda', \check{\alpha}_i\>\in \delta_i\ZZ$. So, for any $i\in\cJ$, $\<\lambda-\lambda', \check{\alpha}_i\>=0$ and $\lambda-\lambda'\in\Lambda^{\sharp}_0$. 
\end{proof}

Using Lemma~\ref{Lm_pr_is_surjective_for_bar_cM}, we pick for each $a\in \Lambda/\Lambda^{\sharp}$ an element $\lambda_a\in \bar\cM$ with $a=\pr(\lambda)$. Set $\cM=\{\lambda_a\mid a\in \Lambda/\Lambda^{\sharp}\}\subset \bar\cM$. The projection $\cM\to \Lambda/\Lambda^{\sharp}$ is bijective. From Theorem~\ref{Thm_Lusztig-Steinberg_analog} we now derive the following.
\index{$\cM, \omega_i$}
 
\begin{Cor} 
\label{Cor_free_module_over_Rep}
Assume $[\check{G}_{\zeta}, \check{G}_{\zeta}]$ simply-connected. Then $\Whit^{\kappa, ss}_x$ is a free module over $\Rep(\check{G})$ with base $\{\cF_{x, \lambda}\mid\lambda\in \cM\}$. \QED
\end{Cor}

\section{Examples of Kazhdan-Lusztig's type polynomials}  
\label{Section_Examples_KL_polynomials}
 
\ssec{Inductive structure} Pick $x\in X$. Recall for $\mu\in \Lambda$ the locally closed immersion $j_{x, \mu}: \wt\gM_{x, \mu}\hook{} \wt\gM_x$. Since the version of the twisted Whittaker category on $\wt\gM_{x, \mu}$ is semi-simple (cf. Section~\ref{Section_2.3}), for each $\mu\le\lambda$ with $\mu, \lambda\in\Lambda^+$ we get 
$$
j_{x,\mu}^*\cF_{x,\lambda}\,\iso\, \cF_{x, \mu}\otimes K^{\lambda}_{\mu},
$$
where $K^{\lambda}_{\mu}$ is a cohomologically graded $\Qlb$-vector space. 
\index{$K^{\lambda}_{\mu}$}
 
  We think of $K^{\lambda}_{\mu}$ as a version of Kazhdan-Lusztig's polynomials expressing the relation between the two bases in the Grothendieck group of $\Whit^{\kappa}_x$, the first constings of $\cF_{x,\lambda, !}$, the second constings of the irreducible objects. 
% According to Gaitsgory's conjecture (\cite{G1}, Conjecture~0.4), $\cF_{x, \lambda}$ should correspond to the irreducible representations of the quantum group, and $\cF_{x,\lambda, !}$ should correspond to the Verma modules. So, these polynomials are supposed to give a relation between the two corresponding bases of the Grothendick group $\Rep(U_q(\check{G}))$.

Let $M\subset G$ be a standard Levi subgroup. Then $M$ is equipped with the metaplectic data induced from those for $G$, so that we have the corresponding twisted Whittaker category $_M\Whit^{\kappa}_x$ for $M$, and its irreducible objects $_M\cF_{x,\lambda}$ for all $\lambda\in \Lambda^+_M$. Now for $\mu\le_M\lambda$ with $\mu,\lambda\in\Lambda^+_M$ one has 
$$
j_{x, \mu}^*(_M\cF_{x,\lambda})\,\iso\, {_M\cF_{x,\mu}}\otimes (_MK^{\lambda}_{\mu})
$$
as above. The multiplicity spaces $K^{\lambda}_{\mu}$ have the following inductive structure when passing from $G$ to $M$.
\index{$_M\Whit^{\kappa}_x, {_M\cF_{x,\lambda}}, \; {_MK^{\lambda}_{\mu}}$}

\begin{Pp} 
\label{Pp_inductive_structure}
Let $\mu,\lambda\in\Lambda^+$ with $\mu\le_M\lambda$. Then 
$K^{\lambda}_{\mu}\,\iso\, {_MK^{\lambda}_{\mu}}$ canonically. 
\end{Pp} 
\begin{proof}
Consider the Zastava space $\cZ^{\mu}_{x,\le\lambda}$ from Section~\ref{section_4.1_Zastava} and its version $_M\cZ^{\mu}_{x,\le\lambda}$ for the Levi $M$. 
The natural map $_M\cZ^{\mu}_{x,\le\lambda}\to \cZ^{\mu}_{x,\le\lambda}$ is an isomorphism. From the factorization property it follows that
the functor $F^{\mu}$ commutes with the Verdier duality. So, $F^{\mu}(\cF_{x,\lambda})$ is an irreducible perverse sheaf on $\wt\cZ^{\mu}_{x,\le\lambda}$. So, it suffices to calculate the $*$-restriction of $F^{\mu}(\cF_{x,\lambda})$ to $\wt\cZ^{\mu}_{x, \mu}:=\wt\cZ^{\mu}_x\times_{\gM_x} \gM_{x, \mu}$.
\end{proof}

\ssec{} Let $i\in \cJ$ and $\lambda\in\Lambda^+$ with $\lambda-\alpha_i\in\Lambda^+$. 

\begin{Pp} Assume the subtop cohomology property is satisfied for $\varrho$. If $\<\lambda, \check{\alpha}_i\> \notin \delta_i\ZZ$ then $K^{\lambda}_{\lambda-\alpha_i}=0$.
\end{Pp}
\begin{proof} By Proposition~\ref{Pp_inductive_structure}, we may and will  assume $G$ of semi-simple rank one. Since $\cJ$ consists of one element, we suppress the index $i$ from the notation $\alpha_i$, $\delta_i$ and so on. 
 
 Consider the diagram 
$$
\gM_{x,\le \lambda}\,\getsup{'\gp} \,\cZ^{\lambda-\alpha}_{x,\le\lambda}\,\toup{\pi^{\lambda-\alpha}} \,X^{\lambda-\alpha}_{x,\le\lambda}\, .
$$
We have an isomorphism $X\;\iso\; X^{\lambda-\alpha}_{x,\le\lambda}$ sending $y$ to $\lambda x-\alpha y$. Pick a fundamental weight $\check{\omega}$ corresponding to $\alpha$. Then 
\begin{equation}
\label{line_bundle}
\pi^{\lambda-\alpha}: \cZ^{\lambda-\alpha}_{x,\le\lambda}\to X^{\lambda-\alpha}_{x,\le\lambda}
\end{equation}
is a line bundle over $X$ whose fibre over $y$ is $\Omega^{-1}(2\<\lambda, \check{\omega}\> -y)_y$. This is the total space of the line bundle $\cO(2\<\lambda, \check{\omega}\>x)$ over $X$. 

 For $\nu\in\Lambda$ set 
$$
\cZ^{\mu}_{x, \nu}=\cZ^{\mu}_x\times_{\gM_x} \gM_{x,\nu} \, .
$$ 
The open subscheme $\cZ^{\lambda-\alpha}_{x, \lambda}\subset \cZ^{\lambda-\alpha}_{x,\le\lambda}$ is the complement to the zero section of the above line bundle. 

 From Theorem~\ref{Thm_3.10.1} we see that $\bar F^{\lambda-\alpha}(\cF_{x,\lambda})$ is the extension by zero under $\cZ^{\lambda-\alpha}_{x,\le\lambda}\hook{} \ov{\cZ}^{\lambda-\alpha}_{x,\le\lambda}$. From Theorem~\ref{Thm_objects_remain_irred} we now derive 
$
(\pi^{\lambda-\alpha})_! F^{\lambda-\alpha}(\cF_{x, \lambda})\,\iso\, \cL_{x,\lambda}
$.  By Lemma~\ref{Lm_fibre_of^omega_cL_bar_kappa}, the line bundle $\cP^{\bar\kappa}$ is constant over $X-x\subset X\,\iso\, X^{\lambda-\alpha}_{x, \le\lambda}$ with fibre 
$$
(\Omega_x^{\frac{1}{2}})^{-\bar\kappa(\lambda, \lambda+2\rho)}\otimes \epsilon^{\bar\lambda}
$$ 
A trivialization of the latter line identifies $\cP^{\bar\kappa}$ over $X\,\iso\, X^{\lambda-\alpha}_{x, \le\lambda}$ with $\cO(-mx)$, where $m=\bar\kappa(\alpha, \lambda)$. 

 We have $m\notin N\ZZ$. Indeed, let $a\in\ZZ$ be such that $(a, \delta)=1$ and $\frac{\bar\kappa(\alpha, \alpha)}{2N}=\frac{a}{\delta}$. From the formula
(\ref{def_bar_kappa}) it follows that $\frac{m}{N}=\frac{a}{\delta} \<\check{\alpha}, \lambda\>\notin \ZZ$, because of our assumption $\<\lambda, \check{\alpha}\> \notin \delta\ZZ$. It follows that the $(\lambda-\alpha)$-component $\cL^{\lambda-\alpha}_{x, \lambda}$ of $\cL_{x,\lambda}\in \wt\FS^{\kappa}_x$ is the extension by zero under $X-x\hook{} X$. 
 
 The fibre of (\ref{line_bundle}) over $x\in X$ is $\cO(2\<\lambda, \check{\omega}\>)_x$. The restriction of 
$$
\cZ^{\lambda-\alpha}_{x, \lambda}\toup{'\gp} \gM_{x,\lambda}\toup{\ev_{x,\lambda}}\A^1
$$ 
to $\cO(2\<\lambda, \check{\omega}\>)_x-\{0\}$ is constant with value $0$. Write $f: \cZ^{\lambda-\alpha}_{x,\lambda}\to X^{\lambda-\alpha}_{x, \le\lambda}$ for the restriction of (\ref{line_bundle}). The fibre of $f_! F^{\lambda-\alpha}(\cF_{x,\lambda})$ at $x\in X$ also vanishes, because $m\notin N\ZZ$. Thus, the $*$-fibre of $F^{\lambda-\alpha}(\cF_{x,\lambda})$ at $0\in \cO(2\<\lambda, \check{\omega}\>)_x$ vanishes. 
\end{proof}

\begin{Rem} Let $\lambda\in\Lambda^{\sharp,+}$, $\mu\in\Lambda^+$ with $\mu\le\lambda$. In the notations of Section~\ref{subsection_CSh_basics} related to the Casselman-Shalika formula, we get 
\begin{equation}
\label{complex_K_generalized_Gauss_sum}
K^{\lambda}_{\mu}\,\iso\, \RG_c(\Gr_B^{-\mu}\cap \ov{\Gr}_G^{-w_0(\lambda)}, (\chi^{-\mu}_{\mu})^*\cL_{\psi}\otimes (s^{-\mu}_B)^*\cA^{-w_0(\lambda)}_{\cE})[-\<\mu, 2\check{\rho}\>]\, .
\end{equation}
For example, if $\alpha_i$ is a simple coroot and $\mu=\lambda-\alpha_i$ then $\Gr_B^{\alpha_i-\lambda}\cap \ov{\Gr}_G^{-w_0(\lambda)}\,\iso\,\A^1$ and (\ref{complex_K_generalized_Gauss_sum}) becomes the Gauss sum. So,  $K^{\lambda}_{\lambda-\alpha_i}\,\iso\, \Qlb[1]$ noncanonically and 
$$
\Ext^1(\cF_{x, \lambda-\alpha_i}, \cF_{x,\lambda})\,\iso\, \Qlb\, .
$$ 
Thus, we may think of (\ref{complex_K_generalized_Gauss_sum}) as a generalization of the Gauss sum. We also see that $\Whit^{\kappa}_x$ is not semi-simple.
\end{Rem}

\section{Hecke functors on $\wt\FS^{\kappa}_x$}
\label{section_H_functors_for_FS}

\ssec{} 
\label{section_Hecke_functors_for_FS}
In this section we define an action of $\Rep(\check{T}_{\zeta})$ on $\wt\FS^{\kappa}_x$ by Hecke functors. The main result of this section is Theorem~\ref{Thm_functoriality} showing that $\ov{\FF}: \Whit^{\kappa}_x\to \FS^{\kappa}_x$ commutes with the actions of Hecke functors. 

Pick $x\in X$. Define an action of $\Rep(\check{T}_{\zeta})$ on $\wt\FS^{\kappa}_x$ by Hecke functors as follows. The definition is analogous to that of (\cite{L3}, Section~5.2.3). 

By (\cite{L3}, Proposition~4.1), for $\mu\in\Lambda$, $\cF\in\Bun_T$ we have canonically 
\begin{equation}
\label{iso_for_Sect_5.1}
(^{\omega}\cL^{\bar\kappa})_{\cF(\mu x)}\,\iso\, (^{\omega}\cL^{\bar\kappa})_{\cF}\otimes (\cL^{\bar\kappa(\mu)}_{\cF})_x\otimes \cL^{\bar\kappa}_{\cO(\mu x)}
\end{equation}
The section $\gt^0_{\EE}: \Lambda^{\sharp}\to  V_{\EE}$ picked in Section~\ref{section_Additional input data} yields for each $\mu\in\Lambda^{\sharp}$ a trivialization of $\cL^{\bar\kappa}_{\cO(\mu x)}$. 

For $\mu\in\Lambda^{\sharp}$, $\lambda\in \Lambda$ let $m_{\mu}: \wt X^{\lambda}_x\to \wt X^{\lambda+\mu}_x$ be the map sending $(D, \cU)$ together with $\cU^N\,\iso\, \cP^{\bar\kappa}_D$ to $(D+\mu x, \cU')$, where 
$$
\cU'=\cU\otimes ((\cL^{-\bar\kappa(\mu)/N})_{\Omega^{\rho}(-D)})_x
$$ 
with the isomorphism $\cU'^N\,\iso\, \cP^{\bar\kappa}_{D+\mu x}$ induced by (\ref{iso_for_Sect_5.1}) and $\gt^0_{\EE}$. The map $m_{\mu}$ is an isomorphism.
\index{$m_{\mu}, \H^{\ra}$}

For $\mu\in\Lambda^{\sharp}$ write $\Qlb^{\mu}\in \Rep(\check{T}_{\zeta})$ for $\Qlb$ on which $\check{T}_{\zeta}$ acts by $\mu$. Let 
$$
\H^{\ra}: \Rep(\check{T}_{\zeta})\times \wt\FS^{\kappa}_x\to \wt\FS^{\kappa}_x
$$ 
be the functor commuting with direct sums in $\Rep(\check{T}_{\zeta})$ and such that $\H^{\ra}(\Qlb^{\mu}, \cdot):  \wt\FS^{\kappa}_x\to \wt\FS^{\kappa}_x$ is the functor $(m_{\mu})_*$. 

An object $\cL\in \wt\FS^{\kappa}_x$ is a collection $\cL=\{\cL^{\nu}\}_{\nu\in\Lambda}$. Then $(m_{\mu})_*\cL$ is understood as the collection such that for any $\nu\in\Lambda$, its $\nu$-th component is $(m_{\mu})_*\cL^{\nu-\mu}$.

For $\mu\in \Lambda^{\sharp}, \lambda\in\Lambda$ one has 
\begin{equation}
\label{action_of_Hecke_on_objects_of_FS}
\H^{\ra}(\Qlb^{\mu}, \cL_{x,\lambda})\,\iso\, \cL_{x,\lambda+\mu}\;\;\;\;\mbox{and}\;\;\;\;\H^{\ra}(\Qlb^{\mu}, \cL_{x,\lambda,*})\,\iso\, \cL_{x,\lambda+\mu,*} \; .
\end{equation}
As in (\cite{ABBGM}, Section~1.1.7), we may call $\cL_{x,\lambda,*}\in \wt\FS^{\kappa}_x$ \select{the baby co-Verma module} of highest weight $\lambda$. 

 Recall that $\FS^{\kappa}_x\subset \wt\FS^{\kappa}_x$ is the full subcategory of objects of finite length. For $V\in \Rep(\check{T}_{\zeta})$ the functor $\H^{\ra}(V,\cdot)$ preserves the subcategory $\FS^{\kappa}_x$. 
 
\begin{Lm} Assume $\varrho$ satisfies the subtop cohomology property. Then $\ov{\FF}$ sends $\Whit^{\kappa}_x$ to $\FS^{\kappa}_x$.
\end{Lm}
\begin{proof} This follows from Theorem~\ref{Th_4.12.1}, since a perverse sheaf is of finite length.
\end{proof}

Write $\Res^{\check{G}_{\zeta}}_{\check{T}_{\zeta}}$ for the restriction functor $\Rep(\check{G}_{\zeta})\to \Rep(\check{T}_{\zeta})$. Recall the twisted Satake equivalence $\Rep(\check{G}_{\zeta})\,\iso\, \PPerv_{G,\zeta}^{\natural}$ from Section~\ref{section_Metaplectic dual group}. The following is an analog of (\cite{ABBGM}, Proposition~1.1.11). 
\index{$\Res^{\check{G}_{\zeta}}_{\check{T}_{\zeta}}$}

\begin{Thm} 
\label{Thm_functoriality}
Assume $\varrho$ satisfies the subtop cohomology property. 
There is an isomorphism in $\FS^{\kappa}_x$
functorial in $V\in \Rep(\check{G}_{\zeta})$ and $K\in \Whit^{\kappa}_x$
$$
\ov{\FF}(\H^{\ra}_G(V, K))\,\iso\, \H^{\ra}(\Res^{\check{G}_{\zeta}}_{\check{T}_{\zeta}}(V), \ov{\FF}(K))\, .
$$
\end{Thm}

 The proof is given in Sections~\ref{Section_512}-\ref{Section_proof_Thm_functoriality}.

\sssec{} 
\label{Section_512}
For $\mu\in\Lambda$ denote by $_{x,\infty}\Bunb^{\mu}_{B^-}$ the stack classifying $\cF\in\Bun_G$, $\cF_T\in\Bun_T$ with $\deg\cF_T=(2g-2)\rho-\mu$, and a collection of nonzero maps of coherent sheaves for $\check{\lambda}\in\check{\Lambda}^+$
$$
\kappa^{\check{\lambda},-}: \cV^{\check{\lambda}}_{\cF}\to \cL^{\check{\lambda}}_{\cF_T}
$$
over $X-x$ satisfying the Pl\"ucker relations. For $\nu\in\Lambda$ we define the diagram 
$$
_{x,\nu}\Bunb^{\mu}_{B^-}\hook{}
{_{x,\ge\nu}\Bunb^{\mu}_{B^-}}\hook{} {_{x,\infty}\Bunb^{\mu}_{B^-}},
$$ 
where $_{x,\ge\nu}\Bunb^{\mu}_{B^-}$ is
the closed substack given by the property that for each $\check{\lambda}\in \check{\Lambda}^+$ the map
\begin{equation}
\label{map_kappa^-_with_poles}
\kappa^{\check{\lambda},-}: \cV^{\check{\lambda}}_{\cF}(\<\nu, \check{\lambda}\>x)
\to \cL^{\check{\lambda}}_{\cF_T}
\end{equation}
is regular over $X$. Further, $_{x,\nu}\Bunb^{\mu}_{B^-}$ is the open substack of $_{x,\ge\nu}\Bunb^{\mu}_{B^-}$ given by the property that for $\check{\lambda}\in \check{\Lambda}^+$ the map (\ref{map_kappa^-_with_poles}) has no zeros on $X$. 
\index{$_{x,\infty}\Bunb^{\mu}_{B^-}$}\index{${_{x,\nu}\Bunb^{\mu}_{B^-}}$}%
\index{${_{x,\ge\nu}\Bunb^{\mu}_{B^-}}$}

\sssec{} We derive Theorem~\ref{Thm_functoriality} essentially from the fact that the formation of the principal geometric Eisenstein series in the twisted setting commutes with Hecke functors, more precisely, from the following reformulation of (\cite{L2}, Proposition~3.2) with $B$ replaced by $B^-$. 

 Recall the stack $\Bunb_{\tilde B^-}$ defined in Section~\ref{section_4.6_for_introducing_IC_zeta}. The stack $_{x,\infty}\Bunb_{\tilde B^-}$ is defined similarly. Write a point of $_{x,\infty}\Bunb_{\tilde B^-}$ as a collection $(\cF, \cF_T, \kappa^-, \cU, \cU_T)$, where $\cU, \cU_T$ are the lines equipped with isomorphisms
$$
\cU^N\,\iso\, (^{\omega}\cL^{\bar\kappa})_{\cF},\;\;\;\; 
\cU_T^N\,\iso\, (^{\omega}\cL^{\bar\kappa})_{\cF_T}\, .
$$
\index{$_{x,\infty}\Bunb_{\tilde B^-}$}%
For $\lambda\in\Lambda^+$ the stack $_x\cH^{\lambda}_{\tilde G}$ defined in Section~\ref{section_543} classifies $(\cF, \cF', \cF\,\iso\,\cF'\mid_{X-x}, \cU, \cU')$ with $\cF'$ at the position $\le\lambda$ with respect to $\cF$ at $x$. Here $\cU, \cU'$ are lines equipped with
\begin{equation}
\label{lines_cU_and_cU'}  
\cU^N\,\iso\, (^{\omega}\cL^{\bar\kappa})_{\cF}, \;\; \cU'^N\,\iso\, (^{\omega}\cL^{\bar\kappa})_{\cF'}
\end{equation}

Assume $\lambda\in\Lambda^{\sharp, +}$. Let $Y^{-w_0(\lambda)}={_x\cH^{-w_0(\lambda)}_{\tilde G}}\times_{\Bunt_G} \Bunb_{\tilde B^-}$, where we used the map $\tilde h^{\la}_G: {_x\cH^{-w_0(\lambda)}_{\tilde G}}\to\Bunt_G$ to define the fibred product. As in  Section~\ref{section_5.2_action_on_D_zeta}, the fibration $h^{\la}: Y^{-w_0(\lambda)}\to \Bunb_{\tilde B^-}$ yields the twisted exterior product $\cK:=(\cA^{-w_0(\lambda)}_{\cE}\tboxtimes \IC_{\zeta})^l$ on $Y^{-w_0(\lambda)}$. Here $\IC_{\zeta}$ is the perverse sheaf on $\Bunb_{\tilde B^-}$ defined in Section~\ref{section_4.6_for_introducing_IC_zeta}. 

 Denote by
$$
\phi: Y^{-w_0(\lambda)}\to {_{x,\ge -\lambda}\Bunb_{\tilde B^-}}
$$ 
the map sending $(\cF,\cF',  \cF\,\iso\,\cF'\mid_{X-x}, \cU, \cU')\in {_x\cH^{-w_0(\lambda)}_{\tilde G}}$, $(\cF,\cF_T, \kappa^-, \cU, \cU_T)\in \Bunb_{\tilde B^-}$ to 
$$
(\cF', \cF_T, \kappa^-, \cU', \cU_T)
$$ 

 For $\nu\in\Lambda^{\sharp}$ define the isomorphism $\tilde i_{\nu}: \Bunb_{\tilde B^-}\,\iso\, {_{x,\ge\nu}\Bunb_{\tilde B^-}}$ sending $(\cF, \cF_T, \kappa^-, \cU, \cU_T)$ to $(\cF, \cF_T(\nu x), \cU, \bar\cU_T)$, where 
$$
\bar\cU_T=\cU_T\otimes (\cL^{\bar\kappa(\nu)/N}_{\cF_T})_x
$$ 
is equipped with the isomorphism $\bar\cU_T^N\,\iso\, (^{\omega}\cL^{\bar\kappa})_{\cF_T(\nu x)}$ obtained from (\ref{iso_for_Sect_5.1}) and $\gt^0_{\EE}$. 
\index{$\tilde i_{\nu}$}

 Recall that $V(\lambda)$ denotes the irreducible $\check{G}_{\zeta}$-representation with highest weight $\lambda$, and $V(\lambda)_{\nu}$ its $\check{T}_{\zeta}$-weight space corresponding to $\nu$. 

\begin{Pp}[\cite{L2}, Proposition~3.2] 
\label{Pp_5.1.7_on_Eis}
For $\lambda\in\Lambda^{\sharp, +}$ there is an isomorphism 
$$
\phi_!((\cA^{-w_0(\lambda)}_{\cE}\tboxtimes \IC_{\zeta})^l)\,\iso\, \mathop{\oplus}\limits_{\nu\ge -\lambda}  (\tilde i_{\nu})_!\IC_{\zeta}\otimes V(\lambda)_{-\nu}
$$
\end{Pp}

\sssec{} Let 
$$
\uZ^{\mu}_x\subset \gM_x\times_{\Bun_G}\, {_{x,\infty}\Bunb^{\mu}_{B^-}}
$$ 
be the open substack given by the property that for each $\check{\lambda}\in\check{\Lambda}^+$ the composition
$$
\Omega^{\<\rho, \check{\lambda}\>}\,\toup{\kappa^{\check{\lambda}}} \,\cV^{\check{\lambda}}_{\cF}\,\toup{\kappa^{\check{\lambda},-}} \,\cL^{\check{\lambda}}_{\cF_T},
$$
which is regular over $X-x$, is not zero. So, $\ov{\cZ}^{\mu}_x\hook{} \uZ^{\mu}_x$ is a closed ind-substack. We get the diagram
$$
\begin{array}{cccc}
\gM_x  \getsup{'\und{\gp}}& \uZ^{\mu}_x & \toup{'\und{\gp}_B} &  {_{x,\infty}\Bunb^{\mu}_{B^-}}\\
& \downarrow\lefteqn{\scriptstyle \und{\pi}^{\mu}} && \downarrow\lefteqn{\scriptstyle \bar\gq^-}\\
& X^{\mu}_x & \toup{AJ} & \Bun_T,
\end{array}
$$
where $'\und{\gp}$ and $'\und{\gp}_B$ are the projections. Here $\und{\pi}^{\mu}$ sends $(\cF, \kappa, \kappa^-, \cF_T)\in \uZ^{\mu}_x$ to $D$ such that $\kappa^-\comp\kappa$ induces an isomorphism $\Omega^{\rho}(-D)\,\iso\, \cF_T$. %For $\nu\in\Lambda$ write $\uZ^{\mu}_{x,\le\nu}$ for the preimage of $\gM_{x,\le\nu}$ under $'\und{\gp}$. 
\index{$\uZ^{\mu}_x, {'\und{\gp}}, {'\und{\gp}_B}, \und{\pi}^{\mu}$}

 Let $\uZ_{\tilde G,x}^{\mu}$ be the stack obtained from $\uZ^{\mu}_x$ by the base change $\wt\gM_x\times \wt X^{\mu}_x\to \gM_x\times X^{\mu}_x$. 
Let $\ov{\cZ}_{\tilde G, x}^{\mu}$ be the restriction of the gerbe $\uZ_{\tilde G,x}^{\mu}$ to the closed substack $\ov{\cZ}_x^{\mu}\hook{} \uZ_x^{\mu}$.
\index{$\uZ_{\tilde G,x}^{\mu}, \ov{\cZ}_{\tilde G, x}^{\mu}$}

\sssec{Proof of Theorem~\ref{Thm_functoriality}} 
\label{Section_proof_Thm_functoriality}
Let $\lambda\in\Lambda^{\sharp, +}$ and $\mu\in\Lambda$. We calculate $\ov{\FF}\H^{\ra}_G(\cA^{\lambda}_{\cE}, K)$ over $\wt X^{\mu}_x$. Recall the stack $Z$ defined in Section~\ref{section_5.2_action_on_D_zeta}. Set $\cY^{\mu}=Z\times_{\wt\gM_x} \ov{\cZ}_{\tilde G, x}^{\mu}$, where  we used the maps $\ov{\cZ}_{\tilde G, x}^{\mu}\toup{'\bar\gp} \wt\gM_x \getsup{'h^{\la}}Z$ to define the fibred product. 

  The stack $\cY^{\mu}$ classifies $(\cF, \cF', \cF\,\iso\, \cF'\mid_{X-x}, \cU, \cU')\in {_x\cH_{\tilde G}}$, $(\cF, \kappa, \kappa^-, \cF_T, \cU_T, \cU)\in  \ov{\cZ}_{\tilde G, x}^{\mu}$. Let $\cY^{\mu, -w_0(\lambda)}\subset \cY^{\mu}$ be the closed substack given by the property that 
$$
(\cF, \cF', \cF\,\iso\, \cF'\mid_{X-x},\cU,\cU')\in {_x\cH^{-w_0(\lambda)}_{\tilde G}}
$$ 
Let $q: \cY^{\mu, -w_0(\lambda)}\to \uZ^{\mu}_{\tilde G,x}$ be the map sending the above collection to $(\cF', \kappa, \kappa^-, \cF_T, \cU', \cU_T)$. 

 The following diagram commutes
$$
\begin{array}{cccccc}
\wt\gM_x \getsup{'h^{\ra}} & Z & \getsup{a} & \cY^{\mu, -w_0(\lambda)} & \toup{q} & \uZ^{\mu}_{\tilde G,x}\\
 & \downarrow\lefteqn{\scriptstyle 'h^{\la}} && \downarrow\lefteqn{\scriptstyle b} && \downarrow\lefteqn{\scriptstyle \und{\pi}^{\mu}}\\
 & \wt\gM_x & \getsup{'\bar\gp} & \ov{\cZ}_{\tilde G, x}^{\mu} & \toup{\ov{\pi}^{\mu}} & X^{\mu}_x
\end{array}
$$ 
Here $a,b$ are the natural maps. Recall that $\H^{\ra}_G(\cA^{\lambda}_{\cE}, K)=('h^{\la})_!(K\tboxtimes \cA^{\lambda}_{\cE})^r$. 
We get 
$$
\ov{\FF}\H^{\ra}_G(\cA^{\lambda}_{\cE}, K)\,\iso\, (\und{\pi}^{\mu})_!q_!((a^*(K\tboxtimes \cA^{\lambda}_{\cE})^r)\otimes b^* ('\bar\gp_B)^*\IC_{\zeta})[-\dim\Bun_G]\, .
$$

 We have the following cartesian square
$$
\begin{array}{ccc}
Y^{-w_0(\lambda)} & \toup{\phi} & {_{x,\infty}\Bunb_{\tilde B^-}}\\
\uparrow\lefteqn{\scriptstyle \alpha} && \uparrow\lefteqn{\scriptstyle \beta}\\
\cY^{\mu, -w_0(\lambda)} & \toup{q} & \uZ^{\mu}_{\tilde G,x},
\end{array}
$$
where the maps $\alpha,\beta$ forget the generalized $B$-structure $\kappa$. Further 
$$
(a^*(K\tboxtimes \cA^{\lambda}_{\cE})^r)\otimes b^* ('\bar\gp_B)^*\IC_{\zeta}\,\iso\, q^*('\und{\gp})^*K\otimes \alpha^* \cK\, .
$$

 The complex $q_!\alpha^*\cK\,\iso\, \beta^*\phi_!\cK$ is calculated in Proposition~\ref{Pp_5.1.7_on_Eis}. Plugging it into the above expression, we get from definitions an isomorphism
$$
\ov{\FF}\H^{\ra}_G(\cA^{\lambda}_{\cE}, K)\,\iso\, \mathop{\oplus}\limits_{\nu\ge -\lambda} ((m_{\nu})^*\ov{\FF}(K))\otimes V(\lambda)_{-\nu}
$$
over $\wt X^{\mu}_x$. Since $m_{\nu}^*=(m_{-\nu})_*$,
the latter identifies with $\H^{\ra}(\Res^{\check{G}_{\zeta}}_{\check{T}_{\zeta}}(V(\lambda)), \ov{\FF}(K))$. Theorem~\ref{Thm_functoriality} is proved. 

\ssec{A subcategory of $\FS^{\kappa}_x$} Here is another application of the Hecke functors on $\FS^{\kappa}_x$. Recall the quadratic form $\varrho: \Lambda\to\QQ$ defined in Section~\ref{section_input_data}.

\begin{Pp} 
\label{Pp_closed_under_extensions}
Assume that $\varrho(\alpha_i)\notin\ZZ$ for any simple coroot $\alpha_i$, and the subtop cohomology property holds. Then the functor $\Rep(\check{T}_{\zeta})\to \FS^{\kappa}_x$, $V\mapsto \H^{\ra}(V, \cL_{\emptyset})$ is fully faithful, and its image is a subcategory of $\FS^{\kappa}_x$ closed under extensions.
\end{Pp}
\begin{proof}
We only have to check the last property. It is clear that $\Ext^i(\cL_{\emptyset}, \cL_{\emptyset})=0$ for $i>0$. So, it suffices to check the following. If $\lambda\in\Lambda^{\sharp}$ with $0<\lambda$ then the $*$-fibre $(\cL_{\emptyset})_{-\lambda x}$ is placed in degrees $\le -2$. By Propositions~\ref{Pp_4.11.1} and \ref{Pp_3.11.2}, this is true unless $\lambda$ is a simple coroot. However, any simple coroot $\alpha_i$ is not in $\Lambda^{\sharp}$ because of our assumption $\varrho(\alpha_i)\notin\ZZ$.
\end{proof}
 
\section{Casselman-Shalika formula: basic observations}
\label{C-Sh-formula_basics}

\ssec{}  In this section we formulate and discuss the metaplectic analog of the Casselman-Shalika formula of \cite{FGV}. 

 Let us introduce some general notation. For a subset $\cB$ of $\cup_{\lambda, \mu\in\Lambda} \Irr(\Gr_B^{\lambda}\cap\Gr_{B^-}^{\mu})$ we set $\und{\cB}=\{b\in \cB\mid\bar b\; \mbox{is special}\}$. Here for $b\in \Irr(\Gr_B^{\lambda}\cap\Gr_{B^-}^{\mu})$ we let $\bar b=t^{-\lambda}b\in B_{\gg}$. The notation $B_{\gg}$ is that of Section~\ref{Section_Recollections on crystals}, and special elements in crystals were  introduced in Definition~\ref{Def_special_elements_in_B_gg}. 
 Given $V\in \Rep(\check{G}_{\zeta}), K\in \Whit^{\kappa}_x$ set for brevity $K\ast V=\H^{\ra}_G(V, K)$. 
\index{$\und{\cB}, K\ast V$}

\ssec{} 
\label{subsection_CSh_basics}
The Casselman-Shalika formula in the non-twisted case is (\cite{FGV}, Theorem~1). The following could be thought of as the metaplectic Casselman-Shalika problem. 

Recall that in Section~\ref{section_Additional input data} we picked a trivialization $\delta_{\lambda}$ of the fibre of $\Gra_G\to\Gr_G$ over $t^{\lambda}G(\cO)$ for each $\lambda\in\Lambda$ (compatible with the action of $\Lambda^{\sharp}$). This provided the sections $s^{\lambda}_B: \Gr_B^{\lambda}\to \wt\Gr_G$ of the gerbe $\wt\Gr_G\to\Gr_G$ (cf. Section~\ref{section_4.12.4}). 

 As in \cite{FGV}, for $\eta\in\Lambda$ we write $\chi_{\eta}: U(F)\to \A^1$ for the additive character of conductor $\bar\eta$, where $\bar\eta$ is the image of $\eta$ in the coweights lattice of $G_{ad}$. 
For $\eta+\nu\in\Lambda^+$ we also write $\chi^{\nu}_{\eta}: \Gr_B^{\nu}\to\A^1$ for any $(U(F), \chi_{\eta})$-equivariant function. The isomorphism $\Gr_B^0\,\iso\, \Gr_B^{\eta}$, $v\mapsto t^{\eta}v$ transforms $\chi^0_0: \Gr_B^0\to\A^1$ to $\chi^{\eta}_{-\eta}: \Gr_B^{\eta}\to\A^1$. 

 For $\nu\in \Lambda^{\sharp, +}$ we denote by $\wt{\ov{\Gr}}{}^{\nu}_G$ the restriction of the gerbe $\wt\Gr_G\to\Gr_G$ to $\ov{\Gr}^{\nu}_G$. Recall the irreducible objects $\cA^{\nu}_{\cE}$ of $\PPerv_{G,\zeta}$ defined in (\cite{L1}, Section~2.4.2), we are using for their definition the choice of $\Omega^{\frac{1}{2}}$ from Section~\ref{sssec_Notation}. The perverse sheaf $\cA^{\nu}_{\cE}$ is defined only up to a scalar automorphism (but up to a unique isomorphism for $\nu$ in the coroot lattice of $G$).
\index{$\wt{\ov{\Gr}}{}^{\nu}_G$} 
 
% Any trivialization of the fibre of $\Gra_G\to \Gr_G$ at $t^{\eta}G(\cO_x)$ yields a section $s^{\eta}_B: \Gr_B^{\eta}\to \wt\Gr_B^{\eta}$ of the gerb $\wt\Gr_B^{\eta}\to \Gr_B^{\eta}$. 
 
 The metaplectic Casselman-Shalika problem is the following. Given $\gamma\in\Lambda^{\sharp, +}$ and $\mu,\nu\in\Lambda$ with $\mu+\nu\in\Lambda^+$, calculate
\begin{equation}
\label{M_C-Sh_formula}
CSh^{\gamma, \mu}_{\mu+\nu}:=
\RG_c(\Gr_B^{\nu}\cap \ov{\Gr}_G^{\gamma}, (\chi^{\nu}_{\mu})^*\cL_{\psi}\otimes (s^{\nu}_B)^*\cA^{\gamma}_{\cE})[\<\nu, 2\check{\rho}\>]
\end{equation}
and describe the answer in terms of the corresponding quantum group. 
\index{$CSh^{\gamma, \mu}_{\mu+\nu}$}

 Pick $x\in X$. Let $j_{x,\mu}: \wt\gM_{x,\mu}\hook{}\wt\gM_x$ be the inclusion of this stratum.  As in (\cite{FGV}, Section~8.2.4) for $\mu+\nu\in\Lambda^+$, $\gamma\in\Lambda^{\sharp, +}$ we can calculate the complex $j_{x,\mu}^*(\cF_{x,\mu+\nu, !}\ast V(\gamma)^*)$ over $\wt\gM_{x,\mu}$. It vanishes unless $\mu\in\Lambda^+$, and in the latter case we get
$$
j_{x,\mu}^*(\cF_{x,\mu+\nu, !}\ast V(\gamma)^*)\,\iso\, 
\cF_{x,\mu,!}\otimes CSh^{\gamma, \mu}_{\mu+\nu}
$$

\smallskip\noindent
By Theorem~\ref{Th_Hecke_action}, (\ref{M_C-Sh_formula}) is placed in degrees $\le 0$. Note that 
$$
CSh^{\gamma, \mu}_{\mu+\nu}\,\iso\, \DD\R\Hom(\cF_{x,\mu+\nu, !}, \, \cF_{x, \mu, *}\ast V(\gamma))
$$ 

The complexes (\ref{M_C-Sh_formula}) describe the action of the Hecke functors on the objects $\cF_{x,\eta, !}$ for $\eta\in\Lambda^+$. 

\sssec{}  For $\gamma\in\Lambda^{\sharp, +}$, $\mu+\nu\in\Lambda^+, \mu\in\Lambda^+$ set
\begin{multline*}
\cB^{\gamma,\mu}_{\mu+\nu}:=
\{a\in \Irr(\Gr_B^{\nu}\cap \Gr_{B^-}^{w_0(\gamma)})\mid a\subset \ov{\Gr}_G^{\gamma}, t^{\mu}a\subset \ov{\Gr}_G^{\mu+\nu}\}\,\iso\,\\
\{b\in \Irr(\Gr_B^{\mu+\nu}\cap \Gr_{B^-}^{\mu+w_0(\gamma)})\mid t^{-\mu}b\subset \ov{\Gr}^{\gamma}_G, b\subset \ov{\Gr}^{\mu+\nu}_G\},
\end{multline*}
the latter map sends $a$ to $b=t^{\mu}a$.     
\index{$\cB^{\gamma,\mu}_{\mu+\nu}$}

\begin{Pp} 
\label{Pp_H^0_of_C-Sh-formula}
The space $\H^0$ of the complex (\ref{M_C-Sh_formula}) admits a canonical base $\und{\cB}^{\gamma,\mu}_{\mu+\nu}$. 
\end{Pp}
\index{$\und{\cB}^{\gamma,\mu}_{\mu+\nu}$}
\begin{proof} Denote for brevity by $K$ the complex (\ref{M_C-Sh_formula}).
For $\tau\le\gamma$, $\tau\in\Lambda^+$ let $K_{\tau}$ be the contribution of the stratum $\Gr_B^{\nu}\cap \Gr_G^{\tau}$ in $K$. First consider $\cA^{\gamma}_{\cE}\mid_{\wt\Gr_G^{\tau}}$, which has constant cohomology sheaves because of $G(\cO)$-equivariance, it is placed in perverse degrees $\le 0$, and the inequality is strict unless $\tau=\gamma$. So, the $*$-restriction of $\cA^{\gamma}_{\cE}\mid_{\wt\Gr_G^{\tau}}$ to $\Gr_B^{\nu}\cap \Gr_G^{\tau}$ is placed in perverse degrees 
$$
\le -\codim(\Gr_B^{\nu}\cap \Gr_G^{\tau}, \Gr_G^{\tau})=\<\nu-\tau, \check{\rho}\>,
$$
(and the inequality is strict unless $\tau=\gamma$). Since $\dim \Gr_B^{\nu}\cap \Gr_G^{\tau}=\<\nu+\tau, \check{\rho}\>$, we conclude that $K_{\tau}$ is placed in degrees $\le 0$, and the inequality is strict unless $\tau=\gamma$. So, only $K_{\gamma}$ contributes to 
$$
\H^0(K)\,\iso\, \H^{top}_c(\Gr_B^{\nu}\cap \Gr_G^{\gamma}, (\chi^{\nu}_{\mu})^*\cL_{\psi}\otimes (s^{\nu}_B)^*\cA^{\gamma}_{\cE})[\<\nu, 2\check{\rho}\>]
$$ 
So, $\H^0(K)$ has a base consisting of those $b\in\Irr(\Gr_B^{\nu}\cap \Gr_G^{\gamma})$ over which the shifted local system $(\chi^{\nu}_{\mu})^*\cL_{\psi}\otimes (s^{\nu}_B)^*\cA^{\gamma}_{\cE}$ is constant. 
By formula (\ref{bij_Andersen}), we have a bijection
$$
\{a\in\Irr(\Gr_B^{\nu}\cap \Gr_{B^-}^{w_0(\gamma)})\mid a\subset \ov{\Gr}_G^{\gamma}\} \,\iso\, \Irr(\Gr_B^{\nu}\cap \Gr_G^{\gamma})
$$
sending $a$ to the closure of $a\cap \Gr_G^{\gamma}$. Our claim is now reduced to Lemma~\ref{Lm_nice_first_for_C-Sh-formula} below.
\end{proof}

\begin{Lm} 
\label{Lm_nice_first_for_C-Sh-formula}
Let $\gamma\in\Lambda^{\sharp, +}$, $\mu+\nu\in\Lambda^+, \mu\in\Lambda^+$. Let $b\in\Irr(\Gr_B^{\nu}\cap \Gr_G^{\gamma})$. The local system $(\chi^{\nu}_{\mu})^*\cL_{\psi}\otimes (s^{\nu}_B)^*\cA^{\gamma}_{\cE}$ is constant on $b$ if and only if two conditions hold: 
\begin{itemize}
\item for $a=b\cap \Gr_{B^-}^{w_0(\gamma)}$ the component $\bar a\in B_{\gg}$ is special;
\item $t^{\mu}a\subset \ov{\Gr}_G^{\mu+\nu}$.
\end{itemize}
\end{Lm}

\begin{proof} Let $a\in \Irr(\Gr_B^{\nu}\cap \Gr_{B^-}^{w_0(\gamma)})$ with $a\subset \ov{\Gr}_G^{\gamma}$. The local system $(\chi^{\nu}_{\mu})^*\cL_{\psi}\otimes (s^{\nu}_B)^*\cA^{\gamma}_{\cE}$ is constant on $a$ if and only if $\chi^{\nu}_{\mu}: a\to\A^1$ is not dominant, and $(s^{\nu}_B)^*\cA^{\gamma}_{\cE}$ is constant on $a$. 

 By Lemma~\ref{Lm_comparing_s_B_and_gamma}, over $\Gr_B^{\nu}\cap \Gr_{B^-}^{w_0(\gamma)}\cap \Gr_G^{\gamma}$ one gets $(s^{\nu}_B)^*\cA^{\gamma}_{\cE}\,\iso\, (\gamma_{\nu}^{w_0(\gamma)})^*\cL_{\zeta}^{-1}$ up to a shift. The notation $\gamma_{\nu}^{w_0(\gamma)}$ is that of Section~\ref{section_4.12.4}. Now, $(\gamma_{\nu}^{w_0(\gamma)})^*\cL_{\zeta}$ is constant on $a$ if and only if $\bar a\in B_{\gg}(\nu-w_0(\gamma))$ is special.

By Lemma~\ref{Lm_2.1.6}, the map $\chi^{\nu}_{\mu}: a\to\A^1$ is not dominant iff $\phi_i(\bar a)\le \<\mu+\nu, \check{\alpha}_i\>$ for all $i\in \cJ$. By Lemma~\ref{Lm_Andersen_again}, the latter property is equivalent to $t^{\mu}a\subset \ov{\Gr}_G^{\mu+\nu}$.
\end{proof}

\sssec{Non-twisted case} 
\label{C-Sh-Non-twisted case}
For simplicity, we do not distinguish between $\VV$ and $(^{\iota}\VV)^*$ for an irreducible $\check{G}$-representation $\VV$, where $\iota$ is the Chevalley automorphism of $\check{G}$. Recall our notation $\VV^{\gamma}$ for the irreducible $\check{G}$-module with highest weight $\gamma\in\Lambda^+$. For $\nu\in\Lambda$ write $\VV^{\gamma}_{\nu}$ for the $\nu$-weight space of $\VV^{\gamma}$. Set 
$$
\cB^{\gamma}_{\nu}=\Irr(\Gr_B^{\nu}\cap \Gr_G^{\gamma}),\;\;\;\;\; {^-\cB^{\gamma}_{\nu}}=\Irr(\Gr_{B^-}^{\nu}\cap \Gr_G^{\gamma})\, .
$$
Recall that $\VV^{\gamma}_{\nu}$ has two canonical bases
$\cB^{\gamma}_{\nu}$ and $^-\cB^{\gamma}_{\nu}$. Combining Lemma~\ref{Lm_nice_first_for_C-Sh-formula} for the trivial metaplectic parameters with (\cite{FGV}, Theorem~1), we see that
$\Hom(\VV^{\gamma}\otimes\VV^{\mu}, \VV^{\mu+\nu})$ admits a canonical base $\cB^{\gamma,\mu}_{\mu+\nu}$. 
\index{$\iota, \, ^{\iota}\VV, \VV^{\gamma}_{\nu}$}
\index{$\cB^{\gamma}_{\nu}, {^-\cB^{\gamma}_{\nu}}$}

Here are two special (limiting) cases. 

\smallskip\noindent
{\bf CASE i)} The canonical inclusion $\cB^{\gamma,\mu}_{\mu+\nu}\hook{} {^-\cB^{\mu+\nu}_{\mu+w_0(\gamma)}}$, $b\mapsto b\cap \Gr_G^{\mu+\nu}$ induces 
\begin{equation}
\label{great_inclusion_one_for_C-Sh}
\Hom(\VV^{\gamma}\otimes\VV^{\mu}, \VV^{\mu+\nu})\hook{} \VV^{\mu+\nu}_{\mu+w_0(\gamma)}
\end{equation}
If for any weight $\tau$ of $\VV^{\mu+\nu}$ one has $-w_0(\gamma)+\tau\in\Lambda^+$ then 
$$
\VV^{\mu+\nu}\otimes (\VV^{\gamma})^*\,\iso\, \oplus_{\tau} \VV^{\tau-w_0(\gamma)}\otimes \VV^{\mu+\nu}_{\tau},
$$ 
so (\ref{great_inclusion_one_for_C-Sh}) is an isomorphism, and $\cB^{\gamma,\mu}_{\mu+\nu}={^-\cB^{\mu+\nu}_{\mu+w_0(\gamma)}}$.

\medskip\noindent
{\bf CASE ii)} The canonical inclusion $\cB^{\gamma,\mu}_{\mu+\nu}\hook{} \cB^{\gamma}_{\nu}$, $a\mapsto a\cap \Gr_G^{\gamma}$ induces 
\begin{equation}
\label{great_inclusion_two_for_C-Sh}
\Hom(\VV^{\gamma}\otimes\VV^{\mu}, \VV^{\mu+\nu})\hook{} \VV^{\gamma}_{\nu}
\end{equation}
If for any weight $\tau$ of $\VV^{\gamma}$ one has $\tau+\mu\in\Lambda^+$ then 
$
\VV^{\gamma}\otimes\VV^{\mu}\,\iso\,\oplus_{\tau} \VV^{\mu+\tau}\otimes \VV^{\gamma}_{\tau}, 
$
so (\ref{great_inclusion_two_for_C-Sh}) is an isomorphism, and $\cB^{\gamma,\mu}_{\mu+\nu}=\cB^{\gamma}_{\nu}$.

\sssec{} Assume that $\varrho$ satisfies the subtop cohomology property. Recall that for $\lambda\in\Lambda^+, \mu\in\Lambda$ with $\mu\le\lambda$ we have the vector spaces $V^{\lambda}_{\mu}$ introduced in Theorem~\ref{Th_4.12.1}. It provides a decomposition
$$
\ov{\FF}(\cF_{x,\lambda})\,\iso\, \mathop{\oplus}\limits_{\mu\le\lambda, \, \lambda-\mu\in\Lambda^{\sharp}} \; \cL_{x, \mu}\otimes V^{\lambda}_{\mu},
$$
and realizes $^-\und{\cB}^{\lambda}_{\mu}$ as a canonical base of  $V^{\lambda}_{\mu}$. 
\index{$^-\und{\cB}^{\lambda}_{\mu}$}

\begin{Cor} Let $\gamma\in\Lambda^{\sharp, +}$, $\mu+\nu\in\Lambda^+, \mu\in\Lambda^+$ and $K$ denote the complex (\ref{M_C-Sh_formula}).

\noindent
i) There is a canonical inclusion $\und{\cB}^{\gamma, \mu}_{\mu+\nu}\hook{}{^-{\und{\cB}}^{\mu+\nu}_{\mu+w_0(\gamma)}}$, hence also 
$$
\H^0(K)\hook{} V^{\mu+\nu}_{\mu+w_0(\gamma)}
$$ 
If for any weight $\tau$ of $\VV^{\mu+\nu}$ one has $-w_0(\gamma)+\tau\in\Lambda^+$ then the above inclusions are a bijection and an isomorphism respectively.

\noindent
ii) There is a canonical inclusion $\und{\cB}^{\gamma, \mu}_{\mu+\nu}\hook{}\und{\cB}^{\gamma}_{\nu}$, hence also 
$$
\H^0(K)\hook{} V(\gamma)_{\nu}=V^{\gamma}_{\nu}
$$ 
If for any weight $\tau$ of $\VV^{\gamma}$ the coweight $\tau+\mu$ is dominant then the above inclusions are a bijection and an isomorphism respectively.
\end{Cor}
\begin{proof}
By (\cite{L1}, Lemma~3.2), $\und{\cB}^{\gamma}_{\nu}$ is canonically a base of $V(\gamma)_{\nu}$. By Theorem~\ref{Thm_special_case_of_V_lambda_mu}, $V(\gamma)_{\nu}\,\iso\, V^{\gamma}_{\nu}$ canonically. Both inclusions are those of Section~\ref{C-Sh-Non-twisted case} restricted to special elements. Both claims follow from Section~\ref{C-Sh-Non-twisted case} and Proposition~\ref{Pp_H^0_of_C-Sh-formula}. 
\end{proof}

\sssec{Metaplectic Casselman-Shalika formula} According to Gaitsgory's conjecture (\cite{G1}, Conjecture~0.4), to our metaplectic data one may associate the category $\cC$ of finite-dimensional representations of the corresponding\footnote{This will not be a quantum group in the usual sense in general, but rather a Hopf algebra in a suitable braided monoidal category.} big quantum group $U_q(\check{G})$ over $\Qlb$, and $\Whit^{\kappa}_x\,\iso\, \cC$ naturally.  

 For $\lambda\in\Lambda^+$ denote by $W^{\lambda, !}$, $W^{\lambda, *}$  the corresponding standard  and costandard objects of $\cC$,  $W^{\lambda, !}$ should be thought of as a Verma module in $\cC$. Write $\D(\cC)$ for the derived category of $\cC$. Then $\cC$ is equipped with a fully faithful functor $\Fr: \Rep(\check{G}_{\zeta})\to \cC$. So, $\check{G}_{\zeta}$ should be thought of as the quantum Frobenius quotient of $U_q(\check{G})$. The notation $\Fr$ is taken from  \cite{ABBGM}.
 
\begin{Con} Given $\gamma\in\Lambda^{\sharp, +}$ and $\mu,\nu\in\Lambda$ with $\mu+\nu\in \Lambda^+$ the complex (\ref{M_C-Sh_formula}) vanishes unless $\mu\in\Lambda^+$. In the latter case
there is an isomorphism
$$
\RG_c(\Gr_B^{\nu}\cap \ov{\Gr}^{\gamma}_G, (\chi^{\nu}_{\mu})^*\cL_{\psi}\otimes (s^{\nu}_B)^*\cA^{\gamma}_{\cE})[\<\nu, 2\check{\rho}\>]\,\iso\, 
\DD\R\Hom_{\D(\cC)}(W^{\mu+\nu, !}, W^{\mu, *}\otimes \Fr(V(\gamma)))
$$
\end{Con}
\begin{Rem} If the metaplectic data is trivial then $\cC$ is semisimple. In this case $\Lambda^{\sharp}=\Lambda$ and $\check{G}_{\zeta}=\check{G}$. In this case $\Fr: \Rep(\check{G}_{\zeta})\to \cC$ is an equivalence. The right hand side becomes $\Hom(\VV^{\mu}\otimes \VV^{\gamma}, \VV^{\mu+\nu})$, and the above conjecture becomes (\cite{FGV}, Theorem~1), so it is true.
\end{Rem}

\section{Comparison with \cite{ABBGM}}
\label{Section_comparison_ABBGM}

\ssec{} In this section we don't need $\ell$-adic numbers, and we reserve the notation $\ell$ for another integer introduced below and referring to \cite{ABBGM}.

Let $Q\subset\Lambda$ be the coroot lattice. Let $(\cdot, \cdot): Q\otimes Q\to \ZZ$ be the canonical $W$-invariant symmetric bilinear even form. So, $(\alpha_i,\alpha_i)=2d_i$, where $d_i\in\{1,2,3\}$ is the minimal set of integers such that the matrix $d_i\<\alpha_i, \check{\alpha}_j\>$ is symmetric. 
\index{$Q, (\cdot, \cdot), d_i$}

 To fall into the setting in \cite{ABBGM} we must assume that the restriction of $\bar\kappa$ to $Q$ is a multiple of the canonical form. 
Assume there is $m\in\ZZ$ such that $\bar\kappa(a,b)=m(a, b)$ for any $a,b\in Q$. Assume also that $N=m\ell$, where $\ell$ is an integer. Assume also that $d_i$ divides $\ell$ for any $i$ and set $\delta_i=\frac{\ell}{d_i}$. Then $\bar\kappa(\alpha_i,\alpha_i)=2md_i$. So, $\frac{\bar\kappa(\alpha_i,\alpha_i)}{2N}=\frac{1}{\delta_i}$, and $\delta_i$ is the denominator of $\frac{\bar\kappa(\alpha_i,\alpha_i)}{2N}$. 
 
 Assume also given a symmetric $W$-invariant form $(\cdot, \cdot)_{\ell}: \check{\Lambda}\otimes\check{\Lambda}\to\ZZ$ such that for any $\check{\lambda}\in\check{\Lambda}, i\in \cJ$,
$$
(\check{\alpha}_i, \check{\lambda})_{\ell}=\delta_i\<\alpha_i, \check{\lambda}\>
$$
Let $\phi_{\ell}: \check{\Lambda}\to \Lambda$ be the map induced by $(\cdot, \cdot)_{\ell}$. So, $\phi_{\ell}(\check{\alpha}_i)=\delta_i\alpha_i$. Since $\bar\kappa(\alpha_i)=m d_i\check{\alpha}_i$, we get $\bar\kappa\phi_{\ell}(\check{\alpha}_i)=N\check{\alpha}_i$. 
\index{$(\cdot, \cdot)_{\ell}, \phi_{\ell}$}

 Assume in addition that the composition $\bar\kappa\comp\phi_{\ell}$ equals the multiplication by $N$. Then $\phi_{\ell}: \check{\Lambda}\hook{} \Lambda$ and $\bar\kappa: \Lambda\hook{}\check{\Lambda}$ are inclusions with finite cokernels. 
 
\begin{Lm} 1) The map $\phi_{\ell}$ takes values in $\Lambda^{\sharp}$, and $\phi_{\ell}: \check{\Lambda}\,\iso\, \Lambda^{\sharp}$ is an isomorphism.\\
2) The map $\phi_{\ell}$ gives an isomorphism of root data of $(G,T)$ with that of $(\check{G}_{\zeta}, \check{T}_{\zeta})$. So, it induces an isomorphism $G\,\iso\, \check{G}_{\zeta}$ identifying $T$ with $\check{T}_{\zeta}$. 
\end{Lm}
\begin{proof}
1) If $\check{\lambda}\in\check{\Lambda}$ then $\bar\kappa\phi_{\ell}(\check{\lambda})\in N\check{\lambda}$, so $\phi_{\ell}: \check{\lambda}\to \Lambda^{\sharp}$. Let now $\lambda\in\Lambda$ with $\bar\kappa(\lambda)\in N\check{\Lambda}$. Then $\bar\kappa \phi_{\ell}(\frac{\bar\kappa(\lambda)}{N})=\bar\kappa(\lambda)$. So, $\lambda= \phi_{\ell}(\frac{\bar\kappa(\lambda)}{N})$, because $\bar\kappa$ is injective. So, $\phi_{\ell}$ is surjective.

\smallskip\noindent
2) The roots of $(\check{G}_{\zeta}, \check{T}_{\zeta})$ are $\delta_i\alpha_i$. One has $\phi_{\ell}(\check{\alpha}_i)=\delta_i\alpha_i$, and the dual map $(\phi_{\ell})^{\check{}}: (\Lambda^{\sharp})^{\check{}}\;\iso\; \Lambda$ sends $\frac{\check{\alpha}_i}{\delta_i}$ to $\alpha_i$.
\end{proof}

\begin{Rem} In \cite{ABBGM} one moreover assumes $\ell$ even. We did not need this assumption. 
\end{Rem}

%%%%%%%%%
\appendix

\section{Property (C)}

\subsection{} In some cases we use the following observation. Let $i\in\cJ$, $\lambda>\alpha_i$ such that $\omega_i-\lambda$ appears as a weight of $\VV^{\omega_i}$. Then there is $\mu\in\Lambda^+$ with $\mu\le \omega_i$, $w\in W$ such that $\lambda=\omega_i-w\mu$. Then the property $\bar\kappa(\omega_i-w\mu-\alpha_i)\in N\check{\Lambda}$ is equivalent to $\bar\kappa(w^{-1}s_i(\omega_i)-\mu)\in N\check{\Lambda}$, where $s_i$ is the reflection corresponding to $\alpha_i$. So, one may first find the $W$-orbit of each $\omega_i$. Second, find for each $i$ all the dominant coweights satisfying $\mu\le \omega_i$. Third, check for each $i\in\cJ$, $\mu\le\omega_i$ dominant with $\mu\ne\nu\in W\omega_i$ the property $\bar\kappa(\nu-\mu)\notin N\check{\Lambda}$.

\medskip\noindent
{\bf Type $A_{n-1}$.} We may assume $G=\GL_n$, $B\subset G$ is the group of upper triangular matrices, $T$ is the group of diagonal matrices. So, $\Lambda=\ZZ^n$. We may assume $\bar\kappa: \Lambda\otimes\Lambda\to\ZZ$ given by $\bar\kappa=m\kappa$, where $m\in\ZZ$ and $\kappa(a,b)=\sum_{i=1}^n a_ib_i$. Then our assumption is $m\notin N\ZZ$. Since $\lambda$ is not a simple coroot, we have $n\ge 3$. We assume $\cJ=\{1,\ldots, n-1\}$ and $\omega_i=(1,\ldots, 1,0,\ldots,0)$, where $1$ appears $i$ times.
The representation $\VV^{\omega_i}$ is minuscule, for any $\mu\le\omega_i$ with $\mu\in\Lambda^+$ we have $\mu=\omega_i$. Any $\nu\in W\omega_i$ is of the form $\nu=e_{j_1}+\ldots+e_{j_i}$ for $1\le j_1<\ldots < j_i\le n$. Let $1\le k\le n$ be the smallest such that $\alpha_k=e_k-e_{k+1}$ appears in the decomposition of $\omega_i-\nu\ne 0$  into a sum of simple coroots. Then $k\le i$ and $m=\bar\kappa(\lambda, e_k)\notin N\ZZ$. We are done.

\medskip\noindent
{\bf Type $C_n$.} We may assume $G=\GSp_{2n}$, the quotient of $\Gm\times\Sp_{2n}$ by the diagonally embedded $\mu_2$. Realize $G\subset \GL_{2n}$ as the subgroup preserving up to scalar the bilinear form given by the matrix
$$
\left(
\begin{array}{cc}
0 & E_n\\
-E_n & 0
\end{array}
\right),
$$
where $E_n$ is the unit matrix of $\GL_n$. The maximal torus $T$ of $G$ is $\{(y_1,\ldots,y_{2n})\mid y_iy_{n+i} \mbox{ does not depend on  } i\}$. Let $\check{\epsilon}_i\in\check{\Lambda}$ 
be the caracter that sends a point of $T$ to $y_i$. 
The roots are
$$
\check{R}=\{\pm\check{\alpha}_{ij} \; (i< j \in 1,\ldots,n),\;
\pm\check{\beta}_{ij}\; (i\le j \in 1,\ldots,n)\},
$$
where
$\check{\alpha}_{ij}=\check{\epsilon}_i-\check{\epsilon}_j$ and 
$\check{\beta}_{ij}=\check{\epsilon}_i-\check{\epsilon}_{n+j}$. 

 We have $\Lambda=\{(a_1,\ldots,a_{2n})\mid a_i+a_{n+i} \mbox{ does not depend on }
i\}$. The weight latice is 
$$
\check{\Lambda}=\ZZ^{2n}/\{\check{\epsilon}_i+
\check{\epsilon}_{n+i}-\check{\epsilon}_j-\check{\epsilon}_{n+j}, \; i<j\} 
$$
Let $e_i$ denote the standard basis of $\ZZ^{2n}$. 
The coroots are 
$$
R=\{\pm\alpha_{ij} \; (i< j \in 1,\ldots,n),\;
\pm\beta_{ij}\; (i\le j \in 1,\ldots,n)\},
$$
where $\beta_{ij}=e_i+e_j-e_{n+i}-e_{n+j}$ for
$i<j$ and $\beta_{ii}=e_i-e_{n+i}$. Besides, $\alpha_{ij}=e_i+e_{n+j}-e_j-e_{n+i}$.   

 Fix positive roots 
$$
\check{R}^+=\{\check{\alpha}_{ij}
\;\; (i<j\in 1,\ldots,n), \; \check{\beta}_{ij} \;\; (i\le j\in 1,\ldots,n)\}
$$
Then the simple roots are $\check{\alpha}_1:=\check{\alpha}_{12},\ldots,\check{\alpha}_{n-1}:=\check{\alpha}_{n-1,n}$ and
$\check{\alpha}_n:=\check{\beta}_{n,n}$. 

 For $1\le i<n$ set $\omega_i=(1,\ldots, 1,0,\ldots, 0; -1,\ldots, -1, 0\ldots, 0)$, where $1$ appears $i$ times then $0$ appears $n-i$ times then $-1$ appears $i$ times, and $0$ appears $n-i$ times. Set $\omega_n=(1,\ldots, 1;0,\ldots,0)$, where $1$ appears $n$ times, and $0$ appears $n$ times. This is our choice of the fundamental coweights corresponding to $\check{\alpha}_i$. 
 
  For $b\in \Lambda$ write $\bar b=b_i+b_{n+i}$, this is independent of $i$. The map $\Lambda_{ab}\,\iso\,\ZZ$, $a\mapsto \bar a$ is an isomorphism. 
Let $\kappa:\Lambda\otimes\Lambda\to\ZZ$ be given by $\kappa(a,b)=\sum_{i=1}^{2n} a_i b_i$. Then $\kappa$ is $W$-invariant symmetric bilinear form. We have $\kappa(\alpha_{ij}, \alpha_{ij})=\kappa(\beta_{ij}, \beta_{ij})=4$ for $i\ne j$, and $\kappa(\beta_{ii}, \beta_{ii})=2$. We may assume $\bar\kappa=m\kappa$ for some $m\in\ZZ$. 

 Note that $\VV^{\omega_n}$ is the spinor representation of $\check{G}\,\iso\, \GSpin_{2n+1}$ of dimension $2^n$, $\VV^{\omega_1}$ is the standard representation of the quotient $\SO_{2n+1}$, and $\VV^{\omega_i}=\wedge^i (\VV^{\omega_1})$ for $1\le i<n$. We have
$0\le \omega_1\le\ldots\le\omega_{n-1}$, and if $\mu\in\Lambda$ is dominant and $\mu\le\omega_{n-1}$ then $\mu$ is in this list.  

 The assumption $\varrho(\alpha_i)\notin \ZZ$ for any simple coroots reads $2m\notin N\ZZ$. Assume $n=2$. In this case it is easy to check the desired property (C). 

 Assume now $n\ge 3$. Then the assumption $\varrho(\alpha_i)\notin \frac{1}{2}\ZZ$ for any simple coroots reads $4m\notin N\ZZ$. 
 
 First, let $1\le i<n$. Suppose $\omega_i-\lambda$ appears in $\VV^{\omega_i}$. Then $\omega_i-\lambda$ is of the form $\sum_{k=1}^j \epsilon_k \beta_{i_k, i_k}$, where $\epsilon_k=\pm 1$, $0\le j\le i$, and $1\le i_1<\ldots < i_j\le n$. Let $\lambda-\alpha_i=(a_1,\ldots, a_{2n})$. If $j<i$ then there is $1\le k\le n$ such that $a_k=1$, and $\kappa(\lambda-\alpha_i, \beta_{k,k})=2$. If $j=i$ and there is no $1\le k\le n$ with this property then there is $1\le k\le n$ such that $a_k=2$, and $\kappa(\lambda-\alpha_i, \beta_{k,k})=4$. The case $i<n$ is done.

 Let now $i=n$. The representation $\VV^{\omega_n}$ is minuscule, its weights are the $W$-orbit of $\omega_n$. The coweight $\lambda$ is of the form $\lambda=\sum_{k\in S}\beta_{k,k}$, where $S\subset \{1,\ldots, n\}$ is a subset, and $\lambda>\alpha_n=\beta_{n,n}$. So, there is $k\in S$ with $k<n$. We have $\kappa(\lambda-\alpha_n, \beta_{k,k})=2$. We are done.
 
\medskip\noindent
{\bf Type $B_n$.}  Assume $n\ge 3$, let $G=\Spin_{2n+1}$. We take $\Lambda=\{(a_1,\ldots, a_n)\in \ZZ^n\mid \sum_k a_k=0\!\mod\! 2\}$, so $\ZZ^n\subset \check{\Lambda}$. The coroots are
$$
R=\{\pm \alpha_{ij} (1\le i<j\le n), \pm \beta_{ij} (1\le i\le j\le n)\},
$$
where $\alpha_{ij}=e_i-e_j$, $\beta_{ij}=e_i+e_j$. The corresponding roots are $\check{\alpha}_{ij}=e_i-e_j$, $\check{\beta}_{ij}=e_i+e_j$ for $1\le i<j\le n$, and $\check{\beta}_{ii}=e_i$. Here $\check{\alpha}_{ij}, \check{\beta}_{ij}\in \ZZ^n\subset \check{\Lambda}$. The simple roots are
$\check{\alpha}_1=\check{\alpha}_{12}, \ldots, \check{\alpha}_{n-1}=\check{\alpha}_{n-1,n}$, $\check{\alpha}_n=\check{\beta}_{n,n}$. 

 Write $\check{G}^{sc}$ for the simply-connected cover of $\check{G}$.
 The fundamental weights of $\check{G}^{sc}$, which we refer to as the fundamental coweights of $G_{ad}$, are $\omega_i=e_1+\ldots+e_i\in \ZZ^n$ for $1\le i\le n$. We use here the canonical inclusion $\Lambda\subset \ZZ^n=\Lambda_{ad}$ as a sublattice of index 2. Here $\Lambda_{ad}$ is the coweights lattice of $G_{ad}=\SO_{2n+1}$. The Weyl group acts on $\Lambda_{ad}$ by any permutations and any sign changes. That is, it contains the maps $\Lambda_{ad}\to\Lambda_{ad}$, $\mu=(a_1,\dots, a_n)\mapsto (\epsilon_1a_1,\ldots, \epsilon_na_n)$ for any $\epsilon_k=\pm 1$. 
 
  Let $\kappa: \Lambda\otimes\Lambda\to\ZZ$ be the unique $W$-invariant symmetric bilinear form such that $\kappa(\alpha, \alpha)=2$ for a short coroot. Then $\kappa$ extends uniquely to $\kappa: \Lambda_{ad}\otimes\Lambda_{ad}\to\ZZ$ as $\kappa(a,b)=\sum_{k=1}^n a_k b_k$. We get $\kappa(\beta_{ii}, \beta_{ii})=4$ for any $1\le i\le n$, and all the other coroots are short. We may assume $\bar\kappa=m\kappa$, $m\in \ZZ$. Then the assumption of Conjecture~\ref{Con_main} reads $2m\notin N\ZZ$. 
  
  Let $\Lambda^+_{ad}$ be the dominant coweigts of $G_{ad}$ then $\Lambda^+_{ad}=\{(a_1,\ldots, a_n)\in \ZZ^n\mid a_1\ge\ldots\ge a_n\ge 0\}$. 
If $\mu\in \Lambda^+_{ad}$ and $\mu\le\omega_i$ then $\mu=(1,\ldots, 1,0,\ldots, 0)$, where $1$ appears $k$ times with $k\le i$ and $k=i\,\mod\, 2$. Any weight of $\VV^{\omega_i}$ is of the form $w\mu$, $w\in W$, where $\mu\in \Lambda^+_{ad}$ and $\mu\le\omega_i$. So, the weights of $\VV^{\omega_i}$ are of the form $\omega_i-\lambda=\sum_{r=1}^k \epsilon_r e_{j_r}$, where $0\le k\le i$, $k=i\mod 2$, and $1\le j_1<\ldots<j_k\le n$, here $\epsilon_r=\pm 1$. 

 If $1\le i<n$ then $\omega_i-\alpha_i=(1,\ldots,1,0,1,0,\ldots,0)$, where $1$ appears first $i-1$ times. If $k<i$ then $\lambda-\alpha_i$ contains an entry $1$ on some $m$-th place and $\kappa(\lambda-\alpha_i, \beta_{m,m})=2$, so $\bar\kappa(\lambda-\alpha_i)$ is not divisible by $N$ in this case. If $k=i$ and $\lambda-\alpha_i$ does not contain the entry 1 then $\lambda-\alpha_i$ is of the form $\sum_{j\in S} \beta_{jj}$ for some subset $S\subset \{1,\ldots, n\}$ that contains at most $i$ elements. Since $i<n$ there is a couple $j_1\in S, j_2\notin S$. Then $\kappa(\lambda-\alpha_i, \beta_{j_1, j_2})=2$, so $\bar\kappa(\lambda-\alpha_i)$ is not divisible by $N$ in this case.
 
 Let $i=n$ then $\omega_n-\alpha_n=(1,\ldots, 1, -1)$. Let $\omega_i-\lambda$ be as above. If $k<n$ then $k\le n-2$, and $\lambda-\alpha_n$  contains an entry $1$ at some place. As above this implies that $\bar\kappa(\lambda-\alpha_i)$ is not divisible by $N$ in this case. If $k=n$ then $\lambda-\alpha_n=\sum_{j\in S} \beta_{jj}+ ae_n$, where $S\subset \{1,\ldots, n-1\}$ is a subset, and $a=0$ or $a=-2$. If $\lambda-\alpha_n$ contains a entry $0$ then as above one shows that $\bar\kappa(\lambda-\alpha_i)$ is not divisible by $N$. The only remaning case is $\lambda-\alpha_n=(2,\ldots, 2, -2)=-\beta_{nn}+\sum_{j=1}^{n-1}\beta_{jj}$. 
 
  Recall that for any coroot $\alpha$ one has $\kappa(\alpha)=\frac{\kappa(\alpha, \alpha)}{2}\check{\alpha}$. We get $\kappa(\beta_{jj})=2\check{\beta}_{jj}$ for any $j$. So, $\kappa(\lambda-\alpha_n)=-2\check{\beta}_{nn}+2\sum_{j=1}^{n-1}\check{\beta}_{jj}$. The root lattice of $G$ is $\ZZ^n\subset \check{\Lambda}$, and $-\check{\beta}_{nn}+\sum_{j=1}^{n-1}\check{\beta}_{jj}$ is divisible in $\check{\Lambda}$, namely $\frac{1}{2}(-\check{\beta}_{nn}+\sum_{j=1}^{n-1}\check{\beta}_{jj})\in \check{\Lambda}$. So, we must require that $4m\notin N\ZZ$ to guarantee that $\bar\kappa(\lambda-\alpha_i)$ is not divisible by $N$.
We are done. 

\medskip\noindent
{\bf Type $G_2$.} Let $G$ be of type $G_2$. Let $\Lambda=\{a\in\ZZ^3\mid \sum_i a_i=0\}$ with the bilinear form $\kappa: \Lambda\otimes\Lambda\to\ZZ$ given by $\kappa(a,b)=\sum_i a_i b_i$ for $a,b\in\Lambda$. The coroots are the vectors $\mu\in\Lambda$ such that $\kappa(\mu,\mu)=2$ or $6$. The coroots are 
$$
\pm\{e_1-e_2, e_1-e_3, e_2-e_3, 2e_1-e_2-e_3, 2e_2-e_1-e_3, 2e_3-e_1-e_2\}
$$ 
The form $\kappa$ induces an inclusion $\kappa: \Lambda\hook{}\check{\Lambda}$ such that $\check{\Lambda}/\kappa(\Lambda)\,\iso\, \ZZ/3\ZZ$. The roots can be found from the property that for any coroot $\alpha$ one has $\kappa(\alpha)=\frac{\kappa(\alpha,\alpha)}{2}\check{\alpha}$. For a short coroot $\alpha$ one gets $\kappa(\alpha)=\check{\alpha}$, and for a long coroot $\alpha$ one gets $\kappa(\alpha)=3\check{\alpha}$. We get the roots
$$
\pm\{e_1-e_2, e_1-e_3, e_2-e_3, e_1, e_2, e_3\}\subset \ZZ^3/(e_1+e_2+e_3)=\check{\Lambda}
$$
The center of $G$ is trivial. Pick positive roots $\check{\alpha}_1=e_1-e_2$ and $\check{\alpha}_2=-e_1$. They correspond to simple coroots $\alpha_1=e_1-e_2$, $\alpha_2=-2e_1+e_2+e_3$. The dominant coweights are $\Lambda^+=\{a\in\Lambda\mid a_2\le a_1\le 0\}$. The fundamental coweights are $\omega_1=(0, -1,1)=2\alpha_1+\alpha_2$ and $\omega_2=(-1,-1,2)=3\alpha_1+2\alpha_2$. The positive coroots are $\{\alpha_1,\alpha_2, \alpha_2+\alpha_1, \alpha_2+2\alpha_1, \alpha_2+3\alpha_1, 3\alpha_1+2\alpha_2\}$. The representation $\VV^{\omega_2}$ is the adjoint representation of $\check{G}$, $\dim \VV^{\omega_2}=14$ and $\dim \VV^{\omega_1}=7$. We have $\omega_1\le\omega_2$. We assume $\bar\kappa=m\kappa$ for some $m\in\ZZ$. 

 The weights of $\VV^{\omega_2}$ are coroots and zero. So, for $i=2$ the coweight $\lambda$ is one of the following
\begin{multline*}
\{\alpha_1+\alpha_2, 2\alpha_1+\alpha_2, 3\alpha_1+\alpha_2, 2\alpha_1+2\alpha_2, 3\alpha_1+2\alpha_2, 4\alpha_1+2\alpha_2, 
\\ 3\alpha_1+3\alpha_2, 4\alpha_1+3\alpha_2, 5\alpha_1+3\alpha_2, 6\alpha_1+3\alpha_2, 6\alpha_1+4\alpha_2\}
\end{multline*}
Since $\kappa(\alpha_1)=\check{\alpha}_1$ and $\kappa(\check{\alpha}_2)=3\check{\alpha}_2$, we get in this case that $\kappa(\lambda-\alpha_2)$ is an element of the set
\begin{multline*}
\{\check{\alpha}_1, 2\check{\alpha}_1, 3\check{\alpha}_1, 
2\check{\alpha}_1+3\check{\alpha}_2, 3\check{\alpha}_1+3\check{\alpha}_2, 4\check{\alpha}_1+3\check{\alpha}_2, 3\check{\alpha}_1+6\check{\alpha}_2,
\\  4\check{\alpha}_1+3\check{\alpha}_2, 5\check{\alpha}_1+6\check{\alpha}_2, 6\check{\alpha}_1+6\check{\alpha}_2, 6\check{\alpha}_1+9\check{\alpha}_2\}
\end{multline*}
An element of this set may be divisible in $\check{\Lambda}$ by $2,3,6$.
So, in order to guarantee that $\frac{m}{N}\kappa(\lambda-\alpha_2)\notin \check{\Lambda}=\ZZ\check{\alpha}_1\oplus \ZZ\check{\alpha}_2$, we
must assume $6m\notin N\ZZ$. In terms of $\varrho$ this assumption reads $\varrho(\alpha_i)\notin \frac{1}{2}\ZZ$ for any simple coroot $\alpha_i$. 

 Let now $i=1$. Then $\kappa(\lambda-\alpha_1)$ is an element of the set
$$
\{3\check{\alpha}_2, \check{\alpha}_1+3\check{\alpha}_2, 2\check{\alpha}_1+3\check{\alpha}_2, 2\check{\alpha}_1+6\check{\alpha}_2, 3\check{\alpha}_1+6\check{\alpha}_2\}
$$
An element of this set may be divisible in $\check{\Lambda}$ by $2,3$. So, we must assume $2m,3m\notin N\ZZ$. Finally, it suffices to assume $6m\notin N\ZZ$. We are done. 

\medskip\noindent
{\bf Type $D_n$.} Let $G=\Spin_{2n}$ with $n\ge 4$. We take $\Lambda=\{(a_1,\ldots, a_n)\in\ZZ^n\mid \sum_j a_j=0\!\mod\! 2\}$, so $\ZZ^n\subset \check{\Lambda}$. The group $\check{\Lambda}$ is generated by $\ZZ^n$ and the element $\frac{1}{2}(1,\ldots, 1)$. The roots are 
$$
\check{R}=\{\pm\check{\alpha}_{ij}=e_i-e_j  (1\le i<j\le n), \pm\check{\beta}_{ij}=e_i+e_j (1\le i<j\le n)\}
$$
The simple roots are $\check{\alpha}_1=\check{\alpha}_{12},\ldots, \check{\alpha}_{n-1}=\check{\alpha}_{n-1,n}, \check{\alpha}_n=\check{\beta}_{n-1,n}$. The coroots are $\alpha_{ij}=e_i-e_j$, $\beta_{ij}=e_i+e_j$. The Weyl group acting on $\Lambda$ contains all the permutations, and also all the sign changes with the even number of sign changes. Let $\kappa: \Lambda\otimes\Lambda\to\ZZ$ be given by $\kappa(a,b)=\sum_{k=1}^n a_kb_k$. Then $\kappa$ is the unique $W$-invariant symmetric bilinear form such that $\kappa(\alpha,\alpha)=2$ for any coroot. Let $\bar\kappa=m\kappa$, $m\in\ZZ$. The assumption of Conjecture~\ref{Con_main} reads $m\notin N\ZZ$. 

The center of $G$ is $\ZZ/2\ZZ\times \ZZ/2\ZZ$ for $n$ even (resp., $\ZZ/4\ZZ$ for $n$ odd).
The group $\Lambda_{ad}$ is generated by $\ZZ^n$ and the vector $\frac{1}{2}(1,\ldots, 1)$. The fundamental coweights of $G_{ad}$ in $\Lambda_{ad}$ are $\omega_i=(1,\ldots, 1, 0,\ldots, 0)\in \ZZ^n$, where $1$ appears $i$ times for $1\le i\le n-2$, and 
$$
\omega_n=\frac{1}{2}(1,\ldots, 1), \;\; \omega_{n-1}=\frac{1}{2}(1,\ldots, 1, -1)
$$ 
Here $\VV^{\omega_{n-1}}$, $\VV^{\omega_n}$ are half-spin representations of $\check{G}^{sc}\,\iso\,\Spin_{2n}$. The representation $\VV^{\omega_1}$ is the standard representation of $\SO_{2n}$, and $\VV^{\omega_i}\,\iso\, \wedge^i \VV^{\omega_1}$ for $1\le i\le n-2$. Both half-spin representations are minuscule of dimension $2^{n-1}$. 

 The weights of $\VV^{\omega_n}$ (resp., of $\VV^{\omega_{n-1}}$) are $\frac{1}{2}(\epsilon_1,\ldots, \epsilon_n)$, where $\epsilon_k=\pm 1$, and the number of negative signs is even (resp., odd). 
 
 If $i=n$ then $\lambda$ is of the form $\lambda=\sum_{k\in S} e_k$, where $S\subset\{1,\ldots, n\}$, and the number of elements of $S$ is even. For $n$ odd one checks that for any such $\lambda$, $\kappa(\lambda-\alpha_n)$ is not divisible in $\check{\Lambda}$, so $\bar\kappa(\lambda-\alpha_n)\notin N\ZZ$. For $n$ even taking $\lambda=(1,\ldots, 1,0,0)$ we get $\lambda-\alpha_n=(1,\ldots, 1, -1,-1)$. For any $\mu\in\Lambda$, $\kappa(\lambda-\alpha_n, \mu)$ is even. So, we have to assume $2m\notin N\ZZ$ for $n$ even. Under this assumption one checks that $\bar\kappa(\lambda-\alpha_n)\notin N\check{\Lambda}$. 
 
 If $i=n-1$ then $\lambda-\alpha_{n-1}$ is of the form $(\epsilon_1,\ldots, \epsilon_{n-2}, 0, \epsilon_n)$, where $\epsilon_k=0$ or $1$, and the number of 1's is even; or of the form $(\epsilon_1,\ldots, \epsilon_{n-2}, -1, \epsilon_n)$, where $\epsilon_k=0$ or $1$, and the number of 1's is odd (and the element $\lambda=0$ is excluded here). In the first case $\bar\kappa(\lambda-\alpha_n)\notin N\check{\Lambda}$, and in the second case the only difficulty comes from $\lambda-\alpha_{n-1}=(1,\ldots, 1, -1, 1)$ for $n$ even. In this case our assumption $2m\notin N\ZZ$ for $n$ even guarantees that $\bar\kappa(\lambda-\alpha_n)\notin N\check{\Lambda}$. 
 
  Let now $i\le n-2$. Note that for any $a=(a_1,\ldots, a_n)\in\Lambda$, $\kappa(a)=(a_1,\ldots, a_n)\in \check{\Lambda}$. If $\mu\in\Lambda^+$ is a weight of $\VV^{\omega_i}$ then $\mu$ is of the form $(1,\ldots, 1,0,\ldots, 0)$, where $1$ appears $m\le i$ times with $i-m$ even. So, any weight of $\VV^{\omega_i}$ is of the form $\sum_{k\in S} \epsilon_k$ with $\epsilon_k=\pm 1$, where $S\subset \{1,\ldots,n\}$ is a subset of order $m\le i$ with $i-m$ even. 
We have $\omega_i-\alpha_i=(1,\ldots, 1,0,1,0,\ldots,0)$, where $1$ first appears $i-1$ times. If $\lambda-\alpha_i$ contains the entry 0 then its other entries could be only $0,1,-1, 2$. So, $\kappa(\lambda-\alpha_i)$ may be divisible at most by 2 in $\check{\Lambda}$. Since $2m\notin N\ZZ$, $\bar\kappa(\lambda-\alpha_i)\notin N\check{\Lambda}$ in this case. If $\lambda-\alpha_i$ does not contains the entry 0 and contains the entry 2 then $\kappa(\lambda-\alpha_i)$ may be divisible at most by 2. If $\lambda-\alpha_i$ does not contains the entries $0,2$ then $i=n/2$, $n$ is even and $\lambda-\alpha_i=(1,\ldots, 1, \epsilon_i, 1, \epsilon_{i+2}, \ldots, \epsilon_n)$ with $\epsilon_k=\pm 1$. Then $\kappa(\lambda-\alpha_i)$ is divisible at most by 2. We are done.

\begin{Rem} Our result for the type $D_n$ could possibly be improved by replacing $\Spin_{2n}$ with the corresponding group with connected center as in Remark~\ref{Rem_passing_to_G_semisimple}.
\end{Rem}

\medskip\noindent
{\bf Type $F_4$.} Let $I=\ZZ^4$, $e=\frac{1}{2}(e_1+e_2+e_3+e_4)\in (\frac{1}{2}\ZZ)^4$ and $\Lambda=I \cup I'$, where $I'=e+I$. So, $\Lambda\subset (\frac{1}{2}\ZZ)^4$. Let $\kappa: \Lambda\otimes\Lambda\to\ZZ$ be the symmetric bilinear form given by $\kappa(a,b)=2\sum_k a_k b_k$. Let $R$ be the set of $\mu\in \Lambda$ with $\kappa(\mu,\mu)=2$ or $4$. The coroots are
$$
R=\{\pm e_i (1\le i\le 4), \pm(e_i-e_j), \pm(e_i+e_j) (1\le i<j\le 4), \frac{1}{2}(\pm 1,\ldots, \pm 1)\}
$$
Pick $\alpha_1=\frac{1}{2}(1,-1,-1,-1)$, $\alpha_2=e_4$, $\alpha_3=e_3-e_4$, $\alpha_4=e_2-e_3$. These are simple coroots (notations from \cite{VO}), and $\Lambda$ is freely generated by $\alpha_i$. The map $\kappa: \Lambda\hook{}\check{\Lambda}$ is an inclusion. The center of $G$ is trivial. 

 We identify $\check{\Lambda}$ with a sublattice of $\QQ^4$ such that the pairing $\<,\>: \Lambda\otimes\check{\Lambda}\to\ZZ$ is the map sending $(a,b)$ to $\sum_k a_k b_k$. The fundamental weights are $\check{\omega}_1=2e_1$, $\check{\omega}_2=3e_1+e_2+e_3+e_4$, $\check{\omega}_3=2e_1+e_2+e_3$, $\check{\omega}_4=e_1+e_2$ in $\check{\Lambda}$. Then $\check{\Lambda}$ is freely generated by $\check{\omega}_i$. So, $\check{\Lambda}=\{a\in\ZZ^4\mid \sum_i a_i=0\!\mod\! 2\}$. The map $\kappa: \Lambda\to\check{\Lambda}$ sends any $a$ to $2a$. We recover the roots in $\check{\Lambda}$ from the property that $\kappa(\alpha)=\frac{\kappa(\alpha,\alpha)}{2}\check{\alpha}$ for any coroot $\alpha$. 
The roots are
$$
\check{R}=\{\pm 2e_i (1\le i\le 4), \pm(e_i-e_j), \pm(e_i+e_j) (1\le i<j\le 4),  (\pm 1,\ldots, \pm 1)\}
$$
The simple roots are $\check{\alpha}_1=(1,-1,-1,-1)$, $\check{\alpha}_2=2e_4$, $\check{\alpha}_3=e_3-e_4, \check{\alpha}_4=e_2-e_3$. 
The fundamental coweights are $\omega_1=e_1$, $\omega_2=\frac{1}{2}(3e_1+e_2+e_3+e_4)$, $\omega_3=2e_1+e_2+e_3$, $\omega_4=e_1+e_2$. The Weyl group acting on $\Lambda$ is generated by all the permutations, all the sign changes, and the element $s_1$ given by
$$
s_1(a_1,\ldots, a_4)=\frac{1}{2}(a_1+\ldots+a_4, a_1+a_2-a_3-a_4, a_1-a_2+a_3-a_4, a_1-a_2-a_3+a_4)
$$
The element $-w_0$ acts trivially on $\Lambda$. The group $W$ acts transitively on long (resp., short) coroots. We have $0\le \omega_1\le \omega_4\le\omega_2\le \omega_3$. The representation $\VV^{\omega_4}$ is the adjoint one, $\dim \VV^{\omega_2}=273, \dim \VV^{\omega_3}=1274$. The 24 positive coroots are
\begin{multline*}
R^+=\{\alpha_i (1\le i\le 4), \alpha_2+\alpha_3+\alpha_4, \alpha_2+\alpha_3, 2\alpha_1+3\alpha_2+2\alpha_3+\alpha_4,\\
2\alpha_1+2\alpha_2+\alpha_3, 2\alpha_1+2\alpha_2+\alpha_3+\alpha_4, 2\alpha_1+2\alpha_2+2\alpha_3+\alpha_4, \alpha_3+\alpha_4,\\
2\alpha_1+4\alpha_2+3\alpha_3+2\alpha_4, 2\alpha_1+4\alpha_2+3\alpha_3+\alpha_4, 
2\alpha_1+4\alpha_2+2\alpha_3+\alpha_4,\\
2\alpha_2+2\alpha_3+\alpha_4,
2\alpha_2+\alpha_3+\alpha_4,
2\alpha_2+\alpha_3,\\
\alpha_1+\alpha_2+\alpha_3+\alpha_4, 
\alpha_1+\alpha_2+\alpha_3,
\alpha_1+\alpha_2, 
\alpha_1+2\alpha_2+2\alpha_3+\alpha_4,\\
\alpha_1+3\alpha_2+2\alpha_3+\alpha_4,
\alpha_1+2\alpha_2+\alpha_3+\alpha_4,
\alpha_1+2\alpha_2+\alpha_3\}
\end{multline*} 

 Let $i=1$. The weights of $\VV^{\omega_1}$ are known from \cite{VO}, they are $\pm e_j$, $\frac{1}{2}(\pm 1,\ldots, \pm 1)$, $0$. We have $\omega_1-\alpha_1=e$. So, $\lambda-\alpha_1$ may be $\frac{1}{2}(a_1,\ldots, a_4)$, where all $a_j=1$ except one, which is $-1$ or $3$; it also may be $(a_1,\ldots, a_4)\ne 0$, where each $a_k$ is $0$ or $1$; it also maybe $e$. We see that $\kappa(\lambda-\alpha_1)$ may be divisible at most by 2. Assume $\bar\kappa=m\kappa$ with $m\in\ZZ$. The assumption of Conjecture~\ref{Con_main} says $2m\notin N\ZZ$. So, in this case $\bar\kappa(\lambda-\alpha_i)$ is not divisible by $N$. 
 
 Let $i=4$. The weights of $\VV^{\omega_4}$ are the coroots and $0$. 
We have $\omega_4=2\alpha_1+4\alpha_2+3\alpha_3+2\alpha_4$. 
If $\omega_4-\lambda$ is a weight of $\VV^{\omega_4}$ then $\lambda\le 2\omega_4$. Under our assumptions, we get $0<\lambda-\alpha_4\le 2\omega_4-\alpha_4=4\alpha_1+8\alpha_2+6\alpha_3+3\alpha_4$. 
Since $\gamma:=2\alpha_1+4\alpha_2+3\alpha_3+\alpha_4$ is a coroot, $\lambda-\alpha_4$ may take value $\omega_4+\gamma-\alpha_4=4\alpha_1+8\alpha_2+6\alpha_3+2\alpha_4$. For this $\lambda$ we see that $\kappa(\lambda-\alpha_4)=4\check{\alpha}_1+8\check{\alpha}_2+12\check{\alpha}_3+4\check{\alpha}_4$ is divisible by $4$. So, the assumption of Conjecture~\ref{Con_main} is not sufficient for our method to work in this case. We need to assume at least that $4m\notin N\ZZ$. 

 Use the method from Section~A.1. The dominant coweights $\mu\in\Lambda^+$ such that $\mu\le\omega_4$ are $\{0, \omega_1, \omega_4\}$. For $\mu=0$ we need to check that $\bar\kappa(\omega_4)\notin N\check{\Lambda}$. Since $\kappa(\omega_4)=2(e_1+e_2)$ is only divisible by $2$, and $2m\notin N\ZZ$, we see that $\bar\kappa(\omega_4)\notin N\check{\Lambda}$.
For $\mu=\omega_1$ this property is easy. The $W$-orbit through $\omega_4$ is the set of long coroots. For $\mu=\omega_4$ and a long coroot $\alpha$, $\kappa(\alpha-\mu)$ may be divisible at most by 4 in the case $\alpha=-e_1-e_2$. The assumption $4m\notin N\ZZ$ guarantees in this case that $\bar\kappa(\lambda-\alpha_i)\notin N\check{\Lambda}$.

 Let $i=2$. The dominant coweights $\mu$ such that $\mu\le \omega_2$ form the set $\{0, \omega_1,\omega_4, \omega_2\}$. The $W$-orbit through $\omega_2$ is the set  
\begin{multline*}
X_2=\{\frac{1}{2}(\pm 3, \pm 1,\pm 1,\pm 1), \frac{1}{2}(\pm 1, \pm 3,\pm 1,\pm 1), \frac{1}{2}(\pm 1, \pm 1,\pm 3,\pm 1), \frac{1}{2}(\pm 1, \pm 1,\pm 1,\pm 3), \\ (\pm 1, \pm 1, \pm 1, 0), (\pm 1, \pm 1, 0, \pm 1), (\pm 1, 0, \pm 1, \pm 1), (0, \pm 1, \pm 1, \pm 1)\},
\end{multline*}
these are all the coweights of length 6. The element $\kappa(\omega_2)$ is not divisible. For $\tau\in X_2$, $\kappa(\tau-\omega_1)$ is divisible at most by 2. For $\tau\in X_2$, $\kappa(\tau-\omega_4)$ is divisible at most by 2. For $\tau\in X_2$, $\kappa(\tau-\omega_2)$ may be divisible by 2 or 3. Namely, if $\tau=\frac{1}{2}(-3, 1,1,1)$ then $\kappa(\tau-\omega_2)=-6e_1$ is divisible in $\check{\Lambda}$ by $3$. So, we must assume $3m\notin N\ZZ$. 
 
 Let $i=3$. The set of $\mu\in\Lambda^+$ such that $\mu\le \omega_3$ is the set $\{0, \omega_1,\omega_4, \omega_2, 2\omega_1, \omega_1+\omega_4, \omega_3\}$. The $W$-orbit through $\omega_3$ is the set $X_3$ of all the coweights of length 12, it consists of $(\pm 2, \pm 1, \pm 1, 0)$ and all their permutations. The element $\kappa(\omega_3)$ is divisible by 2. For $\tau\in X_3$, $\kappa(\tau-\omega_1)$ is not divisible. For $\tau\in X_3$, $\kappa(\tau-\omega_4)$ may be divisible at most by 4. In this case our condition $4m\notin N\ZZ$ guarantees that $\bar\kappa(\lambda-\alpha_i)\notin N\check{\Lambda}$. For $\tau\in X_3$, $\kappa(\tau-\omega_2)$ may be divisible at most by 3. For $\tau\in X_3$, $\kappa(\tau-2\omega_1)$ is divisible at most by 2. For $\tau\in X_3$, $\kappa(\tau-\omega_1-\omega_4)$ is not divisible. For $\tau\in X_3$, $\kappa(\tau-\omega_3)$ may be divisible by 4 and by 6 (it is not divisible by 5 or by $r$ with $r\ge 7$). For example, if $\tau=(-1, -2, 1,0)$ then $\kappa(\tau-\omega_3)=6(-1, -1,0, 0)\in 6\check{\Lambda}$. Our condition $4m, 6m\notin N\ZZ$ guarantees that $\bar\kappa(\lambda-\alpha_i)\notin N\check{\Lambda}$. We are done. 
 
\begin{Rem} The notation from Bourbaki (\cite{B}, chapter 6, Section~4.9) for this root system are obtained from the above by passing to the opposite order in the linearly ordered set $\{1,2,3,4\}$.
\end{Rem} 

\subsection{} Assume $G$ is of type $E_8$. We follow the notations for the corresponding root system from Bourbaki (\cite{B}, chapter 6, Section~4.10). So, $\Lambda=\Lambda_1+\ZZ(\frac{1}{2}\sum_{i=1}^8 e_i)$, where $e_i$ is the canonical (orthonormal) base in $\ZZ^8$. Here $\Lambda_1=\{(a_1,\ldots, a_8)\in \ZZ^8\mid \sum a_i=0\mod 2\}$. The bilinear form $\kappa: \Lambda\otimes\Lambda\to\ZZ$ is induced from the scalar product on $\RR^8$, where $e_i$ is the orthonormal base. Then $\kappa: \Lambda\to \check{\Lambda}$ is an isomorphism. The element $w_0$ acts on $\Lambda$ as $-1$. The structure of $W$ is described in (\cite{B}, exercise 1, paragraph 4, p. 228). It contains all the permutations of $e_i$ and all the even number of sign changes (of the base elements). 
Our notations for $\omega_i$ and $\alpha_i$ is as in (\cite{B}, Section~4.10, p. 213). In particular, $\omega_8$ is the biggest coroot, so $\VV^{\omega_8}$ is the (quasi-minuscule) adjoint representation. We may assume $\bar\kappa=m\kappa$. The assumption of Conjecture~\ref{Con_main} reads $m\notin N\ZZ$. The condition $\bar\kappa(\lambda-\alpha_i)\in N\check{\Lambda}$ is equivalent to $m(\lambda-\alpha_i)\in N\Lambda$.

 We have the following inequalities 
$$
0\le \omega_8\le \omega_1\le\omega_7\le \omega_2\le\omega_6\le \omega_3\le \omega_5\le \omega_4
$$ 

 For $i=8$ we have $\omega_8=e_7+e_8$ and $\alpha_8=e_7-e_6$. So, $\omega_8-\alpha_8=e_6+e_8$, and $\omega_8-\lambda$ is either zero or a coroot. Taking $\omega_i-\lambda=-e_6-e_8$ we get $\lambda-\alpha_i=2(e_6+e_8)\in 2\Lambda$. So, we have to assume $2m\notin N\ZZ$ at least. Clearly, for $\omega_i-\lambda=\pm e_k\pm e_j$ with $k\ne j$ the element $\lambda-\alpha_i$ may be divisible at most by 2 in $\Lambda$. For $\omega_i-\lambda=\frac{1}{2}(a_1+\ldots+a_8)$ with $a_k=\pm 1$, $\sum_k a_k$ even, the element $\lambda-\alpha_8$ is not divisible. So, for $i=8$ we are done. 
 
% For $i=1$ we have $\dim W^{\omega_1}=3875$. If $\mu\in\Lambda^+$ and $\mu\le \omega_1$ then $\mu\in \{0,\omega_8,\omega_1\}$. Now we have to find the $W$-orbit of $\omega_1$. 
 
 In the case $i=4$ consider $\omega_4-\alpha_4=e_2+e_4+e_5+e_6+e_7+5e_8$. Its $W$-orbit contains the element $\omega_4-\lambda=e_2+e_4+e_5+e_6+e_7-5e_8$, for such $\lambda$ we get $\lambda-\alpha_4=10e_8$. So, we must assume $10m\notin N\ZZ$. 
 
 In the case $i=5$ we get $\omega_5-\alpha_5=e_3+e_5+e_6+e_7+4e_8$. The $W$-orbit of this element contains $\omega_5-\lambda=e_3+e_5+e_6+e_7-4e_8$. For this $\lambda$ we get $\lambda-\alpha_5=8e_8$. So, we must assume $8m\notin N\ZZ$. 
 
 In the case $i=6$ we get $\omega_6-\alpha_6=e_4+e_6+e_7+3e_8$. The $W$-orbit of this element contains $\omega_6-\lambda=e_4+e_6+e_7-3e_8$. For this $\lambda$ we get $\lambda-\alpha_6=6e_8$. So, we must assume $6m\notin N\ZZ$. The above assumptions are equivalent to the property that for a simple coroot $\alpha_i$, $\varrho(\alpha_i)\notin \frac{1}{10}\ZZ, \frac{1}{8}\ZZ, \frac{1}{6}\ZZ$. 

\medskip

\section{Proof of Proposition~\ref{Pp_reformulating_subtop_coh_property}}

\subsection{} We use the notations as in Section~\ref{subsection_CSh_basics} for the Casselman-Shalika problem. 

 Properties ii) and iii) are clearly equivalent. For $\eta\in\Lambda$ one has
$$
\Gr_B^0\cap \ov{\Gr}_{B^-}^{-\lambda}\,\iso\, \Gr_B^{\eta}\cap \ov{\Gr}_{B^-}^{\eta-\lambda}
$$ 
%Use Raskin's method. 
By (\cite{R}, Proposition~3.5.1), if $-\eta$ is deep enough in the dominant chamber then 
$$
\Gr_B^{\eta}\cap \ov{\Gr}_{B^-}^{\eta-\lambda}=\Gr_B^{\eta}\cap \ov{\Gr}_G^{w_0(\eta-\lambda)}
$$
Here we assume that for each $-\lambda\le \mu\le 0$ the coweight $\eta+\mu$ is anti-dominant, and $\eta-\lambda\in\Lambda^{\sharp}$. Consider the complex
\begin{equation}
\label{complex_CSh_limit}
\RG_c(\Gr_B^{\eta}\cap \ov{\Gr}_G^{w_0(\eta-\lambda)}, 
(s^{\eta}_B)^*\cA^{w_0(\eta-\lambda)}_{\cE}\otimes (\chi^{\eta}_{-\eta})^*\cL_{\psi})[\<\eta, 2\check{\rho}\>]
\end{equation}
This complex is what should be the limiting case of the metaplectic Casselman-Shalika formula (\ref{M_C-Sh_formula}) as in (\cite{R}, Section~3). As in (\cite{FGV}, Section~8.2.4), the tensor product of $\cF_{x, -\eta}$ by (\ref{complex_CSh_limit}) is isomorphic over $\wt\gM_{x, -\eta}$ to 
$j_{x, -\eta}^*\H^{\ra}_G(\cA^{\lambda-\eta}_{\cE}, \cF_{\emptyset})$. Recall that $\H^{\ra}_G(\cA^{\lambda-\eta}_{\cE}, \cF_{\emptyset})\,\iso\, \cF_{x, \lambda-\eta}$ by Theorem~\ref{Th_Hecke_action}. 

  The contribution of the open stratum $\Gr_B^{\eta}\cap \Gr_G^{w_0(\eta-\lambda)}$ to (\ref{complex_CSh_limit}) is
\begin{equation}
\label{complex_CSh_limit_open}
\RG_c(\Gr_B^{\eta}\cap \Gr_G^{w_0(\eta-\lambda)}, (s^{\eta}_B)^*\cA^{w_0(\eta-\lambda)}_{\cE}\otimes (\chi^{\eta}_{-\eta})^*\cL_{\psi})[\<\eta, 2\check{\rho}\>]
\end{equation}

\begin{Lm} The complex (\ref{complex_CSh_limit_open}) identifies with the complex (\ref{complex_main}) shifted to the left by $\<\lambda, 2\check{\rho}\>$.
\end{Lm} 
\begin{proof}
Recall the local system $\cW^{w_0(\eta-\lambda)}$ on $\wt\Gr_G^{w_0(\eta-\lambda)}$ defined in (\cite{L1}, Section~2.4.2). The perverse sheaf $\cA^{w_0(\eta-\lambda)}_{\cE}$ is the intermediate extension of this (shifted) local system. %Under the action of $T(\cO)$ the local system $(s^{\eta}_B)^*\cW^{w_0(\eta-\lambda)}$ over $\Gr_B^{\eta}\cap \Gr_G^{w_0(\eta-\lambda)}$ changes by what? 
The $\Gm$-torsor $\Gra_G\times_{\Gr_G} \Gr_B^{\eta}\to \Gr_B^{\eta}$ is constant with fibre $\Omega_x^{-\frac{\bar\kappa(\eta,\eta)}{2}}-0$, and $T(\cO)$ acts on it by the character $T(\cO)\to T\toup{-\bar\kappa(\eta)}\Gm$. So, the local system $(s^{\eta})^*\cW^{w_0(\eta-\lambda)}$ over $\Gr_B^{\eta}\cap \Gr_G^{w_0(\eta-\lambda)}$ changes under the action of $T(\cO)$ by the inverse image of $\cL_{\zeta}$ under $T(\cO)\to T\toup{-\bar\kappa(\eta)}\Gm$. Since $\bar\kappa(\eta-\lambda)\in N\check{\Lambda}$, it coincides with the inverse image of $\cL_{\zeta}$ under $T(\cO)\to T\toup{-\bar\kappa(\lambda)}\Gm$. 
Since the isomorphism $\Gr_B^0\cap \Gr_{B^-}^{-\lambda}\,\iso\, \Gr_B^{\eta}\cap \Gr_G^{w_0(\eta-\lambda)}, z\mapsto t^{\eta}z$ is $T(\cO)$-equivariant, we are done. 
\end{proof}
% maybe for the above Lemma to be true, one needs to replace \zeta by \zeta^{-1}.

\begin{Lm} For each $-\lambda<\mu\le 0$ the stratum $\Gr_B^{\eta}\cap \Gr_G^{w_0(\mu+\eta)}$ does not contribute to the cohomology group of (\ref{complex_CSh_limit}) in degrees $\ge -1$.
\end{Lm}
\begin{proof} The $*$-restriction $\cA^{w_0(\eta-\lambda)}$ to $\wt\Gr_G^{w_0(\mu+\eta)}$ is placed in perverse degrees $<0$, that is, in usual degrees $\le \<\mu+\eta, 2\check{\rho}\>-1$. Recall that $\dim \Gr_B^{\eta}\cap \Gr_G^{w_0(\mu+\eta)}=-\<\mu, \check{\rho}\>$. 

If $\mu\ne 0$ then, by (\cite{FGV}, Proposition~7.1.7), $(\chi^{\eta}_{-\eta})^*\cL_{\psi}$ is nonconstant on each irreducible component of $\Gr_B^{\eta}\cap \Gr_G^{w_0(\mu+\eta)}$. So, in this case 
\begin{equation}
\label{complex_three}
\RG_c(\Gr_B^{\eta}\cap \Gr_G^{w_0(\mu+\eta)}, (s^{\eta}_B)^*\cA_{\cE}^{w_0(\eta-\lambda)}\otimes (\chi^{\eta}_{-\eta})^*\cL_{\psi})[\<\eta, 2\check{\rho}\>]
\end{equation}
lives in degrees $\le -2$. 

 If $\mu=0$ then $\Gr_B^{\eta}\cap \Gr_G^{w_0(\eta)}$ is a point, the $*$-restriction of $(s^{\eta}_B)^*\cA^{w_0(\eta-\lambda)}$ to this point lives in degrees $\le \<\eta, 2\check{\rho}\>-1$. Besides, it lives only in usual degrees of the same parity as $\<\eta-\lambda, 2\check{\rho}\>$ by (\cite{L1}, Lemma~2.2). Since $\<\lambda, 2\check{\rho}\>\in 2\ZZ$, it is of the same parity as $\<\eta, 2\check{\rho}\>$. So, it lives in degrees $\le \<\eta, 2\check{\rho}\>-2$.
\end{proof}

 We conclude that the subtop cohomology property is equivalent to requiring that for any $\lambda>0$, which is not a simple coroot, (\ref{complex_CSh_limit}) is placed in degrees $\le -2$. Proposition~\ref{Pp_reformulating_subtop_coh_property} is proved. 
 
\bigskip\noindent
\select{Acknowledgements.} 
I am very grateful to M. Finkelberg and V. Lafforgue for constant support and numerous discussions on the subject. The reference to Kashiwara's result (\cite{K}, Proposition~8.2) was indicated to me by M. Finkelberg. I thank P. Baumann, J. Campbell, D. Gaitsgory, J.~Kamnitzer, G. Laumon, S. Raskin for answering my questions and useful comments. I also benefited a lot from discussions with W. T. Gan, M. Weissman and thank the University of Singapore, where this work was initiated, for hospitality. The author was supported by the ANR project ANR-13-BS01-0001-01. 

\printindex

\end{document}

\title{Erratum to:  Twisted Whittaker models for metaplectic groups,  GAFA 2017, vol. 27, Issue 2, 289-372}

\author{S. Lysenko}

\maketitle

\sssec{} In this paper we correct a gap in (\cite{L}, 4.7.1). We are deeply grateful to D. Gaitsgory, who has pointed out this mistake.

\sssec{} Use notations from \cite{L}. The following claim in (\cite{L}, Prop. 4.7.1) is not justified in \select{loc.cit.}:
$$
\add^*_{\mu_1,\mu_2,disj}('\bar\gp^*_B\IC_{\zeta})
$$ 
goes under (31) to $'\bar\gp^*_B\IC_{\zeta}\boxtimes ('\bar\gp_B)^*\IC_{\zeta}$ up to a shift. Indeed, ([33], Th. 4.1) only implies this after restrictions on a stratification. 

 We show below not that $'\bar\gp^*_B\IC_{\zeta}$ factorizes but rather that for any $n\ge 0, K\in\Whit^{\kappa}_n$ the complex $\bar F^{\mu}(K)$ factorizes as in (\cite{L}, 4.7.1) functorially on $K$. 
 
\sssec{} For $\mu\in -\Lambda^{pos}$ let $\cN^{reg,-}_{\mu}$, $\cN^{mer,-}_{\mu}$ be the group scheme (resp., group ind-scheme) over $X^{\mu}$, whose fibre over $D$ are automorphisms of $\Omega^{\rho}\times_T{B^-}$ over the formal neighbourhood (resp., the punctured formal neighbourhood) of $D$. 
Let $(\Bunb_{B^-})_{\goodat\, \mu}\subset X^{\mu}\times\Bunb_{B^-}$ be the open substack given by the property that $\cV^{\check{\lambda}}_{\cF}\to \cL^{\check{\lambda}}_{\cF_T}$ is surjective in a neighbourhood of $D\in X^{\mu}$. For a point of $(\Bunb_{B^-})_{\goodat\, \mu}$ we get a $B^-$-torsor $\cF_{B^-}$ over the formal neighbourhood $\bar D$ of $D$ with a trivialization $\epsilon_{B^-}: \cF_{B^-}\times_{B^-} T\,\iso\, \Omega^{\rho}$ over $\bar D$. We have a $\cN^{reg,-}_{\mu}$-torsor $_{\mu}\Bunb_{B^-}$ over $(\Bunb_{B^-})_{\goodat\, \mu}$ classifying a point of $(\Bunb_{B^-})_{\goodat\, \mu}$ and a trivialization $\cF_{B^-}\,\iso\, \Omega^{\rho}\times_T B^-$ compatible with $\epsilon_{B^-}$. 
Let 
$$
\act_{\mu}: \cN^{mer,-}_{\mu}\times^{\cN^{reg,-}_{\mu}} (_{\mu}\Bunb_{\tilde B^-})\to \Bunb_{\tilde B^-}
$$ 
be the corresponding action map. We have canonically for this map
$$
\act_{\mu}^*\IC_{\zeta}\,\iso\, \Qlb\boxtimes\IC_{\zeta}
$$
(as for $B$ in the proof of (\cite{L}, 4.3.1)). This implies the following.

\begin{Lm} 
\label{Lm_1}
Let $\mu_1\in -\Lambda^{pos}, \mu_2\in\Lambda, \mu=\mu_1+\mu_2$. There is a canonical isomorphism
$$
\add_{\mu_1,\mu_2,disj}^*('\bar\gp^*_B \IC_{\zeta})[\dimrel('\bar\gp_B)]
\,\iso\,
'\bar\gp^*_B \IC_{\zeta}[\dimrel('\bar\gp_B)]\boxtimes ('\bar\gp_B)^* \IC_{\zeta}[\dimrel('\bar\gp_B)]
$$
over 
$$
(X^{\mu_1}\times X^{\mu_2}_n)_{disj}\times_{(X^{\mu_1}\times X^{\mu_2}_n)} (\cZ^{\mu_1}\times \ov{\cZ}^{\mu_2}_n)
$$ 
\QED
\end{Lm}

 For this reason the proof of (\cite{L}, 4.9.1) holds without changes. Indeed, in that proof we appeal to the isomorphism of (\cite{L}, 4.7.1) only over the open part, where it holds by the above lemma.
 
\begin{Lm}
\label{Lm_104}
Let $\mu_1,\mu_2\in -\Lambda^{pos}, \mu=\mu_1+\mu_2$. There is a canonical isomorphism
$$
\add_{\mu_1,\mu_2,disj}^*('\bar\gp^*_B \IC_{\zeta})[\dimrel('\bar\gp_B)]
\,\iso\,
'\bar\gp^*_B \IC_{\zeta}[\dimrel('\bar\gp_B)]\boxtimes ('\bar\gp_B)^* \IC_{\zeta}[\dimrel('\bar\gp_B)]
$$
over 
$$
(X^{\mu_1}\times X^{\mu_2})_{disj}\times_{(X^{\mu_1}\times X^{\mu_2})} (\oo{\ov{\cZ}}{}^{\mu_1}\times \oo{\ov{\cZ}}{}^{\mu_2})
$$ 
\end{Lm}
\begin{proof} This is essentially explained in \cite{BFGM}. It follows from the factorization property and the fact that the map $\Bun_B\to \Bun_G$ is smooth over the components where $\<\check{\alpha}, \deg\cF_B\><2-2g$ for any positive root $\check{\alpha}$.
\end{proof}

\sssec{} Let $\mu_1\in -\Lambda^{pos}, \mu_2\in\Lambda, \mu=\mu_1+\mu_2$. Lemma~\ref{Lm_1} shows that the isomorphism of (\cite{L}, 4.7.1) holds over the open part
$$
(X^{\mu_1}\times X^{\mu_2}_n)_{disj}\times_{(X^{\mu_1}\times X^{\mu_2}_n)} (\cZ^{\mu_1}\times \ov{\cZ}^{\mu_2}_n)
$$ 
with the proof of \select{loc.cit.}

\sssec{} To correct (\cite{L}, 4.7.1) it suffices to show the following two claims. 

\begin{Pp} 
\label{Pp_1}
For any $K\in\Whit_n^{\kappa}, \mu\in\Lambda$ the complex $\bar F^{\mu}(K)\in \D_{\zeta}(\ov{\cZ}^{\mu}_n)$ is perverse.
\end{Pp}
\begin{Pp} 
\label{Pp_2}
Given $\mu_1\in -\Lambda^{pos}, \mu_2\in\Lambda$ with $\mu=\mu_1+\mu_2$, $K\in\Whit_n^{\kappa}$, the perverse sheaf 
$$
\add_{\mu_1,\mu_2,disj}^*\bar F^{\mu}(K)
$$
over $(X^{\mu_1}\times X^{\mu_2})_{disj}\times_{(X^{\mu_1}\times X^{\mu_2})} \ov{\cZ}^{\mu}_n$ is an intermediate extension from the open part
$$
(X^{\mu_1}\times X^{\mu_2})_{disj}\times_{(X^{\mu_1}\times X^{\mu_2})}
(\cZ^{\mu_1}\times \ov{\cZ}^{\mu_2}_n)
$$
\end{Pp}
(The formulation of Proposition~\ref{Pp_2} makes sense as soon as Proposition~\ref{Pp_1} is established). Indeed, we know already that (\cite{L}, 4.7.1) holds over this open part, and 
$\bar F^{\mu_1}(\cF_0)$ is an irreducible perverse sheaf on $\ov{\cZ}^{\mu}_1$. This yields the desired isomorphism of (\cite{L}, 4.7.1).

\begin{proof}[Proof of Proposition~\ref{Pp_1}] We may and do assume $K$ irreducible. In view of (\cite{L}, 4.9.1) it suffices to show that $\bar F^{\mu}(K)$ is placed in perverse degrees $\le 0$. Pick $\lambda_i\in\Lambda$, we may assume $K$ is the extension by zero from $\gM_{n,\le\bar\lambda}$. For any $\nu$ with $\nu\le \sum_{i=1}^n\lambda_i$ we have the map $\pi_{\gM}: {_{\nu}\gM_{n,\le \bar\lambda}}\to X^{\nu}_{n,\le \bar\lambda}$ from (\cite{L}, 2.6). According to (\cite{L}, Lemma~2.6.1), the $*$-resriction of $K$ to ${_{\nu}\gM_{n,\le \bar\lambda}}$ is of the form $\pi_{\gM}^*K'\otimes\ev^*\cL_{\psi}[\dimrel(\pi_{\gM})]$, where $K'$ is the extension by zero from $^+X^{\nu}_{n,\le\bar\lambda}$ and
is placed in perverse degrees $\le 0$. Here $^+X^{\nu}_{n,\le\bar\lambda}$ is the preimage of $^+X^{\nu}_n$ in $X^{\nu}_{n,\le\bar\lambda}$. The projection $^+X^{\nu}_{n,\le\bar\lambda}\to X^n$ is finite. 

 Denote by $b$ the composition $_{\nu}\gM_{n,\le \bar\lambda}\times_{\gM_n} \ov{\cZ}^{\mu}_n\to \ov{\cZ}^{\mu}_n\to \Bunb_{B^-}$. As in \cite{BFGM}, the complex $b^*\IC_{\zeta}[\dimrel(b)]$ is an irreducible perverse sheaf. Besides, it is ULA over $X^{\nu}_{n,\le\bar\lambda}$ (see \cite{L2}, 5.1.2). Now (\cite{BG}, 5.1.2, property 4) implies that $\bar F^{\mu}(K)$ is placed in perverse degrees $\le 0$. 
\end{proof} 
 
\begin{proof}[Proof of Proposition~\ref{Pp_2}] Since (\cite{L}, 4.9.1) holds without changes, it suffices to show that the $*$-restriction of $\add_{\mu_1,\mu_2,disj}^*\bar F^{\mu}(K)$ to 
$$
(X^{\mu_1}\times X^{\mu_2}_n)_{disj}\times_{(X^{\mu_1}\times X^{\mu_2}_n)} ((\ov{\cZ}^{\mu_1}-\cZ^{\mu_1})\times \ov{\cZ}^{\mu_2}_n)
$$ 
is placed in perverse degrees $<0$. 

 Pick $\bar\lambda$ such that $K$ is an extension by zero from $\gM_{n,\le\bar\lambda}$. Consider the stratification of $\gM_{n,\le \bar\lambda}$ by the substacks $_{\nu}\gM_{n,\le\bar\lambda}$ defined in (\cite{L}, 2.6). We know that 
$\add_{\mu_1,\mu_2,disj}^*\bar F^{\mu}(K)$
and
$$
\bar F^{\mu_1}(\cF_0)\boxtimes \bar F^{\mu_2}(K)
$$
over 
$$
(X^{\mu_1}\times X^{\mu_2}_n)_{disj}\times_{(X^{\mu_1}\times X^{\mu_2}_n)} (\ov{\cZ}^{\mu_1}\times \ov{\cZ}^{\mu_2}_n)
$$ 
become isomorphic after restriction to the stratum given by base change $_{\nu}\gM_{n,\le\bar\lambda}\times_{\gM_n} \ov{\cZ}^{\mu}_n\to \ov{\cZ}^{\mu}_n$ for any $\nu$. Since $\bar F^{\mu_1}(\cF_0)$ is an irreducible perverse sheaf, its $*$-restriction to $\ov{\cZ}^{\mu_1}-\cZ^{\mu_1}$ is placed in perverse degrees $<0$. Our claim follows.
\end{proof} 

 The proof of (\cite{L}, 4.7.1) is complete.